\documentclass[leqno,11pt]{amsart}

\usepackage[letterpaper,margin=1in]{geometry}

\usepackage[all]{xy} 
\usepackage{amsmath, amssymb, amsfonts, latexsym, mdwlist, amsthm, amscd,mathabx,braket}
\usepackage{subfig}
\usepackage{graphicx}
\usepackage{wrapfig}

\usepackage[bookmarks, colorlinks, breaklinks, pdftitle={},
pdfauthor={}]{hyperref}
\hypersetup{linkcolor=blue,citecolor=blue,filecolor=black,urlcolor=blue}

\usepackage{tikz}
\usetikzlibrary{calc,trees,positioning,arrows,chains,shapes.geometric,%
    decorations.pathreplacing,decorations.pathmorphing,shapes,%
    matrix,shapes.symbols}

\tikzset{
>=stealth',
  punktchain/.style={
    rectangle,
    rounded corners,
    draw=black, thick,
    minimum height=3em,
    text centered,
    on chain},
  line/.style={draw, thick, <-},
  element/.style={
    tape,
    top color=white,
    bottom color=blue!50!black!60!,
    minimum width=8em,
    draw=blue!40!black!90, very thick,
    text width=10em,
    minimum height=3.5em,
    text centered,
    on chain},
  every join/.style={->, thick,shorten >=1pt},
  decoration={brace},
  tuborg/.style={decorate},
  tubnode/.style={midway, right=2pt},
}

\usepackage{paralist}
\usepackage{enumitem} 
\setdefaultenum{(1)}{(a)}{(i)}{}
\setlist[enumerate,1]{label={\upshape(\arabic*)}}
\setlist[enumerate,2]{label={\upshape(\alph*)},ref=\theenumi\upshape(\alph*)}
\setlist[enumerate,3]{label={\upshape(\roman*)},ref=\theenumi\theenumii\upshape(\roman*)}
\usepackage[capitalize]{cleveref}
\crefname{Prop}{Proposition}{Propositions}
\crefname{Thm}{Theorem}{Theorems}
\crefname{Lem}{Lemma}{Lemmas}
\crefname{enumi}{Case}{Cases}

\def\C{\ensuremath{\mathbb{C}}}

\def\H{\ensuremath{\mathbb{H}}}
\def\N{\ensuremath{\mathbb{N}}}
\def\P{\ensuremath{\mathbb{P}}}
\def\Q{\ensuremath{\mathbb{Q}}}
\def\R{\ensuremath{\mathbb{R}}}
\def\T{\ensuremath{\mathbb{T}}}
\def\Z{\ensuremath{\mathbb{Z}}}

\def\alg{\mathrm{alg}}
\def\Amp{\mathrm{Amp}}
\def\Aut{\mathop{\mathrm{Aut}}\nolimits}

\def\Br{\mathop{\mathrm{Br}}\nolimits}
\def\ch{\mathop{\mathrm{ch}}\nolimits}

\def\Cl{\mathop{\mathrm{Cl}}}
\def\Coh{\mathop{\mathrm{Coh}}\nolimits}
\def\codim{\mathop{\mathrm{codim}}\nolimits}
\def\Cone{\mathop{\mathrm{cone}}}

\def\deg{\mathop{\mathrm{deg}}}

\def\dim{\mathop{\mathrm{dim}}\nolimits}

\def\Ext{\mathop{\mathrm{Ext}}\nolimits}
\def\ext{\mathop{\mathrm{ext}}\nolimits}
\def\lExt{\mathop{\mathcal Ext}\nolimits} 

\def\GL{\mathop{\mathrm{GL}}\nolimits}
\def\Hal{H^*_{\alg}}
\def\Hilb{\mathop{\mathrm{Hilb}}\nolimits}

\def\Hom{\mathop{\mathrm{Hom}}\nolimits}
\def\lHom{\mathop{\mathcal Hom}\nolimits}
\def\RlHom{\mathop{\mathbf{R}\mathcal Hom}\nolimits}
\def\RHom{\mathop{\mathbf{R}\mathrm{Hom}}\nolimits}
\def\id{\mathop{\mathrm{id}}\nolimits}
\def\Id{\mathop{\mathrm{Id}}\nolimits}
\def\im{\mathop{\mathrm{im}}\nolimits}
\def\coker{\mathop{\mathrm{coker}}\nolimits}

\def\Isom{\mathop{\mathrm{Isom}}\nolimits}

\def\Mov{\mathop{\mathrm{Mov}}\nolimits}
\def\mod{\mathop{\mathrm{mod}}\nolimits}
\def\min{\mathop{\mathrm{min}}\nolimits}

\def\Nef{\mathrm{Nef}}

\def\Num{\mathop{\mathrm{Num}}\nolimits}
\def\NS{\mathop{\mathrm{NS}}\nolimits}

\def\Pic{\mathop{\mathrm{Pic}}\nolimits}
\def\PGL{\mathop{\mathrm{PGL}}\nolimits}

\def\Quot{\mathop{\mathrm{Quot}}\nolimits}
\def\rk{\mathop{\mathrm{rk}}}

\def\Sing{\mathop{\mathrm{Sing}}}

\def\Sym{\mathop{\mathrm{Sym}}\nolimits}

\def\td{\mathop{\mathrm{td}}\nolimits}

\def\Cone{\mathop{\mathrm{Cone}}\nolimits}
\def\chr{\mathop{\mathrm{char}}\nolimits}

\def\MG13{\ensuremath{{\mathcal M}_{\Gamma_1(3)}}}
\def\tildeMG13{\ensuremath{\widetilde{\mathcal M}_{\Gamma_1(3)}}}
\def\Stab{\mathop{\mathrm{Stab}}\nolimits}
\def\Stabd{\mathop{\Stab^{\dagger}}\nolimits}
\def\into{\ensuremath{\hookrightarrow}}
\def\onto{\ensuremath{\twoheadrightarrow}}

\def\blank{\underline{\hphantom{A}}}

\def\Db{\mathrm{D}^{\mathrm{b}}}

\newcommand\TFILTB[3]{%
\xymatrix@=1pc{
{0 = {#1}_0} \ar[rr]&&
{{#1}_1} \ar[rr]\ar[ld] &&
{{#1}_2} \ar[r]\ar[ld] &
{\cdots} \ar[r] & { {#1}_{#3-1}} \ar[rr] &&
{{#1}_{#3} = {#1}} \ar[ld]
\\
& *{{#2}_1} \ar@{.>}[ul] &&
{{#2}_2} \ar@{.>}[ul] & &&&
{{#2}_{{#3}}} \ar@{.>}[ul]
}}

\makeatletter
\newtheorem*{rep@theorem}{\rep@title}
\newcommand{\newreptheorem}[2]{%
\newenvironment{rep#1}[1]{%
 \def\rep@title{#2 \ref{##1}}%
 \begin{rep@theorem}}%
 {\end{rep@theorem}}}
\makeatother

\newtheorem{Thm}{Theorem}[section]
\newreptheorem{Thm}{Theorem}
\newtheorem{Prop}[Thm]{Proposition}
\newtheorem{PropDef}[Thm]{Proposition and Definition}
\newtheorem{Lem}[Thm]{Lemma}

\newtheorem{Cor}[Thm]{Corollary}
\newreptheorem{Cor}{Corollary}

\newreptheorem{Con}{Conjecture}

\newtheorem{thm-int}{Theorem}

\theoremstyle{definition}
\newtheorem{Def-s}[Thm]{Definition}
\newtheorem{Def}[Thm]{Definition}
\newtheorem{Rem}[Thm]{Remark}

\newtheorem{Ex}[Thm]{Example}

\def\C{\ensuremath{\mathbb{C}}}

\def\H{\ensuremath{\mathbb{H}}}
\def\N{\ensuremath{\mathbb{N}}}
\def\P{\ensuremath{\mathbb{P}}}
\def\Q{\ensuremath{\mathbb{Q}}}
\def\R{\ensuremath{\mathbb{R}}}
\def\Z{\ensuremath{\mathbb{Z}}}

\def\AA{\ensuremath{\mathcal A}}
\def\BB{\ensuremath{\mathcal B}}
\def\CC{\ensuremath{\mathcal C}}
\def\DD{\ensuremath{\mathcal D}}
\def\EE{\ensuremath{\mathcal E}}
\def\FF{\ensuremath{\mathcal F}}
\def\GG{\ensuremath{\mathcal G}}
\def\HH{\ensuremath{\mathcal H}}

\def\MM{\ensuremath{\mathcal M}}

\def\OO{\ensuremath{\mathcal O}}
\def\PP{\ensuremath{\mathcal P}}

\def\QQ{\ensuremath{\mathcal Q}}
\def\TT{\ensuremath{\mathcal T}}
\def\VV{\ensuremath{\mathcal V}}
\def\WW{\ensuremath{\mathcal W}}

\def\ZZ{\ensuremath{\mathcal Z}}

\newcommand{\mor}[1][]{\xrightarrow{#1}}
\newcommand{\isomor}{\mor[\sim]}

\def\X{\ensuremath{\widetilde{X}}}
\def\a{\mathbf{a}}
\def\b{\mathbf{b}}
\def\d{\mathbf{d}}
\def\u{\mathbf{u}}
\def\v{\mathbf{v}}
\def\w{\mathbf{w}}
\def\z{\mathbf{z}}

\def\cal{\mathcal}
\def\Bbb{\mathbb}

\newcommand{\todo}[1]{\vspace{5 mm}\par \noindent
\marginpar{\textsc{ToDo}}
\framebox{\begin{minipage}[c]{0.95 \textwidth}
\tt \textcolor{red}{#1} \end{minipage}}\vspace{5 mm}\par}

\newcommand{\ignore}[1]{}

\begin{document}
\author{Howard Nuer}
\address{H.N.: Department of Mathematics, Statistics, and Computer Science,
University of Illinois at Chicago,
851 S. Morgan Street
Chicago, IL 60607}
\email{hjnuer@gmail.com}

\author{K\={o}ta Yoshioka}
\address{K.Y.: Department of Mathematics, Faculty of Science, Kobe University, Kobe, 657, Japan}
\email{yoshioka@math.kobe-u.ac.jp}
\title[MMP via wall-crossing for moduli sheaves on Enriques surfaces]{MMP via wall-crossing for moduli spaces of stables sheaves on an Enriques surface}
\begin{abstract}
    We use wall-crossing in the Bridgeland stability manifold to systematically study the birational geometry of the moduli space $M_\sigma(\v)$ of $\sigma$-semistable objects of class $\v$ for a generic stability condition $\sigma$ on an arbitrary Enriques surface $X$.  In particular, we show that for any other generic stability condition $\tau$, the two moduli spaces $M_\tau(\v)$ and $M_\sigma(\v)$ are birational.  As a consequence, we show that for primitive $\v$ of odd rank $M_\sigma(\v)$ is birational to a Hilbert scheme of points.  Similarly, in even rank we show that  $M_\sigma(\v)$ is birational to a moduli space of torsion sheaves supported on a hyperelliptic curve when $\ell(\v)=1$.  As an added bonus of our work, we prove that the Donaldson-Mukai map $\theta_{\v,\sigma}:\v^\perp\to\Pic(M_\sigma(\v))$ is an isomorphism for these classes.  Finally, we use our classification to fully describe the geometry of the only two examples of moduli of stable sheaves on $X$ that are uniruled (and thus not K-trivial).
\end{abstract}
\maketitle
\setcounter{tocdepth}{1}
\tableofcontents
\section{Introduction}
For almost forty years, moduli spaces of stable sheaves have attracted great interest from mathematicians and physicists alike.  They have been studied using vastly different mathematical disciplines, but more recently the development and application of Bridgeland stability conditions to their study has unified most of these tools.  Introduced by Bridgeland \cite{Bri07} to formulate a rigorous definition of Douglas's $\pi$-stability for branes in string theory, stability conditions on the derived category $\Db(X)$ of a smooth projective variety $X$ provide an adequately robust arena in which to study moduli spaces of sheaves using tools such as Fourier-Mukai transforms, enumerative and motivic invariants, and the minimal model program.  In particular, as stability conditions move around in a complex manifold $\Stab(X)$ which admits a wall-and-chamber for any given Chern character, there is a strong connection between wall-crossing in $\Stab(X)$ on the one hand, and wall-crossing formulae for enumerative invariants and birational transformations on moduli spaces on the other.

In this paper, we bring this toolbox to bear on the study of moduli spaces of sheaves on an Enriques surface $X$.  In previous work \cite{Nue14a,Nue14b,Yos03,Yos14,Yos16b,Yos16a}, we had shown that for a generic stability condition $\sigma$ in a certain distinguished connected component $\Stabd(X)\subset\Stab(X)$ there exist projective coarse moduli spaces $M_\sigma(\v)$ parametrizing (S-equivalence classes) of $\sigma$-semistable objects of Mukai vector $\v$ (see \cref{sec:ReviewStabilityK3Enriques} for definitions).  We also classified precisely for which Mukai vectors $\v$ (or equivalently Chern character) the moduli space $M_\sigma(\v)$ is nonempty, and we studied some coarse geometric and topological invariants of $M_\sigma(\v)$ (refer again to \cref{sec:ReviewStabilityK3Enriques} for a brief recap of these results).  In particular, for the Mukai vectors (or Chern characters) of stable sheaves, we described some of these invariants for the moduli space of stable sheaves using a combination of modern and classical techniques in conjunction with Bridgeland stability.

We continue our investigation of these moduli spaces in this paper with two more specific goals in mind.  The first is to intimately study the effect on moduli spaces of crossing a wall $\WW$.  More specifically, given a Mukai vector $\v$, a wall $\WW\subset\Stabd(X)$ for $\v$, and two stability conditions $\sigma_\pm$ in the opposite and adjacent chambers separated by $\WW$, we seek a precise answer to the question: how are $M_{\sigma_+}(\v)$ and $M_{\sigma_-}(\v)$ related?  The basic answer is given by \cref{Thm:MainTheorem1} which says that $M_{\sigma_+}(\v)$ and $M_{\sigma_-}(\v)$ are birational.  This question, and others associated with it, are motivated by a larger trend in moduli theory, wherein minimal models of a given moduli space are shown to be moduli spaces in and of themselves, just of slightly different objects.  In the case of K3 surfaces, such a result, stating that all minimal models of $M_\sigma(\v)$ are isomorphic to some $M_\tau(\v)$ for a different $\tau\in\Stabd(X)$, was shown to be true by Bayer and Macr\`{i} in \cite{BM14b}.  While we had hoped to prove such a result in the case of Enriques surfaces, our investigation instead led to a possible counterexample, see \cref{Ex:ConfusingSmallContraction}.  

Our second goal is to use wall-crossing, Fourier-Mukai transforms, and \cref{Thm:MainTheorem1} to pin-point precisely where the moduli spaces $M_\sigma(\v)$ live in the classification of algebraic varieties.  We accomplish this goal with success for almost all Mukai vectors in \cref{Thm:application1,Thm:application2}.  It has become more apparent with these results and similar results for other surfaces that Bridgeland stability conditions are crucial tool not only for studying the birational geometry of moduli spaces, but also more intrinsic and more classical questions about the geometry of moduli spaces.

\subsection*{Summary of Results and Techniques}
Let us turn now to stating our main results and the main tool we use to prove them.  While we briefly introduce notation, the reader is invited to see \cref{sec:ReviewStabilityK3Enriques} for more details.

For an Enriques surface $X$, denote by $\varpi:\widetilde{X}\to X$ the two-to-one K3 universal covering.  The topological invariants of a coherent sheaf or object $E$ in the derived category $\Db(X)$ are encoded in its Mukai vector $$\v(E):=\ch(E)\sqrt{\td(X)}\in H^*(X,\Q).$$  We consider the Mukai lattice $$\Hal(X,\Z):=\v(K(X)),$$ where $K(X)$ is the Grothendieck group, along with the induced pairing $\langle\v(E),\v(F)\rangle=-\chi(E,F)$.  For any primitive $\v\in\Hal(X,\Z)$, if we write $\varpi^*\v=\ell(\v)\w\in\Hal(\widetilde{X},\Z)$ with $\w$ primitive, then $\ell(\v)=1$ or 2 (see \cref{primitive}).  The stability conditions that we consider are all contained in the distinguished connected component $\Stabd(X)$ of $\Stab(X)$ containing those stability conditions $\sigma$ such that skyscraper sheaves of points are $\sigma$-stable.  For $\sigma\in\Stabd(X)$, we denote by $M_\sigma(\v)$ the moduli space of (S-equivalence classes of) $\sigma$-semistable objects $E\in\Db(X)$ with $\v(E)=\v$.  These moduli spaces admit a decomposition $$M_\sigma(\v)=M_{\sigma}(\v,L)\bigsqcup M_\sigma(\v,L+K_X),$$ where $L\in\Pic(X)$ satisfies $c_1(L)=c_1(\v)$ in $H^2(X,\Q)$ and $M_\sigma(\v,L)$ parametrizes $E\in M_\sigma(\v)$ such that $\det(E)=L$.

Our first main result states that any two Bridgeland moduli spaces are birational.  
\begin{Thm}\label{Thm:MainTheorem1}
Let $\v\in\Hal(X,\Z)$ satisfy $\v^2>0$, and let $\sigma,\tau\in\Stabd(X)$ be generic stability conditions with respect to $\v$ (that is, they are not contained on any wall for $\v$).
\begin{enumerate}
\item\label{enum:MT1-two moduli are birational} The two moduli spaces $M_{\sigma}(\v)$ and $M_{\tau}(\v)$ of Bridgeland-semistable objects are birational to each other.
\item\label{enum:MT1-birational map given by FM transform} More precisely, there is a birational map induced by a derived (anti-)autoequivalence $\Phi$ of $\Db(X)$ in the following sense: there exists a common open subset $U\subset M_{\sigma}(\v),U\subset M_{\tau}(\v)$ such that for any $u\in U$, the corresponding objects $E_u\in M_{\sigma}(\v)$ and $F_u\in M_{\tau}(\v)$ satisfy $F_u=\Phi(E_u)$.  If $\Pic(\widetilde{X})=\varpi^*\Pic(X)$, then the complements of $U$ in $M_{\sigma}(\v)$ and $M_{\tau}(\v)$ have codimension at least two. 
\end{enumerate}
\end{Thm}
We remark that the generic Enriques surface $X$ satisfies $\varpi^*\Pic(X)=\Pic(\widetilde{X})$.  The proof of the analogous statement to \cref{Thm:MainTheorem1} for K3 surfaces \cite[Theorem 1.1]{BM14b} relies heavily on the fact that $M_\sigma(\v)$ is a projective hyperk\"{a}hler manifold in the K3 case.  Our proof is based on studying stacks of Harder-Narasimhan filtrations and is thus more universal.  In fact, the first author has already used this tool to obtain analogous results for bielliptic surfaces \cite{Nue18}.  

Let $Y$ be a smooth projective variety, and suppose we are given a stability condition $\sigma\in\Stabd(Y)$ on a wall $\WW$ for $\v$ and a sequence of Mukai vectors (or Chern characters) $\v_1,\dots,\v_s$ of descending slope with respect to a nearby generic stability condition $\sigma_-$.  Then under some mild assumptions on $Y$, we prove in \cref{Thm:DimensionOfHNFiltrationStack} that the substack of $\sigma$-semistable objects whose Harder-Narasimhan filtration with respect to $\sigma_-$ has $i$-th $\sigma_-$-semistable factor of class $\v_i$ has dimension $$\sum_{i=1}^s\dim\MM_{\sigma_-}(\v_i)+\sum_{i<j}\langle\v_i,\v_j\rangle.$$  See \cref{sec:DimensionsOfHarderNarasimhan} for more details.  In particular, in the case of an Enriques surface, \cref{Thm:DimensionOfHNFiltrationStack} allows us to determine what objects in $M_{\sigma_+}(\v,L)$ are destabilized when crossing the wall $\WW$, and most importantly the dimension of this locus, purely in terms of a hyperpolic lattice $\HH_\WW$ associated with the wall.  This is the content of \cref{classification of walls} which gives a much more refined and detailed classification of the type of birational map in \cref{Thm:MainTheorem1} in terms of the lattices $\HH_\WW$ associated to the walls $\WW$ that are crossed on along a path from $\sigma$ to $\tau$.

One of the most powerful uses of a result like \cref{Thm:MainTheorem1} appears in the next three results.  They all follow a similar pattern.  To study a given moduli space $M_\sigma(\v,L)$, we apply a Fourier-Mukai transform $\Phi:\Db(X)\to\Db(X)$ inducing an isomorphism $M_\sigma(\v,L)\isomor M_{\Phi(\sigma)}(\Phi_*(\v),L')$, where $\Phi_*(\v)$ has a more familiar form, say that of a well-understand type of coherent sheaf.  Using \cref{Thm:MainTheorem1}, we know that there is a birational map $M_{\Phi(\sigma)}(\Phi_*(\v),L')\dashrightarrow M_\tau(\Phi_*(\v),L')$, where $\tau$ is now in the so-called Gieseker chamber so that $M_\tau(\Phi_*(\v),L')$ is simply the moduli space of stable sheaves of Mukai vector $\Phi_*(\v)$ and determinant $L'$.  An argument exactly along these lines gives the next two results:   

\begin{Thm}\label{Thm:application1}
Let $\v$ be a primitive Mukai vector such that $\v^2>0$ is odd.
Then for a general $\sigma$,
there is an (anti-)autoequivalence $\Phi$ of $\Db(X)$
which induces an isomorphism $\Phi:U \to U'$
where $U \subset M_\sigma(\v,L)$ and 
$U' \subset \Hilb^{\frac{\v^2+1}{2}}(X)$ are dense open subsets.  
In particular, $M_\sigma(\v,L)$ is birationally equivalent to
$\Hilb^{\frac{\v^2+1}{2}}(X)$, 
$\pi_1(M_\sigma(\v,L)) \cong \Z/2 \Z$,
$K_{M_\sigma(\v,L)} \not \cong  \OO_{M_\sigma(\v,L)}$ and 
$K_{M_\sigma(\v,L)}^{\otimes 2} \cong  \OO_{M_\sigma(\v,L)}$
for $\v^2 \geq 1$. 
Moreover, if $\Pic(\widetilde{X})=\varpi^*\Pic(X)$, then the complements of $U,U'$ have codimension at least two.
\end{Thm}

\begin{Thm}\label{Thm:application2}
Let $\v$ be a primitive Mukai vector such that $\v^2$ is even
 and $\ell(\v)=1$.
Then there is an elliptic fibration $\pi:X \to \P^1$, a curve $C\subset X$,
and a Mukai vector $\w=(0,C,\chi)$ with $\chi\neq0$
such that 
(1) $\pi|_C:C \to \P^1$ is a double cover, and 
(2) there is an (anti-)autoequivalence $\Phi$ of $\Db(X)$
which induces an isomorphism $\Phi:U \to U'$
where $U \subset M_\sigma(\v,L)$ and 
$U' \subset M_H(\w,L')$ are open subsets.  In particular,
$M_\sigma(\v,L)$ is 
birationally equivalent to
$M_H(\w,L')$
for a general $\sigma$. 
If $\Pic(\widetilde{X})=\varpi^*\Pic(X)$ and $\v^2 \geq 2$, then the complements of $U$ and $U'$ are of codimension at least two, $\pi_1(M_\sigma(\v,L)) \cong \Z/2 \Z$,
$K_{M_\sigma(\v,L)} \cong \OO_{M_\sigma(\v,L)}$,
and $h^{p,0}(M_\sigma(\v,L))=0$ for $p \ne 0, \v^2+1$.
\end{Thm}
The remaining case, when $\v^2$ is even and $\ell(\v)=2$, can still be attacked in the same way, but the corresponding moduli space of torsion sheaves is not well understood due to the presence of non-reduced curves of larger multiplicity.

Our third use of \cref{Thm:MainTheorem1} goes in a slightly different direction.  To prove that all minimal models of $M_\sigma(\v,L)$ come from Bridgeland wall-crossing, an important step, arguably a necessary one, is to show that the Donaldson-Mukai map is surjective.  Set $K(X)_\v:=\Set{x \in K(X) \ |\ \langle \v(x),\v \rangle=0}$.
For a universal family $\EE$ on $M_{\sigma}(\v,L) \times X$,
the Donaldson-Mukai map is defined by
\begin{equation}
\begin{matrix}
\theta_{\v,\sigma}:& K(X)_\v & \to & \Pic(M_{\sigma}(\v,L))\\
& x &\mapsto & \det (p_{M_\sigma(\v,L)!}(\EE \otimes p_X^*(x^{\vee}))).
\end{matrix} 
\end{equation}
Then using \cref{Thm:MainTheorem1}, we obtain the following result.
\begin{Cor}\label{Cor:Picard}
Let $\v$ be a primitive Mukai vector with $\v^2 \geq 3$ and $\sigma\in\Stabd(X)$ a generic stability condition with respect to $\v$.  If either
\begin{enumerate}
\item
$\v^2$ is odd, or
\item
$\v^2$ is even, $\ell(\v)=1$, and $\Pic(\widetilde{X})=\varpi^*\Pic(X)$,
\end{enumerate}
then $\theta_{\v, \sigma}$ is an isomorphism.
\end{Cor}
In an appendix, we use our results to describe the geometry of the two examples of moduli spaces of sheaves on an Enriques surface $X$ that are uniruled.  These constitute the only counterexamples to the fact that  moduli spaces of stable sheaves on an Enriques surface have numerically trivial canonical divisors.  One of these examples only exists on nodal Enriques surfaces, while the other is ubiquitous, but only applies for Mukai vectors of the form $\v=2\v_0$ where $\v_0^2=1$.  In each case we show that there is always a wall $\WW\subset\Stabd(X)$ that induces the structure of a $\P^1$-fibration on $M_\sigma(\v)$.  Correspondingly, these are the only counterexamples to the conjecture of the first author in \cite{Nue14a} that the Bayer-Macr\`{i} divisor at the wall is big.
\subsection*{Relation to other work}  The relation between wall-crossing and the minimal model program studied here has been explored for many other surfaces.  We refer the reader to \cite{BM14a,BM14b,MYY14b,MYY14} for K3 surfaces, to \cite{ABCH13,BMW14,CH15,CHW14,LZ13} for $\P^2$, to \cite{BC13,AM17} for Hirzebruch and del Pezzo surfaces, and to \cite{MM13,YY14,Yos12} for abelian surfaces.

During the writing phase of this project, the authors became aware of the recent preprint \cite{Bec18}.  In that paper, the author proves part \ref{enum:MT1-two moduli are birational} of \cref{Thm:MainTheorem1} as well as \cref{Thm:application1,Thm:application2} for generic Enriques surfaces (that is, when $\Pic(\widetilde{X})=\varpi^*\Pic(X)$) and for Mukai vectors $\v$ such that $\varpi^*\v$ is primitive (see Theorem 4.4, Proposition 4.6, and Proposition 4.7, respectively, in \cite{Bec18} for the precise results).  The approach is entirely different from ours and is again based on hyperk\"{a}hler geometry, specifically the concept of a constant cycle subvariety of a hyperk\"{a}hler manifold.  Nevertheless, our approach has numerous advantages.  The first is that we do not assume in \cref{Thm:MainTheorem1} that $\v$ is primitive or that $X$ is generic.\footnote{In the case of \cite[Proposition 4.6]{Bec18}, the author actually does not need to assume that $X$ is generic, but instead uses a clever deformation argument with relative moduli spaces of stable sheaves to prove his analogue of \cref{Thm:application1} only for moduli of Gieseker stable sheaves.  This is another benefit of our method; although, as the author points out, one can obtain the same generality as in \cref{Thm:application1} using the construction of relative moduli spaces of Bridgeland stable objects in the forthcoming article \cite{BLMNPS19}}
  The second advantage, which is a consequence of the first, is that we obtain (in \cref{classification of walls}) a complete classification of the birational behavior induced by any wall $\WW$ without any assumptions on $X$ or the divisibility of $\v$.  Even when $\v$ is primitive, our approach has the added advantage of giving necessary and sufficient conditions for a potential numerical wall $\WW$ to actually be a wall, which is not guaranteed by just using the corresponding condition on the covering K3 surface $\widetilde{X}$.
\subsection*{Acknowledgements} 
The authors would like to thank Arend Bayer, Izzet Coskun, Emanuele Macr\`{i}, Benjamin Schmidt, and Xiaolei Zhao for very helpful discussions related to this article.

This project originally began while the authors were attending the workshop, ``Derived categories and Chow groups of Hyperk\"{a}ehler and Calabi-Yau varieties," at the Simons Center in fall of 2016, and we would like to thank the center for its warm hospitality and stimulating environment.  The collaboration was aided by the second author's visits to the University of Illinois at Chicago, and we would like to thank the mathematics department for providing a productive working atmosphere.  H.N. was partially supported by the NSF postdoctoral fellowship DMS-1606283, by the NSF RTG grant DMS-1246844, and by the NSF FRG grant DMS-1664215.
K.Y. was partially supported by the Grant-in-aid for 
Scientific Research (No.\ 18H01113, \ 17H06127, \ 26287007), JSPS.

\subsection*{Notation}
For a complex number $z\in\C$ we denote its real and imaginary parts by $\Re z$ and $\Im z$, respectively.

We will denote by $\Db(X)$ the bounded derived category of coherent sheaves on a smooth projective variety $X$.  On occasion we will consider the bounded derived category $\Db(X,\alpha)$ of the abelian category $\Coh(X,\alpha)$ of $\alpha$-twisted coherent sheaves, where $\alpha\in\Br(X)$ is a Brauer class.  See \cite[pp. 515-516]{BM14b} and the references contained therein for more background on twisted sheaves.

We will use non-script letters ($E,F,G,\dots$) for objects on a fixed scheme and reserve curly letters ($\EE,\FF,\GG,\dots$) for families of such objects.

For a vector $\v$ in a lattice $\HH$ with pairing $\langle\blank,\blank\rangle$, we abuse notation and write $$\v^2:=\langle\v,\v\rangle.$$
For a given lattice $\HH$ and integer $k\in\Z$, we denote by $\HH(k)$ the same underlying lattice with pairing multiplied by the integer $k$.

The intersection pairing on a smooth surface $X$ will be denoted by $(\blank,\blank)$ and the self-intersection of a divisor $D$ by $(D^2)$.  The fundamental class of a smooth projective variety $X$ will be denoted by $\varrho_X$.

By an \emph{irreducible} object of an abelian category, we mean an object that has no non-trivial subobjects.  These are sometimes called \emph{simple} objects in the literature.

Recall that an object $S$ in the derived category of a K3 or Enriques surface is called \emph{spherical} if $\RHom(S,S)=\C\oplus\C[-2]$.  We denote the associated spherical reflection by $R_S$; it is defined by 
\begin{equation}\label{eqn:spherical reflection}
R_S(E):=\Cone{(\RHom(S,E)\otimes S\to E).}
\end{equation}
Similarly, an object $E_0$ in the derived category of an Enriques surface is called \emph{exceptional} if $\RHom(E_0,E_0)=\C$.  In analogy to the spherical case, we denote the associated exceptional , or weakly spherical, reflection by $R_{E_0}$; it is defined by 
\begin{equation}\label{eqn:weakly spherical reflection}
R_{E_0}(E):=\Cone{(\RHom(E_0,E)\otimes E_0\oplus\RHom(E_0(K_X),E)\otimes E_0(K_X)\to E)}.
\end{equation}
This has also been called the Fourier-Mukai transform associated to $(-1)$-reflection in the literature.
\section{Review: Bridgeland stability conditions}

In this section, we summarize the notion of Bridgeland stability conditions on a triangulated category $\DD$.  The main reference is \cite{Bri07}.

\begin{Def}\label{def:slicing}
A slicing $\PP$ of the category $\DD$ is a collection of full extension-closed subcategories $\PP(\phi)$ for $\phi\in\R$ with the following properties :
\begin{enumerate}
    \item $\PP(\phi+1)=\PP(\phi)[1]$.
    \item If $\phi_1>\phi_2$, then $\Hom(\PP(\phi_1),\PP(\phi_2))=0$.
    \item For any $E\in\DD$, there exists a collection of real numbers $\phi_1>\phi_2>\cdots>\phi_n$ and a sequence of triangles
\begin{equation} \label{eq:HN-filt}
 \TFILTB E A n
\end{equation}
with $A_i \in \PP(\phi_i)$.
\end{enumerate}
\end{Def}

The subcategory $\PP(\phi)$ is abelian; its nonzero objects are called semistable of phase $\phi$, and its simple objects are called stable.  Appropriately, the collection of triangles in \eqref{eq:HN-filt} is called the \emph{Harder-Narasimhan (HN) filtration} of $E$, and we define $\phi_{\max}(E):=\phi_1$ and $\phi_{\min}(E):=\phi_n$.  For any $\phi\in\R$, we denote by $\PP(\phi-1,\phi]$ the full subcategory of objects with $\phi_{\min}>\phi-1$ and $\phi_{\max}\leq\phi$.  This is the heart of a bounded t-structure.  We usually consider $\AA=\PP(0,1]$.


Let us fix a lattice of finite rank $\Lambda$ and a surjective map $\v:K(\DD)\onto\Lambda$.

\begin{Def}\label{def:stability condition}
A Bridgeland stability condition on $\DD$ is a pair $\sigma=(Z,\PP)$, where 
\begin{itemize}
		\item the \emph{central charge} $Z:\Lambda\to \C$ is a group homomorphism, and 
		\item $\PP$ is a slicing of $Z$,
\end{itemize}satisfying the following compatibility conditions:
\begin{enumerate}
		\item For all non-zero $E\in\PP(\phi)$, $\frac{1}{\pi}\arg Z(\v(E))=\phi$;
		\item For a fixed norm $\lvert\blank\rvert$ on $\Lambda_\R$, there exists a constant $C>0$ such that 
		\begin{equation*}|Z(\v(E))|\geq C\lvert\v(E)\rvert
		\end{equation*}
		for all semistable $E$.
\end{enumerate}
\end{Def}

We will write $Z(E)$ for $Z(\v(E))$ from here on.  Furthermore, when we wish to refer to the central charge, the heart, or the slicing of a stability condition $\sigma$, we will denote it by $Z_{\sigma}$, $\AA_{\sigma}$, and $\PP_{\sigma}$, respectively.  It is worth noting that giving a stability condition $(Z,\PP)$ is equivalent to giving a pair $(Z,\AA)$ where $Z:\Lambda\to\C$ is a stability function with the HN-property in the sense of \cite{Bri08}.  See \cite[Proposition 3.5]{Bri08} specifically for this equivalence.

The main theorem in \cite{Bri07} asserts that the set $\Stab(\DD)$ of stability conditions is a complex manifold of dimension $\rk(\Lambda)$.  The manifold $\Stab(\DD)$ carries two group actions: the group $\Aut(\DD)$ of autoequivalences acts on the left by $\Phi(Z,\PP)=(Z\circ\Phi_*^{-1},\Phi(\PP))$, where $\Phi\in\Aut(\DD)$ and $\Phi_*$ is the action $K(\DD)$, and the universal cover $\widetilde{\GL}_2^+(\R)$ of matrices in $\GL_2(\R)$ with positive determinant acts on the right.  This second action lifts the action of $\GL_2(\R)$ on $\Hom(K(\DD),\C)=\Hom(K(\DD),\R^2)$.

\section{Review: Stability conditions on Enriques surfaces, K3 surfaces, and moduli spaces}\label{sec:ReviewStabilityK3Enriques}We give here a review of Bridgeland stability conditions on Enriques surfaces and K3 surfaces and their moduli spaces of stable complexes.  The main references are \cite{Nue14a,Yos16b}.  Throughout this section, $Y$ will denote a K3 or Enriques surface.
\subsection{The algebraic Mukai lattice}\label{subsec:algMukaiLattice}Let $X$ be an Enriques surface over an algebraically 
closed field $k$ of $\chr(k) \ne 2$, and let $\widetilde{X}$ be its covering K3 surface with covering map $\varpi:\widetilde{X}\to X$ and covering involution $\iota$.  We denote by $H^*_{\alg}(\widetilde{X},\Z)$ the algebraic part of the whole cohomology of $\widetilde{X}$, namely
\begin{equation}\label{eq:AlgebraicMukaiLattice}
H^*_\alg(\widetilde{X},\Z) = H^0(\widetilde{X},\Z) \oplus \mathrm{NS}(\widetilde{X}) \oplus H^4(\widetilde{X},\Z).
\end{equation}  

Similarly, for the Enriques surface $X$, we define \begin{equation}\label{EnriquesLattice}
H^*_{\alg}(X,\Z):=\Set{(r,D,\tfrac{s}{2}) \ |\ 
r,s \in \Z, r \equiv s \mod 2, D \in \Num(X) }\subset H^{*}(X,\Q),\end{equation}  where $\Num(X)=\NS(X)/\langle K_X\rangle$.  

\begin{Def} Let $Y=X$ or $\widetilde{X}$.  
\begin{enumerate}
    \item We denote by $\v:K(Y)\onto\Hal(Y,\Z)$, the \emph{Mukai vector} $$\v(E):=\ch(E)\sqrt{\td(Y)}.$$  When $Y=\widetilde{X}$, the Mukai vector takes the form \[\v(E)=(r(E),c_1(E),r(E)+\ch_2(E)),\] in the decomposition \eqref{eq:AlgebraicMukaiLattice}, and when $Y=X$ it takes the form \[\v(E)=(r(E),c_1(E),\frac{r(E)}{2}+\ch_2(E)),\] in the decomposition \eqref{EnriquesLattice}. 
    \item The \emph{Mukai pairing} $\langle\blank,\blank\rangle$ is defined on $\Hal(Y,\Z)$ by $$\langle(r,c,s),(r',c',s')\rangle:=(c,c')-rs'-r's\in\Z,$$ where $(\blank,\blank)$ is the intersection pairing on $H^2(Y,\Z)$.  The Mukai pairing has signature $(2,\rho(Y))$ and satisfies $\langle\v(E),\v(F)\rangle=-\chi(E,F)=-\sum_i(-1)^i\ext^i(E,F)$ for all $E,F\in\Db(Y)$.
    \item The \emph{algebraic Mukai lattice} is defined to be the pair $(\Hal(Y,\Z),\langle\blank,\blank\rangle)$.
\end{enumerate}
\end{Def}

Given a Mukai vector $\v\in\Hal(Y,\Z)$, we denote
its orthogonal complement by
\[
\v^\perp:=\Set{\w\in\Hal(Y,\Z)\ | \ \langle\v,\w\rangle=0 }.
\]
We call a Mukai vector $\v$ \emph{primitive} if it is not divisible in $\Hal(Y,\Z)$.  

The covering map $\varpi$ induces an embedding $$\varpi^*:\Hal(X,\Z)\into\Hal(\X,\Z)$$ such that $\langle\varpi^*\v,\varpi^*\w\rangle=2\langle\v,\w\rangle$ and identifies $\Hal(X,\Z)$ with an index 2 sublattice of the $\iota^*$-invariant component of $\Hal(\X,\Z)$.  The following lemma makes this precise:

\begin{Lem}[{\cite[Lem. 2.1]{Nue14a}}]\label{primitive} A Mukai vector $\v=(r,c_1,\frac{s}{2})\in\Hal(X,\Z)$ is primitive if and only if $$\gcd(r,c_1,\frac{r+s}{2})=1.$$ For primitive $\v$, we define $\ell(\v)$ by $\varpi^*\v=\ell(\v)\w$, where $\w$ is primitive in $\Hal(\X,\Z)$.  Then $\ell(\v)=\gcd(r,c_1,s)$ and can be either $1$ or $2$.  Moreover,
\begin{itemize}
\item if $\ell(\v)=1$, then either $r$ or $c_1$ is not divisible by 2;
\item if $\ell(\v)=2$, then $c_2$ must be odd and $r+s\equiv 2\pmod{4}$.
\end{itemize}
\end{Lem}

In particular, for odd rank Mukai vectors or Mukai vectors with $c_1$ primitive, $\varpi^*\v$ is still primitive, while primitive Mukai vectors with $\gcd(r,c_1)=2$ (and thus necessarily $\gcd(r,c_1,s)=2$) must satisfy $\v^2\equiv 0 \pmod 8$, as can be easily seen.

\subsection{Stability conditions on Enriques and K3 surfaces}\label{subsec:StabilityCondOnEnriquesK3}We continue to let $Y=X$ or $\widetilde{X}$.  
\begin{Def}
A (full, numerical) \emph{stability condition} on $Y$ is a Bridgeland stability condition on $\Db(Y)$ with $\Lambda=\Hal(Y,\Z)$ and $\v$ defined as in  \cref{subsec:algMukaiLattice}.
\end{Def}
In particular, for a stability condition $\sigma=(Z,\PP)$ on $Y$, the category $\PP(\phi)$ has finite length for $\phi\in\R$, so any $\sigma$-semistable object $E\in\PP(\phi)$ admits a filtration with $\sigma$-stable objects $E_i\in\PP(\phi)$.  While the filtration itself, called a \emph{Jordan-H\"{o}lder (JH) filtration}, is not unique, the $\sigma$-stable factors $E_i$ are unique, up to reordering.

A connected component $\Stabd(Y)$ of the space of full numerical stability conditions on $\Db(Y)$ is described in \cite{Bri08,Yos16b}.  Let $\beta,\omega\in\NS(Y)_\R$ be two real divisor classes, with $\omega$ ample.  For $E\in\Db(Y)$, define $$Z_{\omega,\beta}(E):=\langle e^{\beta+i\omega},\v(E)\rangle,$$ and consider the heart $\AA_{\omega,\beta}$ defined by 
\begin{equation*} \label{eq:AK3}
\AA_{\omega,\beta}:=\Set{E\in\Db(Y)\ |\ \begin{array}{l}
\bullet\;\;\HH^p(E)=0\mbox{ for }p\not\in\{-1,0\},\\\bullet\;\;
\HH^{-1}(E)\in\FF_{\omega,\beta},\\\bullet\;\;\HH^0(E)\in\TT_{\omega,\beta}\end{array}},
\end{equation*} where $\FF_{\omega,\beta}$ and $\TT_{\omega,\beta}$ are defined by
\begin{enumerate}
    \item $\FF_{\omega,\beta}$ is the set of torsion-free sheaves $F$ such that every subsheaf $F'\subseteq F$ satisfies $\Im Z_{\omega,\beta}(F')\leq 0$;
    \item $\TT_{\omega,\beta}$ is the set of sheaves $T$ such that, for every non-zero torsion-free quotient $T\onto Q$, we have $\Im Z_{\omega,\beta}(Q)>0$.
\end{enumerate}

We have the following result:
\begin{Thm}[{\cite[Prop. 10.3]{Bri08},\cite{Yos16b}}]\label{thm:GeometricStabilityConditions}
Let $\sigma$ be a stability condition such that all skyscraper sheaves $k(y)$ of points $y\in Y$ are $\sigma$-stable.  Then there are $\omega,\beta\in\NS(Y)_\R$ with $\omega$ ample, such that, up to the $\widetilde{\GL}_2^+(\R)$-action, $Z_\sigma=Z_{\omega,\beta}$ and  $\PP_{\sigma}(0,1]=\AA_{\omega,\beta}$.
\end{Thm}

We call such stability conditions \emph{geometric}, and we let $U(Y)$ be the open subset of $\Stab(Y)$ consisting of geometric stability conditions and denote by $\Stabd(Y)$ the connected component of $\Stab(Y)$ containing $U(Y)$.  

Using the Mukai pairing, for any stability condition $\sigma=(Z_\sigma,\PP_\sigma)$,
we can find $\mho_\sigma \in \Hal(Y,\C)$ 
such that $$Z_\sigma(\blank)=\langle \mho_\sigma,\v(\blank) \rangle.$$  For $\sigma\in\Stabd(Y)$, $\Re\mho_\sigma$ and $\Im\mho_\sigma$ span a positive definite 2-plane in $\Hal(Y,\R)$.  In fact, we can say more.  Let $P(Y)$ be the set of $\mho\in\Hal(Y,\C)$ with components spanning a positive definite 2-plane, which consists of two connected components, and denote by $P^+(Y)$ the component containing vectors of the form $e^{\beta+i\omega}$.  When $Y=\X$, define the subset 
\begin{equation}\label{eqn:DefOfK3Roots}
\Delta(Y):=\Set{\w\in\Hal(Y,\Z)\ | \ \w^2=-2}.
\end{equation}
When $Y=X$, we define the subsets
\begin{equation}\label{eqn:DefOfEnriquesRootsSep}
\begin{split}
    \Delta(Y)_{-1}&:=\Set{\w\in\Hal(Y,\Z)\ |\ \w^2=-1}\\
    \Delta(Y)_{-2}&:=\Set{\w\in\Hal(Y,\Z)\ |\ \begin{array}{c}\w^2=-2 \mbox{ and }c_1(\v)\equiv D\pmod 2,\\
        \;\;\;\;\mbox{where $D$ is a nodal cycle} 
    \end{array}},
\end{split}
\end{equation}
and take their union 
\begin{equation}\label{eqn:DefOfEnriquesRoots}
    \Delta(Y):=\Delta(Y)_{-1}\cup\Delta(Y)_{-2}.
\end{equation}
In either case, we consider the subset $$P_0^+(Y):=P^+(Y)\backslash\bigcup_{\w\in\Delta(Y)}\w^\perp.$$  Then we have the following fundamental theorem:
\begin{Thm}[{\cite[Prop. 8.3,Thm. 13.2]{Bri08},\cite[Cor. 3.8]{MMS09}}]\label{thm:CoveringMap}
The map $$\ZZ:\Stabd(Y)\to\Hal(Y,\C),\;\;\;\;\sigma\mapsto\mho_\sigma$$ is a covering map of the open subset $P_0^+(Y)$, where $Z_\sigma(\blank)=\langle\mho_\sigma,\v(\blank)\rangle$.  If we let $$\Aut_0^\dagger(Y):=\Set{\Phi\in\Aut(\Db(Y))\ | \ \Phi_*=\id_{\Hal(Y,\Z)},\Phi(\Stabd(Y))=\Stabd(Y)},$$ then the group of deck transformations for the covering map $\ZZ$ is precisely $\Aut_0^\dagger(Y)$.  Finally, let $\T$ be the subgroup of $\Aut(\Db(Y))$ generated by $R_{\OO_C(k)}$, where $C$ is a $(-2)$-curve, and $R^2_T$, where $T$ is a spherical or exceptional object (of nonvanishing rank).  Then $$\Stabd(Y)=\bigcup_{\Phi\in\T}\Phi(\overline{U(Y)}).$$
\end{Thm} 
Here $R_T$ is the spherical or exceptional twist.
\subsubsection{Inducing stability conditions}\label{subsubsec:inducing stability}The natural pull-back and push-forward functors relate stability conditions on $\widetilde{X}$ to those on $X$ and vice-versa.  Indeed, a stability condition $\sigma=(Z_\sigma,\PP_\sigma)\in\Stabd(X)$ \emph{induces} a stability condition $\varpi^*(\sigma)=(Z_{\varpi^*(\sigma)},\PP_{\varpi^*(\sigma)})\in\Stabd(\widetilde{X})$ by defining 
$$Z_{\varpi^*(\sigma)}:=Z_\sigma\circ\varpi_*,\;\;\;\PP_{\varpi^*(\sigma)}(\phi):=\Set{E\in\Db(\widetilde{X})\ |\ \varpi_*(E)\in\PP_\sigma(\phi)}.$$  In the opposite direction, for a stability condition $\sigma'=(Z_{\sigma'},\PP_{\sigma'})\in\Stabd(\widetilde{X})$, we can induce a stability condition $\varpi_*(\sigma')=(Z_{\varpi_*(\sigma')},\PP_{\varpi_*(\sigma')})\in\Stabd(X)$ via the definition
$$Z_{\varpi_*(\sigma')}:=Z_{\sigma'}\circ\varpi^*,\;\;\;\PP_{\varpi_*(\sigma')}(\phi):=\Set{E\in\Db(X)\ |\ \varpi^*(E)\in\PP_{\sigma'}(\phi)}.$$

It was shown in \cite{MMS09} that $\varpi^*:\Stabd(X)\to\Stabd(\widetilde{X})$ is a closed embedding onto the submanifold of $\Stabd(\widetilde{X})$ consisting of $\iota^*$-invariant stability conditions, albeit using different notation.  
\subsection{Walls}\label{subsec:Walls}Of paramount importance to our investigation here, the space of Bridgeland stability conditions admits a well-behaved wall and chamber structure.  For a fixed Mukai vector $\v\in\Hal(Y,\Z)$, there exists a locally finite set of \emph{walls} (real codimension one submanifolds with boundary) in $\Stabd(Y)$, depending only on $\v$, with the following properties:
\begin{enumerate}
    \item When $\sigma$ varies in a chamber, that is, a connected component of the complement of the union of walls, the sets of $\sigma$-semistable and $\sigma$-stable objects of class $\v$ do not change.
    \item When $\sigma$ lies on a single wall $\WW\subset\Stabd(Y)$, there is a $\sigma$-semistable object that is unstable in one of the adjacent chambers and semistable in the other adjacent chamber.
    \item These same properties remain true for the wall and chamber structure on $\WW_1\cap\cdots\cap\WW_k$ whose walls are $\WW_1\cap\cdots\cap\WW_k\cap\WW$ for an additional wall $\WW$ for $\v$.
\end{enumerate}
These walls were originally defined in \cite[Prop. 2.3]{Bri08} for K3 surfaces, and more generally in \cite{Tod08}.  By the construction of these walls, it follows that, for primitive $\v$ and $\sigma$ in a chamber for $\v$, a JH-filtration factor of a $\sigma$-semistable object of class $m\v$ ($m\in\N$) must have class $m'\v$ for $1\leq m'<m$.  In particular, for $\sigma$ in a chamber for $\v$, $\sigma$-stability coincides with $\sigma$-semistability.

\begin{Def}\label{def:generic} Let $\v\in\Hal(Y,\Z)$.  A stability condition $\sigma\in\Stabd(Y)$ is called \emph{generic} with respect to $\v$ if it does not lie on any wall for $\v$.
\end{Def}

It is worth recalling that, given a polarization $H\in\Amp(Y)$ and the Mukai vector $\v$ of an $H$-Gieseker semistable sheaf, there exists a chamber $\CC$ for $\v$, the \emph{Gieseker chamber}, where the set of $\sigma$-semistable objects of class $\v$ coincides with the set of $H$-Gieseker semistable sheaves \cite[Prop. 14.2]{Bri08}.

\subsection{Moduli stacks and moduli spaces}

For $\sigma \in \Stab^\dagger(Y)$, 
let $\MM_\sigma(\v)$ be the moduli stack of $\sigma$-semistable objects $E$
with $\v(E)=\v$ and $\MM_\sigma^s(\v)$ 
the open substack of $\sigma$-stable
objects. That is, for a scheme $T$,
$\MM_{\sigma}(\v)(T)$ is the category of $\EE \in\Db(Y \times T)$ such that $\EE$ is relatively perfect over $T$ (\cite[Def. 2.1.1]{Lie}) and $\EE_t$ are $\sigma$-semistable objects
with $\v(\EE_t)=\v$ for all $t \in T$. 
By (the proof of) \cite[Thm. 4.12]{Tod08}, 
$\MM_\sigma(\v)$ is an Artin stack of finite type which is an open substack of Lieblich's ``mother of all moduli spaces" $\MM$ which parametrizes families $\EE\in\Db(Y\times T)$ with $\EE$ relatively perfect over $T$ and such that $\Ext^i(\EE_t,\EE_t)=0$ for all $i<0$ and $t\in T$ (see \cite{Lie}).

We say two objects $E_1$ and $E_2$ in $\MM_\sigma(\v)(k)$ are S-equivalent if they have the same JH-filtration factors.  For $\sigma\in\Stabd(Y)$ generic with respect to $\v$, $\MM_{\sigma}(\v)$ (resp. $\MM_\sigma^s(\v)$) admits a projective coarse moduli scheme $M_\sigma(\v)$ (resp. $M_\sigma^s(\v)$), which parametrizes S-equivalence
classes of $\sigma$-semistable (resp. $\sigma$-stable) objects $E$ with $\v(E)=\v$ (see \cite{BM14a} for the K3 case and \cite[sect. 9]{Nue14b},\cite{Yos16b} for the Enriques case).  

It was shown in \cite{BM14a} and \cite{Yos16b} that there exists an autoequivalence $\Phi\in\Aut(\Db(Y))$ inducing an isomorphism between $\MM_\sigma(\v)$ and the stack $\MM_H(\v')$ of $H$-Gieseker semistable sheaves of Mukai vector $\v'=\Phi(\v)$ for some polarization $H$ generic with respect to $\v'$.  Thus $$\MM_\sigma(\v)\cong[Q(\v')^{ss}/\GL_N],$$ where $Q(\v')$ is the open subscheme of the Quot scheme parametrizing quotients $$\lambda:\OO_Y(-mH)^{\oplus N}\onto E$$ such that 
\begin{enumerate}
    \item $\v(E)=\v'$;
    \item $\lambda$ induces an isomorphism $H^0(Y,\OO_Y^{\oplus N})\cong H^0(Y,E(mH))$;
    \item $H^i(Y,E(mH))=0$ for $i>0$,
\end{enumerate}
where $m\gg 0$ is fixed and $Q(\v')^{ss}$ (resp. $Q(\v')^s$) is the open sublocus where $E$ is semistable (resp. stable).  It follows that \begin{equation}\label{eqn:DimensionOfStackAndCoarse}\dim\MM_\sigma^s(\v)=\dim Q(\v')^{s}-\dim\GL_N=(\dim Q(\v')^s-\dim\PGL_N)-1=\dim M_\sigma^s(\v)-1,\end{equation} since $\PGL_N$ acts freely on $Q(\v')^s$.  

For $L \in \NS(Y)$, we let $\MM_\sigma(\v,L)$ be the substack of $\MM_\sigma(\v)$
consisting of $E$ with $c_1(E)=L$.
We define $\MM_\sigma^s(\v,L)$, 
$M_\sigma(\v,L)$ and $M_\sigma^s(\v,L)$ similary.  When the determinant is irrelevant, we drop $L$ from the notation.  In particular, as $\NS(\widetilde{X})=\Num(\widetilde{X})$, we drop it from the notation in the K3 case.  On the other hand, in the Enriques case, when $Y=X$, $$\MM_\sigma(\v)=\MM_\sigma(\v,L)\bigsqcup\MM_\sigma(\v,L+K_X).$$


\subsection{Some properties of moduli spaces}We recall here what is known about the moduli spaces $\MM_\sigma(\v)$ and their coarse moduli spaces $M_\sigma(\v)$.  Before we get into details for K3 surfaces and Enriques surfaces individually, let us point out that using the definition of inducing stability conditions in \cref{subsubsec:inducing stability} we can relate their respective moduli spaces.  Indeed, for $\sigma\in\Stabd(X)$, $\v\in\Hal(X,\Z)$, and $\w\in\Hal(\widetilde{X},\Z)$, there are morphisms of stacks 
$$\MM_{\varpi^*(\sigma)}(\w)\to\MM_\sigma(\varpi_*(\w)),\;\;\;E\mapsto\varpi_*(E)$$
$$\MM_\sigma(\v)\to\MM_{\varpi^*(\sigma)}(\varpi^*\v),\;\;\;E\mapsto\varpi^*(E).$$  Only the second of these requires comment.  For the stability condition $\varpi_*(\varpi^*(\sigma))$, we have $Z_{\varpi_*(\varpi^*(\sigma))}=2Z_\sigma$ and $\PP_{\varpi_*(\varpi^*(\sigma))}(\phi)=\PP_\sigma(\phi)$, so in particular $E\in\Db(X)$ is $\sigma$-semistable if and only if $\varpi^*(E)$ is $\varpi^*(\sigma)$-semistable. 
\subsubsection{$Y=\widetilde{X}$ is a K3 surface} The following result gives precise conditions on nonemptiness of the moduli spaces $M_\sigma(\v)$ in the K3 case and is proven in \cite{BM14a} and \cite{BM14b}.
\begin{Thm}[{\cite[Thm. 2.15]{Bri08}}]\label{thm:nNnemptinessModuliK3}
Let $\widetilde{X}$ be a K3 surface over $k$, and let $\sigma\in\Stabd(\widetilde{X})$ be a generic stability condition with respect to $\v=m\v_0\in\Hal(\widetilde{X},\Z)$, where $\v_0$ is primitive and $m>0$.
\begin{enumerate}
    \item The coarse moduli space $M_\sigma(\v)$ is non-empty if and only if $\v_0^2\geq-2$.
    \item Either $\dim M_\sigma(\v)=\v^2+2$ and $M_\sigma^s(\v)\neq\varnothing$, or $m>1$ and $\v_0^2\leq0$.
    \item When $\v_0^2>0$, $M_\sigma(\v)$ is a normal irreducible projective variety with $\Q$-factorial singularities.  
\end{enumerate}
\end{Thm}

\subsubsection{$Y=X$ is an Enriques surface}

The following results follow from \cite{Nue14a}. 
Since $\MM_\sigma(v)$ is isomorphic to a moduli stack of Gieseker semi-stable
sheaves \cite{Yos16b}, they also follow from corresponding results for
Gieseker semistable sheaves \cite{Yos14,Yos16a}.  Since the Enriques case is more subtle, we break the statement into smaller pieces, beginning with the primitive case:

\begin{Thm}[{cf. \cite{Nue14b},\cite[Thm. 3.1]{Yos14},\cite[Theorem 4.10]{Yos16a}}]\label{Thm:exist:nodal}
Let $X$ be an Enriques surface over $k$, and let $\sigma\in\Stabd(X)$ be a generic stability condition with respect to primitive $\v\in\Hal(X,\Z)$.  Then for $L\in\NS(X)$ such that $[L\mod K_X]=c_1(\v)$, $M_\sigma(\v,L) \ne \varnothing$ if and only if
\begin{enumerate}
\item
$\ell(\v)=1$ and $\v^2\geq -1$ or 
\item 
$\ell(\v)=2$ and $\v^2>0$ or 
\item
$\ell(\v)=2$, $\v^2=0$, and $L \equiv \frac{r}{2}K_X \pmod 2$ or
\item\label{enum:SphericalNonemptiness}
$\v^2=-2$,
$L \equiv D+\frac{r}{2}K_X \pmod 2$, where 
$D$ is a nodal cycle, that is, $(D^2)=-2$ and $H^1({\cal O}_D)=0$.
\end{enumerate}
Furthermore, when non-empty,
\begin{enumerate}
\item
$M_\sigma(\v,L)$ is connected, and
\item
if $X$ is unnodal or $\v^2\geq 4$, 
then $M_\sigma(\v,L)$ is irreducible.
\end{enumerate}
\end{Thm}
\cref{enum:SphericalNonemptiness} in Theorem \ref{Thm:exist:nodal} only occurs when $X$ contains a smooth rational curve $C$ which necessarily satisfies $C^2=-2$, in which case $X$ is called a \emph{nodal} Enriques surface.  An Enriques surfaces containing no smooth rational curve is called \emph{unnodal}.  A Mukai vector $\v$ as in \ref{enum:SphericalNonemptiness} and an object $E\in M_\sigma^s(\v)$ are called \emph{spherical}, and it can be shown that the existence of a spherical object on $X$ is equivalent to $X$ being nodal \cite{Kim94}.  Similarly, an object $E\in M_\sigma^s(\v)$ with $\v^2=-1$ is called \emph{exceptional}.

For non-primitive Mukai vectors, we phrase the results in terms of the moduli stacks, as it is in this form that we will use them.  We state the positive square case first:

\begin{Prop}[{cf. \cite[Lem. 1.5, Cor. 1.6]{Yos16a}, \cite[Thm. 8.2]{Nue14b}}]\label{prop:pss}
Let $\v\in\Hal(X,\Z)$ be a Mukai vector with $\v^2 >0$, and let $\sigma\in\Stabd(X)$ be generic with respect to $\v$.  For $L\in\NS(X)$ such that $[L\mod K_X]=c_1(\v)$, we set
\begin{equation*}
\MM_\sigma(\v,L)^{pss}:=\Set{E \in \MM_\sigma(\v,L) \ | \
\text{$E$ is properly $\sigma$-semistable}}
\end{equation*}

Then
\begin{enumerate}
\item
$\dim \MM_\sigma(\v,L)^{pss} \leq \v^2-1$.
Moreover $\dim \MM_\sigma(\v,L)^{pss} \leq \v^2-2$ unless
$\v=2\v_0$ with $\v_0^2=1$.
 \item
$\MM_\sigma(\v,L)^{s} \ne \varnothing$, $\MM_\sigma(\v,L)$ is reduced, and 
$\dim \MM_\sigma(\v,L)=\v^2$.
\item
$\MM_\sigma(\v,L)$ is normal, unless
\begin{enumerate}
\item $\v=2\v_0$ with $\v_0^2=1$ and 
$L \equiv \frac{r}{2}K_X \pmod 2$, or
\item $\v^2=2$.
\end{enumerate}
\end{enumerate}
\end{Prop}
The statements in Proposition \ref{prop:pss} remain true for the coarse moduli spaces with dimensions adjusted in accordance with \eqref{eqn:DimensionOfStackAndCoarse}.  In particular, $\dim M_\sigma(\v,L)=\v^2+1$.  

For Mukai vectors with $\v^2\leq 0$ it is particularly useful to use the moduli stacks for dimension estimates as we now see:
\begin{Prop}[cf. {\cite[Proposition 1.9]{Yos16a}}]\label{prop:isotropic}
Let $\u\in\Hal(X,\Z)$ be an isotropic and primitive Mukai vector, and let $\sigma\in\Stabd(X)$ be generic with respect to $\u$.
\begin{enumerate}
\item
If $\MM_\sigma^s(m\u) \ne \varnothing$, then
$m=1,2$.
\item
$\MM_\sigma^s(2\u,L) \ne \varnothing$
if and only if $\ell(\u)=1$ and $L \equiv 0 \pmod 2$.
Moreover
$$
\MM_\sigma^s(2\u)=\Set{\varpi_*(F) \ | \ F \in 
\MM_{\varpi^*(\sigma)}^s(\varpi^*\v),\;
\iota^*(F) \not \cong F }.
$$ 
In particular, $\MM_\sigma^s(2\u)$ is smooth of
dimension 1.
\item
$\dim \MM_\sigma(m\u) \leq\lfloor\frac{m\ell(\u)}{2}\rfloor$.
\end{enumerate} 
\end{Prop}

Finally, we consider the negative square case:
\begin{Lem}\label{Lem:dimension negative}
Let $\w\in\Hal(X,\Z)$ be a primitive spherical or exceptional class, and let $\sigma\in\Stabd(X)$ be a generic stability condition with respect to $\w$.  Then
$$\dim\MM_{\sigma}(m\w)= \begin{cases}
 -m^2, & \text{ if }\w^2=-2,\\

-\frac{m^2}{2}, & \text{ if }\w^2=-1,m\equiv 0\pmod 2,\\

-\frac{m^2+1}{2}, & \text{ if }\w^2=-1,m\equiv 1\pmod 2.\\
 \end{cases}$$
\end{Lem}
\begin{proof}
By \cite[Proposition 9.9]{Nue14b}, in case $\w^2=-2$, the coarse moduli space $M_{\sigma}(m\w)$ consists of a single point, $S^{\oplus m}$, where $S$ is the unique $\sigma$-stable spherical object of class $\w$.  As $\Aut(S^{\oplus m})=\GL_m(k)$, we get $$\dim \MM_{\sigma}(m\w)=\dim M_{\sigma}(m\w)-\dim\Aut(S^{\oplus m})=-m^2,$$   If $\w^2=-1$, then by \cite[Lemma 9.2]{Nue14b} the coarse moduli space $M_{\sigma}(m\w)$ consists of the $m+1$ points $\Set{E^{\oplus i}\oplus E(K_X)^{\oplus m-i}}_{i=0}^m$, where $E$ and $E(K_X)$ are the two $\sigma$-stable exceptional objects of class $\w$.  As $E$ and $E(K_X)$ are both exceptional, $\Aut(E^{\oplus i}\oplus E(K_X)^{\oplus m-i})=\GL_i(k)\times \GL_{m-i}(k)$ of dimension $i^2+(m-i)^2$.  But then 
\begin{align}
\begin{split}\dim\MM_{\sigma}(m\w)&=\max_{0\leq i\leq m}\dim_{E^{\oplus i}\oplus E(K_X)^{\oplus m-i}} M_{\sigma}(m\w)-\Aut(E^{\oplus i}\oplus E(K_X)^{\oplus m-i})\\
&=-\min_{0\leq i\leq m} i^2+(m-i)^2,\\
\end{split}
\end{align}
which gives the dimension as claimed.
\end{proof}

\subsection{Line bundles on moduli spaces}\label{subsec:LineBundles}
We again let $Y=X$ or $\X$, and we recall the definition of the Donaldson-Mukai morphism.  Fix a Mukai vector $\v\in\Hal(Y,\Z)$, a stability condition $\sigma\in\Stabd(Y)$, and a universal family $\EE\in\Db(M_\sigma(\v,L)\times Y)$.  Then we have the following definition.
\begin{Def} Set $K(Y)_\v:=\Set{x\in K(Y)\ |\ \langle\v(x),\v\rangle=0}$.  The Donaldson-Mukai morphism from $K(Y)_\v$ to $\Pic(M_\sigma(\v,L))$ is defined by:  
\begin{equation}\label{eqn:DonaldsonMukai}
\begin{matrix}
\theta_{\v,\sigma}:& K(Y)_\v & \to & \Pic(M_{\sigma}(\v,L))\\
 & x & \mapsto & \det (p_{M_\sigma(\v,L)!}(\EE \otimes p_Y^*(x^{\vee}))).
\end{matrix} 
\end{equation}
More generally, for a scheme $S$ and a family $\EE\in\Db(S\times Y)$ over $S$ of objects in $M_\sigma(\v,L)$, there is a Donaldson-Mukai morphism $\theta_\EE:K(Y)_\v\to\Pic(S)$ associated to $\EE$, defined as in \eqref{eqn:DonaldsonMukai}, which satisfies $\theta_\EE=\lambda_\EE^*\theta_{\v,\sigma}$, where $\lambda_\EE:S\to M_{\sigma}(\v,L)$ is the associated classifying map.  See \cite[Section 8.1]{HL10} for more details.

Setting 
\begin{equation}\label{eqn:def of xi}
\xi_\sigma:=\Im\frac{ \mho_\sigma}{\langle \mho_\sigma, \v \rangle}
\in \v^\perp,
\end{equation} we define the numerical divisor class
\begin{equation}\label{eqn:def of ell(sigma)}
\ell_\sigma:=\theta_{\v,\sigma}(\xi_\sigma)\in\Num(M_\sigma(\v,L)),
\end{equation}
where we abuse notation by also using $\theta_{\v,\sigma}$ for the extension of \eqref{eqn:DonaldsonMukai} to the Mukai lattice.
\end{Def}

From \cref{subsec:Walls} it follows that the moduli space $M_\sigma(\v,L)$ and $\EE$ remain constant when varying $\sigma$ in a chamber for $\v$, so for each chamber $\CC$, we get a map $$\ell_{\CC}:\CC\to\NS(M_\CC(\v,L)),\;\;\sigma\mapsto\ell_\sigma,$$ where the notation $M_\CC(\v,L)$ denotes the fixed moduli space.  By the proof of the projectivity of $M_\sigma(\v,L)$ in \cite{Yos16b} with the argument
in \cite{MYY14} or \cite{BM14a}, we get a generalization of 
\cite[Theorem 10.3]{Nue14b} and can say even more:
\begin{Thm}[{cf. \cite[Theorem 4.1]{BM14a} in the K3 case}]\label{Thm:NefAmpleDivisor}
For a generic $\sigma$,
$\ell_\sigma=\theta_{\v,\sigma}(\xi_\sigma)$ is an ample divisor on $M_\sigma(\v,L)$.
\end{Thm}

\subsection{Wall-crossing and Birational transformations} In this last subsection, we recall one more result that will be essential for our study of the connection between crossing walls in $\Stabd(Y)$ and birational transformations of the moduli space $M_\sigma(\v,L)$.  Let $\v\in\Hal(Y,\Z)$ with $\v^2>0$, and let $\WW$ be a wall for $\v$.  We say $\sigma_0\in\WW$ is \emph{generic} if it does not belong to any other wall, and we denote by $\sigma_+$ and $\sigma_-$ two generic stability conditions nearby $\WW$ in two opposite adjacent chambers.  Then all $\sigma_\pm$-semistable objects are still $\sigma_0$-semistable, and thus the universal familes $\EE^\pm$ on $M_{\sigma_\pm}(\v)\times Y$ induce nef divisors $\ell_{\sigma_0,\pm}$ on $M_{\sigma_\pm}(\v)$ by $$\ell_{\sigma_0,\pm}:=\theta_{\v,\sigma_\pm}(\xi_{\sigma_0}).$$  The main result about $\ell_{\sigma_0,\pm}$ is the following:
\begin{Thm}[{\cite[Thm. 1.4(a)]{BM14a},\cite[Thm. 11.3]{Nue14b}}]\label{Thm:WallContraction} Let $\v\in\Hal(Y,\Z)$ satisfy $\v^2>0$, and let $\sigma_\pm$ be two stability conditions in opposite chambers nearby a generic $\sigma_0\in\WW$.  Then:
\begin{enumerate}
\item The divisors $\ell_{\sigma_0,\pm}$ are semiample on $M_{\sigma_\pm}(\v)$.  In particular, they induce contractions $$\pi^\pm:M_{\sigma_\pm}(\v)\to\overline{M}_\pm,$$ where $\overline{M}_\pm$ are normal projective varieties.  When $Y=\X$, the divisors $\ell_{\sigma_0,\pm}$ are big so that $\pi^\pm$ are birational and $\overline{M}_\pm$ are irreducible.
\item For any curve $C\subset M_{\sigma_\pm}(\v)$, $\ell_{\sigma_0,\pm}.C=0$ if and only if the two objects $\EE_c^\pm$ and $\EE_{c'}^\pm$ corresponding to two general points $c,c'\in C$ are S-equivalent.  In particular, the curves contracted by $\pi^\pm$ are precisely the curves of objects that are S-equivalent with respect to $\sigma_0$.
\end{enumerate}
\end{Thm}

This theorem leads us to the following definition describing a wall $\WW$ in terms of the geometry of the induced morphisms $\pi^\pm$.
\begin{Def}
We call a wall $\WW$:
\begin{enumerate}
\item a \emph{fake wall}, if there are no curves contracted by $\pi^\pm$;
\item a \emph{totally semistable wall}, if $M_{\sigma_0}^s(\v)=\varnothing$;
\item a \emph{flopping wall}, if we can identify $\overline{M}_+=\overline{M}_-$ and the induced map $M_{\sigma_+}(\v)\dashrightarrow M_{\sigma_-}(\v)$ induces a flopping contraction;
\item a \emph{divisorial wall}, if the morphisms $\pi^\pm$ are both divisorial contractions;
\item a \emph{$\P^1$-wall}, if the morphisms $\pi^\pm$ are both $\P^1$-fibrations.
\end{enumerate}
\end{Def}

A non-fake wall $\WW$ such that $\codim(M_{\sigma_\pm}(\v)\backslash M^s_{\sigma_0}(\v))\geq 2$ is necessarily a flopping wall by \cite[Thm. 1.4(b)]{BM14a} and \cite[Thm. 11.3]{Nue14b}.  

\section{Dimension estimates of substacks of Harder-Narasimhan filtrations}\label{sec:DimensionsOfHarderNarasimhan}
 Having finished our review of known tools for studying wall-crossing, we develop here the first tool we will use to classify the behavior induced by crossing a wall.  In this section, we will denote by $Y$ any smooth projective variety satisfying openness of stability and boundedness of Bridgeland semistable objects as in \cite[Lemma 3.4]{Tod08}.

For Mukai vectors $\v_1,\v_2,\dots,\v_s$ with the same phase $\phi$ with respect to $\sigma$, 
let $\FF(\v_1,\dots,\v_s)$ be the stack of filtrations:
for a scheme $T$, 
\begin{equation}\FF(\v_1,\dots,\v_s)(T):=\Set{0 \subset \FF_1 \subset \cdots \subset \FF_s\ | \ \FF_i/\FF_{i-1} \in \MM_{\sigma}(\v_i)(T), 1\leq i\leq s,\FF_s\in\MM_{\sigma}(\v)(T)},
\end{equation}
where $\v=\v_1+\cdots+\v_s$.  Then we have the following result.

\begin{Prop}With the notation as above, let $Y$ be a smooth projective variety satisfying boundedness and openness of stability.  Then $\FF(\v_1,\dots,\v_s)$ is an Artin stack of finite type.
\end{Prop}
\begin{proof}
We prove the proposition by induction on $s$.  Assuming that $\FF(\v_1,\dots,\v_{s-1})$ is an Artin stack of finite type,
we shall prove that 
$\FF(\v_1,\dots,\v_{s-1},\v_s)$ is also an Artin stack of finite type.
We set $\v\:=\sum_{i=1}^s \v_i$.
It is sufficient to show that
\begin{enumerate}
\item\label{enum:MorphismRepresentable} the natural morphism
$\FF(\v_1,\dots,\v_{s-1},\v_s) \to 
\FF(\v_1,\dots,\v_{s-1}) \times \MM_{\sigma}(\v)$
is representable by schemes, and
\item\label{enum:DiagonalRepresentable} the diagonal morphism $\Delta:\FF(\v_1,\dots,\v_s)\to\FF(\v_1,\dots,\v_s)\times\FF(\v_1,\dots,\v_s)$ is representable.
\end{enumerate}
Indeed, if we take a smooth surjective morphism $M\to\FF(\v_1,\dots,\v_{s-1})\times\MM_\sigma(\v)$ from a scheme $M$ of finite type, then we get a smooth surjective morphism $\FF(\v_1,\dots,\v_s)\times_{\FF(\v_1,\dots,\v_{s-1})\times\MM_\sigma(\v)}M\to\FF(\v_1,\dots,\v_s)$, where $\FF(\v_1,\dots,\v_s)\times_{\FF(\v_1,\dots,\v_{s-1})\times\MM_\sigma(\v)}M$ is a finite type scheme by \ref{enum:MorphismRepresentable}.  The statement in \ref{enum:DiagonalRepresentable} is simply the other condition in the definition of an Artin stack \cite[Def. 8.1.4]{Ols16}

We prove \ref{enum:MorphismRepresentable} first.  Let $T$ be a scheme and 
$T \to \FF(\v_1,\dots,\v_{s-1}) \times \MM_{\sigma}(\v)$
a morphism. Then
we have a family of filtrations
\begin{equation}
0 \subset \FF_1 \subset \cdots \subset \FF_{s-1}
\end{equation} 
on $T \times Y$
such that $\FF_i/\FF_{i-1}$ are relatively perfect over $T$, 
$(\FF_i/\FF_{i-1})_t \in \MM_{\sigma}(\v_i)$ for all $t \in T$ 
and a family of objects $\FF_s$ such that
$\FF_s$ is relatively perfect over $T$ and
$(\FF_s)_t \in \MM_{\sigma}(\v)$.
By \cite[Prop. 1.1]{Ina02},
there is a scheme $p:Q \to T$ 
which represents the functor $\QQ:(Sch/T) \to (Sets)$ defined by 
\begin{equation}
\QQ(U \overset{\varphi}{\to} T)=
\{f \mid f:(\varphi \times 1_Y)^*(\FF_{s-1}) \to (\varphi \times 1_Y)^*(\FF_s)\}.
\end{equation}
Let 
$\xi:(p \times 1_Y)^*(\FF_{s-1}) \to (p \times 1_Y)^*(\FF_s)$
be the universal family of homomorphisms.
Let $Q^0$ be the subscheme
of $Q$ such that $\Cone(\xi_q)=\Cone(\xi)_q\in\AA_\sigma$ for all $q\in Q^0$, which is open by the Open Heart Property \cite[Theorem 3.8]{Tod08},\cite[Theorem 3.3.2]{AP06}.  It follows that $\xi_q$ is injective in $\AA_{\sigma}$ for all $q \in Q^0$. 
Then on $Q^0$ we have a family of filtrations
\begin{equation}
0 \subset \FF_1  \subset \cdots \subset \FF_s.
\end{equation} 
Therefore
\begin{equation}
Q^0 \cong \FF(\v_1,\dots,\v_s) 
\times_{\FF(\v_1,\dots,\v_{s-1}) \times \MM_{\sigma}(\v)}
T.
\end{equation}
In particular, the morphism $\FF(\v_1,\dots,\v_s)\to\FF(\v_1,\dots,\v_{s-1})\times\MM_\sigma(\v)$ is representable by schemes, as claimed.

Now let us prove \ref{enum:DiagonalRepresentable}.  By \cite[Lem. 8.1.8]{Ols16}, it is equivalent to showing that for every scheme $T$ and two families of filtrations $\FF,\FF'\in\FF(\v_1,\dots,\v_s)(T)$, the sheaf $\Isom_{\FF(\v_1,\dots,\v_s)/T}(\FF,\FF')$ is an algebraic space.  So let $T$ be a scheme and consider two families
\begin{equation}
\begin{split}
\FF: & 0 \subset \FF_1 \subset\cdots \subset \FF_s\\
\FF': & 0 \subset \FF_1'\subset \cdots \subset \FF_s'
\end{split}
\end{equation} 
of relatively perfect filtrations.  An isomorphism $\phi:\FF \to \FF'$ is
an isomorphism $\FF_s \to \FF_s'$ as families in $\MM_\sigma(\v)(T)$ which preserves the filtration. But $\phi$ preserves the filtration if and only if the induced maps
$\FF_i \to \FF_s'/\FF_i'$ are the 0-map for all $0<i<s$.
This is a closed condition by \cite[Prop. 1.1]{Ina02}, and if $\phi(\FF_i) \subseteq \FF_i'$, then we must in fact have equality as they are both families of $\sigma$-semistable
objects with the same Mukai vector.  Hence the sheaf $\Isom_{\FF(\v_1,\dots,\v_s)/T}(\FF,\FF')$ is parametrized by a closed algebraic subspace of the algebraic space
$\Isom_{\MM_\sigma(\v)/T}(\FF_s,\FF_s')$, which shows that $\Isom_{\FF(\v_1,\dots,\v_s)/T}(\FF,\FF')$ is an algebraic space, as required.
\end{proof}

We have a natural morphism
\begin{equation*}
\begin{matrix}
\FF(\v_1,\dots,\v_s)&\longrightarrow&\FF(\v_1,\dots,\v_{s-1})\times\MM_{\sigma}(\v_s)\\
(0\subset\FF_1\subset\cdots\subset\FF_s)&\longmapsto&((0\subset\FF_1\subset\cdots\subset\FF_{s-1}),\FF_s/\FF_{s-1})
\end{matrix},
\end{equation*} 
and hence a morphism
$$
\Pi:\FF(\v_1,\dots,\v_s) \to \prod_{i=1}^s \MM_{\sigma}(\v_i).
$$
Let $$\FF(\v_1,\dots,\v_s)^*:=\Pi^{-1}\left(\prod_{i=1}^s(\MM_{\sigma_-}(\v_i)\cap\MM_\sigma(\v_i))\right)\subset \FF(\v_1,\dots,\v_s)$$ be the open substack of
$\FF(\v_1,\dots,\v_s)$ where each $\FF_i/\FF_{i-1}$ is $\sigma_-$-semistable as well, where $\sigma_-$ is sufficiently close to
$\sigma$.  The intersections are taken within the large moduli space $\MM$, and as $\MM_{\sigma}(\v)$ is an open substack of $\MM$ for any $\sigma\in\Stab(Y)$ by openness of stability, it follows that $\MM_{\sigma_-}(\v_i)\cap\MM_\sigma(\v_i)$ is open in $\MM_\sigma(\v_i)$.  Thus $\FF(\v_1,\dots,\v_s)^*$ is indeed well-defined and an open substack of $\FF(\v_1,\dots,\v_s)$, as claimed.

While we cannot say much more about the stack of filtrations in general, if we assume that $\v_1,\dots,\v_s$ are the Mukai vectors of the semistable factors of the Harder-Narasimhan
filtration with respect to $\sigma_-$ of an object $E\in\MM_{\sigma}(\v)$, with $\v=\sum_{i=1}^s\v_i$, then the natural map 
\begin{equation}
\begin{matrix}
 \FF(\v_1,\dots,\v_s)^*& \to &\MM_{\sigma}(\v)\\
 (0\subset\FF_1\subset\cdots\subset\FF_s)&\mapsto&\FF_s
\end{matrix}   
\end{equation}
 is injective with image the substack of $\MM_{\sigma}(\v)$ parameterizing objects with Harder-Narasimhan filtration factors having Mukai vectors
$\v_1,\dots,\v_s$.  In this case, we can say even more and prove the following theorem, whose proof is similar to that of \cite[Prop. 6.2]{Bri12}.
\begin{Thm}\label{Thm:DimensionOfHNFiltrationStack} As above, suppose that $Y$ satisfies boundedness and openness of stability, and let $\v_1,\dots,\v_s$ be the Mukai vectors of the semistable factors of the Harder-Narasimhan filtration with respect to $\sigma_-$ of some object $E\in\MM_{\sigma}(\v)$, where $\v=\sum_{i=1}^s\v_i$.  Suppose further that for any $\sigma\in\Stab(Y)$ and $E,E'\in\AA_\sigma$ such that $\phi_{\min}(E')>\phi_{\max}(E)$ we have $\Hom(E,E'[k])=0$ for $2\leq k\leq\dim Y$.  Then 
\begin{equation}\label{eqn:DimensionOfHNFiltraionStack}
\begin{split}
\dim \FF(\v_1,\dots,\v_s)^*=&
\dim \FF(\v_1,\dots,\v_{s-1})^*+\dim \MM_{\sigma_-}(\v_s)+
\langle \v-\v_s,\v_s \rangle\\
=& \sum_{i=1}^s \dim \MM_{\sigma_-}(\v_i)+\sum_{i<j}\langle \v_i,\v_j \rangle.
\end{split}
\end{equation}

\end{Thm}

\begin{proof}
For an atlas $\varphi:T \to \FF(\v_1,\dots,\v_{s-1})^* \times \MM_{\sigma_-}(\v_s)$,
we set 
$R:=T \times_{\FF(\v_1,\dots,\v_{s-1})^* \times \MM_{\sigma_-}(\v_s)} T$.
Let $0 \subset \FF_1 \subset \cdots \subset \FF_{s-1}$ and
$\EE_s$ be objects on $T \times Y$ 
corresponding to the morphism $\varphi$. 
We note that for all $t\in T$, $\Hom((\EE_s)_t,(\FF_{s-1})_t [k])=0$ for
$k \ne 0,1$.  Indeed, for $k\not\in[0,\dim Y]$, this is clear as $(\EE_s)_t,(\FF_{s-1})_t\in\AA_{\sigma_-}$, while for $k\in[2,\dim Y]$, this follows from the hypothesis of the theorem.  We shall stratify $T=\bigcup_i T_i$ by $n(t):=\hom((\EE_s)_t,(\FF_{s-1})_t)$ so that $n(t)$ is constant on $T_i$ and $n_{|T_i} \ne n_{|T_j}$ for $i \ne j$.  It follows that $$\chi((\EE_s)_t,(\FF_{s-1})_t)=\hom((\EE_s)_t,(\FF_{s-1})_t)-\hom((\EE_s)_t,(\FF_{s-1})_t[1])$$ is constant for all $t\in T$, and thus on each $T_i$, both $n(t)$ and $\hom((\EE_s)_t,(\FF_{s-1})_t[1])$ are constant.
We set $R_{ij}:=R \times_{T \times T} T_i \times T_j$.
Then $R_{ij}=\varnothing $ if $i \ne j$ and we have a stratification
$R=\bigcup_i R_{ii}$.  Let $p_i^k:V_i^k \to T_i$ $(k=0,1)$ be the vector bundles associated to 
$\lHom_{p_i}((\EE_s)_{|T_i},(\FF_{s-1})_{|T_i}[k])$, 
where $p_i:T_i \times Y \to T_i$ is the projection.

As in \cite[Lemma 6.1]{Bri12}, there is a universal extension $$0\to (p_i^1)^*\FF_{s-1}\to \FF\to(p_i^1)^*\EE_s\to 0$$ over $V_i^1$, and the family $\FF$ determines a morphism $q_i:V_i^1\to\FF(\v_1,\dots,\v_s)^*$  which factors through
$$
V_i^1 \to 
T_i \times_{ \FF(\v_1,\dots,\v_{s-1})^* \times \MM_{\sigma_-}(\v_s)}
\FF(\v_1,\dots,\v_s)^*.
$$
As in \cite[p. 131]{Bri12}, one can show that there is an isomorphism
$$
V_i^1 \times_{\FF(\v_1,\dots,\v_s)^*} V_i^1 \cong 
V_i^0 \times_{R_{ii}} V_i^1
$$
with a commutative diagram
\begin{equation}
\begin{CD}
V_i^0 \times_{R_{ii}} V_i^1 @>>> V_i^1 \times V_i^1 \\
@VVV @VVV\\
R_{ii} @>>> T_i \times T_i.
\end{CD}
\end{equation}
It follows that in this description $\Set{V_i^1\times V_i^1}_i$ provide an atlas for $\FF(\v_1,\dots,\v_s)^*$ with relations given by $\Set{V_i^0\times_{R_{ii}}V_i^1}_i$ just as $\Set{T_i\times T_i}_i$ and $\Set{R_{ii}}_i$, respectively, do for $\FF(\v_1,\dots,\v_{s-1})^*\times\MM_{\sigma_-}(\v_s)$.  Since
\begin{equation}
\begin{split}
& \dim V_i^1 \times V_i^1-\dim V_i^0 \times_{R_{ii}} V_i^1\\
=& \dim T_i \times T_i +2\rk V_i^1-(\dim R_{ii}+\rk V_i^1+\rk V_i^0)\\
=& (\dim T_i \times T_i-\dim R_{ii})+\langle \v-\v_s, \v_s \rangle,
\end{split}
\end{equation}
we get 
\begin{equation}
\begin{split}
\dim \FF(\v_1,\dots,\v_s)^* =& 
\max_i \{\dim V_i^1 \times V_i^1-\dim V_i^0 \times_{R_{ii}} V_i^1\}\\
=& \max_i \{(\dim T_i \times T_i-\dim R_{ii})\}+\langle \v-\v_s, \v_s \rangle\\
=&
\dim \FF(\v_1,\dots,\v_{s-1})^*+
\dim \MM_{\sigma_-}(\v_s)+\langle \v-\v_s, \v_s \rangle.
\end{split}
\end{equation}
This gives the first equation in \eqref{eqn:DimensionOfHNFiltraionStack}, while the second follows by induction.  
\end{proof}
  
We can apply the above theorem to study the locus in $\MM_{\sigma_+}(\v)$ of strictly $\sigma_0$-semistable objects.  Recall our setup: $X$ is an Enriques surface, and for a given $\v\in\Hal(X,\Z)$ with $\v^2>0$, we take a generic stability condition $\sigma_0$ on a wall $\WW$ for $\v$ and two generic nearby stability conditions $\sigma_\pm$ in opposite adjacent chambers.  By letting $\sigma=\sigma_0$ in the above theorem, we get the following result:
\begin{Prop}\label{Prop:HN codim} Let $X$ be an Enriques surface, and suppose that $\v_1,\dots,\v_s$ are the Mukai vectors of the semistable factors of the Harder-Narasimhan filtration with respect to $\sigma_-$ of an object $E\in\MM_{\sigma_+}(\v)$, where $\v=\sum_{i=1}^s\v_i$ satisfies $\v^2>0$.  Then letting $\FF(\v_1,\dots,\v_s)^o:=\FF(\v_1,\dots,\v_s)^*\cap\MM_{\sigma_+}(\v)$, where the intersection is taken in $\MM_{\sigma_0}(\v)$, we have
\begin{equation}\label{eqn:HNFiltrationCodim}\codim\FF(\v_1,\dots,\v_s)^o
\geq\sum_{i=1}^s \left(\v_i^2-\dim\MM_{\sigma_-}(\v_i)\right)+\sum_{i<j}\langle \v_i,\v_j\rangle,
\end{equation}
where the codimension is taken with respect to $\MM_{\sigma_+}(\v)$.
\end{Prop}
\begin{proof}
In order to apply Theorem \ref{Thm:DimensionOfHNFiltrationStack}, we first observe that the hypothesis of the theorem is met where $\sigma=\sigma_0$ in this case.  Indeed, for $E,E'\in\AA_{\sigma_0}$ such that $\phi_{\min}(E')>\phi_{\max}(E)$, Serre duality gives $$\Hom(E,E'[2])=\Hom(E',E(K_X))=0,$$ where the last equality follows since $X$ is numerically $K$-trivial so that $\phi_{\max}(E(K_X))=\phi_{\max}(E)$.

Noting that $\FF(\v_1,\dots,\v_s)^o$ is an open substack of $\FF(\v_1,\dots,\v_s)^*$ by openness of stability, we get that $\dim\FF(\v_1,\dots,\v_s)^o\leq\dim\FF(\v_1,\dots,\v_s)^*$, with equality if and only if the component of $\FF(\v_1,\dots,\v_s)^*$ of largest dimension contains a $\sigma_+$-semistable object.  As $\v^2>0$, $\dim\MM_{\sigma_+}(\v)=\v^2$ by Proposition \ref{prop:pss}, so 
\begin{equation}
    \begin{split}
        \codim\FF(\v_1,\dots,\v_s)^o&=\v^2-\dim\FF(\v_1,\dots,\v_s)^o\geq\v^2-\dim\FF(\v_1,\dots,\v_s)^*\\
        &=\v^2-\left(\sum_{i=1}^s\dim\MM_{\sigma_-}(\v_i)+\sum_{i<j}\langle\v_i,\v_j\rangle\right)\\
        &=\sum_{i=1}^s\left(\v_i^2-\dim\MM_{\sigma_-}(\v_i)\right)+\sum_{i<j}\langle\v_i,\v_j\rangle,
    \end{split}
\end{equation}
as claimed.
\end{proof}
While Proposition \ref{Prop:HN codim} is phrased for Enriques surfaces, it applies to any $K$-trivial surface, with \eqref{eqn:HNFiltrationCodim} modified appropriately.  In particular, the first author uses it in \cite{Nue18} to study wall-crossing for bielliptic surfaces.

\section{The hyperbolic lattice associated to a wall}\label{sec:HyperbolicLattice}
In order to effectively use the estimates provided by Proposition \ref{Prop:HN codim}, we need to gain some understanding of the Mukai vectors $\v_i$ which appear as Harder-Narasimhan factors of an object $E\in\MM_{\sigma_+}(\v)$ when we cross a wall $\WW$.  From their definition, these walls are associated to the existence of another Mukai vector with the same phase as $\v$, so to any wall $\WW$ it is natural to consider the set of these ``extra'' classes, as in the following definition.  As it turns out, this set will contain all of the Mukai vectors we are interested in.
\begin{PropDef}\label{hyperbolic} To a wall $\WW$, let $\HH_{\WW}\subset\Hal(X,\Z)$ be the set of Mukai vectors $$\HH_\WW:=\Set{\w\in\Hal(X,\Z)\ |\ \Im\frac{Z(\w)}{Z(\v)}=0\mbox{ for all }\sigma\in\WW}.$$
Then $\HH_{\WW}$ has the following properties:
\begin{enumerate}
\item It is a primitive sublattice of rank two and of signature $(1,-1)$ (with respect to the restriction of the Mukai form).
\item Let $\sigma_+,\sigma_-$ be two sufficiently close and generic stability conditions on opposite sides of the wall $\WW$, and consider any $\sigma_+$-stable object $E\in M_{\sigma_+}(\v)$.  Then any HN-filtration factor $A_i$ of $E$ with respect to $\sigma_-$ satisfies $\v(A_i)\in\HH_{\WW}$.
\item If $\sigma_0$ is a generic stability condition on the wall $\WW$, the conclusion of the previous claim also holds for any $\sigma_0$-semistable object $E$ of class $\v$.
\item Similarly, let $E$ be any object with $\v(E)\in\HH_{\WW}$, and assume that it is $\sigma_0$-semistable for a generic stability condition $\sigma_0\in\WW$.  Then every Jordan-H\"{o}lder factor of $E$ with respect to $\sigma_0$ will have Mukai vector contained in $\HH_{\WW}$.
\end{enumerate}
\end{PropDef}
\begin{proof} The proof of \cite[Proposition 5.1]{BM14a} carries over word for word.
\end{proof}

We would like to characterize the type of the wall $\WW$, i.e. the type of birational transformation induced by crossing it, in terms of the lattice $\HH_{\WW}$.  We will find it helpful to also go in the opposite direction as in \cite[Definition 5.2]{BM14a}:
\begin{Def} Let $\HH\subset\Hal(X,\Z)$ be a primitive rank two hyperbolic sublattice containing $\v$.  A \emph{potential wall} $\WW$ associated to $\HH$ is a connected component of the real codimension one submanifold consisting of those stability conditions $\sigma$ such that $Z_\sigma(\HH)$ is contained in a line.
\end{Def}

We will also have cause to consider two special convex cones in $\HH_{\R}$.  The first is defined as follows (see \cite[Definition 5.4]{BM14a}):

\begin{Def} Given any hyperbolic lattice $\HH\subset \Hal(X,\Z)$ of rank two containing $\v$, denote by $P_{\HH}\subset \HH_{\R}$ the cone generated by classes $\u\in\HH$ with $\u^2\geq 0$ and $\langle\v,\u\rangle>0$.  We call $P_{\HH}$ the \emph{positive cone} of $\HH$, and a class in $P_{\HH}\cap\HH$ a \emph{positive class}.
\end{Def}

The next cone, called the \emph{effective cone} and whose integral classes are \emph{effective classes}, is classified by the following proposition (see \cite[Proposition 5.5]{BM14a} for the analogue in the K3 case):

\begin{Prop} Let $\WW$ be a potential wall associated to a hyperbolic rank two sublattice $\HH\subset\Hal(X,\Z)$.  For any $\sigma\in\WW$, let $C_{\sigma}\subset\HH_{\R}$ be the cone generated by classes $\u\in\HH$ satisfying the two conditions $$\Re\frac{Z_{\sigma}(\u)}{Z_{\sigma}(\v)}>0\mbox{ and }\begin{cases}
\u^2\geq -1, & \mbox{or}\\
\u^2=-2, & c_1(\u)\equiv D\pmod2,\mbox{ $D$ a nodal cycle}
\end{cases}.$$  Then this cone does not depend on the choice of $\sigma\in\WW$, so we may and will denote it by $C_{\WW}$.  Moreover, it contains $P_{\HH}$.

If $\u\in C_{\WW}$, then there exists a $\sigma$-semistable object of class $\u$ for every $\sigma\in\WW$, and if $\u\notin C_{\WW}$, then for a generic $\sigma\in\WW$ there does not exist a $\sigma$-semistable object of class $\u$.
\end{Prop}
\begin{proof} The proof is identical to that of \cite[Proposition 5.5]{BM14a} except that for the statements about the existence of semistable objects we must use Theorem \ref{Thm:exist:nodal}.  This accounts for the more subtle condition on $\u^2$ compared to the corresponding condition $\u^2\geq -2$ for K3 surfaces.
\end{proof}

We also recall \cite[Remark 5.6]{BM14a}:
\begin{Rem} From the positivity condition on $\Re\frac{Z_{\sigma}(\u)}{Z_{\sigma}(\v)}$, it is clear that $C_{\WW}$ contains no line through the origin, i.e. if $\u\in C_{\WW}$ then $-\u\notin C_{\WW}$.  Thus there are only finitely many classes in $C_{\WW}\cap(\v-C_{\WW})\cap\HH$.

We use this fact to make the following assumption: when we refer to a generic $\sigma_0\in\WW$, we mean that $\sigma_0$ is not in any of the other walls associated to the finitely many classes in $C_{\WW}\cap(\v-C_{\WW})\cap\HH$.  Likewise, $\sigma_{\pm}$ will refer to stability conditions in adjacent chambers to $\WW$ in this more refined wall-and-chamber decomposition.
\end{Rem}

Finally, we single-out two types of primitive hyperbolic lattices as the nature of our arguments differ greatly between them:

\begin{Def} We say that $\WW$ is \emph{isotropic} if $\HH_{\WW}$ contains an isotropic class and \emph{non-isotropic} otherwise.
\end{Def}

We begin our investigation by determining precisely the kind of Mukai vectors that can be contained in $\HH_\WW$.
\begin{Prop}\label{Prop:lattice classification}Let $\HH$ be the hyperbolic lattice associated to a wall $\WW$ and $\sigma_0=(Z,\PP_0)\in\WW$ generic.  Then $\HH$ and $\sigma_0$ satisfy one of the following mutually exclusive conditions:
\begin{enumerate}
\item\label{enum:nonegativeclasses} $\HH$ contains no effective spherical or exceptional classes.  
\item\label{enum:OneNegative} \begin{enumerate}
	\item\label{enum:OneSpherical} $\HH$ contains precisely one effective spherical class, and there exists a unique $\sigma_0$-stable spherical object $S$ with $\v(S)\in\HH$.
    \item\label{enum:OneExceptional} $\HH$ contains precisely one effective exceptional class, and there exists exactly two $\sigma_0$-stable exceptional objects $E,E(K_X)$ with $\v(E)=\v(E(K_X))\in\HH$.
    \end{enumerate}
\item\label{enum:TwoNegative} There are infinitely many effective spherical or exceptional classes in $\HH$, and either 
\begin{enumerate}
\item\label{enum:TwoSpherical} there exist exactly two $\sigma_0$-stable spherical objects $S,T$ whose classes are in $\HH$; or
\item\label{enum:TwoExceptional} there exist exactly four $\sigma_0$-stable exceptional objects $E_1,E_1(K_X),E_2,E_2(K_X)$ with $\v(E_1)=\v(E_1(K_X)),\v(E_2)=\v(E_2(K_X))\in\HH$; or
\item\label{enum:OneExceptionalOneSpherical} there exists exactly one $\sigma_0$-stable spherical object $S$ and exactly two $\sigma_0$-stable exceptional objects $E,E(K_X)$ with $\v(S),\v(E)=\v(E(K_X))\in\HH$.
\end{enumerate}
In case \ref{enum:TwoNegative}, $\HH$ is non-isotropic.
\end{enumerate}
\end{Prop}
\begin{proof}
Suppose that $\HH$ contains precisely one effective spherical (resp. exceptional) class $\w$.  Then by Theorem \ref{Thm:exist:nodal}, there exists a unique $\sigma_+$-stable object $S$ with $\v(S)=\w$ (resp. precisely two $\sigma_+$-stable objects $E$ and $E(K_X)$ with $\v(E)=\v(E(K_X))=\w$), which must then be spherical (resp. exceptional) by \cite[Lemma 4.3]{Yos16b}.  Suppose that $S$ (resp. $E$) is strictly $\sigma_0$-semistable.  Then by \cite[Lemma 4.6]{Yos16b} every $\sigma_0$-stable factor $F$ of $S$ (resp. $E\oplus E(K_X)$) must satisfy $\Ext^1(F,F)=0$.  But then $\v(F)^2<0$, so by \cite[Lemma 4.3]{Yos16b} $\v(F)^2=-1$ or $\v(F)^2=-2$ and $c_1(F)\equiv D\pmod 2$ for some nodal cylce $D$, i.e. $\v(F)$ is either an effective spherical or exceptional class.  But this is a contradiction to the assumption, so $S$ is (resp. $E,E(K_X)$ are) $\sigma_0$-stable, giving Case \ref{enum:OneNegative}.

It remains to consider Case \ref{enum:TwoNegative}.  Let $\phi$ be the phase of $\v$ with respect to $\sigma_0$.  It will suffice for our purposes to show that, up to twisting by $K_X$, there cannot be any combination of three stable spherical or exceptional objects $S_1,S_2,S_3$ in $\PP_0(\phi)$.  Since each $S_i$ is $\sigma_0$-stable of the same phase and distinct up to twisting by $K_X$, we must have $\Hom(S_i,S_j)=\Hom(S_j,S_i(K_X))=0$ for each $i\neq j$.  Thus if $\w_i=\v(S_i)$, then $\langle \w_i,\w_j\rangle=\ext^1(S_i,S_j)\geq 0$.

Now any two of the $\w_i$ must be linearly independent, and we may choose, say, $\w_1$ and $\w_2$ to represent either both spherical or both exceptional $\sigma_0$-stable objects.    Denote by $m:=\langle \w_1,\w_2\rangle\geq 0$.  Since $\HH$ has signature $(1,-1)$, $$\langle \w_1,\w_2\rangle^2> \w_1^2\w_2^2=\begin{cases}
1, & \text{ if }\w_1^2=\w_2^2=-1,\\
4, & \text{ if }\w_1^2=\w_2^2=-2.\\
\end{cases}$$
So $m\geq 2$ or $3$.  We write $\w_3=x\w_1+y\w_2$ with $x,y\in\Q$, and from $\langle \w_3,\w_1\rangle,\langle \w_3,\w_2\rangle\geq 0$, we get that \begin{equation}\label{eq:positivity}\begin{cases}
\frac{1}{m}\leq \frac{y}{x}\leq m, & \text{ if }\w_1^2=\w_2^2=-1,\\
\frac{2}{m}\leq \frac{y}{x}\leq \frac{m}{2}, & \text{ if }\w_1^2=\w_2^2=-2.\\  
\end{cases}\end{equation}    But then since $$m-\sqrt{m^2-1}\leq\frac{1}{m}\leq\frac{y}{x}\leq m\leq m+\sqrt{m^2-1}$$ for $m\geq 2$ and $$\frac{m-\sqrt{m^2-4}}{2}\leq\frac{2}{m}\leq\frac{y}{x}\leq\frac{m}{2}\leq\frac{m+\sqrt{m^2-4}}{2}$$ for $m\geq 3$, we see that $$\w_3^2=x^2\begin{cases}
-1+2m(\frac{y}{x})-(\frac{y}{x})^2, &\text{ if }\w_1^2=\w_2^2=-1,\\
-2+2m(\frac{y}{x})-2(\frac{y}{x})^2, &\text{ if }\w_1^2=\w_2^2=-2,
\end{cases}$$ must be positive, in contradiction to the fact that $\w_3^2=-1$ or $-2$.  Thus, we see that, up to tensoring by $K_X$, there can only be at most two $\sigma_0$-stable spherical or exceptional objects with Mukai vectors in $\HH$.  Notice further that if $\HH$ admits any combination of two linearly independent spherical or exceptional classes, then the group generated by the associated spherical and $(-1)$ reflections is infinite, so the orbit of a spherical or exceptional class gives infinitely many Mukai vectors of the same kind. 

Furthermore, we see that solving the quadratic equation $$0=\u^2=(x\w_1+y\w_2)^2$$ gives \begin{equation}\label{eq:isotropic solutions}\frac{y}{x}=\begin{cases}
m\pm\sqrt{m^2-1}, & \text{ if }\w_1^2=\w_2^2=-1,\\

\frac{m\pm\sqrt{m^2-4}}{2}, & \text{ if }\w_1^2=\w_2^2=-2,
\end{cases}\end{equation} which are irrational.  Thus in subcases \ref{enum:TwoSpherical} and \ref{enum:TwoExceptional}, there can be no isotropic classes, as these would give rational solutions in \eqref{eq:isotropic solutions}.

Finally, it only remains to show that in subcase \ref{enum:OneExceptionalOneSpherical} $\HH$ is non-isotropic.  Similar to the previous  subcases, we write an integral isotropic class $\u=x\w_1+y\w_2$ with $x,y\in\Q$, where $\w_1=\v(S)$ is a spherical class and $\w_2=\v(E)=\v(E(K_X))$ is an exceptional class.  Solving the quadratic equation $\u^2=0$ gives $$\frac{y}{x}=m\pm\sqrt{m^2-2},$$ where $m=\langle \w_1,\w_2\rangle$.  But $S$ and $E$ are $\sigma_0$-stable objects of the same phase with classes in a lattice of signature $(1,-1)$, so we must have $m\geq 2$, as in the arguments for the preceeding cases.  But this gives a contradiction as $m^2-2$ cannot be a square for $m\geq 2$.
\end{proof}

All of our main results are a consequence of the following classification theorem, which essentially says that the birational behavior induced by crossing $\WW$ is entirely determined by $\HH_\WW$ and $C_\WW$.
\begin{Thm}\label{classification of walls}
Let $\HH\subset \Hal(X,\Z)$ be a primitive hyperbolic rank two sublattice containing $\v$, and let $\WW\subset\Stabd(X)$ be a potential wall associated to $\HH$.  

The set $\WW$ is a totally semistable wall if and only if one of the following conditions hold:
\begin{description}
\item[(TSS1)] there exists a spherical or exceptional class $\w\in C_\WW\cap\HH$ such that $\langle \v,\w\rangle<0$; 
\item[(TSS2)] there exists an isotropic class $\u\in\HH$ with $\ell(\u)=2$ and $\langle \v,\u\rangle=1$; or
\item[(TSS3)] there exists a primitive isotropic class $\u\in\HH$ such that $\langle \v,\u\rangle=\ell(\u)$ and $\langle \v,\w\rangle=0$ for a spherical $\w\in C_\WW\cap\HH$; or
\item[(TSS4)] there exists a primitive isotropic class $\u\in\HH$ such that $\langle\v,\u\rangle=2=\ell(\u)$ and $\langle\v,\w\rangle=0$ for an exceptional $\w\in C_{\WW}\cap\HH$.
\end{description}
In addition,
\begin{enumerate}
\item\label{thm:Classification,Divisorial} The set $\WW$ is a wall inducing a divisorial contraction if one of the following conditions hold:
\begin{description*}
\item[(Brill-Noether)] there exists a spherical class $\w\in C_\WW\cap\HH$ such that $\langle \w,\v\rangle=0$, or there exists an exceptional class $\w\in C_\WW\cap\HH$ such that $\langle\w,\v\rangle=0$ and $\v-2\w\in\Delta(X)_{-2}\cap C_{\WW}\cap\HH$; or
\item[(Hilbert-Chow)] there exists an isotropic class $\u$ with $\langle \v,\u\rangle=1$ and $\ell(\u)=2$; or
\item[(Li-Gieseker-Uhlenbeck)] there exists a primitive isotropic class $\u\in\HH$ with $\langle \v,\u\rangle=2=\ell(\u)$; or
\item[(induced Li-Gieseker-Uhlenbeck)] there exists an isotropic class $\u\in\HH$ with $\langle \v,\u\rangle=1=\ell(\u)$ and $\v^2\geq 3$.
\end{description*} 
\item\label{thm:Classification,Fibration} The set $\WW$ is a wall inducing a $\P^1$-fibration on $M_{\sigma_+}(\v,L)$ if one of the following conditions hold:
\begin{description*}
\item[(Exceptional case)] there exists a primitive isotropic class $\u$ with $\langle \v,\u\rangle=2=\ell(\u)$, an exceptional class $\w\in C_\WW\cap\HH$ with $\langle \v,\w\rangle=0$, and $L\equiv K_X\pmod 2$, or
\item[(Spherical case)] there exists an isotropic class $\u$ with $\langle \v,\u\rangle=\ell(\u)$, a spherical class $\w\in C_\WW\cap\HH$ with $\langle \v,\w\rangle=0$, and $L\equiv D+\frac{\rk \v}{2}K_X\pmod 2$, where $D$ is a nodal cycle.
\end{description*}
\item\label{thm:Classification,Flops}
Otherwise, if $\v$ is primitive and either
\begin{enumerate}
\item \label{enum:sum2positive}
$\v^2\geq 3$ and $\v$ can be written as the sum 
$\v = \a_1 + \a_2$ with $\a_i\in P_\HH$ such that $L\equiv \frac{r}{2}K_X\pmod 2$ if for each $i$, $\a_i^2=0$ and $\ell(\a_i)=2$; or 
\item\label{enum:exceptional} there exists an exceptional class $\w\in\HH$ and either
\begin{enumerate}
\item\label{enum:exceptionalflop1}
$0< \langle  \w,\v\rangle\leq\frac{\v^2}{2}$, or
\item\label{enum:exceptionalflop2}
$\langle \v,\w\rangle=0$ and $\v^2\geq 3$; or
\end{enumerate}
\item\label{enum:spherical} there exists a spherical class $\w\in \pm (C_\WW\cap\HH)$ and either
\begin{enumerate}
\item\label{enum:sphericalflop1}
$0 < \langle \w, \v\rangle < \frac{\v^2}2$, or
\item\label{enum:sphericalflop2}
$\langle \w,\v\rangle=\frac{\v^2}{2}$ and $\v-\w$ is a spherical class,
\end{enumerate}
\end{enumerate}
then $\WW$ induces a small contraction.
\item In all other cases, $\WW$ is either a fake wall or not a wall at all.
\end{enumerate}

\end{Thm}

The proof of the above theorem will occupy us for the next four sections, but before we enter into a more involved and lattice specific analysis of the wall-crossing behavior, we present a general result on the codimension of the strictly $\sigma_0$-semistable locus corresponding to the simplest Harder-Narasimhan filtration as above:
\begin{Prop}\label{Prop:HN filtration all positive classes}
As above, let $\FF(\a_1,\dots,\a_n)^o$ be the substack of $\MM_{\sigma_+}(\v)$ parametrizing objects with $\sigma_-$ Harder-Narasimhan filtration factors of classes $\a_1,\dots,\a_n$ (in order of descending phase with respect to $\phi_{\sigma_-}$), and suppose that $\a_i^2>0$ for all $i$.  Then $\codim\FF(\a_1,\dots,\a_n)^o\geq 2$. \end{Prop}
\begin{proof}
By Theorem \ref{Thm:exist:nodal}, the assumption that $\a_i^2>0$ implies that $\dim\MM_{\sigma_-}(\a_i)=\a_i^2$.  Thus by Proposition \ref{Prop:HN codim}, $$\codim\FF(\a_1,\dots,\a_n)^o\geq\sum_{i<j}\langle \a_i,\a_j\rangle.$$  But as $\a_i^2\geq 1$ and $\HH$ has signature $(1,-1)$, we must have $$\langle \a_i,\a_j\rangle>\sqrt{\a_i^2 \a_j^2}\geq 1,$$ for $i<j$.  Thus $\langle \a_i,\a_j\rangle\geq 2$.  It follows that $$\codim\FF(\a_1,\dots,\a_n)^o\geq n(n-1)\geq 2,$$ as $n\geq 2$.
\end{proof}
It follows from the proposition that in order for there to be more interesting wall-crossing behavior, $\HH$ must contain some class $\w$ with $\w^2\leq 0$.  We will begin with the non-isotropic case first in the next section.

\section{Totally semistable non-isotropic walls}\label{Sec:TotallySemistable-non-isotropic}
In this section we describe the criterion for a potential non-isotropic wall $\WW$ to be totally semistable, that is, every $E\in M_{\sigma_+}(\v)$ is strictly $\sigma_0$-semistable.  We will see that by applying an appropriate sequence of spherical or weakly spherical twists, we can always reduce to the case of a non-totally semistable wall to study the birational behavior of crossing $\WW$.  Let us begin with a sufficient condition for $\WW$ to be totally semistable in general.  

\begin{Lem}\label{Lem: condition for totally semistable wall}
Let $\WW$ be a potential wall such that $\langle \v,\w\rangle<0$ for an effective spherical or exceptional class $\w\in\HH_{\WW}$.  Then $\WW$ is totally semistable.
\end{Lem}

\begin{proof}
Suppose there were a $\sigma_0$-stable object $E$ of class $\v$.  Let $\tilde{E_0}$ be a $\sigma_0$-semistable object $\tilde{E_0}$ with $\v(\tilde{E_0})=\w$.  As all stable factors of $\tilde{E_0}$ are spherical or exceptional \cite[Lemma 4.3, Lemma 4.6]{Yos16b}, we may find a $\sigma_0$-stable object $E_0$ such that $\langle \v,\v(E_0)\rangle<0$ and $\v(E_0)^2=-1$ or $-2$.  As $E$ and $E_0$ (resp. $E$ and $E_0(K_X)$) are non-isomorphic $\sigma_0$-stable objects of the same phase, we must have $\Hom(E,E_0)=\Hom(E_0(K_X),E)=0$.  But then $0>\langle \v,\v(E_0)\rangle=\ext^1(E,E_0)\geq 0$, a contradiction.
\end{proof}

For a non-isotropic wall, the condition in Lemma \ref{Lem: condition for totally semistable wall} is actually necessary as we see now.

\begin{Lem}\label{Lem:non-isotropic no totally semistable wall}Suppose that $\HH$ is non-isotropic and $\langle \v,\w\rangle\geq 0$ for all spherical or exceptional classes $\w\in C_\WW\cap\HH$.  Then $\WW$ cannot be a totally semistable wall, and if $\codim(M_{\sigma_+}(\v)\backslash M^s_{\sigma_0}(\v))=1$ then $\langle \v,\w\rangle=0$ for some spherical or exceptional class $\w\in C_\WW\cap\HH$ or $\langle\v,\w\rangle=1=\langle\v,\v-\w\rangle$ for spherical classes $\w,\v-\w\in C_\WW\cap\HH$.  Moreover, the generic member of $M_{\sigma_+}(\v)\backslash M^s_{\sigma_0}(\v)$ has HN-filtration factors of classes $\w$ and $\v-\w$ with respect to $\sigma_-$-stability.
\end{Lem}
\begin{proof}
Consider the stack $\FF(\a_1,\dots,\a_n)^o$ of Harder-Narasimhan filtrations with respect to $\sigma_-$-stability as in Section \ref{sec:DimensionsOfHarderNarasimhan}.  We wish to estimate the codimension of $\FF(\a_1,\dots,\a_n)^o$.    

Suppose that $I=\Set{i\ |\ \a_i^2>0}\neq\varnothing$ and let $\a:=\sum_{i\in I}\a_i$.  Write $\b:=\v-\a$.  If $\b^2>0$, then we automatically have $\langle \a,\b\rangle>\sqrt{\a^2 \b^2}\geq 1$, so $\v^2=\a^2+2\langle \a,\b\rangle+\b^2\geq \a^2+5$.  If, instead $\b^2<0$, then note that $\b$ is the sum of effective spherical and/or exceptional classes.  Indeed, for $i\notin I$, $\a_i^2<0$ and $M_{\sigma_-}(\a_i)\ne\varnothing$, so $\a_i$ is the sum of effective spherical and/or exceptional classes by \cite[Lemma 4.3, Lemma 4.6]{Yos16b}, so the same holds for $\b=\sum_{i\notin I}\a_i$.  By the assumption that $\v$ pairs non-negatively with all effective spherical and exceptional classes, we see that $\langle\v,\b\rangle\geq 0$, so $\a^2=\v^2-2\langle \v,\b\rangle+\b^2\leq \v^2-1$.  

In any case we have $\a^2<\v^2$, and we write $\a^2=\v^2-k$ with $k\in\Z_{\geq 0}$ and $k>0$ if $\v\neq \a$.  Expanding the squares on each side gives $$0=\sum_{i\in I^c}\a_i^2+2\sum_{i<j,(i,j)\in (I^2)^c}\langle \a_i,\a_j\rangle-k,$$ and rearranging gives \begin{equation}\label{non-isotropic estimate}
\sum_{i<j,(i,j)\in (I^2)^c}\langle \a_i,\a_j\rangle =\frac{k}{2} -\frac{1}{2}\sum_{i\in I^c}\a_i^2.\end{equation}

By Proposition \ref{Prop:HN codim},
we get that \begin{align}\label{codim estimate}
\begin{split}\codim\FF(\a_1,\dots,\a_n)^o&\geq\sum_i (\a_i^2-\dim\MM_{\sigma_-}(\a_i))+\sum_{i<j}\langle \a_i,\a_j\rangle\\
&=\sum_{i\in I^c}(\a_i^2-\dim\MM_{\sigma_-}(\a_i))+\sum_{i<j,(i,j)\in I^2}\langle \a_i,\a_j\rangle+\sum_{i<j,(i,j)\in (I^2)^c}\langle \a_i,\a_j\rangle, 
\end{split}
\end{align}
since $\dim\MM_{\sigma_-}(\a_i)=\a_i^2$ for $i\in I$.  Using \eqref{non-isotropic estimate} and writing $\a_i=m_i \w_i$ with $\w_i$ primitive for $i\in I^c$, we have
\begin{align}\label{eqn: non-isotropic codimension}
\begin{split}
\codim\FF(\a_1,\dots,\a_n)^o&\geq\frac{k}{2}+\sum_{i\in I^c}\left(\frac{\a_i^2}{2}-\dim\MM_{\sigma_-}(\a_i)\right)+\sum_{i<j,(i,j)\in I^2}\langle \a_i,\a_j\rangle\\
&=\frac{k}{2}+\sum_{i\in I^c:\w_i^2\equiv m_i\equiv 1\pmod 2}\frac{1}{2}+\sum_{i<j,(i,j)\in I^2}\langle \a_i,\a_j\rangle,
\end{split}
\end{align}
where the final equality follows from Lemma \ref{Lem:dimension negative}.  Moreover, for $(i,j)\in I^2$ with $i\neq j$, the signature of $\HH$ forces $\langle \a_i,\a_j\rangle\geq 2$.  Thus $$\codim\FF(\a_1,\dots,\a_n)^o \geq \frac{k}{2}+|I|(|I|-1)>0.$$  As this holds true for all possible HN-filtrations of objects in $M_{\sigma_+}(\v)$ with respect to $\sigma_-$-stability, $\WW$ cannot be totally semistable.  Moreover, note that if $|I|=0$, so that $\a=0$, the estimate in \eqref{eqn: non-isotropic codimension} is still valid.  In that case, we must have $k=\v^2>0$ from which we see that $\codim\FF(\a_1,\dots,\a_n)^o >0$.

For the second and third claim, we note that if $\codim\FF(\a_1,\dots,\a_n)^o=1$, then from \eqref{eqn: non-isotropic codimension} we see that $|I|=0,1$ and $k=1,2$.  If $|I|=1$, then $\b=\v-\a$ must satisfy $\b^2<0$ (otherwise $k\geq 5$ as we saw above), and solving for $\langle \v,\b\rangle$ in $\v^2-k=(\v-\b)^2$ gives $$0\leq 2\langle \v,\b\rangle=\b^2+k\leq \b^2+2<2,$$ so $\langle \v,\b\rangle=0$ and $\b^2=-k$.  Thus $\b$ is an effective spherical or exceptional class orthogonal to $\v$.  Moreover, as $\b=\sum_{i\notin I}\a_i$, it follows from the assumption on $\v$ and $\langle\b,\v\rangle=0$ that $n=2$ and $\b=\a_1$ or $\a_2$. Letting $\w=\b$, we get the claim about the HN-filtration of the generic member of $M_{\sigma_+}(\v)\backslash M_{\sigma_0}^s(\v)$.

If instead $|I|=0$, then $k=\v^2=1,2$, and we rearrange the $\a_i$, if necessary, so that $\langle\v,\w_i\rangle\leq\langle\v,\w_{i+1}\rangle$ for all $i$.   Let us assume first that $\v^2=2$, and note that it follows from \eqref{eqn: non-isotropic codimension} that $m_i$ must be even if $\w_i^2=-1$.
If $\langle \v,\w_1 \rangle=0$, then as $(\v-\w_1)^2=2+\w_1^2$ and $\HH$ is non-isotropic, we must have $\w_1^2=-1$.  It is easy to see that then $\HH=\Z \v \perp \Z \w_1$.
For $\w_j$ ($j \ne 1$), we set $\w_j=x_j \v+y_j \w_1$ for $x_j,y_j\in\Z$.
Then $0 \leq \langle \v,\w_j \rangle=2x_j$.
Since $\v=\sum_{j=1}^n m_j\w_j$, it follows that $1=\sum_{j \ne } m_j x_j$, so $n=2$ and
$m_2=x_2=1$. Hence $\v=m_1 \w_1+\w_2$.
Since $2-m_1^2=\w_2^2$, $m_1=2$ and $\w_2^2=-2$.

If $\langle \v,\w_i \rangle>0$ for all $i$, then from $$2=\v^2=\sum_{i=1}^n m_i\langle\v,\w_i\rangle,$$ we see that $n=2$ and
$\langle \v,\w_1 \rangle=\langle \v,\w_2 \rangle=m_1=m_2=1$.
Thus $\w_1^2=(\v-\w_2)^2=\v^2-2\langle\v,\w_2\rangle+\w_2^2=\w_2^2$.
As the $m_i$ are odd, we must have $\w_i^2=-2$, which implies $\langle \w_1,\w_2 \rangle=3$.

Now assume that $\v^2=1$.
Then from $$1=\v^2=\sum_{i=1}^n m_i\langle\v,\w_i\rangle,$$ we see that $n=2$ and $\v=m_1 \w_1+m_2 \w_2$ with $\langle \v,\w_1 \rangle=0$ and $\langle \v,\w_2 \rangle=m_2=1$.
Since $(\v-\w_1)^2=1+\w_1^2$ and $\HH$ is not isotropic, we must have $\w_1^2=-2$.
Then $$\v^2-2m_1^2=(\v-m_1\w_1)^2=\w_2^2$$ implies $\w_2^2=-1$ and $m_1=1$.
Hence $\langle \w_1,\w_2 \rangle=2$.
\end{proof}

The previous lemma justifies singling out Mukai vectors which pair non-negatively with all effective spherical or exceptional classes, and we will spend the remainder of this section using the theory of Pell's equation to reduce to this case.  Such a Mukai vector is called \emph{minimal in its $G_{\HH}$-orbit}, or simply \emph{minimal} for short, because of the following definition.

\begin{PropDef}\label{PropDef: minimal vectors}
Let $G_{\HH}\subset\Aut(\HH)$ be the group generated by spherical and exceptional reflections associated to effective spherical and exceptional classes in $C_{\WW}$.  For a given positive class $\v\in P_{\HH}\cap\HH$, the $G_{\HH}$-orbit of $\v$ contains a unique class $\v_0$ such that $\langle \v,\w\rangle\geq 0$ for all effective spherical and exceptional classes $w\in C_{\WW}$.  We call $\v_0$ the minimal class of the orbit $G_{\HH}\cdot\v$.
\end{PropDef}

The proof of the existence of $\v_0$ is almost identical to that of \cite[Proposition and Definition 6.6]{BM14b}, so we omit it.  In the remainder of this section we will explore the action of $G_\HH$ more fully.  We will consider explicitly only Cases \ref{enum:TwoExceptional} and \ref{enum:OneExceptionalOneSpherical}, as Case \ref{enum:TwoSpherical} is covered in detail in \cite[Section 6]{BM14b}.  Furthermore, our discussion is clearly irrelevant to Case \ref{enum:nonegativeclasses}, and applies in a much simpler but completely analogous way in Case \ref{enum:OneNegative}.  

For simplicity, we observe that up to the action of $\widetilde{\GL}_2^+(\R)$, we may assume that the phase of all objects in $C_\WW$ is 1.  We make this assumption throughout the rest of this section.

\subsection{Stability}
In order to study the behavior of stability under $G_\HH$,
we must study the interplay between rotating a stability function and tilting.  

Let $\BB$ be an abelian category with a stability function $Z:\Db(\BB) \to \C$ such that $Z(\BB \setminus \{0 \})= \H \cup \R_{<0}$.  We set $\phi(E):=\frac{1}{\pi}\arg Z(E)$.  
\begin{Def}
$E \in \BB$ is $Z$-semistable if $\phi(F) \leq \phi(E)$ for all subobjects $F$ of $E$ in $\BB$.
\end{Def}
We assume that $Z$ satisfies the HN-filtration property and that the category $$\Set{E\in\BB\ | \ E\mbox{ is $Z$-semistable of phase $\phi$}}$$
is finite so that JH-filtrations exist, see \cite[Sections 3 and 4]{Bri08}. 
Now we recall the definition of a torsion pair:
\begin{Def}[{\cite{HRS96}}]\label{defn:TorsionPair}
A torsion pair in an abelian category $\BB$ is a pair of full subcategories $(\TT, \FF)$ of $\BB$ which satisfy $\Hom(T,F) = 0$ for $T\in\TT$ and $F\in\FF$, and such that every object $E\in\BB$ fits into a short exact sequence
$$0\to T\to E\to F\to 0$$ for some pair of objects $T\in\TT$ and $F\in\FF$.  We write $T=\TT(E)$ and $F=\FF(E)$.
\end{Def}
One often constructs torsion pairs via HN-filtrations as in the following definition of the rotation of a stability function $Z$ and the tilt of corresponding torsion pair.
\begin{Def}\label{defn:B'}
For a real number $\theta \in (-1,1)$, 
we set $Z_\theta:=e^{\pi \sqrt{-1} \theta} Z$.
\begin{enumerate}
\item[(1)]
If $\theta \geq 0$, then
let $(\TT',\FF')$ be the torsion pair of $\BB$
such that
\begin{enumerate}
\item
$\TT'$ is generated by $Z$-stable objects $E \in \BB$ with
$\phi(E)+\theta>1$. 
\item
$\FF'$ is generated by $Z$-stable objects $E \in \BB$ with
$\phi(E)+\theta \leq 1$. 
\end{enumerate}
We set $\BB':=\langle \TT'[-1],\FF' \rangle$
and $Z':=Z_\theta$.  We denote the corresponding phase by $\phi'$.
\item[(2)]
If $\theta \leq 0$, then
let $(\TT'',\FF'')$ be the torsion pair of $\BB$
such that
\begin{enumerate}
\item
$\TT''$ is generated by $Z$-stable objects $E \in \BB$ with
$\phi(E)+\theta>0$. 
\item
$\FF''$ is generated by $Z$-stable objects $E \in \BB$ with
$\phi(E)+\theta \leq 0$. 
\end{enumerate}
We set $\BB'':=\langle \TT'',\FF''[1] \rangle$
and $Z'':=Z_\theta$.  We denote the corresponding phase by $\phi''$.
\end{enumerate}
\end{Def}
We determine now precisely when an object in the tilted category is semistable with respect to the rotated stability function, generalizing \cite[Lemma 6.10]{BM14b}.
\begin{Prop}\label{Prop:BB}
\begin{enumerate}
\item\label{enum:B'}
Assume that $E\in \BB'$.  Then $E$ is $Z'$-semistable if and only if 
\begin{enumerate}
\item
 $E \in \FF'$
and $E$ is $Z$-semistable or 
\item
 $E[1] \in \TT'$ and $E[1]$ is $Z$-semi-stable.
\end{enumerate}
\item\label{enum:B''}
Assume that $E\in \BB''$.  Then $E$ is $Z''$-semistable if and only if 
\begin{enumerate}
\item
 $E[-1] \in \FF''$ and $E[-1]$ is $Z$-semistable or
\item
 $E \in \TT''$ and
$E$ is $Z$-semistable.
\end{enumerate}
\end{enumerate}
\end{Prop}
Although this result might be known to experts, we provide a proof for the sake of convenience.
\begin{proof}
We prove \ref{enum:B'} first.  Note that $\phi'(E_1)>\phi'(E_2)$ for all $E_1 \in \FF'$ and $E_2 \in \TT'[-1]$.  Then, if $E\in\BB'$ is $Z'$-semistable, it follows from the canonical exact sequence in $\BB'$,
\begin{equation}
0 \to \HH^0(E) \to E \to \HH^1(E)[-1] \to 0,
\end{equation}
that either $\HH^0(E)=0$ or $\HH^1(E)=0$.

Assume first that $\HH^1(E)=0$, that is, $E \in \FF'$.  Let $E_1$ be a subobject of $E$ in $\BB$.  By considering the HN-filtration of $E/E_1$, we see that there is a subobject $E_1' \subset E$ in $\BB$ such that $E_1 \subset E_1'$, $E_1'/E_1 \in \TT'$ and $E/E_1' \in \FF'$.  Since $E/E_1'$ is a quotient object of $E$ in $\BB'$, $Z'$-semistability of $E$ implies $$\phi(E)+\theta=\phi'(E) \leq \phi'(E/E_1')=\phi(E/E_1')+\theta.$$  In particular, $\phi(E)\leq\phi(E/E_1')$.  From the definitions of $\TT'$ and $\FF'$ and our choice of $E_1'$, we see that $\phi(E_1'/E_1) \geq \phi(E/E_1) \geq \phi(E/E_1')$, so it follows that $\phi(E) \leq \phi(E/E_1)$.  Therefore $E$ is $Z$-semistable.

Now assume that $\HH^0(E)=0$, that is, $E[1] \in \TT'$.  Let $E_1$ be a subobject of $E[1]$ in $\BB$.  Then there is a subobject $E_1' \subset E_1$ in $\BB$ such that $E_1' \in \TT'$ and $E_1/E_1' \in \FF'$.  Since $E[1]/E_1' \in \TT'$, $E_1'$ is a subobject of $E[1]$ in $\BB'[1]$.  By the $Z'$-semi-stability of $E$, $$\phi(E_1')+\theta=\phi'(E_1') \leq \phi'(E[1])=\phi(E[1])+\theta,$$ so in particular $\phi(E_1')\leq\phi(E[1])$.  Since $\phi(E_1') \geq \phi(E_1) \geq \phi(E_1/E_1')$, from the definition of $\TT'$ and $\FF'$,  we see that $\phi(E_1) \leq \phi(E[1])$.  Therefore $E$ is $Z$-semistable.

Next we shall prove the converse direction.  Let $E$ be a $Z$-semistable object of $\BB$.  We first assume that $E \in \FF'$.  Taking the long exact sequence associated to a given short exact sequence in $\BB'$,
\begin{equation}
0 \to E_1 \to E \to E_2 \to 0,
\end{equation}
 we get
\begin{equation*}
0 \to \HH^0(E_1) \to E \overset{\varphi}{\to} \HH^0(E_2)
\to \HH^1(E_1) \to 0.
\end{equation*}
Then $\phi(\im \varphi) \leq \phi(E_2)$ by $\HH^1(E_1)\in \TT'$.  By the $Z$-semistability of $E$, $\phi(E) \leq \phi(\im \varphi)$ so that $$\phi'(E)-\theta=\phi(E) \leq\phi(E_2)=\phi'(E_2)-\theta.$$  Therefore $E$ is $Z'$-semistable.

We next assume that $E \in \TT'$.  Again we take the long exact sequence associated to a short exact sequence in $\BB'$,
\begin{equation}
0 \to E_1 \to E[-1] \to E_2 \to 0,
\end{equation}
and we get the exact sequence
\begin{equation}
0 \to \HH^0(E_2) \to \HH^1(E_1) \overset{\varphi}{\to} E \to \HH^1(E_2) \to 0.
\end{equation}
Then $\phi(\im \varphi) \geq \phi(H^1(E_1))$ by $\HH^0(E_2) \in \FF'$.  From the $Z$-semistability of $E$, it follows that $\phi(E) \geq \phi(\im \varphi)$, and thus we have $$\phi'(E[-1])-\theta=\phi'(E)-\theta=\phi(E) \geq\phi(H^1(E_1))=\phi(E_1[1])=\phi'(E_1[1])-\theta=\phi'(E_1)-\theta.$$  Therefore $E$ is $Z'$-semi-stable, as required.

To prove \ref{enum:B''}, we note that for the abelian category $\BB[1]$ with 
the stability function 
$$
Z^*:\BB[1] \overset{[-1]}{\to} \BB \overset{Z}{\to}\C,
$$  
we have
$(\BB[1])'=\BB''$, where $Z_{1+\theta}^*=Z_\theta$.
Hence the claim follows from \ref{enum:B'}.
\end{proof}

We will use Proposition \ref{Prop:BB} to study how stability is effected under the action of $G_\HH$.  To do so, we must investigate Cases \ref{enum:TwoExceptional} and \ref{enum:OneExceptionalOneSpherical} separately.
\subsection{\ref{enum:TwoExceptional}: Exactly two $\sigma_0$-stable exceptional objects up to $-\otimes\OO_X(K_X)$}\label{subsec:TwoExceptional}

Recall from Proposition \ref{Prop:lattice classification} that in this case $\HH$ contains infinitely many exceptional classes, precisely two of which represent $\sigma_0$-stable objects (up two tensoring with $\OO_X(K_X)$).  Denote one of these classes by $\w_0$.

We may complete $\w_0$ to a basis so that $\HH=\Z\w_0+\Z\z$,
where $\langle\w_0,\z\rangle=0$ and $D:=\z^2>0$.
Recalling the definition of $\Delta(X)_{-1}$ in \eqref{eqn:DefOfEnriquesRootsSep}, we see that $\Delta(X)_{-1}\cap\HH$ is described by the Pell equation
\begin{equation}\label{eq:Pell}
x^2-Dy^2=1.
\end{equation}
Note that $\sqrt{D}$ must be irrational since $\HH$ would be isotropic otherwise, in contradiction to Case \ref{enum:TwoNegative} of Proposition \ref{Prop:lattice classification}.

Recall, for example from \cite[Theorem 8.6]{Lev96}, that the solutions to Pell's equation \eqref{eq:Pell} form a group isomorphic to $\Z\oplus\Z/2\Z$.  Indeed, let $(p_1,q_1)$ be the fundamental solution of 
\eqref{eq:Pell} with $p_1<0$ and $q_1>0$, and define $p_n, q_n \in\Z$ by
\begin{equation}
p_n+q_n \sqrt{D}=
\begin{cases}
-(-p_1-q_1 \sqrt{D})^n, & n > 0\\
(-p_1-q_1 \sqrt{D})^n, & n \leq 0
\end{cases}.
\end{equation}
Then setting $\w_n:=p_n\w_0+q_n\z$, we see that all solutions to \eqref{eq:Pell} are given by 
$$
\Delta(X)_{-1}\cap\HH=\Set{\pm\w_n \ | \ n \in \Z}.
$$ 
Recall that to an exceptional object $E_0\in\Db(X)$, we get a weakly spherical reflection $R_{E_0}$ as in 
\eqref{eqn:weakly spherical reflection}.  We abuse notation by using $R_{\v(E_0)}$ to denote the action on cohomology, which is given by $$R_{\v(E_0)}(\v)=\v+2\langle\v,\v(E_0)\rangle\v(E_0).$$  Then it is easy to see that
\begin{equation}\label{eqn:ExceptionalVectorsReflections}
\begin{split}
\w_{n+1}=&-R_{\w_n}(\w_{n-1}),\; (n \geq 2)\\
\w_{n-1}=&-R_{\w_n}(\w_{n+1}),\; (n \leq -1)\\
\w_2=&R_{\w_1}(\w_0),\;\; \w_{-1}=R_{\w_0}(\w_1).
\end{split}
\end{equation}

\begin{figure}
   \begin{tikzpicture}[scale=1]
   \draw [->] (-4,0) -- (4,0);
   \draw[->] (0,-3) -- (0,3);
   \path [fill=gray!50,opacity=0.2] (-4,3) -- (-4,.8) -- (4,-.8) -- (4,3) -- cycle;
   \draw[gray,domain=-4:4] plot (\x,{-.2*\x});
           \draw [red,domain=-4:4] plot (\x,{(1/sqrt(2))*\x});
   \draw [red,domain=-4:4] plot (\x,{-(1/sqrt(2))*\x});
   \draw [blue,domain=-4:-1] plot (\x,{sqrt(.5*(pow(\x,2)-1))});
   \draw [blue,domain=1:4] plot (\x,{sqrt(.5*(pow(\x,2)-1))});
   \draw [blue,domain=-4:-1] plot (\x,{-sqrt(.5*(pow(\x,2)-1))});
   \draw [blue,domain=1:4] plot (\x,{-sqrt(.5*(pow(\x,2)-1))});
   \filldraw [gray] (1,0) circle (1.5pt) node [anchor=south west] {$\w_0$};
   \filldraw [gray] (-1,0) circle (1.5pt) node [anchor=north east] {$-\w_0$};
   \filldraw [gray] (2.00149,-1.21698) circle (1.5pt);
   \filldraw [gray] (1.43593,-0.733002) circle (1.5pt) node [anchor=west] {$-\w_1$};
   \filldraw [gray] (2.00149,1.21698) circle (1.5pt) node [anchor=west] {$\w_{-2}$};
   \filldraw [gray] (-1.43593,0.733002) circle (1.5pt) node [anchor=east] {$\w_1$};
   \filldraw [gray] (-2.00149,1.21698) circle (1.5pt) node [anchor=east] {$\w_2$};
   \filldraw [gray] (1.43593,0.733002) circle (1.5pt) node [anchor=west] {$\w_{-1}$};
	\node[below] at (-3.5,0.75) {$Z^{-1}(0)$};
	\node[left] at (0,2.5) {$\Re\frac{Z_{\sigma_0}(\blank)}{Z_{\sigma_0}(\v)}>0$};
	\node[left] at (0,-2.5) {$\Re\frac{Z_{\sigma_0}(\blank)}{Z_{\sigma_0}(\v)}<0$};
   \end{tikzpicture}
   \caption{The shaded gray area is the half plane containing $\u$ such that $\Re\frac{Z_{\sigma_0}(\u)}{Z_{\sigma_0}(\v)}>0$, which is bounded by the line $Z^{-1}(0)$. The hyperbola is defined by $\u^2=-1$, and the lines by $\u^2=0$.}
   \label{fig:TwoExceptional}
\end{figure}
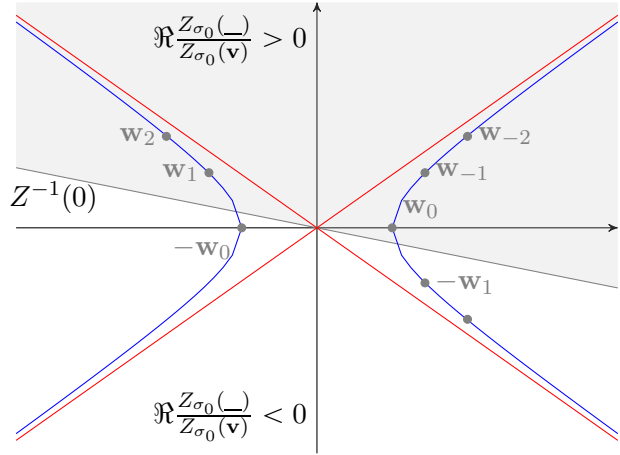

Since $Z^{-1}(0) \cap \HH_\R$ is negative definite and $\lim_{n \to \pm \infty}\frac{p_n}{q_n}=\mp \sqrt{D}$, it is not difficult to see that for $\w'=-\w_1$, $Z^{-1}(0) \cap (\R_{>0}\w_0+\R_{>0}\w') \ne \varnothing$ and $\Delta(X)_{-1} \cap (\R_{>0}\w_0+\R_{>0}\w')= \varnothing$, see Figure \ref{fig:TwoExceptional}.  Then $\Re\frac{Z_{\sigma_0}(-\w')}{Z_{\sigma_0}(\v)}>0$, and since $\w_1=-\w'$ we get
$$
\Set{\u \in \HH \ | \ \Re\frac{Z_{\sigma_0}(\u)}{Z_{\sigma_0}(\v)}>0, \u^2 \geq -1}\subset\Q_{\geq 0}\w_0+\Q_{\geq 0}\w_1.
$$
It follows that $C_\WW=\R_{>0}\w_0+\R_{>0}\w_1$.  Furthermore, the positive cone can be described as 
$$
P_{\HH}=\Set{x\w_0+y\z \in \HH \ |\ y^2 D-x^2>0, y \geq 0 }.
$$
It follows that $\w_n\in C_{\WW}$ for all $n\in\Z$.  Let $T_0$ and $T_1$ be $\sigma_0$-semistable objects with $\v(T_i)=\w_i$ for $i=0$ and $1$, respectively.  By construction, $T_0$ is $\sigma_0$-stable, and we claim that $T_1$ is as well.  Indeed, by \cite[Lemma 4.3, Lemma 4.6]{Yos16b}, any $\sigma_0$-stable factor $\tilde{T}$ of $T_0$ must be exceptional so that $\v(\tilde{T})\in\Delta(X)_{-1}\cap C_{\WW}$.  In particular, by the description in \eqref{eqn:ExceptionalVectorsReflections}, $\v(\tilde{T})=a\w_0+b\w_1$, where $a$ and $b$ are nonnegative integers.  This gives a contradiction unless $\v(\tilde{T})=\w_1$ so that $T_1$ is $\sigma_0$-stable, as claimed.

We set
\begin{equation}
\CC_n:=\Set{x \w_0+y\z \ |\
 \frac{Dq_{n+1}}{p_{n+1}}y<x<\frac{Dq_n}{p_n}y, y>0}.
\end{equation}
Note that for $n<0$, we have 
\begin{equation}
    \CC_n=\Set{\u\in C_\WW\ | \ \langle\u,\w_{n+1}\rangle<0<\langle\u,\w_n\rangle},
\end{equation}
and for $n>0$, 
\begin{equation}
    \CC_n=\Set{\u\in C_\WW\ | \ \langle\u,\w_{n+1}\rangle>0>\langle\u,\w_n\rangle},
\end{equation}
while for $n=0$,
\begin{equation}
    \CC_n=\Set{\u\in C_\WW\ | \ 0<\langle\u,\w_n\rangle,\langle\u,\w_{n+1}\rangle}.
\end{equation}
 Then $\{\CC_n \mid n \in \Z \}$ is the chamber decomposition
of $P_{\HH}$ under the action of $G_\HH$.

For $\v_0 \in \CC_0$, 
we set 
\begin{equation}\label{eqn:OrbitOfv0}
\v_n:= 
\begin{cases}
R_{\w_n} \circ R_{\w_{n-1}} \circ \cdots \circ R_{\w_1}(\v_0), & n>0\\
R_{\w_{n+1}}^{-1} \circ R_{\w_{n+2}}^{-1} 
\circ \cdots \circ R_{\w_0}^{-1}(\v_0), & n < 0.
\end{cases}
\end{equation}
Then for $\v \in \CC_n$, there is $\v_0 \in \CC_0$ such that
$\v=\v_n$ for this $\v_0$.

\subsubsection{The abelian categories $\AA_i$}

Up to reordering, we may assume that $\phi^+(T_1)>\phi^+(T_0)$ (and hence $\phi^-(T_1)<\phi^-(T_0)$), where $\phi^{\pm}$ denotes the phase with respect to $\sigma_{\pm}$, respectively.  For $i\in\Z$, let $T_i^\pm\in\PP_0(1)$ be $\sigma^\pm$-stable objects with
$\v(T_i^\pm)=\w_i$.
Then 
\begin{equation}
\phi^+(T_1^+) > \phi^+(T_2^+)>\cdots>\phi^+(E)>
\cdots >\phi^+(T_{-1})>\phi^+(T_0^+)
\end{equation}
for any $\sigma_+$-stable object $E$ with
$\v(E)^2 \geq 0$.
We note that $T_i^+=T_i^-=T_i$ $(i=0,1)$ are $\sigma_0$-stable objects.

We make the following definition which generalizes the approach of \cite[Lemma 6.8]{BM14b}.
\begin{Def}
Assume that $i  \geq 0$.
\begin{enumerate}
\item[(1)]
Let $(\TT_i,\FF_i)$ be the torsion pair of $\PP_0(1)$ such that
\begin{enumerate}
\item
$\TT_i=\langle T_1^+,T_1^+(K_X),T_2^+,T_2^+ (K_X),...,
T_i^+,T_i^+ (K_X) \rangle$ 
is the subcategory of $\PP_0(1)$ generated by $\sigma_+$-stable objects
$F$ with $\phi^+(F)>\phi^+(T_{i+1}^+)$ and
\item
$\FF_i$ is the subcategory of $\PP_0(1)$ generated by 
$\sigma_+$-stable objects $F$
with $\phi^+(F) \leq \phi^+(T_{i+1}^+)$.
\end{enumerate}
Let $\AA_i:=\langle \TT_i[-1],\FF_i \rangle$ be the tilting.
\item[(2)]
Let $(\TT_i^*,\FF_i^*)$ be the torsion pair of $\PP_0(1)$ such that
\begin{enumerate}
\item
$\TT_i^*$ is the subcategory of $\PP_0(1)$ generated by 
$\sigma_-$-stable objects $F$
with $\phi^-(F) \geq \phi^-(T_{i+1}^-)$.
\item
$\FF_i^*=\langle T_1^-,T_1^-(K_X),T_2^-,T_2^-(K_X),...,
T_i^-,T_i^-(K_X) \rangle$ 
is the subcategory of $\PP_0(1)$ generated by $\sigma_-$-stable objects
$F$ with $\phi^-(F)<\phi^-(T_{i+1}^-)$.
\end{enumerate}
Let $\AA_i^*:=\langle \TT_i^*,\FF_i^*[1] \rangle$ be the tilting.
\end{enumerate}
\end{Def}
Since
$\TT_0=0$ and $\FF_0^*=0$,
we have $\AA_0=\AA_0^*=\PP_0(1)$.

\begin{Rem}\label{rem:simple-objects}
We note that from the definition of $\AA_i$ (resp. $\AA_i^*$), it follows that:
\begin{enumerate}
\item 
$T_{i+1}^+,T_{i+1}^+(K_X), T_i^+[-1],T_i^+(K_X)[-1]$ are irreducible objects of $\AA_i$.
\item
$T_{i+1}^-,T_{i+1}^-(K_X), T_i^-[1],T_i^-(K_X)[1]$ are irreducible objects of $\AA_i^*$.
\end{enumerate}
\end{Rem}
With these notions in place, we prove the following stronger form of the induction claim in \cite[p. 541]{BM14b}.

\begin{Prop}\label{Prop:equiv1}
For $i\geq 0$, $R_{T_{i+1}^+}$ induces an equivalence
$\AA_i \to \AA_{i+1}$.
\end{Prop}

\begin{proof}
We set $\Phi:=R_{T_{i+1}^+}$ and $\Phi^p(E):=\HH^p(\Phi(E))$ for $E \in \Db(X)$.  

We first prove that 
$R_{T_{i+1}^+}(\AA_i) \subset \AA_{i+1}$ by showing that $\Phi(\FF_i)\subset\AA_{i+1}$ and $\Phi(\TT_i)\in\AA_{i+1}[1]$.  Then the claim follows for a general $E\in\AA_i$ from the short exact sequence 
\begin{equation}\label{eqn:canonical short exact sequence}
0\to \HH^0(E)\to E\to\HH^1(E)[-1]\to 0,
\end{equation}
as $\HH^0(E)\in\FF_i$ and $\HH^1(E)\in\TT_i$.

Observe first that for $E \in \PP_0(1)$, 
$\Ext^p(T_{i+1}^+,E)=\Ext^p(T_{i+1}^+(K_X),E)=0$ for $p \ne 0,1,2$.
Hence, from the definition of $R_{T_{i+1}^+}$ in \eqref{eqn:weakly spherical reflection}, we have an exact sequence
\begin{equation}
\begin{CD}
0 @>>> \Phi^{-1}(E) @>>> \Hom(T_{i+1}^+,E) \otimes T_{i+1}^+ \oplus  
\Hom(T_{i+1}^+(K_X),E) \otimes T_{i+1}^+ (K_X) @>{\varphi}>> E \\
 @>>> \Phi^0(E) @>>> \Ext^1(T_{i+1}^+,E) \otimes T_{i+1}^+ \oplus  
\Ext^1(T_{i+1}^+(K_X),E) \otimes T_{i+1}^+ (K_X) @>>> 0
\end{CD},
\end{equation}
and also an isomorphism
\begin{equation}
\Phi^1(E) \cong \Ext^2(T_{i+1}^+,E) \otimes T_{i+1}^+\oplus  
\Ext^2(T_{i+1}^+(K_X),E) \otimes T_{i+1}^+ (K_X) \in \TT_{i+1}.
\end{equation}

Suppose first that $E \in \FF_i$ so that
$\phi_{\max}^+(E) \leq \phi^+(T_{i+1}^+)$. Then it follows that 
$\varphi$ is injective (so that $\Phi^{-1}(E)=0$) and $\coker \varphi \in \FF_i$.  Thus $\Phi^0(E) \in \FF_i$, but in fact we can say more.  Noting that $\Phi(T_{i+1}^+(K_X))=T_{i+1}^+[-1]$ (and similarly that $\Phi(T_{i+1}^+)=T_{i+1}^+(K_X)[-1]$), we get
\begin{equation}
\begin{split}
\Hom(T_{i+1}^+,\Phi(E))& =\Hom(\Phi(T_{i+1}^+(K_X))[1],\Phi(E))
=\Hom(T_{i+1}^+(K_X),E[-1])=0,\\
\Hom(T_{i+1}^+(K_X),\Phi(E))&=\Hom(\Phi(T_{i+1}^+)[1],\Phi(E))
=\Hom(T_{i+1}^+,E[-1])=0,
\end{split}
\end{equation}
where the final equality follows from $\phi_{\max}^+(E[-1])=\phi_{\max}^+(E)-1<\phi^+(T_{i+1}^+)=\phi^+(T_{i+1}^+(K_X))$.  From the triangle $$\Phi^0(E)\to\Phi(E)\to\Phi^{\geq 1}(E)[-1]\to\Phi^0(E)[1]$$ we see that this implies $\Hom(T_{i+1}^+,\Phi^0(E))=\Hom(T_{i+1}^+(K_X),\Phi^0(E))=0$ so that $\Phi^0(E) \in \FF_{i+1}$.  Therefore $\Phi(E) \in \AA_{i+1}$, as claimed.

Now we assume that $E \in \TT_i$, from which it follows 
$\coker \varphi \in \TT_i \subset \TT_{i+1}$.
Since $T_{i+1}^+, T_{i+1}^+ (K_X)\in \TT_{i+1}$,
we get $\Phi^0(E) \in \TT_{i+1}$.
By $T_{i+1}^+,T_{i+1}^+(K_X) \in \FF_i$, 
we get $\Phi^{-1}(E) \in \FF_i$.
Since 
\begin{equation}
\begin{split}
\Hom(T_{i+1}^+,\Phi(E)[-1])& =\Hom(\Phi(T_{i+1}^+(K_X)),\Phi(E)[-2])
=\Hom(T_{i+1}^+(K_X),E[-2])=0,\\
\Hom(T_{i+1}^+(K_X),\Phi(E)[-1])&=\Hom(\Phi(T_{i+1}^+),\Phi(E)[-2])
=\Hom(T_{i+1}^+,E[-2])=0,
\end{split}
\end{equation}
for the same reasons as above, we see from the triangle $$\Phi^{-1}(E)\to\Phi(E)\to\Phi^{\geq 0}(E)\to\Phi^{-1}(E)[1]$$ that $\Hom(T_{i+1}^+,\Phi^{-1}(E))=\Hom(T_{i+1}^+(K_X),\Phi^{-1}(E))=0$.  Thus we get $\Phi^{-1}(E) \in \FF_{i+1}$.  Finally, by Serre duality we have 
\begin{equation*}
\begin{split}
    \Ext^2(T_{i+1}^+,E)^\vee\cong\Hom(E,T_{i+1}^+(K_X))=0,\\
    \Ext^2(T_{i+1}^+(K_X),E)^\vee\cong\Hom(E,T_{i+1}^+)=0,
\end{split}
\end{equation*} since $E\in\TT_i$ implies that $\phi^+_{\min}(E)>\phi^+(T_{i+1}^+)=\phi^+(T_{i+1}^+(K_X))$.  Therefore, $\Phi^1(E)=0$, so $\Phi(E) \in \AA_{i+1}[1]$, as required.

We next claim that $R_{T_{i+1}^+}^{-1}(\AA_{i+1}) \subset \AA_i$.
Let $\Psi$ be the inverse of $\Phi$, and set
$\Psi^p(E):=\HH^p(\Psi(E))$ for any $E\in\Db(X))$.  Recall that for any $E\in\Db(X)$ we have a distinguished triangle 
\begin{equation}
    \Psi(E)\to E\to T_{i+1}^+[2]\otimes\RHom(T_{i+1}^+,E)\oplus T_{i+1}^+(K_X)[2]\otimes\RHom(T_{i+1}^+(K_X),E)\to\Psi(E)[1].
\end{equation}
Then
\begin{equation}\label{eq:Psi-1}
\Hom(T_{i+1}^+,E(K_X)) \otimes T_{i+1}^+ \oplus
\Hom(T_{i+1}^+,E) \otimes T_{i+1}^+(K_X) \cong \Psi^{-1}(E)
\end{equation}
and we have an exact sequence
\begin{equation}\label{eqn:Psi-2}
\begin{CD}
0 @>>> \Ext^1(T_{i+1}^+,E(K_X)) \otimes T_{i+1}^+ \oplus
\Ext^1(T_{i+1}^+,E) \otimes T_{i+1}^+(K_X) @>>> \Psi^0(E) @>>>E \\
 @>{\psi}>> \Ext^2(T_{i+1}^+,E(K_X)) \otimes T_{i+1}^+ \oplus
\Ext^2(T_{i+1}^+,E) \otimes T_{i+1}^+(K_X) @>>> \Psi^1(E) @>>>0.
\end{CD}
\end{equation}
We prove the claim by showing that $\Psi(\FF_{i+1})\subset\AA_i$ and $\Psi(\TT_{i+1})\subset\AA_i[1]$ which suffices by considering the exact sequence \eqref{eqn:canonical short exact sequence}.

First assume that 
$E \in \FF_{i+1}$.
Then $\Psi^{-1}(E)=0$ by \eqref{eq:Psi-1}.  From \eqref{eqn:Psi-2}, we get $\Psi^0(E) \in \FF_i$ and $\Psi^1(E) \in \TT_{i+1}$.
Since
\begin{equation}\label{eq:T_{i+1}}
\begin{split}
\Hom(\Psi(E),T_{i+1}^+[p])& =\Hom(E,\Phi(T_{i+1}^+)[p])=
\Hom(E, T_{i+1}^+(K_X)[p-1])=0,\\
\Hom(\Psi(E),T_{i+1}^+(K_X)[p])& =\Hom(E,\Phi(T_{i+1}^+(K_X))[p])=
\Hom(E, T_{i+1}^+[p-1])=0
\end{split}
\end{equation}
for $p \leq 0$,
$\Psi^1(E) \in \TT_i$.
Therefore $\Psi(E) \in \AA_i$.

Now assume that $E \in \TT_{i+1}$.  Then $\Psi^{-1}(E)\in\FF_i$ by \eqref{eq:Psi-1}.  We must show that $\Psi^0(E)\in\TT_i$ and $\Psi^{1}(E)=\coker \psi=0$.  As a quotient of an object generated by
$T_{i+1}^+$ and $T_{i+1}^+(K_X)$, $\Psi^1(E)=\coker \psi \in \TT_{i+1}$.
Moreover, for the same reason we note that $\im \psi \in \TT_{i+1}$ as $E \in \TT_{i+1}$.  As a subobject of an object generated by
$T_{i+1}^+$ and $T_{i+1}^+(K_X)$, we also have $\im \psi \in \FF_i$.
Hence $\phi^+(\im \psi)=\phi^+(T_{i+1}^+)$.  Similarly, $\coker \psi$ is a $\sigma_+$-semistable object of $\phi^+(\coker
\psi)=\phi^+(T_{i+1}^+)$.  

It follows that $\coker \psi$ and $\im \psi$ are direct sums of $T_{i+1}^+$ and
$T_{i+1}^+(K_X)$.  Now by using \eqref{eq:T_{i+1}} we get 
$$0=\Hom(\Psi(E),T_{i+1}^+(D)[-1])=\Hom(\Psi^1(E),T_{i+1}^+(D)),$$
for $D=0,K_X$.  Thus we have $\Psi^1(E)=\coker \psi=0$.

Writing $\TT_i(E)$ and $\FF_i(E)$ for the components of $E$ in the torsion pair $(\TT_i,\FF_i)$, it follows from $\im\psi\in\FF_i$ that $\TT_i(E)\subset\ker\psi$, and it is easy to see that then $\TT_i(\ker\psi)=\TT_i(E)$.  Applying $\Hom(-,T_{i+1}^+(D))$ $(D=0,K_X)$ to the short exact sequence $$0\to\TT_i(E)\to E\to\FF_i(E)\to 0,$$ we see that $\hom(E,T_{i+1}^+(D))=\hom(\FF_i(E),T_{i+1}^+(D))$ for $D=0,K_X$.  As 
\begin{equation*}
    \begin{split}
        \ext^2(T_{i+1}^+,E)=\hom(E,T_{i+1}^+(K_X)),\\
        \ext^2(T_{i+1}^+,E(K_X))=\hom(E,T_{i+1}^+),
    \end{split}
\end{equation*} it follows from the short exact sequence \begin{equation}\label{eqn:ses for ker psi}0\to\ker\psi/\TT_i(E)\to\FF_i(E)\to\Ext^2(T_{i+1}^+,E(K_X)) \otimes T_{i+1}^+ \oplus
\Ext^2(T_{i+1}^+,E) \otimes T_{i+1}^+(K_X)\to 0\end{equation} that $\Hom(\ker\psi/\TT_i(E),T_{i+1}^+(D))=0$ by the same reasoning.  But this forces $\ker\psi/\TT_i(E)=0$.  Indeed, as $E\in\TT_{i+1}$, we have $\FF_i(E)\in\FF_i\cap\TT_{i+1}$, so $\FF_i(E)$ is a direct sum of copies of $T_{i+1}^+$ and $T_{i+1}^+(K_X)$, and thus so is $\ker\psi/\TT_i(E)$ from \eqref{eqn:ses for ker psi}.  Thus we have $\ker\psi=\TT_i(E)\in\TT_i$ so that $\Psi^0(E)\in\TT_{i+1}$.  Then \eqref{eq:T_{i+1}} implies $\Psi^0(E) \in \TT_i$.  Therefore $\Psi(E)\in \AA_i[1]$, as required.
\end{proof}

As $\AA_i^*$ deals in parallel with objects in $\PP_0(1)$ considered with respect to $\sigma_-$-stability, we have the following result:
\begin{Prop}\label{Prop:equiv2}
$R_{T_{i+1}^-}^{-1}$ induces an equivalence
$\AA_i^* \to \AA_{i+1}^*$.
\end{Prop}

\begin{proof}
We set $\Phi:=R_{T_{i+1}^-}$ and $\Psi:=R_{T_{i+1}^-}^{-1}$.
We only show that 
$\Phi(\AA_{i+1}^*) \subset \AA_i^*$.
We note that 
\begin{equation}\label{eq:Phi2}
\begin{split}
\Hom(T_{i+1}^-,\Phi(E)[p])& =\Hom(\Phi(T_{i+1}^-(K_X))[1],\Phi(E)[p])
=\Hom(T_{i+1}^-(K_X),E[p-1])=0,\\
\Hom(T_{i+1}^-(K_X),\Phi(E)[p])&=\Hom(\Phi(T_{i+1}^-)[1],\Phi(E)[p])
=\Hom(T_{i+1}^-,E[p-1])=0
\end{split}
\end{equation}
for $E \in \PP_0(1)$ and $p \leq 0$.
Assume that $E \in \FF_{i+1}^*$.
For the morphism
\begin{equation}
\varphi:\Hom(T_{i+1}^-,E) \otimes T_{i+1}^- \oplus  
\Hom(T_{i+1}^-(K_X),E) \otimes T_{i+1}^- (K_X)
\to E,
\end{equation}
$\ker\varphi$ and $\im \varphi$ are generated by
$T_{i+1}^-,T_{i+1}^- (K_X)$.
By \eqref{eq:Phi2}, $\Phi^{-1}(E)=0$ and $\Phi^0(E) \in \FF_i^*$.
Since $\Phi^1(E)$ is generated by
$T_{i+1}^-,T_{i+1}^- (K_X)$,
$\Phi^1(E) \in \TT_i^*$.
Therefore $\Phi(E[1]) \in \AA_i^*$.

Assume that $E \in \TT_{i+1}^*$.
Then $\Phi^1(E)=0$.
Since $\Phi^{-1}(E) \in \FF_{i+1}^*$, \eqref{eq:Phi2} implies
$\Phi^{-1}(E) \in \FF_i^*$.
Since $\coker \varphi, T_{i+1}^-, T_{i+1}^- (K_X) \in \TT_i^*$,
$\Phi^0(E) \in \TT_i^*$.
Therefore $\Phi(E) \in \AA_i^*$.
\end{proof}

For negative $i$, we must make the following definition.
\begin{Def}
Assume that $i \leq 0$.
\begin{enumerate}
\item[(1)]
Let $(\TT_i^*,\FF_i^*)$ be the torsion pair of $\PP_0(1)$ such that
\begin{enumerate}
\item
$\TT_i^*$ is generated by $\sigma_+$-stable objects $E$ with 
$\phi^+(E) \geq \phi^+(T_i^+)$.
\item
$\FF_i^*:=\langle T_0^+,T_0^+ (K_X),...,T_{i+1}^+, T_{i+1}^+ (K_X) \rangle$.
\end{enumerate}
Let $\AA_i^*=\langle \TT_i^*,\FF_i^*[1] \rangle$ be the tilting.
\item[(2)]
Let $(\TT_i,\FF_i)$ be a torsion pair of $\PP_0(1)$ such that
\begin{enumerate}
\item
$\TT_i:=\langle T_0^-,T_0^- (K_X),...,T_{i+1}^-, T_{i+1}^- (K_X) \rangle$.
\item
$\FF_i$ is generated by $\sigma_-$-stable objects $E$ with 
$\phi^-(E) \leq \phi^-(T_i^-)$.
\end{enumerate}
Let $\AA_i=\langle \TT_i[-1],\FF_i\rangle$ be the tilting.
\end{enumerate}
\end{Def}
Since $\FF_0^*=\TT_0=0$,
we have $\AA_0^*=\AA_0=\PP_0(1)$.
We also have a similar result to Remark \ref{rem:simple-objects}.

Moreover, we also have the following claims whose proofs are similar to those of Proposition \ref{Prop:equiv1} and Proposition \ref{Prop:equiv2}.
\begin{Prop}\label{Prop:equiv3}
Assume that $i \leq 0$.
\begin{enumerate}
\item
We have an equivalence
$R_{T_i^+}^{-1}:\AA_i^* \to \AA_{i-1}^*$.
\item
We have an equivalence
$R_{T_i^-}:\AA_i \to \AA_{i-1}$.
\end{enumerate}
\end{Prop}

\subsubsection{Preservation of stability}
Having defined the abelian categories $\AA_i$ and $\AA_i^*$, we will relate a certain stability on them to $\sigma_\pm$-stability.  
\begin{Ex}\label{ex:BB}
Recall Definition \ref{defn:B'} and Proposition \ref{Prop:BB}, and let $\BB=\PP_0(1)$.
\begin{enumerate}
\item[(1)]
We take an orientation preserving injective homomorphism $Z:\HH \to \C$ such that
$$Z(\w_0),Z(\w_1) \in \H \cup \R_{<0},\mbox{ and }Z(\w_1)/Z(\w_0) \in \H.$$
The second condition means that $\phi_Z(\w_1)>\phi_Z(\w_0)$.  Then $E \in \PP_0(1)$ is $\sigma_+$-semistable if and only if $E$ is $Z$-semistable.  In this case, $\AA_i$ $(i \geq 0)$ is an example of $\BB'$ for some $\theta \geq 0$, and $\AA_i^*$ $(i \leq 0)$ is an example of $\BB''$ for some $\theta \leq 0$. 
\item[(2)]
We take an orientation reversing injective homomorphism $Z:\HH \to \C$ such that 
$$Z(\w_0),Z(\w_1) \in \H \cup \R_{<0},\mbox{ and }Z(\w_0)/Z(\w_1) \in \H.$$
Then $E \in \PP_0(1)$ is $\sigma_-$-semistable if and only if $E$ is $Z$-semistable.  In this case, $\AA_i$ $(i \leq 0)$ is an example of $\BB'$, and $\AA_i^*$ $(i \geq 0)$ is an example of $\BB''$. 
\end{enumerate}
\end{Ex}
Now we can finally prove a sequence of comparison results that allow us to reduce our analysis to the case of minimal Mukai vectors.
\begin{Prop}\label{Prop:isom-pm}
\begin{enumerate}
\item
\begin{enumerate}
\item
$R_{T_1^+}:\AA_0 \to \AA_1$ induces an isomorphism
$M_{\sigma_-}(\v) \to M_{\sigma_+}(\v')$, where  $\v'=R_{\w_1}(\v)$.
\item
$R_{T_1^-}^{-1}:\AA_0^* \to \AA_1^*$ induces an isomorphism
$M_{\sigma_+}(\v) \to M_{\sigma_-}(\v')$, where  $\v'=R_{\w_1}^{-1}(\v)$.
\end{enumerate}
\item
\begin{enumerate}
\item
$R_{T_0^-}:\AA_0 \to \AA_{-1}$ induces an isomorphism
$M_{\sigma_+}(\v) \to M_{\sigma_-}(\v')$, where  $\v'=R_{\w_0}(\v)$.
\item
$R_{T_0^+}^{-1}:\AA_0^* \to \AA_{-1}^*$ induces an isomorphism
$M_{\sigma_-}(\v) \to M_{\sigma_+}(\v')$, where  $\v'=R_{\w_0}^{-1}(\v)$.
\end{enumerate}
\end{enumerate}
\end{Prop}

\begin{proof}
By Example \ref{ex:BB}, we can apply Proposition \ref{Prop:BB}
to compare the stabilities on $\AA_i$ and $\AA_i^*$ to 
$\sigma_\pm$-stability on $\AA_0=\AA_0^*$.
Since the orientation of $\HH$ is reversed under the reflection,
the claims follow from Proposition \ref{Prop:BB} and 
Propositions \ref{Prop:equiv1},
\ref{Prop:equiv2}, \ref{Prop:equiv3}.
\end{proof}
This result has the following significant corollary:

\begin{Cor}\label{Cor:OrthogonalIsomorphism 1}
If $\langle \v,\w_1 \rangle=0$,
then $R_{T_1}$ induces an isomorphism
$M_{\sigma_-}(\v) \to M_{\sigma_+}(\v)$. 
If $\langle \v,\w_0 \rangle=0$,
then $R_{T_0}$ induces an isomorphism
$M_{\sigma_+}(\v)\to M_{\sigma_-}(\v)$. 
\end{Cor}
\begin{proof}
We apply Proposition \ref{Prop:isom-pm} and observe that for $i=0,1$ we get $R_{\w_i}(\v)=\v+2\langle\v,\w_i\rangle=\v$.
\end{proof}

In order to use the sequence of weakly spherical reflections to reduce any $\v$ to the minimal Mukai vector in its orbit, as in \eqref{eqn:OrbitOfv0}, we consider the next step in the program beyond that in Proposition \ref{Prop:isom-pm}.
\begin{Prop}\label{Prop:isom-2}
\begin{enumerate}
\item\label{enum:iPositive}
Assume that $i>0$.
\begin{enumerate}
\item
$R_{T_{i+1}^+} \circ R_{T_i^+}:\AA_{i-1} \to \AA_{i+1}$
induces an isomorphism
$M_{\sigma_+}(\v) \to M_{\sigma_+}(\v')$, where 
$\v'=R_{\w_{i+1}} \circ R_{\w_i}(\v)$.
\item
$R_{T_{i+1}^-}^{-1} \circ R_{T_i^-}^{-1}:\AA_{i-1}^* \to \AA_{i+1}^*$
induces an isomorphism
$M_{\sigma_-}(\v) \to M_{\sigma_-}(\v')$, where
$\v'=R_{\w_{i+1}}^{-1} \circ R_{\w_i}^{-1}(\v)$.
\end{enumerate}
\item\label{enum:iNegative}
Assume that $i < 0$.
\begin{enumerate}
\item
$R_{T_{i}^-} \circ R_{T_{i+1}^-}:\AA_{i+1} \to \AA_{i-1}$
induces an isomorphism
$M_{\sigma_-}(\v) \to M_{\sigma_-}(\v')$, where
$\v'=R_{\w_{i}} \circ R_{\w_{i+1}}(\v)$.
\item
$R_{T_{i}^+}^{-1} \circ R_{T_{i+1}^+}^{-1}:\AA_{i+1}^* \to \AA_{i-1}^*$
induces an isomorphism
$M_{\sigma_+}(\v) \to M_{\sigma_+}(\v')$, where
$\v'=R_{\w_{i}}^{-1} \circ R_{\w_{i+1}}^{-1}(\v)$.
\end{enumerate}
\end{enumerate}
\end{Prop}

\begin{proof}
\ref{enum:iPositive}
Since $R_{T_{i+1}^+} \circ R_{T_i^+}$ and 
$R_{T_{i+1}^-}^{-1} \circ R_{T_i^-}^{-1}$
preserve the orientation,
we get the claims by a similar argument as in Proposition \ref{Prop:isom-pm}.  
The proof of \ref{enum:iNegative} is similar.
\end{proof}
It turns out that the composition of two consecutive weakly spherical reflections as in the previous proposition is independent of $i$, which is the content of the following lemma.
\begin{Lem}
\begin{enumerate}
\item\label{enum:ReflectionsOfExceptionals}

$R_{T_i^+}(T_{i-1}^+) =
\begin{cases}
T_{i+1}^+[1],T_{i+1}^+(K_X)[1], & i \ne 0,1,\\
T_{i+1}^+, T_{i+1}^+(K_X), & i =0,1.
\end{cases}$

\item\label{enum:CompositionConsecutiveReflections}
We have $R_{T_i^+} \circ R_{T_{i-1}^+}=R_{T_{i+1}^+} \circ R_{T_i^+}$.  In particular,
$R_{T_{i+1}^+} \circ R_{T_i^+} =R_{T_1} \circ R_{T_0}$ for all $i$.
\end{enumerate}
\end{Lem}

\begin{proof}
\ref{enum:ReflectionsOfExceptionals}
Assume that $i \geq 2$.
By Proposition \ref{Prop:isom-2} \ref{enum:iPositive},
$R_{T_{i}^+}\circ R_{T_{i-1}^+}:\AA_{i-2} \to \AA_{i}$
induces an isomorphism
$$
M_{\sigma_+}(\w_{i-1}) \to M_{\sigma_+}(\w_{i+1}).
$$
Hence $R_{T_{i}^+}\circ R_{T_{i-1}^+}(T_{i-1}^+(K_X))
=R_{T_{i}^+}(T_{i-1}^+[-1])$
is a $\sigma_+$-stable object with Mukai vector $\w_{i+1}$.
Then we get 
$R_{T_{i}^+}(T_{i-1}^+[-1])=T_{i+1}^+,T_{i+1}^+(K_X)$.
If $i=0,1$, then Proposition \ref{Prop:isom-pm}
implies $R_{T_i^+}(T_{i-1}^+)=T_{i+1}^+,T_{i+1}^+(K_X)$.
Assume that $i<0$.
Then Proposition \ref{Prop:isom-2} \ref{enum:iNegative} implies
$R_{T_{i}^+}\circ R_{T_{i-1}^+}:\AA_{i-2}^* \to \AA_{i}^*$
induces an isomorphism
$$
M_{\sigma_+}(-\w_{i-1}) \to M_{\sigma_+}(-\w_{i+1}).
$$
Hence we get 
$R_{T_{i}^+}(T_{i-1}^+)=T_{i+1}^+[1],T_{i+1}^+(K_X)[1]$.

\ref{enum:CompositionConsecutiveReflections}
Since $R_{T_i^+} \circ R_{T_{i-1}^+} \circ R_{T_i^+}^{-1}=
R_{R_{T_i^+}(T_{i-1}^+)}$ by \cite[Lemma 8.21]{HL10}, the claim follows from \ref{enum:ReflectionsOfExceptionals}.
\end{proof}

In the same way, we also see that
 $R_{T_i^-} \circ R_{T_{i+1}^-} =R_{T_0} \circ R_{T_1}$ for all $i$.  This leads us to make the following definition.

\begin{Def}
We set $R_+:=R_{T_1} \circ R_{T_0}$ and
$R_-:=R_{T_0} \circ R_{T_1}$.
\end{Def}
We have finally studied the action of $G_\HH$ enough to prove our main reduction result.
\begin{Prop}\label{Prop:NonMinimalIsomorphism} Let $\v_n\in\CC_n$ be defined as in \eqref{eqn:OrbitOfv0}.  That is, $\v_n$ is in the orbit of $\v_0\in\CC_0$.
\begin{enumerate}
\item\label{enum:NonMinimalIsomorphism n even}
If $n$ is even, then $R_+^{\frac{n}{2}} \circ R_-^{\frac{n}{2}}$ induces a birational map
\begin{equation}
M_{\sigma_-}(\v_n) \cong M_{\sigma_-}(\v_0) \dashrightarrow M_{\sigma_+}(\v_0)
\cong M_{\sigma_+}(\v_n),
\end{equation}
which is isomorphic in codimension one.
\item\label{enum:NonMinimalIsomorphism n odd}
If $n$ is odd, then $R_+^{\frac{n-1}{2}} \circ R_{T_1} \circ R_{T_1} \circ R_-^{\frac{n-1}{2}}$ induces a birational map
\begin{equation}
M_{\sigma_-}(\v_n) \cong M_{\sigma_-}(\v_1) \cong M_{\sigma_+}(\v_0) 
\dashrightarrow M_{\sigma_-}(\v_0) \cong M_{\sigma_+}(\v_1)
\cong M_{\sigma_+}(\v_n),
\end{equation}
which is isomorphic in codimension one.
\end{enumerate}
\end{Prop}
\begin{proof}
\ref{enum:NonMinimalIsomorphism n even}
Using $n/2$ applications of Proposition \ref{Prop:isom-2}, we get that $R_-^{\frac{n}{2}}$ induces an isomorphism $M_{\sigma_-}(\v_n)\isomor M_{\sigma_-}(\v_0)$, and by Lemma \ref{Lem:non-isotropic no totally semistable wall} the open subset of $\sigma_0$-stable objects $M_{\sigma_0}^s(\v_0)$ provides a birational map $$M_{\sigma_-}(\v_0)\dashrightarrow M_{\sigma_+}(\v_0).$$  Moreover, the complement of $M_{\sigma_0}^s(\v_0)$ in $M_{\sigma_\pm}(\v_0)$ has codimension at least 2.  Using Proposition \ref{Prop:isom-2} again, we get that $R_+^{\frac{n}{2}}$ gives an isomorphism $M_{\sigma_+}(\v_0)\isomor M_{\sigma_+}(\v_n)$, which gives the result.

The proof of \ref{enum:NonMinimalIsomorphism n odd} follows similarly by using Proposition \ref{Prop:isom-pm} twice in the middle.
\end{proof}

The complementary result for $\v$ on the boundary of some $\CC_n$, which generalizes Corollary \ref{Cor:OrthogonalIsomorphism 1}, is the following.
\begin{Prop}\label{Prop:OrthgonalIsomorphism 2}
Suppose that $\v\in C_\WW$ satisfies $\langle\v,\w_n\rangle=0$ for some $n$.  Then $M_{\sigma_+}(\v)\cong M_{\sigma_-}(\v)$.
\end{Prop}
\begin{proof}
First let us suppose that $n$ is odd.  Then we  take $\v_1 =R_-^{\frac{n-1}{2}}(\v)$ which satisfies $\langle\v_1,\w_1\rangle=0$.
In this case, by Propositions \ref{Prop:isom-pm} and \ref{Prop:isom-2}, the composition
$R_+^{\frac{n-1}{2}} \circ R_{T_1^+}
\circ R_-^{\frac{n-1}{2}}$ induces
an isomorphism
\begin{equation}\label{eq:FM-birat3}
M_{\sigma_-}(\v) \cong M_{\sigma_-}(\v_1) \cong M_{\sigma_+}(\v_1) 
\cong M_{\sigma_+}(\v).
\end{equation}

If instead $n$ is even, then we take $\v_0=R_-^{\frac{n}{2}}(\v)$, which satisfies $\langle\v_0,\w_0\rangle=0$.  In this case, $R_+^{\frac{n}{2}} \circ R_{T_0^+}
\circ R_-^{\frac{n}{2}}$ induces
an isomorphism
\begin{equation}\label{eq:FM-birat4}
M_{\sigma_-}(\v) \cong M_{\sigma_-}(\v_0) \cong M_{\sigma_+}(\v_0) 
\cong M_{\sigma_+}(\v).
\end{equation}
\end{proof}
We can conclude from Propositions \ref{Prop:NonMinimalIsomorphism} and \ref{Prop:OrthgonalIsomorphism 2} that, in the case of a non-isotropic wall of type \ref{enum:TwoExceptional}, the two moduli spaces $M_{\sigma_+}(\v)$ and $M_{\sigma_-}(\v)$ are birational.  As identical statements to these propositions hold for non-isotropic walls of type \ref{enum:TwoSpherical} (see \cite[Prop. 6.8 and Lem. 7.5]{BM14b}), the same conclusion holds in that case.  Now we move on to the final non-isotropic case we need to deal with.


\subsection{\cref{enum:OneExceptionalOneSpherical}: Exactly one $\sigma_0$-stable spherical and exceptional object, respectively, up to $-\otimes\OO_X(K_X)$}\label{subsec:OneSphericalOneExceptional}

We shall briefly explain the case where $\HH$ contains 
effective exceptional and spherical vectors.  Denote the Mukai vector of the unique $\sigma_0$-stable exceptional object (up to $-\otimes\OO_X(K_X)$) by $\w_0$, and complete it to a basis of $\HH$ such that $\HH=\Z\w_0 +\Z\z$, where $\z$ satisfies $\langle \w_0,\z \rangle=0$ and $D:=\z^2>0$.  

Furthermore, we know there is an effective $\w \in C_\WW\cap \HH$ with $\w^2=-2$.  Then $x^2-D y^2=2$ has an integral solution, which implies $\sqrt{D}$ is irrational.  Let $(s,t)$ be a solution of $x^2-D y^2=2$, and set $\alpha:=s+t \sqrt{D}$.  Then $-\alpha^2/2=(1-s^2)-s t \sqrt{D},$ and $(x,y)=(1-s^2,-s t)$ is a solution of $x^2-y^2D=1$.  Let $(x_1,y_1)$ be a minimal solution of $x^2-y^2D=1$ such that $x_1<0$ and $y_1>0$, and set $\beta:=x_1+y_1 \sqrt{D}$.  Then $-1<\beta<0$ (see \cite[Thm. 11.3.1]{AW04}) and $\alpha^2/2=\pm\beta^n$ for some $n$.  

We claim that $n$ is odd.  If not, then $u+v\sqrt{D}:=\alpha/\beta^{\frac{n}{2}}$
satisfies $(u+v\sqrt{D})^2=2$.  But then $\sqrt{D}=\frac{2-u^2-Dv^2}{2uv}$, which contradicts the irrationality of $\sqrt{D}$.  Hence $n$ is odd and $\alpha^2/2=-\beta^n$.  Set $n=2k+1$.  Then $s_1+t_1 \sqrt{D}=\alpha/\beta^{k}$ satisfies $(s_1+t_1 \sqrt{D})^2=-2\beta$ and $s_1^2-t_1^2 D=2$.  So replacing $(s,t)$ by $(s_1,t_1)$ (and possibly taking a conjugate), we may assume that $\alpha:=s_1 +t_1 \sqrt{D}$ satisfies $N(\alpha)=2$, $\beta=-\alpha^2/2$, and $s_1<0$ and $t_1>0$.  Moreover, the same argument applied to any other solution $\gamma$ of the equation $x^2-Dy^2=2$ shows that $\gamma=\pm\frac{\alpha^{2n+1}}{2^n}$, where $\gamma^2/2=-\beta^{2n+1}$.  

We define $(s_n,t_n) \in \Z \times \Z$ by 
\begin{equation}
\begin{split}
s_{2n}+t_{2n} \sqrt{D}:=& (-1)^{n+1}\beta^n=-\alpha^{2n}/2^n,\; (n \geq 1)\\
s_{2n-1}+t_{2n-1}\sqrt{D}:=& \alpha^{2n-1}/2^{n-1},\; (n \geq 1)\\
s_0+t_0\sqrt{D}:=& 1,\\
s_{-2n}+t_{-2n} \sqrt{D}:=&-s_{2n}+t_{2n}\sqrt{D}=(-\beta)^{-n},\; (n \geq 1)\\
s_{-2n+1}+t_{-2n+1} \sqrt{D}:=& -s_{2n-1}+t_{2n-1}\sqrt{D}
=2^{n} (-\alpha)^{-(2n-1)},\; (n \geq 1).\\
\end{split}
\end{equation}
For $n \in \Z$, we set $\w_n:=s_n \w_0+t_n \z$.
Then 
$$
\w_n^2=\begin{cases}
-1,& 2 \mid n\\
-2, & 2 \nmid n
\end{cases} 
$$
It is easy to see that
\begin{equation}\label{eqn:SphericalAndExceptionalVectorsReflections}
\begin{split}
\w_{n+1}=& -R_{\w_n}(\w_{n-1}),\;(n \geq 2)\\
\w_{n-1}=& -R_{\w_n}(\w_{n+1}),\;(n \leq -1)\\
\w_2=& R_{\w_1}(\w_0),\;\w_{-1}=R_{\w_0}(\w_1),
\end{split}
\end{equation}
where $R_{\w_n}$ is, by abuse of notation, the action on $\Hal(X,\Z)$ induced by spherical/weakly spherical reflection through $\w_n$.  To be concrete, we have 
\begin{equation*}
    \begin{split}
        R_{\w_n}(\u)&=\u+2\langle\u,\w_n\rangle\w_n,\;2\mid n\\
        R_{\w_n}(\u)&=\u+\langle\u,\w_n\rangle\w_n,\;2\nmid n.
    \end{split}
\end{equation*}

From this description it is clear that 
\begin{equation}\label{eqn:RootsInH-1}
\Delta(X)\cap\HH\subseteq\Set{\pm \w_n\ | \ n\in \Z },
\end{equation}
where $\Delta(X)$ was defined in \eqref{eqn:DefOfEnriquesRoots}.  By the assumption on $\HH$ and $\WW$, there exists a $\sigma_0$-stable spherical object $T$, so we must have $\v(T)=\pm\w_{2n+1}$ for some $n\in\Z$.  In particular, $c_1(\w_{2n+1})\equiv Z\pmod 2$, where $Z$ is a nodal cycle by Theorem \ref{Thm:exist:nodal}.  From \eqref{eqn:SphericalAndExceptionalVectorsReflections}, we see that 
\begin{equation}\label{eqn:RootsAreSpherical}
    c_1(\w_{2j+1})\equiv c_1(\w_1)\pmod 2,\mbox{ for all }j\in\Z,
\end{equation} so the inclusion in \eqref{eqn:RootsInH-1} is an equality:
\begin{equation}\label{eqn:RootsInH-2}
\Delta(X)\cap\HH=\Set{\pm\w_n\ |\ n\in\Z}.
\end{equation}

As in \cref{subsec:TwoExceptional}, it is not difficult to see that 
$$
\Set{\u \in \HH \ | \ \frac{Z_{\sigma_0}(\u)}{Z_{\sigma_0}(\v)}>0,\u^2\geq -2 } \subset
\Q_{\geq 0}\w_0+\Q_{\geq 0}\w_1, 
$$ 
and thus $C_\WW=\R_{>0}\w_0+\R_{>0}\w_1$.  It follows from this and \eqref{eqn:RootsInH-2} that there are $\sigma_0$-stable objects, $T_0$ and $T_1$, of classes $\w_0$ and $\w_1$, respectively.  Moreover, as we then have $$\Delta(X)\cap C_\WW=\Set{\w_n \ |\ n \in \Z },$$ it follows from \eqref{eqn:RootsAreSpherical} and Theorem \ref{Thm:exist:nodal} that there are $\sigma_\pm$-stable objects $T_n^\pm$ with
$\v(T_n^\pm)=\w_n$.  Note that $T_i^+=T_i^-=T_i$ for $i=0,1$.

Let $\CC_n \subset C_\WW$ be the region between 
$\w_n^\perp$ and $\w_{n+1}^\perp$ as in \cref{subsec:TwoExceptional}. 
Then $\CC_0$ is again the fundamental domain of the Weyl group $G_\HH$ associated to
$\Delta(X)\cap\HH$.  Up to reordering, we may assume that  
\begin{equation}
\phi^+(T_1^+) > \phi^+(T_2^+)>\cdots>\phi^+(E)>
\cdots >\phi^+(T_{-1}^+)>\phi^+(T_0^+)
\end{equation}
for any $\sigma_+$-stable object $E$ with
$\v(E)^2 \geq 0$
and
\begin{equation}
\phi^-(T_1^-) < \phi^-(T_2^-)<\cdots<\phi^-(E)<
\cdots <\phi^-(T_{-1}^-)<\phi^-(T_0^-)
\end{equation}
for any $\sigma_-$-stable object $E$ with
$\v(E)^2 \geq 0$.

We note that $T_n^\pm (K_X)=T_n^\pm$ if $n$ is odd and
$T_n^\pm (K_X) \not \cong T_n^\pm$ if $n$ is even.

As in \cref{subsec:TwoExceptional}, we see that
\begin{equation}
R_+:=R_{T_1} \circ R_{T_0} =R_{T_{i+1}^+} \circ R_{T_i^+} 
\end{equation}
is an equivalence which preserves $\sigma_+$-semistability
and
\begin{equation}
R_-:=R_{T_0} \circ R_{T_1}=R_{T_{i-1}^-} \circ R_{T_i^-} 
\end{equation}
is an equivalence which preserves $\sigma_-$-semistability.  Here, $R_{T_0}$ and $R_{T_1}$ are the weakly-spherical and spherical reflections associated to the exceptional and spherical objects $T_0$ and $T_1$, respectively.  Then we have the following results.

\begin{Prop}\label{Prop:CompositionSphericalExceptional}
Let $\v\in C_\WW\cap\HH$.  
\begin{enumerate}
\item Suppose that $\v\in\CC_n$.
\begin{enumerate}
\item
If $n$ is even, then $R_+^{\frac{n}{2}} \circ R_-^{\frac{n}{2}}$ induces a birational map
\begin{equation}
M_{\sigma_-}(\v) \cong M_{\sigma_-}(\v_0)\dashrightarrow M_{\sigma_+}(\v_0)
\cong M_{\sigma_+}(\v),
\end{equation}
isomorphic in codimension one, where $\v_0=R_-^{\frac{n}{2}}(\v)$.
\item
If $n$ is odd, then $R_+^{\frac{n-1}{2}} \circ R_{T_1}^2 \circ R_-^{\frac{n-1}{2}}$ induces a birational map
\begin{equation}
M_{\sigma_-}(\v) \cong M_{\sigma_-}(\v_1) \cong M_{\sigma_+}(\v_0) 
\dashrightarrow M_{\sigma_-}(\v_0) \cong M_{\sigma_+}(\v_1)
\cong M_{\sigma_+}(\v_n),
\end{equation}
isomorphic in codimension one, where $\v_1=R_-^{\frac{n-1}{2}}(\v)$ and $\v_0=R_{T_1}(R_-^{\frac{n-1}{2}}(\v))$.
\end{enumerate}
\item
Suppose that $\langle \v, \w_n \rangle=0$.
\begin{enumerate}
\item
If $n$ is even, then $R_+^{\frac{n}{2}} \circ R_{T_0^+}
\circ R_-^{\frac{n}{2}}$ induces
an isomorphism
\begin{equation}
M_{\sigma_-}(\v) \cong M_{\sigma_-}(\v_0) \cong M_{\sigma_+}(\v_0) 
\cong M_{\sigma_+}(\v),
\end{equation}
where $\v_0=R_-^{\frac{n}{2}}(\v)$.
\item
If $n$ is odd, then $R_+^{\frac{n-1}{2}} \circ R_{T_1^+}\circ R_-^{\frac{n-1}{2}}$ induces an isomorphism
\begin{equation}
M_{\sigma_-}(\v) \cong M_{\sigma_-}(\v_1) \cong M_{\sigma_+}(\v_1) 
\cong M_{\sigma_+}(\v),
\end{equation}
where $\v_1=R_-^{\frac{n-1}{2}}(\v)$.
\end{enumerate}
\end{enumerate}
\end{Prop}
Proposition \ref{Prop:CompositionSphericalExceptional} is proven by defining $\AA_i$ and $\AA_i^*$ analogously to \cref{subsec:TwoExceptional} and by showing that $R_{T_i^{\pm}}$ induces isomorphisms on moduli with identical proofs except for the minor adjustments when $i$ is odd so that $T_i^{\pm}$ is spherical.  To avoid this word-for-word repetition, we omit these proofs.  Note that the birational map $M_{\sigma_+}(\v_0)\dashrightarrow M_{\sigma_-}(\v_0)$ follows again from Lemma \ref{Lem:non-isotropic no totally semistable wall}.

We have seen in this section that, regardless of its type, a non-isotropic wall $\WW$ induces a birational map $M_{\sigma_+}(\v)\dashrightarrow M_{\sigma_-}(\v)$, and this map is isomorphic outside of a subvariety of codimension at least 2 unless $\langle\v,\w\rangle=0$ for some $\w\in C_\WW\cap\Delta(X)$.  We finish off the section by explaining how this birational map relates to the contraction maps $$\pi^\pm:M_{\sigma_\pm}(\v)\to\overline{M}_\pm$$ induced by $\ell_{\sigma_0,\pm}$, as in Theorem \ref{Thm:WallContraction}.  Recall that the curves contracted by $\pi^\pm$ are precisely those curves of $\sigma_+$-semistable objects that become S-equivalent with respect to $\sigma_0$.  

Let $E$ be a $\sigma_0$-stable object of class $\v_0\in\overline{\CC_0}$. Then
it is an irreducible object of $\AA_0$.
By Propositions \ref{Prop:equiv1}, \ref{Prop:equiv2}, and \ref{Prop:equiv3}, and their analogues in the other subcases of Proposition \ref{Prop:lattice classification}\ref{enum:TwoNegative}, it follows that for a non-negative integer $n$, $E_n^+:=R_{T_n^+} \circ R_{T_{n-1}^+} \circ  \cdots \circ R_{T_1^+}(E)$
is an irreducible object of $\AA_n$
and $E_n^-:=R_{T_n^-}^{-1} \circ 
R_{T_{n-1}^-}^{-1} \circ  \cdots \circ R_{T_1^-}^{-1}(E)$
is an irreducible object of $\AA_n^*$.
By Remark \ref{rem:simple-objects} and 
the definition of $R_{T_i^\pm}$,
$E_n^\pm$ are successive extensions of $E$ by $T_i^\pm,T_i^\pm(K_X)$.
In particular, $E_n^\pm$ is S-equivalent to 
\begin{equation}\label{eqn:Sequivalenceunderreflections}
\begin{cases}
E \oplus T_0^{\oplus k_0} \oplus T_1^{\oplus k_1}, &\mbox{ in Case \ref{enum:TwoSpherical}}\\
E\oplus \left(T_0\oplus T_0(K_X)\right)^{\oplus k_0}\oplus \left(T_1\oplus T_1(K_X)\right)^{\oplus k_1}, &\mbox{ in Case \ref{enum:TwoExceptional}}\\
E\oplus \left(T_0\oplus T_0(K_X)\right)^{\oplus k_0}\oplus T_1^{\oplus k_1}, &\mbox{ in Case \ref{enum:OneExceptionalOneSpherical}}
\end{cases}\end{equation}
with respect to $\sigma_0$, where 
\begin{equation}\label{eqn:Sequivalenceunderreflections-classes}
\v(E_n^\pm)=\begin{cases}
\v(E)+k_0 \w_0+k_1 \w_1, &\mbox{ in Case \ref{enum:TwoSpherical}}\\
\v(E)+2k_0\w_0+2k_1\w_1, &\mbox{ in Case \ref{enum:TwoExceptional}}\\
\v(E)+2k_0\w_0+k_1\w_1, &\mbox{ in Case \ref{enum:OneExceptionalOneSpherical}}
\end{cases}.\end{equation} 
Thus if $\v=R_{\w_n} \circ R_{\w_{n-1}} \circ  \cdots \circ R_{\w_1}(\v_0)$ for $\v_0\in\CC_0$, then the contraction map $\pi^\pm:M_{\sigma_\pm}\to\overline{M}_\pm$ is injective on the image of $M_{\sigma_0}^s(\v_0)$ under the isomorphism $M_{\sigma_\pm}(\v_0)\isomor M_{\sigma_\pm}(\v)$ given by the above composition of Fourier-Mukai functors.\footnote{As written, this isomorphism exists for $n$ even.  For $n$ odd, the composition of Fourier-Mukai functors would give $M_{\sigma_\mp}(\v_0)\isomor M_{\sigma_\pm}(\v)$}  In particular, $\pi^\pm$ is birational.  Moreover, the S-equivalence in \eqref{eqn:Sequivalenceunderreflections} true even if $E$ is strictly $\sigma_0$-semistable, so the curves contracted by $\pi^\pm:M_{\sigma_\pm}(\v)\to\overline{M}_\pm$ are in direct correspondence with the curves of S-equivalent objects contracted by $\pi_0^\pm:M_{\sigma_\pm}(\v_0)\to\overline{M}_{0,\pm}$, where this is the analogous morphism for $\v_0$.

When $\v_0\in\CC_0$ proper, that is, $\langle\v,\w\rangle>0$ for all $\w\in\Delta(X)\cap C_\WW$, the codimension of the exceptional locus of $\pi_0^\pm$ is at least two, so the same remains true of $\pi^\pm$.  When $\v\in \w_n^\perp$, however, this exceptional locus has codimension one and the two moduli spaces $M_{\sigma_\pm}(\v)$ are isomorphic.  We explore in the next section when this divisor is contracted upon crossing the wall $\WW$.

\section{Divisorial contractions in the non-isotropic case}
In this section we aim to prove the following result.

\begin{Prop}\label{Prop:NonisotropicDivisorialContraction}
Assume that the potential wall $\WW$ is non-isotropic.  Then $\WW$ induces a divisorial contraction on $M_{\sigma_+}(\v,L)$ if and only if either there exists a spherical class $\w\in\Delta(X)_{-2}\cap C_\WW$ such that $\langle\v,\w\rangle=0$ or there exists an exceptional class $\w\in\Delta(X)_{-1}\cap C_{\WW}$ such that $\langle\v,\w\rangle=0$, $\v-2\w$ is spherical, and $L\equiv D+\frac{\rk\v}{2}K_X\pmod 2$ for a nodal cycle $D$.  If $T$ is a $\sigma_{\pm}$-stable spherical object of class $\w$ (or $\v-2\w$ in the second case), then the contracted divisor can be described as a Brill-Noether divisor for $T$: it is given either by the condition $\Hom(T,\blank)\neq 0$ or by $\Hom(\blank,T)\neq 0$.
\end{Prop}

By Lemma \ref{Lem:non-isotropic no totally semistable wall}, we know that the locus of strictly $\sigma_0$-semistable objects has codimension one if and only if $\langle\v,\w\rangle=0$ for some class $\w\in \Delta(X)\cap C_\WW$.  In particular, this condition must be  met if $\WW$ induces a divisorial contraction.  In the next two lemmas, we prove that only when $\w$ is spherical does this divisor get contracted.

\begin{Lem}\label{Lem:NonisotropicDivisorialContraction}
Suppose that $\HH$ is non-isotropic and $\WW$ is a potential wall associated to $\HH$.  If there exists an effective spherical class $\w$ with $\langle \v,\w\rangle=0$, then $\WW$ induces a divisorial contraction.  

A generic element $E$ in the contracted divisor $D$ admits a short exact sequence $$0\to T\to E\to F\to 0\mbox{ or }0\to F\to E\to T\to 0,$$ where $T\in M^s_{\sigma_+}(\w)$ and $F\in M^s_{\sigma_+}(\v-\w)$, such that the inclusion $T\into E$ or $F\into E$ is one of the filtration steps in a JH-filtration for $E$ with respect to $\sigma_0$.
\end{Lem}
\begin{proof}
By the discussion at the end of Section \ref{Sec:TotallySemistable-non-isotropic}, we can use a composition of spherical or weakly-spherical reflections, as in \cite[Corollary 7.3, Lemma 7.5]{BM14b} (in Case \ref{enum:TwoSpherical} and Proposition  \ref{Prop:CompositionSphericalExceptional} (in Case \ref{enum:OneExceptionalOneSpherical}, to reduce the discussion to the case of a minimal Mukai vector.  We assume this to be the case.  Then the spherical class must be $\w_0$ or $\w_1$ (in Case \ref{enum:TwoSpherical}) or $\w_1$ (in Case \ref{enum:OneExceptionalOneSpherical}), and we assume it is $\w_1$ with the other case being  dealt with similarly.  As in \cite[Lemma 7.4]{BM14b}, we first prove that $\v-\w_1$ is also minimal.

Let us start by assuming that $\v^2\geq 2$.  We note that we must in-fact have $\v^2\geq3$, since $\v^2=2$ gives $(\v-\w_1)^2=0$, contrary to the assumption that $\HH$ is non-isotropic.  Write $\v=x\w_0+y\w_1$ with $x,y\in\Q$.  Then $0=\langle\w_1,\v\rangle=0$ gives $y=\frac{m}{2}x$, where recall that $m=\langle\w_0,\w_1\rangle$.  As $\langle\w_1,\v-\w_1\rangle=2$, to show that $\v-\w_1$ is minimal it suffices to check that 
$$
0\leq \langle \w_0,\v-\w_1\rangle=(x\w_0^2+ym)-m=my\left(\frac{2\w_0^2}{m^2}+1\right)-m=m \left(y \left(1+\frac{2\w_0^2}{m^2}\right)-1 \right).
$$  
We now consider the Cases \ref{enum:TwoSpherical} and \ref{enum:OneExceptionalOneSpherical} separately. 

First suppose we are in Case \ref{enum:TwoSpherical}.  Then $\w_0^2=-2$, so $\v^2\geq 3$ implies that $\frac{3}{2}\leq y^2(1-\frac{4}{m^2})$, and as in the proof of Proposition \ref{Prop:lattice classification}, we have $m\geq 3$.  If $m=3$, then it is easy to show that the equations $\v^2=3$ and $\langle \v,\w_1\rangle=0$ have no rational solutions, so we may assume that $v^2\geq 4$.  But then $2\leq y^2(1-\frac{4}{m^2})$, so $$y^2 \left(1-\frac{4}{m^2}\right)^2=
y^2 \left(1-\frac{4}{m^2}\right)
\frac{5}{9}\geq\frac{10}{9}>1.$$  Taking square-roots, we see that $$y\left(1+\frac{2\w_0^2}{m^2}\right)>1,$$ and therefore we have $\langle\w_0,\v-\w_1\rangle>0$.  If, instead, $m\geq 4$, then we get $(1-\frac{4}{m^2})\geq\frac{3}{4}$, from which it follows that 
$$y^2 \left(1-\frac{4}{m^2}\right)^2\geq
\frac{9}{8}>1,$$ so indeed $\langle \w_0,\v-\w_1\rangle>0$. 

Now suppose we are in \cref{enum:OneExceptionalOneSpherical}.  Then $\w_0^2=-1$, so $\v^2\geq 3$ is equivalent to $\frac{3}{2}\leq y^2(1-\frac{2}{m^2})$, and now $m\geq 2$.  If $m=2$, then again one can easily check that the equations $\v^2=3$ and $\langle \v,\w_1\rangle =0$ have no rational solutions, so $\v^2\geq 4$, i.e. $y^2(1-\frac{2}{m^2})\geq 2$.  Thus 
$$y^2 \left(1-\frac{2}{m^2}\right)^2=
y^2 \left(1-\frac{2}{m^2}\right)\frac{1}{2}\geq 1,$$ so $\langle \w_0,\v-\w_1\rangle\geq 0$. If, instead, $m\geq 3$, then $\left(1-\frac{2}{m^2}\right)\geq\frac{7}{9}$ and thus $$y^2 \left(1-\frac{2}{m^2}\right)^2\geq y^2 \left(1-\frac{2}{m^2}\right)\frac{7}{9}\geq\frac{7}{6}>1,$$ so we get $\langle \w_0,\v-\w_1\rangle>0$. 

As we have shown that $\v-\w_1$ is minimal, Lemma \ref{Lem:non-isotropic no totally semistable wall} guarantees that the generic element $F\in M_{\sigma_+}(\v-\w_1)$ is also $\sigma_0$-stable.  But then for the unique $\sigma_0$-stable spherical object $T_1$ with $\v(T_1)=\w_1$ we have $\ext^2(F,T_1)=\hom(T_1(K_X),F)=\hom(T_1,F)=0=\hom(F,T_1)$ by stability.  Thus $\ext^1(F,T_1)=\langle \v-\w_1,\w_1\rangle=2$, so there is a family of extensions $$0\to F\to E_p\to T_1\to 0,$$ parametrized by $p\in\P^1=\P(\Ext^1(T_1,F))$, which are all S-equivalent with respect to $\sigma_0$.  By \cite[Lemma 6.9]{BM14b}, they are $\sigma_+$-stable.  Thus $\pi^+$ contracts this rational curve.  Varying $F\in M_{\sigma_0}^s(\v-\w_1)$ sweeps out a family of $\sigma_+$-stable objects in $M_{\sigma_+}(\v)$ of dimension $1+(\v-\w_1)^2+1=\v^2=\dim M_{\sigma_+}(\v)-1$.  Thus we get a divisor contracted by $\pi^+$, which must then have relative Picard rank one, so this is the only component contracted by $\pi^+$.  

Finally, suppose that $\v^2=1$.  Then $(\v-\w_1)^2=-1$, so we must be in Case \ref{enum:OneExceptionalOneSpherical} with $\v-\w_1=\pm\w_n$ for some $n\in\Z$.  As $\w_n\in\Z\w_0+\Z\w_1$, it follows that $\v=x\w_0+y\w_1$ with $x,y\in\Z$.  But then $\langle\v,\w_1\rangle=0$ is equivalent to $x=\frac{2}{m}y$ so that $\v^2=1$ is equivalent to $$1=x^2\left(\frac{m^2}{2}-1\right)=\frac{x^2(m^2-2)}{2},$$ whose only solution in the positive integers is $x=1$ and $m=2$.  Thus we have $\v-\w_1=\w_0$, so we are, as above, guaranteed that the unique member of $M_{\sigma_+}(\v-\w_1)$ is $\sigma_0$-stable.  The same argument gives a family of extensions
$$0\to T_0\to E_p\to T_1\to 0,$$ parametrized by $p\in\P^1=\P(\Ext^1(T_1,T_0))$, which are S-equivalent with respect to $\sigma_0$ but are $\sigma_+$-stable.  This curve is contracted by $\pi^+$ and is a divisor in the two-dimensional moduli space $M_{\sigma_+}(\v)$, so $\pi^+$ again has relative Picard rank one.  These extensions thus give the unique curve contracted by $\pi^+$, as required.
\end{proof}

Having confirmed that when $\w$ is spherical we do get a divisorial contraction, we prove that this is not the case when $\w$ is exceptional with one exception.

\begin{Lem}\label{Lem:ExceptionalDivisorialNonContraction}
Suppose that $\HH$ is non-isotropic and $\WW$ is a potential wall associated to $\HH$.  If there exists an effective exceptional class $\w$ with $\langle \v,\w\rangle=0$, then $\WW$ only induces a divisorial contraction on $M_{\sigma_+}(\v,L)$ if $\v^2=2$, $L\equiv D+\frac{\rk\v}{2}K_X\pmod 2$ for a nodal cycle $D$, and $\HH$ falls into \cref{enum:OneExceptionalOneSpherical} of \cref{Prop:lattice classification}.  

In this case, for $E$ in the contracted divisor, there is a short exact sequence $$0\to T\oplus T(K_X)\to E\to S\to 0\mbox{ or }0\to S\to E\to T\oplus T(K_X)\to 0,$$ where $T\in M^s_{\sigma_+}(\w)$ and $S\in M^s_{\sigma_+}(\v-2\w)$, such that the inclusion $T\into E$ or $S\into E$ is one of the filtration steps in a JH-filtration for $E$ with respect to $\sigma_0$, while the generic $E\in M_{\sigma_+}(\v)$ satisfies $\Hom(T,E)=\Hom(E,T)=0$.  

In general, there is nevertheless a divisor $D_{\sigma_+}(\v)$ whose generic element $E$ admits a short exact sequence $$0\to T\to E\to F\to 0\mbox{ or }0\to F\to E\to T\to 0,$$ where $T\in M^s_{\sigma_+}(\w)$ and $F\in M^s_{\sigma_+}(\v-\w)$, such that the inclusion $T\into E$ or $F\into E$ is one of the filtration steps in a JH-filtration for $E$ with respect to $\sigma_0$, while the generic $E\in M_{\sigma_+}(\v)$ satisfies $\Hom(T,E)=\Hom(E,T)=0$.  Moreover, when $\v$ is minimal, this divisor is precisely the locus of strictly $\sigma_0$-semistable objects.
\end{Lem}
\begin{proof}
As before, we may assume that $\v$ is minimal.  Then in terms of Proposition \ref{Prop:lattice classification}, $\HH_\WW$ must fall into cases \ref{enum:OneExceptional}, \ref{enum:TwoExceptional}, or \ref{enum:OneExceptionalOneSpherical}.  By minimality of $\v$, $\w$ must be $\w_0$ or $\w_1$ in Case \ref{enum:TwoExceptional} or $\w_0$ in Case \ref{enum:OneExceptionalOneSpherical}, and we assume it is $\w_0$ with the other case being dealt with similarly.  Furthermore, observe that there cannot exist any $n\in\Z$ such that $\v^2=n^2$ as then $(\v-n\w_0)^2=0$, contrary to the hypothesis that $\WW$ is non-isotropic.  So, in particular, $\v^2\geq 2$, from which it follows that $(\v-\w_0)^2>0$.

Let us first show that $\v-\w_0$ is minimal.  As $\langle \v-\w_0,\w_0\rangle=1>0$, this is clear in \cref{enum:OneExceptional}, so we can restrict ourselves to Cases \ref{enum:TwoExceptional} and \ref{enum:OneExceptionalOneSpherical}.  Thus it remains to show that $0\leq\langle \v-\w_0,\w_1\rangle$.  Writing $\v=x\w_0+y\w_1$, the conditions $\langle \v,\w_0\rangle=0, \v^2\geq 2$, and $\langle \v-\w_0,\w_1\rangle\geq 0$ become $$y=\frac{x}{m},x^2\left(1+\frac{\w_1^2}{m^2}\right)\geq 2,\mbox{ and }m\left[x\left(1+\frac{\w_1^2}{m^2}\right)-1\right]\geq 0,$$ respectively.    As $m\geq 2$ in either case, we get $$\left(1+\frac{\w_1^2}{m^2}\right)=\begin{cases}
1-\frac{1}{m^2}, &\text{if }\w_1^2=-1,\\
1-\frac{2}{m^2},&\text{if }\w_1^2=-2
\end{cases}\geq\begin{cases}
\frac{3}{4}, &\text{if }\w_1^2=-1,\\
\frac{1}{2}, &\text{if }\w_1^2=-2,
\end{cases}\geq\frac{1}{2},$$ in either case.
Thus $$x^2\left(1+\frac{\w_1^2}{m^2}\right)^2\geq x^2\left(1+\frac{\w_1^2}{m^2}\right)\left(\frac{1}{2}\right)\geq 2\left(\frac{1}{2}\right)\geq 1.$$  Taking square-roots gives that indeed $$x\left(1+\frac{\w_1^2}{m^2}\right)-1\geq 0,$$ as required.

We consider the case when $\v^2=2$, $L\equiv D+\frac{\rk\v}{2}K_X\pmod 2$, and $\HH$ falls into \cref{enum:OneExceptionalOneSpherical} of \cref{Prop:lattice classification}.  As $\langle\v-\w_0,\v-2\w_0\rangle=0$ and $(\v-2\w_0)^2=-2$, it follows from the minimality of $\v-\w_0$ that $\v-2\w_0=\w_1=\v(S)$, for the unique $\sigma_0$-stable spherical object $S$.  We denote by $T_0$ the unique $\sigma_0$-stable object of class $\w_0$ (up to $-\otimes\OO_X(K_X)$).  By stability, we have $$\hom(S,T_0(D))=\ext^2(S,T_0(D))=\hom(T_0(D),S)=0,$$ for $D=0,K_X$.  Thus $\ext^1(S,T_0\oplus T_0(K_X))=\langle\v-2\w_0,2\w_0\rangle=4$, so by \cite[Lemma 6.1-6.3]{CH15} there is a $\P^1\times \P^1$ worth of non-isomorphic $\sigma_+$-stable $E$ fitting into a short exact sequence $$0\to T_0\oplus T_0(K_X)\to E\to S\to 0.$$  As this gives a contracted divisor, it must be the only contracted divisor, as claimed.

We now treat the general case, that is, either $\v^2>2$ or $\v^2=2$ and either $L\nequiv D+\frac{\rk\v}{2}K_X\pmod 2$ or $\HH$ does not fall into \cref{enum:OneExceptionalOneSpherical} of \cref{Prop:lattice classification}.  As $\v-\w_0$ is minimal, it follows from Lemma \ref{Lem:non-isotropic no totally semistable wall} that there exists a $\sigma_0$-stable object $F$ of class $\v-\w_0$.  By stability $\hom(F,T_0)=\ext^2(F,T_0)=\hom(F,T_0(K_X))=0$, so $\ext^1(F,T_0)=\langle \v-\w_0,\w_0\rangle=1$, and there exists a unique non-trivial extension $$0\to T_0\to E\to F\to 0,$$ which is $\sigma_+$-stable by \cite[Lemma 6.9]{BM14b}.  By a dimension count, upon varying $F\in M_{\sigma_0}^s(\v-\w_0)$ these extensions sweep out a divisor of strictly $\sigma_0$-semistable objects which does not get contracted by $\pi^+$.  Moreover, from the proof of \cref{Lem:non-isotropic no totally semistable wall}, it follows that this is precisely the locus strictly $\sigma_0$-semistable objects.
\end{proof}

\begin{Rem}
We will see in Section \ref{Sec:FloppingWalls} that, if $\v^2\geq 3$, then in the setup of Lemma \ref{Lem:ExceptionalDivisorialNonContraction}, $\WW$ induces a small contraction, contracting a $\P^1\times\P^1$.  As the weakly-spherical reflection $R_{T}$ induces an isomorphism $M_{\sigma_+}(\v)\isomor M_{\sigma_-}(\v)$ that acts as the identity on $M_{\sigma_+}(\v)\backslash D_{\sigma_+}(\v)$, where $\Hom(T,E)=\Hom(E,T(K_X))=0$, we see that $M_{\sigma_-}(\v)$ cannot be the flop of $\pi^+$.  It is unclear whether examples such as this show that there are minimal models of $M_{\sigma_+}(\v)$ which cannot be obtained by Bridgeland wall-crossing.  On the other hand, it may be possible to reach this minimal model by crossing a different wall bounding the chamber containing $\sigma_+$.
\end{Rem}

\begin{proof}[Proof of Proposition \ref{Prop:NonisotropicDivisorialContraction}]
This follows directly from Lemma \ref{Lem:non-isotropic no totally semistable wall}, Lemma \ref{Lem:NonisotropicDivisorialContraction}, and Lemma \ref{Lem:ExceptionalDivisorialNonContraction}.
\end{proof}

\begin{Rem}\label{Rem:NonIsotropicDeterminantIrrelevant}
It is important to note that with the exception of \cref{Lem:ExceptionalDivisorialNonContraction}, everything we have said thus far in the non-isotropic case applies to each component $M_{\sigma_+}(\v,L),M_{\sigma_+}(\v,L+K_X)$, where $c_1(\v)=[L\mod K_X]$.  In particular, by taking $F\in M^s_{\sigma_0}(\v-\w,L'),M_{\sigma_0}^s(\v-\w,L'+K_X)$ in Lemmas \ref{Lem:NonisotropicDivisorialContraction} and \ref{Lem:ExceptionalDivisorialNonContraction}, we get divisors with the described properties in each component $M_{\sigma_+}(\v,L),M_{\sigma_+}(\v,L+K_X)$.  We will see in the next section that we must take great care to treat the determinants differently as the wall-crossing behavior is often radically different in each component, in a similar way to the case $\v^2=2$, $L\equiv D+\frac{\rk\v}{2}K_X\pmod2$, and $\v-2\w$ is spherical in \cref{Lem:ExceptionalDivisorialNonContraction}.
\end{Rem}

\section{Isotropic walls}\label{Sec:Isotropic walls}
We finally treat the case of isotropic walls.  We divide our discussion in two.  We will first discuss the case that $\HH$ contains a primitive isotropic vector $\u$ with $\ell(\u)=2$, in which case the wall $\WW$ corresponds, after a Fourier-Mukai transform, to the contraction to the Uhlenbeck compactification, see \cite{Li97,Lo12}.  We will consider separately the case that $\HH$ only contains primitive isotropic vectors $\u$ with $\ell(\u)=1$.  In both cases, we again use the stack of Harder-Narasimhan filtrations, as in Section \ref{sec:DimensionsOfHarderNarasimhan}, to study the wall-crossing behavior.  We begin by studying in more detail the isotropic lattice $\HH$, its isotropic vectors, and the associated moduli spaces.
\subsection{Preliminaries}
We state here a result that summarizes the facts we will need for a more detailed study of wall-crossing in the isotropic case.
\begin{Prop}\label{Prop:isotropic lattice} Assume that there exists an isotropic class $\u\in\HH$.  Then there are two effective, primitive, isotropic classes $\u_1$ and $\u_2$ in $\HH$, which satisfy $P_{\HH}=\R_{\geq 0}\u_1+\R_{\geq 0}\u_2$ and $\langle\v',\u_i\rangle\geq 0$ for $i=1,2$ and any $\v'\in P_{\HH}$.  Moreover, one of the following mutually exclusive conditions holds:
\begin{enumerate}
    \item\label{enum:IsotropicLatticeNoEffectiveNegatives} $C_\WW=P_\HH$ and $\ell(\u_1)\geq\ell(\u_2)$.  
    In this case, $M^s_{\sigma_0}(\u_i)=M_{\sigma_0}(\u_i)$ for each $i=1,2$ and a generic $\sigma_0\in\WW$; or 
    \item\label{enum:IsotropicLatticeEffectiveExceptional} There exists an exceptional class $\w\in C_\WW\cap\HH$.  In this case $\ell(\u_1)=\ell(\u_2)$, $\HH=\Z\w+\Z\u_1$ and $\u_2=\u_1+2\langle\u_1,\w\rangle\w$.  Consequently, $\langle\u_1,\u_2\rangle=2\langle\u_1,\w\rangle^2$ and $C_\WW=\R_{\geq 0}\u_1+\R_{\geq 0}\w$.  Finally, in this case $M_{\sigma_0}^s(\u_1)=M_{\sigma_0}(\u_1)$ for a generic $\sigma_0\in\WW$, while $\WW$ is a totally semistable wall for $\u_2$; or
    \item\label{enum:IsotropicLatticeEffectiveSpherical} There exists a spherical class $\w\in C_\WW\cap\HH$.  In this case, we again have $\ell(\u_1)=\ell(\u_2)$, $\HH=\Z\w+\Z\u_1$, and $\u_2=\u_1+\langle\u_1,\w\rangle\w$.  Consequently, $\langle\u_1,\u_2\rangle=\langle\u_1,\w\rangle^2$ and $\HH$ is even in this case.\footnote{All of these conclusions continue to hold in Case \ref{enum:IsotropicLatticeNoEffectiveNegatives} if $\HH$ admits a class $\w\notin C_\WW$ with $\w^2=-2$.}  Finally, $C_\WW=\R_{\geq 0}\u_1+\R_{\geq 0}\w$ and $M_{\sigma_0}^s(\u_1)=M_{\sigma_0}(\u_1)$ for a generic $\sigma_0\in\WW$, while $\WW$ is a totally semistable wall for $\u_2$.
\end{enumerate}
\end{Prop}
\begin{proof}
If $\u\in\HH$ is a primitive isotropic class, then up to replacing $\u$ by $-\u$, we may assume that $\u$ is effective, so we set $\u_1=\u$.  Completing $\u_1$ to a basis $\HH=\Z\u_1+\Z\v'$, we see that \begin{equation}\label{eqn:OtherIsotropic}0=(x\u_1+y\v')^2=2xy\langle \u_1,\v'\rangle+y^2 (\v')^2\end{equation} has a second integral solution, since we can assume $y\neq 0$ and $\langle \u_1,\v'\rangle\neq 0$ from the signature of $\HH$.  Taking the unique effective primitive class on the corresponding line, we get $\u_2$.  Clearly $P_{\HH}$ is as claimed, and the inequality $\langle \v',\u_i\rangle\geq 0$ follows accordingly.

If $C_\WW=P_\HH$, then the claim about moduli spaces in Case \ref{enum:IsotropicLatticeNoEffectiveNegatives} follows from the fact that $\u_1$ and $\u_2$ are primitive vectors on extremal rays of $C_\WW$.  Moreover, up to renumbering, we assume that $\ell(\u_1)\geq \ell(\u_2)$ and, in case of equality, $\langle\v,\u_1\rangle\geq\langle\v,\u_2\rangle$.

Suppose that there exists a class $\w\in C_\WW$ that is not in $P_\HH$.  Then either $\w$ is exceptional or spherical as in Case \ref{enum:IsotropicLatticeEffectiveExceptional} or \ref{enum:IsotropicLatticeEffectiveSpherical}, respectively.

Let us consider first Case \ref{enum:IsotropicLatticeEffectiveExceptional}, and write $\w\in C_\WW\cap\Delta(X)_{-1}$ as $\w=x\u_1+y\v'$.  Then 
$$-1=\w^2=y(y \v'^2+2x \langle \v',\u_1 \rangle)$$ implies
$y=\pm 1$.  Replacing $\v'$ by $\w$, we see that $\HH=\Z\u_1+\Z\w$.  Then it is easy to that the other primitive effective isotropic vector must satisfy \begin{equation}\label{eqn:U2Exceptional}\u_2=\u_1+2 \langle \u_1,\w \rangle \w.\end{equation}  Pairing this equality with $\u_1$, it is clear that $\langle \u_1,\u_2\rangle=2\langle\u_1,\w\rangle^2$.  Moreover, we see that $c_1(\u_2)\equiv c_1(\u_1)\pmod 2$ from \eqref{eqn:U2Exceptional}, so we get the last statement that $\ell(\u_1)=\ell(\u_2)$.

Now consider Case \ref{enum:IsotropicLatticeEffectiveSpherical}, and write $\w\in C_\WW\cap\Delta(X)_{-2}$ as $\w=x\u_1+y\v'$.  Then $$-2=\w^2=y(y\v'^2+2x\langle\u_1,\v'\rangle)$$ implies that $y=\pm1,\pm2$.  If $y=\pm2$, then $$\mp1=\pm 2{\v'}^2+2x\langle\u_1,\v'\rangle=2(\pm\v'+x\langle\u_1,\v'\rangle),$$ which is absurd.  Thus $y=\pm1$, so that replacing $\v'$ by $\w$ we have $\HH=\Z\u_1+\Z\w$.  In this case, the other primitive effective isotropic vector must satisfy \begin{equation}\label{eqn:U2Spherical}
    \u_2=\u_1+\langle\u_1,\w\rangle\w.
\end{equation}
Pairing \eqref{eqn:U2Spherical} with $\u_1$ gives $\langle\u_1,\u_2\rangle=\langle\u_1,\w\rangle^2$.  To see that $\ell(\u_1)=\ell(\u_2)$, observe that if $\ell(\u_i)=2$ for say $i=1$, then $2\mid c_1(\u_1)$ implies that $2\mid\langle\u_1,\w\rangle$ as $\rk\w\equiv\rk\u_1\equiv 0\pmod 2$.  Thus $\u_1\equiv\u_2\pmod 2$ so that $\ell(\u_2)=2$ as well.  Otherwise, $\ell(\u_1)=\ell(\u_2)=1$, and we have equality again.  Finally, note that for any $\v=x\u_1+y\w$ with $x,y\in\Z$, we have $$\v^2=2xy\langle\u_1,\w\rangle+y^2\w^2=2xy\langle\u_1,\w\rangle-2y^2$$ is even, as claimed.

Observe that Cases \ref{enum:IsotropicLatticeEffectiveExceptional} and \ref{enum:IsotropicLatticeEffectiveSpherical} are indeed mutually exclusive since $\HH$ is an odd lattice in the first case and an even lattice in the latter.

For the statements about $C_\WW$ and moduli spaces in Cases \ref{enum:IsotropicLatticeEffectiveExceptional} and \ref{enum:IsotropicLatticeEffectiveSpherical}, observe that $C_\WW=P_\HH+\R_{\geq 0}\w$ and $\langle\u_2,\w\rangle=-\langle\u_1,\w\rangle$ in either case.  So, up to reordering, we may suppose that $\langle\u_1,\w\rangle>0$ and $\langle\u_2,\w\rangle<0$.  In particular, $\u_1$ is an extremal ray of $C_\WW$ (see Figure \ref{fig:IsotropicWithNegative}), and $\WW$ is totally semistable for $\u_2$ by Lemma \ref{Lem: condition for totally semistable wall}.  Thus $M^s_{\sigma_0}(\u_1)=M_{\sigma_0}(\u_1)$, as claimed.
\end{proof}
\begin{figure}
   \begin{tikzpicture}[scale=1]
   \draw [->] (-4,0) -- (4,0);
   \draw[->] (0,-1) -- (0,4);
   \path [fill=gray!50,opacity=0.2] (0,4) -- (0,0) -- (4,0) -- (4,4) -- cycle;
	\path [fill=gray!60,opacity=0.4] (0,0) -- (4,0) -- (4,4) -- (4/3,4) -- cycle;
           \draw [red,domain=-4:4] plot (\x,{0});
   \draw [red,domain=-1/3:4/3] plot (\x,{3*\x});
   \filldraw [black] (0,3) circle (1.5pt) node [anchor=north east] {$\w$};
	\filldraw [black] (1,0) circle (1.5pt) node [anchor=north east] {$\u_1$};
	\filldraw [black] (1,3) circle (1.5pt) node [anchor=north east] {$\u_2$};
	\node[below] at (2,2) {$P_\HH$};
   \end{tikzpicture}
   \caption{The red lines are defined by $\u^2=0$.  The dark gray region is the positive cone $P_\HH$, while the first quadrant is the effective cone $C_\WW$.}
   \label{fig:IsotropicWithNegative}
\end{figure}
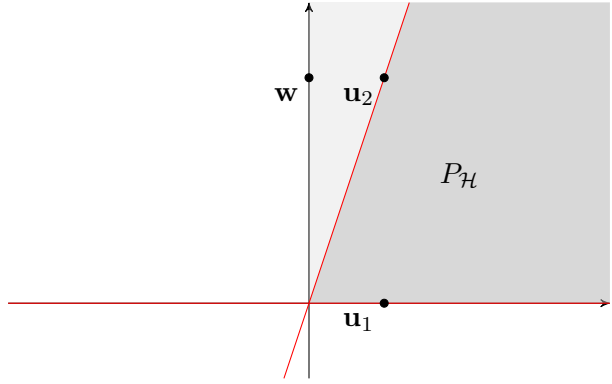
\begin{Rem}\label{Rem:Odd lattice}
It is worth noting that if $\HH$ contains a vector $\v$ such that $\v^2$ is odd, then $\ell(\u_1)=\ell(\u_2)$.  Indeed, as above, we may complete $\u_1$ to a basis so that $\HH=\Z\u_1+\Z\v'$, and as $\v^2$ is odd, we must have $\v'^2$ is odd as well.  From \eqref{eqn:OtherIsotropic}, we see immediately that in writing $\u_2=x\u_1+y\v'$ with $\gcd(x,y)=1$, we must have $y$ even and $x$ odd.  But then $c_1(\u_1)\equiv c_1(\u_2)\pmod 2$ so that $\ell(\u_1)=\ell(\u_2)$, as claimed.
\end{Rem}
\begin{Rem}\label{Rem:Even and Odd pairings}
We note that $\ell(\u_1)=2$ and $2\mid\rk(\u_2)$ force $\langle\u_1,\u_2\rangle$ to be even.  Let us be more specific each case.  

In Case \ref{enum:IsotropicLatticeNoEffectiveNegatives}, we must have $\langle\u_1,\u_2\rangle\geq 4$ if $\ell(\u_2)=2$.  Indeed, if $\ell(\u_2)=2$, then $\u_1-\u_2$ is divisible by 2 in $\Hal(X,\Z)$, and thus in the saturated sublattice $\HH$ as well, so $\langle\u_1,\u_2\rangle=2$ would imply that $\HH$ contains the exceptional class $\frac{\u_1-\u_2}{2}$, an impossibility.  If, instead, $\ell(\u_2)=1$, then $\HH$ must be even by Remark \ref{Rem:Odd lattice}.  

In Case \ref{enum:IsotropicLatticeEffectiveSpherical}, when there exists an effective spherical class, we also have $\langle\u_1,\u_2\rangle\geq 4$ since then $\langle\u_1,\w\rangle$ is even so that  $\langle\u_1,\u_2\rangle=\langle\u_1,\w\rangle^2$ is divisible by 4. 

In \cref{enum:IsotropicLatticeEffectiveExceptional}, we have $\langle\u_1,\w\rangle$ is odd so that $4\nmid\langle\u_1,\u_2\rangle=2\langle\u_1,\w\rangle^2$.  Indeed, writing $\u_1=(2r_1,2c_1,s_1)$ and $\w=(r,c,\frac{s}{2})$ with $r,s$ odd integers, we observe that $$\langle\u_1,\w\rangle=2c_1.c-rs_1-r_1s\equiv -(r_1+s_1)\equiv1\pmod2$$ by \cref{primitive}.
\end{Rem}
\begin{Rem}\label{Rem:IsotropicOrientation}
Up to relabeling $\sigma_+$ and $\sigma_-$, we may assume that the orientation on $\HH_\R$ is as in Figure \ref{fig:IsotropicWithNegative}, even in Case \ref{enum:IsotropicLatticeNoEffectiveNegatives}.  That is, we will assume for the remainder of this section that $\phi^+(\u_1)<\phi^+(\u_2)<\phi^+(\w)$.
\end{Rem}

Our main result about isotropic walls is the following classification:

\begin{Prop}\label{Prop:isotropic-classification}
Assume that $\HH_{\WW}$ is isotropic and $\v\in\HH_\WW$ with $\v^2>0$. Set $r:=\rk \v$.
\begin{enumerate}
\item
If $\WW$ is totally semistable, that is, $M_{\sigma_0}^s(\v,L)=\emptyset$, then
\begin{enumerate}
\item\label{enum:IostropicTSS-NonMinimal} $\HH_\WW$ contains an effective exceptional or spherical class $\w$ such that $\langle\v,\w\rangle<0$; or
\item\label{enum:IsotropicTSS-HC} $\HH_\WW$ contains a primitive isotropic $\u$ such that $\ell(\u)=2$ and $\langle \v,\u \rangle=1$; or 
\item\label{enum:IsotropicTSS-P1Fibration Spherical l=1} $\HH_\WW$ contains a primitive isotropic $\u$ and an effective spherical class $\w$ such that $\langle \v,\u \rangle=\ell(\u)=2$, $\langle \v,\w \rangle=0$, $L\equiv D+\frac{r}{2}K_X \pmod 2$,
where $D$ is a nodal cycle; or 
\item\label{enum:IsotropicTSS-P1Fibration Exceptional} $\HH_\WW$ contains a primitive isotropic $\u$ and an effective exceptional class $\w$ such that
$\langle \v,\u \rangle=\ell(\u)=2$, $\langle \v,\w \rangle=0$, and $L \equiv K_X \pmod 2$; 
or
\item\label{enum:IsotropicTSS-P1Fibration Spherical l=2}$\HH_\WW$ contains a primitive isotropic $\u$ and an effective spherical class $\w$ such that $\langle \v,\u \rangle=1=\ell(\u)$, $\langle\v,\w\rangle=0$, and $L \equiv D+\frac{r}{2}K_X \pmod 2$, 
where $D$ is a nodal cycle.
\end{enumerate}
\item
$\WW$ induces a divisorial contraction if and only if  
\begin{enumerate}
\item\label{enum:IsotropicDivisorialContraction-HC} $\HH_\WW$ contains a primitive isotropic $\u$ and an effective exceptional class $\w$ such that $\langle\v,\u\rangle=1$, $\ell(\u)=2$, and $\langle\v,\w\rangle\neq0$; or
\item\label{enum:IsotropicDivisorialContraction-<v.w>=0} $\HH_\WW$ contains a primitive isotropic $\u$ and an effective spherical class $\w$ such that $\langle\v,\w\rangle=0$ and $\langle\v,\u\rangle>\ell(\u)$; or 
\item\label{enum:IsotropicDivisorialContraction-<v.u>=l(u)} $\HH_\WW$ contains a primitive isotropic $\u$ and an effective exceptional or spherical class $\w$ such that $\langle\v,\u\rangle=\ell(\u)$ and $\langle\v,\w\rangle\neq0$; or
\item\label{enum:IsotropicDivisorialContraction-<v.u>=l(u) no negative} $\HH_\WW$ satisfies $C_\WW=P_\HH$, contains a primitive isotropic $\u$ such that $\langle\v,\u\rangle=\ell(\u)$, and $\v^2\geq \ell(\u)+2$.
\end{enumerate}
\end{enumerate}
\end{Prop}
\begin{Rem}Let us make a few comment about irreducible components and determinants.
\begin{enumerate}
\item
In \cref{enum:IsotropicTSS-P1Fibration Spherical l=1}, then
$\v=\w+\u$ and $\v^2=2$.
If $L \equiv D+\frac{r}{2}K_X+K_X \pmod 2$, then $M_{\sigma_0}^s(\v,L)=M_{\sigma_0}(\v,L)$.  In \cref{enum:IsotropicTSS-P1Fibration Exceptional}, $\v=2(\w+\u)$. If $L \equiv K_X \pmod 2$, then
a connected component of $M_{\sigma_+}(\v,L)$ has two irreducible
components (Proposition \ref{prop:connected}) and $\WW$ is totally semistable for one of them.
If $L \equiv 0 \pmod 2$, then $M_{\sigma_+}(\v,L)\setminus M_{\sigma_0}^s(\v,L)$ is a divisor.  In \cref{enum:IsotropicTSS-P1Fibration Spherical l=2},
we also see that
$\v=\w+2\u$. By Proposition \ref{prop:irred-comp:v^2=2},
$M_\sigma (\v)$ has two irreducible components and
each component becomes totally semistable at walls
$\WW$ (of type \ref{enum:IsotropicTSS-P1Fibration Spherical l=1}) and $\WW'$ (of type \ref{enum:IsotropicTSS-P1Fibration Spherical l=2}).
\item
In \cref{enum:IsotropicDivisorialContraction-<v.u>=l(u) no negative} with $\ell(\u)=1$, 
assume that $\u=\u_1$.  Then we can show
$\v=\frac{\v^2}{2}\u_1+\u_2$ with $\langle \u_1,\u_2 \rangle=1$.  In particular $\ell(\u_1)=\ell(\u_2)=1$.
\end{enumerate}
\end{Rem}

We will prove \cref{Prop:isotropic-classification} in the next two subsections.  For clarity of discussion, in the next section we will first tackle the case that $\v$ is minimal, that is, $\langle\v,\w\rangle\geq0$ for the unique exceptional/spherical class $\w\in\HH_\WW$.  Then we will show in \cref{subsec:non-minimal case} that if $\langle\v,\w\rangle<0$, we may apply the spherical or weakly-spherical twist associated to $T$ of class $\w$ to reduce the wall-crossing behavior to that of a minimal $\v$.
\subsection{Minimal Mukai vectors}
In this section we assume that $\v$ is \emph{minimal}, i.e. that $\langle \v,\w\rangle\geq 0$ for the (unique) spherical or exceptional effective class $\w$, if it exists.  When such $\w$ exists, let us denote by $T$ the spherical or exceptional $\sigma_0$-stable object of class $\w$.  We divide our analysis into two based on whether $\ell(\u_i)=2$ for some $i=1,2$ or not.  We begin with the case where $\ell(\u_i)=2$ for some $i$.
\subsubsection{$\ell(\u_i)=2$ for some $i$}
By Proposition \ref{Prop:isotropic lattice} above, we may assume that $i=1$ so that $M_{\sigma_0}^s(\u_1)=M_{\sigma_0}(\u_1)$ and $\ell(\u_1)=2$ implies that $M_{\sigma_0}(\u_1)\cong X$ by \cite[Lemma 9.3]{Nue14b}.  Using the Fourier-Mukai transform with kernel given by the universal family of $M_{\sigma_0}(\u_1)$, $$\Phi:\Db(X)\cong\Db(X),$$ we get $\Phi(\u_1)=(0,0,1)$.  By construction of $\Stabd(X)$, skyscraper sheaves of points on $X$ are $\Phi(\sigma_0)$-stable. By Bridgeland's Theorem \ref{thm:GeometricStabilityConditions}, there exist divisor classes $\omega,\beta\in\NS(X)_{\Q}$, with $\omega$ ample, such that up to the $\GL_2(\R)$-action, $\Phi(\sigma_0)=\sigma_{\omega,\beta}$. In particular, the category $\PP_{\omega,\beta}(1)$ is the extension-closure of skyscraper sheaves of points and the shifts, $F[1]$, of $\mu_{\omega}$-stable torsion-free sheaves $F$ with slope $\mu_{\omega}(F) =\omega\cdot\beta$. Since $\sigma_0$ by assumption does not lie on any other wall with respect to $\v$, the divisor $\omega$ is generic with respect to  $\Phi(\v)$. Under these identifications, we have the following result whose proof is identical to that of \cite[Theorem 3.2.7]{MYY14b}, \cite[Proposition 8.2]{BM14b}, and \cite[Section 5]{LQ14}.

\begin{Prop}\label{Prop:Uhlenbeck morphism}
An object $E$ of class $\v$ is $\sigma_+$-stable if and only if $\Phi(E)$ is the shift $F[1]$ of a $(\beta,\omega)$-Gieseker stable sheaf $F$ on $X$; therefore $[-1]\circ\Phi$ induces the following identification of moduli spaces: $$M_{\sigma_+}(\v) = M_{\omega}^{\beta}(-\Phi(\v)).$$
Moreover, the contraction morphism $\pi^+$ induced by the wall $\WW$ is the Li-Gieseker-Uhlenbeck (LGU) morphism to the Uhlenbeck compactification.

Similarly, an object $F$ of class $\v$ is $\sigma_-$-stable if and only if it is the shift $F^\vee[1]$ of the derived dual of a $(-\beta,\omega)$-Gieseker stable sheaf on $X$.
\end{Prop}
It follows from the above description that a $\sigma_+$-stable object $E$ becomes $\sigma_0$-semistable if and only if $F=\Phi(E)[-1]$ is not locally free or if $F$ is not $\mu$-stable, as these are the sheaves contracted by the Uhlenbeck contraction.
\begin{Prop}\label{Prop:LGU walls of low codimension}
Assume that $\HH_{\WW}$ contains a primitive isotropic vector $\u$ with $\ell(\u)=2$.  Suppose that $\v\in\HH_{\WW}$ satisfies $\v^2>0$ and $\langle \v,\w\rangle\geq 0$ for the (unique) effective spherical or exceptional class $\w$ (if it exists).
\begin{enumerate}
\item\label{enum:IsotropicTotallySemistable} If $\WW$ is totally semistable for $M_{\sigma_+}(\v)$, then either 
\begin{enumerate}
    \item\label{enum:IsotropicTotallySemistable-HC} $\langle\v,\u\rangle=1$ for primitive isotropic $\u\in\HH$ with $\ell(\u)=2$; or
    \item\label{enum:IsotropicTotallySemistable-Exceptional/Spherical} $\langle\v,\u\rangle=2$ for primitive isotropic $\u\in\HH$ with $\ell(\u)=2$ and $\langle\v,\w\rangle=0$ for the unique spherical or exceptional class $\w$.
\end{enumerate}
\item\label{enum:IsotropicCodimOne} If $\codim(M_{\sigma_+}(\v)\backslash M^s_{\sigma_0}(\v))=1$, then either 
\begin{enumerate}
    \item\label{enum:IsotropicCodimOne <v.u>=2=l(u)}$\langle\v,\u\rangle=2=\ell(\u)$ for primitive isotropic $\u\in\HH$; or
    \item\label{enum:IsotropicCodimOne <v.w>=0}$\langle\v,\w\rangle=0$ for the unique spherical or exceptional class $\w$; or 
\end{enumerate}
\item In all other cases, $\codim{M_{\sigma_+}(\v)\backslash M_{\sigma_0}^s(\v)}\geq 2$.
\end{enumerate}
\end{Prop}
\begin{proof}
We assume that $\u_1$, $\u_2$, and $\w$ are labelled and oriented in accordance with \cref{Rem:IsotropicOrientation} and the discussion preceeding \cref{Prop:Uhlenbeck morphism}.  In particular, we assume that $\ell(\u_1)=2$ and $\langle\u_1,\w\rangle>0$.

For a given $E\in M_{\sigma_+}(\v)$, let the Harder-Narasimhan filtration of $E$ with respect to $\sigma_-$ correspond to a decomposition $\v=\sum_i \a_i$.  Using Proposition \ref{Prop:HN codim}, we shall estimate the codimension of the sublocus $\FF(\a_0,\ldots,\a_n)^o$ of destabilized objects which is equal to
 \begin{equation}
\sum_i (\a_i^2-\dim \MM_{\sigma_-}(\a_i))+\sum_{i<j}\langle \a_i,\a_j \rangle.
\end{equation}

(I) We first assume that one of the $\a_i$ satisfies $\a_i^2<0$, say $\a_0=b_0 \w$ for an effective spherical or exceptional class $\w$.  Then we are Cases \ref{enum:IsotropicLatticeEffectiveExceptional} or \ref{enum:IsotropicLatticeEffectiveSpherical} in Proposition \ref{Prop:isotropic lattice} so that $\ell(\u_1)=\ell(\u_2)=2$.

Assume that $\a_1$ and $\a_2$ are isotropic.  We may set $\a_1=b_1 \u_1$ and $\a_2=b_2 \u_2$.  Then 
 \begin{equation}\label{eq:l=2, case I}
\begin{split}
& \sum_i (\a_i^2-\dim \MM_{\sigma_-}(\a_i))+\sum_{i<j}\langle \a_i,\a_j \rangle\\
\geq & (\a_0^2-\dim \MM_{\sigma_-}(\a_0))+\sum_{i \geq 1}b_0 \langle \w,\a_i \rangle
-b_1-b_2+b_1 b_2 \langle \u_1,\u_2 \rangle\\
= & -\dim\MM_{\sigma_-}(\a_0)+b_0 \langle \w,\v \rangle-b_1-b_2+b_1 b_2 \langle \u_1,\u_2 \rangle\\
\geq & -\dim\MM_{\sigma_-}(\a_0)-b_1-b_2+b_1 b_2 \langle \u_1,\u_2 \rangle,
\end{split}
\end{equation} 
where the first inequality follows from Proposition \ref{prop:isotropic} and the second inequality follows from the assumption that $\langle\v,\w\rangle\geq 0$.

First suppose that $\w^2=-2$.  Then we note that  $\dim\MM_{\sigma_-}(\a_0)=-b_0^2$ and $\langle\u_1,\w\rangle$ is even because $\ell(\u_1)=2$ and $\w^2=-2$ is even.  From Proposition \ref{prop:isotropic} we also have $\langle\u_1,\u_2\rangle=\langle\u_1,\w\rangle^2\geq 4$.  Thus 
\begin{equation}\label{eq:l=2, case I spherical}
\begin{split}
\codim\FF(\a_0,\ldots,\a_n)^o &\geq b_0^2+4b_1 b_2-b_1-b_2\geq 1+2b_1b_2+b_1(b_2-1)+b_2(b_1-1)\geq 3.
\end{split}
\end{equation}
So we must have $\codim\FF(\a_0,\dots,\a_n)^o\geq 3$ in this case.

If instead $\w^2=-1$, then $\dim\MM_{\sigma_-}(\a_0)=\left\lfloor-\frac{b_0^2}{2}\right\rfloor$.  Moreover, by \cref{Rem:Even and Odd pairings} $\langle\u_1,\w\rangle$ and $\frac{\langle\u_1,\u_2\rangle}{2}$ are odd.  Thus
\begin{equation}\label{eq:l=2, case I exceptional}
\begin{split}
\codim\FF(\a_0,\ldots,\a_n)^o &\geq \frac{b_0^2}{2}+2b_1b_2-b_1-b_2\geq\frac{1}{2}+b_1(b_2-1)+b_2(b_1-1)\geq\frac{1}{2}.
\end{split}
\end{equation}
Thus $\codim\FF(\a_0,\dots,\a_n)^o\geq 1$ with equality only if $b_0=b_1=b_2=1$.  But then $\v=\w+\u_1+\u_2$ so that $\langle\v,\w\rangle=\w^2=-1<0$, contrary to the assumption.  Thus $\codim\FF(\a_0,\dots,\a_n)^o\geq 2$ in this case.

Now assume that $\a_1=b_1 \u_j$ and $\a_i^2>0$ for $i>1$.  Then 
\begin{equation}
\begin{split}
 \sum_i (\a_i^2-\dim \MM_{\sigma_-}(\a_i))+\sum_{i<j}\langle \a_i,\a_j \rangle
&\geq  \a_0^2-\dim\MM_{\sigma_-}(\a_0)+\sum_{i \geq 1}b_0 \langle \w,\a_i \rangle
-b_1+\sum_{i \geq 2} b_1 \langle \u_j,\a_i \rangle\\
&\geq -\dim\MM_{\sigma_-}(\a_0)+b_0 \langle \w,\v \rangle
-b_1+b_1 \langle \u_j,\a_2 \rangle\\ 
&\geq -\dim\MM_{\sigma_-}(\a_0)+b_1(\langle\u_j,\a_2\rangle-1).
\end{split}
\end{equation}

If $\w^2=-2$, then $\HH$ is even by Proposition \ref{Prop:isotropic lattice}, so $\ell(\u_j)=2$ implies that $2\mid\langle\u_j,\a_2\rangle$.  Thus $$\codim\FF(\a_0,\dots,\a_n)^o\geq-\dim\MM_{\sigma_-}(\a_0)+b_1=b_0^2+b_1\geq2.$$  

If $\w^2=-1$, then $$\codim\FF(\a_0,\ldots,\a_n)^o\geq-\dim\MM_{\sigma_-}(\a_0)=-\left\lfloor-\frac{b_0^2}{2}\right\rfloor\geq 1,$$ with equality in the last inequality only if $b_0=1$, $\langle\v,\w\rangle=0$ and $\langle\u_j,\a_2\rangle=1$.  But $$1\leq\langle\v,\u_j\rangle=\langle\w,\u_j\rangle+\langle\a_2,\u_j\rangle=\langle\w,\u_j\rangle+1,$$ which forces $j=1$.  Writing $\a_2=x\u_1+y\w$, we see that $1=\langle\u_1,\a_2\rangle$ forces $y=1$ and $\langle\u_1,\w\rangle=1$.  Thus, in addition to $\langle\v,\w\rangle=0$, we also have $\langle\v,\u_1\rangle=2$.  So $\v$ falls into both types \ref{enum:IsotropicCodimOne <v.u>=2=l(u)} and \ref{enum:IsotropicCodimOne <v.w>=0}. Otherwise, we get $\codim\FF(\a_0,\dots,\a_n)^o\geq 2$.

We can now assume that there are no positive classes in the Harder-Narasimhan factors, i.e. $\v=b_0 \w+b_1 \u_j$.  But $\v^2>0$ forces $j=1$ (see Figure \ref{fig:IsotropicWithNegative}), so we may assume this outright.  Then $0 \leq \langle \v,\w \rangle=b_0 \w^2+b_1 \langle \u_1,\w \rangle$, so our estimate becomes 

 \begin{equation}\label{eq:l=2, case I no positive}
\begin{split}
\codim\FF(\a_0,\a_1)^o=&\sum_i (\a_i^2-\dim \MM_{\sigma_-}(\a_i))+\sum_{i<j}\langle \a_i,\a_j \rangle\\
= & b_0^2 \w^2-\dim\MM_{\sigma_-}(\a_0)-b_1+b_0b_1\langle \w,\u_1\rangle\\
\geq & b_0^2 \w^2-\dim\MM_{\sigma_-}(\a_0)+\frac{b_0b_1\langle \w,\u_1\rangle}{2}+\frac{b_1}{2}\left(b_0\langle \w,\u_1\rangle-2\right)\\
\geq & b_0^2 \w^2-\dim\MM_{\sigma_-}(\a_0)+\frac{b_0b_1\langle \w,\u_1\rangle}{2}+\frac{b_1}{2}\left(\langle\v,\u_1\rangle-2\right).
\end{split}
\end{equation}

If $\w^2=-2$, then $b_0^2\w^2-\dim\MM_{\sigma_-}(\a_0)=-b_0^2$, and $\HH$ is even so that again $\ell(\u_1)=2$ implies that $2\mid\langle\v,\u_1\rangle$.  The last line of \eqref{eq:l=2, case I no positive} then becomes $$\codim\FF(\a_0,\a_1)^o\geq-b_0^2+\frac{b_0 b_1\langle \w,\u_1\rangle}{2}+\frac{b_1}{2}(\langle \v,\u_1\rangle-2)=\frac{b_0}{2}\langle \v,\w\rangle+\frac{b_1}{2}(\langle \v,\u_1\rangle-2)\geq 0,$$ with equality only if $\langle \v,\w\rangle=0$ and $\langle \v,\u_1\rangle=2$, as in Case \ref{enum:IsotropicTotallySemistable-Exceptional/Spherical}.  
Moreover, $\codim \FF(\a_0,\a_1)^o\geq 2$ unless $\langle \w,\u_1\rangle=2$ and $\v=\w+2\u_1$.  In this case we have $\codim\FF(\a_0,\a_1)^o=1$ as in Case \ref{enum:IsotropicCodimOne <v.u>=2=l(u)}.

If $\w^2=-1$, then $b_0^2\w^2-\dim\MM_{\sigma_-}(\a_0)=\left\lceil-\frac{b_0^2}{2}\right\rceil$, so the last line of \eqref{eq:l=2, case I no positive} gives \begin{equation}\label{eqn:l=2 case I no positive exceptional}\codim\FF(\a_0,\a_1)^o\geq-\frac{b_0^2}{2}+\frac{b_0 b_1\langle \w,\u_1\rangle}{2}+\frac{b_1}{2}(\langle \v,\u_1\rangle-2)=\frac{b_0}{2}\langle \v,\w\rangle+\frac{b_1}{2}(\langle \v,\u_1\rangle-2)\geq0\end{equation}
unless $\langle \v,\u_1\rangle=1$, in which case $b_0=1=\langle\w,\u_1\rangle$.  But then $b_0^2\w^2-\dim\MM_{\sigma_-}(\a_0)=0$,  so $$\codim\FF(\a_0,\a_1)^o\geq \frac{b_1}{2}-\frac{b_1}{2}=0,$$ and indeed $\codim\FF(\a_0,\a_1)^o=0$ when $\langle\v,\u_1\rangle=1$ as in Case \ref{enum:IsotropicTotallySemistable-HC}.  We get equality in \eqref{eqn:l=2 case I no positive exceptional} only if $\langle\v,\w\rangle=0$ and $\langle\v,\u_1\rangle=2$.  We can derive from these two equations that $\langle\w,\u_1\rangle=1$ and $\v=2\w+2\u_1$, in which case indeed $\codim\FF(\a_0,\a_1)^o=0$ as in Case \ref{enum:IsotropicTotallySemistable-Exceptional/Spherical}.  Moreover, $\codim \FF(\a_0,\a_1)^o\geq 2$ unless $\frac{b_0}{2}\langle\v,\w\rangle=1$ and $\langle \v,\u_1\rangle=2$; or $\langle\v,\w\rangle=0$ and $\frac{b_1}{2}(\langle\v,\u_1\rangle-2)=1$;  or $\frac{b_0}{2}\langle\v,\w\rangle=\frac{1}{2}$ and $\langle\v,\u_1\rangle=2$.  In the latter two cases, however, we would have $\langle\w,\u_1\rangle=2$, which is impossible, as noted in \cref{Rem:Even and Odd pairings}.  Thus we are left with  the first case, in which $\v=2\w+3\u_1$ and $\langle\w,\u_1\rangle=1$.  Notice that $\langle\v,\u_1\rangle=2$ so that we are in \cref{enum:IsotropicCodimOne <v.u>=2=l(u)}.

Finally, assume that other than $\a_0=b_0 \w$, $\a_i^2>0$ for all $i>0$.  Then the estimate becomes 

\begin{equation}\label{eq:l=2 case I no isotropic}
\begin{split}
\codim\FF(\a_0,\ldots,\a_n)^o=&\sum_i (\a_i^2-\dim \MM_{\sigma_-}(\a_i))+\sum_{i<j}\langle \a_i,\a_j \rangle\\
= & -\dim\MM_{\sigma_-}(\a_0)+b_0\langle \w,\v\rangle+\sum_{0<i<j}\langle \a_i,\a_j\rangle\geq \frac{b_0^2}{2}>0.
\end{split}
\end{equation}
Moreover, $\codim\FF(\a_0,\ldots,\a_n)^o\geq 2$ unless $\langle \w,\v\rangle=0$ and $\v=\w+\a_1$, in which case $\FF(\a_0,\a_1)^o$ has codimension one as in \cref{enum:IsotropicCodimOne <v.w>=0}.  Note that in this case we require $0<\a_1^2=(\v-\w)^2=\v^2+\w^2$, so $\v^2>2$ or $\v^2>1$ if $\w^2=-2$ or $\w^2=-1$, respectively.

(II) We next assume that $\a_i^2 \geq 0$ for all $i$.

Suppose first that $\a_1=b_1 \u_1$ and $\a_2=b_2 \u_2$.  Then we can be in any case of Proposition \ref{Prop:isotropic lattice}, and our estimate now becomes
 \begin{equation}\label{eq:l=2 case II}
\begin{split}
\codim\FF(\a_1,\dots,\a_n)^o=& \sum_i (\a_i^2-\dim \MM_{\sigma_-}(\a_i))+\sum_{i<j}\langle \a_i,\a_j \rangle\\
\geq &
-b_1-\left\lfloor\frac{b_2\ell(\u_2)}{2}\right\rfloor+b_1 b_2 \langle \u_1,\u_2 \rangle\\
\geq &b_1(b_2-1)+b_2(b_1-1)\geq0,
\end{split}
\end{equation}
with equality only if $b_1=b_2=1$, $\ell(\u_2)=2$, and $\langle\u_1,\u_2\rangle=2$, so that $\HH$ falls into \cref{enum:IsotropicLatticeEffectiveExceptional} of \cref{Prop:isotropic lattice}.  In particular, we have $\langle\v,\u_1\rangle=2$ and $\langle\v,\w\rangle=\langle\u_1,\w\rangle+\langle\u_2,\w\rangle=0$, as in  \cref{enum:IsotropicTotallySemistable-Exceptional/Spherical}.  Furthermore, by \cref{Rem:Even and Odd pairings} we have $\codim\FF(\a_1,\dots,\a_n)^o=1$ only if $\HH$ falls into Case \ref{enum:IsotropicLatticeEffectiveExceptional} of Proposition \ref{Prop:isotropic lattice} and $\v=2\u_1+\u_2$ or $\u_1+2\u_2$ with $\langle\u_1,\u_2\rangle=2$, so that $\langle\v,\u\rangle=2=\ell(\u)$ as in \cref{enum:IsotropicCodimOne <v.u>=2=l(u)}, or $\HH$ falls into Case \ref{enum:IsotropicLatticeNoEffectiveNegatives} of \cref{Prop:isotropic lattice} and $\v=\u_1+\u_2$ with $\ell(\u_2)=1$ and $\langle\u_1,\u_2\rangle=2$, so that $\langle\v,\u_1\rangle=2$, as in \cref{enum:IsotropicCodimOne <v.u>=2=l(u)}.  Otherwise, $\codim\FF(\a_1,\dots,\a_n)^o\geq 2$.

Now we assume that $\a_1=b_1 \u_j$ and $\a_i^2>0$ for $i \geq 2$.
In this case, we also see that
\begin{equation}
\begin{split}
\codim\FF(\a_1,\dots,\a_n)^o=&\sum_i (\a_i^2-\dim \MM_{\sigma_-}(\a_i))+\sum_{i<j}\langle \a_i,\a_j \rangle\\
=&-\left\lfloor\frac{b_1\ell(\u_j)}{2}\right\rfloor+\sum_{i>1}b_1\langle \u_j,\a_i\rangle+\sum_{1<i<k}\langle \a_i,\a_k\rangle\\
\geq&b_1(\langle \v,\u_j\rangle-1)+\sum_{1<i<k}\langle \a_i,\a_k\rangle\\
\geq&b_1(\langle \v,\u_j\rangle-1)\geq 0.
\end{split}
\end{equation}
Thus $\codim\FF(\a_1,\dots,\a_n)^o=0$ only if $\langle\v,\u_j\rangle=1$, $\ell(\u_j)=2$, and $\v=b_1\u_j+\a_2$, as in \cref{enum:IsotropicTotallySemistable-HC}.  Similarly, $\codim\FF(\a_1,\dots,\a_n)^o=1$ only if either $\v=\u_j+\a_2$ with $\langle\v,\u_j\rangle=2$ and $\ell(\u_j)=2$ as in \cref{enum:IsotropicCodimOne <v.u>=2=l(u)}, or $\v=\u_2+\a_2$, $2\u_2+\a_2$ with $\langle\v,\u_2\rangle=1=\ell(\u_2)$.  But this latter case is impossible.  Indeed, $\HH$ must be even by \cref{Rem:Odd lattice} and $\ell(\u_1)\neq\ell(\u_2)$.  But then we see that we can write  $\u_1=-\frac{\a_2^2}{2}\u_2+\a_2$ so that $\langle\u_1,\u_2\rangle=1$, which is impossible as $\ell(\u_1)=2$.

Finally, if $\a_i^2>0$ for all $i$, then $\codim\FF(\a_1,\dots,\a_n)^o\geq 2$ by Proposition \ref{Prop:HN filtration all positive classes}.
\end{proof}

\begin{Rem}
Proposition \ref{Prop:LGU walls of low codimension} can be proven by using the Fourier-Mukai transform $\Phi$ to translate the problem to the equivalent problem of determining the codimensions of the strictly $\mu$-semistable locus and the non-locally free locus.  One could then use the estimates in \cite{Yos16a} to get the result.  However, there is a small error in Case C there, which misses the spherical case of \cref{enum:IsotropicTotallySemistable-Exceptional/Spherical}, so we use the method above.  We explore this example more fully in \cref{prop:irred-comp:v^2=2}.  
\end{Rem}

\begin{Rem}
The reader may notice that when $\langle \v,\w\rangle=0$ and $\langle \v,\u\rangle=2=\ell(\u)$, we simultaneously claim that $\codim(M_{\sigma_+}(\v)\backslash M^s_{\sigma_0}(\v))$ is both zero and one.  Indeed, we shall prove in \cref{App: exceptional case} that for one choice of the determinant $L$, $M_{\sigma_+}(\v,L)$ contains a connected component with two irreducible components, $M_0$ and $M_1$.  For $M_1$, $\WW$ is a totally semistable wall inducing a $\P^1$-fibration over the singular locus, where it meets $M_0$.  The strictly $\sigma_0$ locus on $M_0$ is this singular locus, which is a divisor.  When $\w^2=-2$, this describes all of $M_{\sigma_+}(\v,L)$ as it is connected.  For the other determinant, $L+K_X$, the strictly $\sigma_0$-semistable locus is a divisor, if non-empty.
\end{Rem}
Now we demonstrate the converse of \ref{Prop:LGU walls of low codimension} in the following sequence of lemmas.  Furthermore, we determine precisely when the divisor in \cref{enum:IsotropicCodimOne} gets contracted.  We make free use of the Fourier-Mukai transform $\Phi$ to translate the problem to that of moduli of sheaves.
\begin{Lem}\label{Lem: Hilbert-Chow}
Assume that $\HH_{\WW}$ contains a primitive isotropic vector $\u$ such that $\ell(\u)=2$ and that $\langle \v,\u\rangle=1$ for $\v$ minimal.  Then $\v^2$ is odd and $\HH_{\WW}$ contains an exceptional class.  Moreover, $\WW$ is totally semistable and, if $\v^2>1$, induces a divisorial contraction.  
\end{Lem}
\begin{proof}
Let us first prove that $\v^2$ is odd and $\HH$ contains an exceptional class.  Write $\v=(r,c,\frac{s}{2})$ and $\u=(2r',2c',s')$, with $r'+s'$ odd, as $\ell(\u)=2$.  Then as $r\equiv s\pmod 2$, it follows that $$1=\langle\v,\u\rangle=2c.c'-r's-rs'\equiv-r(r'+s')\equiv r\pmod2,$$ from which we see that $r$ is odd, or equivalently $\v^2$ is odd.  Thus $\v-\frac{\v^2+1}{2}\u\in\HH_\WW$ is an exceptional class.

As $\HH$ falls into \cref{enum:IsotropicLatticeEffectiveExceptional} of \cref{Prop:isotropic lattice}, we see that $\ell(\u_2)=2$ and \begin{equation}\label{eqn:u2 and u1}
    \u_2=\u_1+2\langle\u_1,\w\rangle\w.
\end{equation}  
We observe from \eqref{eqn:u2 and u1} that we may assume that $\u=\u_1$.  Indeed, if $\langle\v,\u_2\rangle=1$, then pairing \eqref{eqn:u2 and u1} with $\v$, we get  $$1=\langle \v,\u_2\rangle=\langle\v,\u_1\rangle+2\langle\u_1,\w\rangle\langle\v,\w\rangle\geq \langle\v,\u_1\rangle>0,$$ so that $\langle\v,\u_1\rangle=1$ as well.  It follows that $M_{\sigma_+}(\v)\cong M_{\omega}^{\beta}(-\Phi(\v))$ is isomorphic to the Hilbert scheme of points.    Finally, $\WW$ is the Hilbert-Chow wall inducing the Hilbert-Chow morphism $\Hilb^n(X)\to\Sym^n(X)$, which is a divisorial contraction for $1<n=\frac{\v^2+1}{2}$, and every ideal sheaf is strictly semistable as in \cite[Proposition 13.1]{Nue14b}.
\end{proof}

Although the behavior in \cref{Lem: Hilbert-Chow} is analogous with the corresponding case on K3 surfaces, we see some new behavior in the next two lemmas.  We begin with the exceptional case of \cref{enum:IsotropicTotallySemistable-Exceptional/Spherical} in \cref{Prop:LGU walls of low codimension}.  

\begin{Lem}\label{Lem:P1FibrationExceptional}
Suppose that $\HH_{\WW}$ contains a primitive isotropic $\u$ such that  $\langle\v,\u\rangle=2=\ell(\u)$ and $\langle \v,\w\rangle=0$ for an exceptional class $\w\in C_\WW\cap\HH_{\WW}$.  Then $\WW$ is totally semistable for, and induces a $\P^1$-fibration on, precisely one component of $M_{\sigma_+}(\v,2L'+K_X)$.
\end{Lem}
\begin{proof}
As $\langle\v,\w\rangle=0$, we see from \eqref{eqn:u2 and u1} that  $\langle\v,\u_2\rangle=\langle\v,\u_1\rangle$, so we may assume that $\u=\u_1$.  After applying $[-1]\circ\Phi$, and possibly tensoring by a line bundle, we may assume that $\v=(2,0,-1)$, $\w=(1,0,\frac{1}{2})$, and $M_{\sigma_+}(\v)$ is isomorphic to $M_H(-\Phi(\v))$, the moduli space of H-Gieseker semistable sheaves of Mukai vector $-\Phi(\v)$ with respect to a generic polarization $H$.  The contraction $\pi^+$ is then the LGU-contraction morphism as in Proposition \ref{Prop:Uhlenbeck morphism}.  As mentioned in \cite[Remark 2.3]{Yos16a} and proven in Section \ref{App: exceptional case} (see \cref{prop:connected}), there is precisely one component of $M_{\sigma_+}(\v,K_X)$ consisting of stable non-locally free sheaves $E$ fitting into the short exact sequence $$0\to E\to F\mor[(\phi_1,\phi_2)]\C_p\oplus\C_q\to 0,$$ where $F:=\OO_X\oplus\OO_X(K_X)$ and  $\phi_1:\OO_X\to\C_p\oplus\C_q$ and $\phi_2:\OO_X(K_X)\to\C_p\oplus\C_q$ are both surjective.  The polystable object in the same S-equivalence class as $E$ with respect to $\sigma_0$ is $(\C_p\oplus\C_q)[-1]\oplus F$, and the set of distinct $\sigma_+$ stable objects in the same S-equivalence class is parametrized by $\P\Hom(F,\C_p)\times\P\Hom(F,\C_q)/(\Aut(F)/\C^*)$, which is a curve birational to $\P^1$.  Thus $\WW$ is totally semistable for this component and induces a $\P^1$-fibration.  
\end{proof}

We get similar behavior in the spherical case:

\begin{Lem}\label{Lem:P1FibrationSpherical}
Suppose that $\HH_{\WW}$ contains a primitive isotropic $\u$ such that $\langle \v,\u\rangle=2=\ell(\u)$ and $\langle \v,\w\rangle=0$ for a spherical class $w\in C_\WW\cap\HH_{\WW}$.  
Then $\WW$ is totally semistable for, and induces a $\P^1$-fibration on, precisely one component of $M_{\sigma_+}(\v,L)$ with $L \equiv Z+\frac{\rk \v}{2}K_X \mod 2$,
where $Z$ is a nodal cycle. Moreover
$M_{\sigma_+}(\v,L+K_X)=M_{\sigma_0}^s(\v,L+K_X)$.
\end{Lem}

\begin{proof}
As in the proof of \cref{Lem:P1FibrationExceptional}, $\langle\v,\w\rangle=0$ means that $\langle\v,\u_2\rangle=\langle\v,\u_1\rangle$, so we may assume that $\u=\u_1$, and applying $[-1]\circ\Phi$, we may assume that $\v=(2,D,s)$ with $\gcd(2,D)=1$ and $s\in\Z$ and $\w=(2,D,\frac{D^2+2}{4})$, where $4\mid D^2+2$ since $D=c_1(\v)\equiv Z\pmod 2$ for a nodal cycle $Z$.  Then $M_{\sigma_+}(\v)$ is isomorphic to
$M_H(2,D,s)$.  Let $T$ be the stable spherical bundle with $\v(T)=\w$.  Then we have a family of non-locally free sheaves $E$ parameterized by a $\P^1$-bundle over $X$: 
$$
0 \to E \to T \to \C_p \to 0.
$$
This $\P^1$-bundle is one component of $M_{\sigma_+}(\v,L)$, and each such $E$ is strictly $\sigma_0$-semistable, S-equivalent to $T\oplus\C_p[-1]$.  As $\det(T)\equiv Z+\frac{\rk(\v)}{2}K_X\pmod 2$ and $\det(\C_p)=0$ so that $\det(E)\equiv Z+\frac{\rk(\v)}{2}K_X\pmod2$, we get the first claim.

For the other determinant, notice that as $\gcd(2,D)=1$, all stable sheaves are $\mu$-stable, so for any non-locally free $E\in M_H((2,D,s),L+K_X)$, $E^{\vee\vee}$ would be a $\mu$-stable locally free sheaf in $M_H((2,D,\frac{D^2+2}{4}),L+K_X)=\emptyset$.  Thus every $E\in M_H((2,D,s),L+K_X)$ is a $\mu$-stable locally free sheaf so that $M_{\sigma_+}(\v,L+K_X)=M_{\sigma_0}^s(\v,L+K_X)$, as claimed.
\end{proof}

Having considered the behavior of totally semistable walls, we move on to determining when the codimension one strictly $\sigma_0$-semistable locus gets contracted.  We begin with \cref{enum:IsotropicCodimOne <v.u>=2=l(u)}:
\begin{Lem}\label{Lem: isotropic divisorial l=2 1}
Suppose that $\langle \v,\u_1\rangle=2=\ell(\u_1)$ for $\v$ minimal.  Assume further that 
\begin{equation}\label{eqn:RestrictionsOnv^2}
    \begin{cases}
    \v^2\geq 4,\;\;\;&\mbox{in \cref{enum:IsotropicLatticeNoEffectiveNegatives} of \cref{Prop:isotropic lattice}};\\
    \v^2\geq 3,\v^2\neq 4,\;\;\;&\mbox{in \cref{enum:IsotropicLatticeEffectiveExceptional} of \cref{Prop:isotropic lattice}};\\   \v^2>2,\;\;\;&\mbox{ in \cref{enum:IsotropicLatticeEffectiveSpherical} of \cref{Prop:isotropic lattice}}.
    \end{cases}
\end{equation}
Then $\WW$ is not a totally semistable wall and induces a divisorial contraction on $M_{\sigma_+}(\v,L)$.
\end{Lem}
\begin{proof}
In all cases, we will show that our assumptions imply that $\WW$ is not totally semistable for $M_{\sigma_+}(\v-\u_1,L'),M_{\sigma_+}(\v-\u_1,L'+K_X)$.  Assuming we have done this, then taking $F\in M_{\sigma_0}^s(\v-\u_1,L')$ (or $M_{\sigma_0}^s(\v-\u_1,L'+K_X)$) and $G\in M_{\sigma_0}(\u_1)$, we get a $\P^1$ worth of distinct extensions $$0\to G\to E\to F\to 0$$ of objects in $M_{\sigma_+}(\v,L)$ (or $M_{\sigma_+}(\v,L+K_X)$, respectively) that are S-equivalent with respect to $\sigma_0$.  A quick dimension count shows that these sweep out a divisor.  

Now we show that $\WW$ is not totally semistable for $M_{\sigma_+}(\v-\u_1,L'),M_{\sigma_+}(\v-\u_1,L'+K_X)$.  Under the assumptions in \eqref{eqn:RestrictionsOnv^2}, we may have $(\v-\u_1)^2=-1$ in \cref{enum:IsotropicLatticeEffectiveExceptional} of \cref{Prop:isotropic lattice}.  But then $\v-\u_1=\w$ and $\langle\u_1,\w\rangle=2$, which is impossible by \cref{Rem:Even and Odd pairings}.  Otherwise, $(\v-\u_1)^2=\v^2-4\geq 0$, with equality only if $\v=\u_1+k\u_2$.  In this case, $\langle\v,\u_1\rangle=2$ implies that $k=1$, $\langle\u_1,\u_2\rangle=2$, $\ell(\u_2)=1$, and $\HH$ falls into \cref{enum:IsotropicLatticeNoEffectiveNegatives} by \cref{Rem:Even and Odd pairings} and \eqref{eqn:RestrictionsOnv^2}.  So $\v-\u_1=\u_2$, and  $M_{\sigma_0}^s(\u_2)=M_{\sigma_0}(\u_2)$ so that $\WW$ is not totally semistable for $M_{\sigma_+}(\u_2)$ (or a wall at all).

It remains to show that $\WW$ is not totally semistable when $(\v-\u_1)^2>0$.  As both conditions \ref{enum:IsotropicTotallySemistable-HC} and \ref{enum:IsotropicTotallySemistable-Exceptional/Spherical} of \cref{Prop:LGU walls of low codimension} for $\WW$ being totally semistable require the existence of a spherical/exceptional class (see \cref{Lem: Hilbert-Chow} for \cref{enum:IsotropicTotallySemistable-HC}), $\WW$ is automatically not totally semistable in \cref{enum:IsotropicLatticeNoEffectiveNegatives} of \cref{Prop:isotropic lattice}.  In \cref{enum:IsotropicLatticeEffectiveExceptional,enum:IsotropicLatticeEffectiveSpherical}, if we write $\v=x\u_1+y\w$ with $x,y\in\Z_{\geq 0}$, then by \cref{Rem:Even and Odd pairings}, $\langle\v,\u_1\rangle=2$ and $\v^2>4$ is equivalent to $x>2$, $y=2$ and $\langle\u_1,\w\rangle=1$ in \cref{enum:IsotropicLatticeEffectiveExceptional}, while in \cref{enum:IsotropicLatticeEffectiveSpherical}, it is equivalent to $x>1$, $y=1$ and $\langle\u_1,\w\rangle=2$.  Thus $\langle\v-\u_1,\w\rangle\geq 0$, and in case of equality, we may choose $F$ above to be in the component of $M_{\sigma_+}(\v-\u_1,L')$ (or $M_{\sigma_+}(\v-\u_1,L'+K_X)$) that contains $\sigma_0$-stable objects, as guaranteed by \cref{Lem:P1FibrationExceptional,Lem:P1FibrationSpherical}.  If $\langle\v-\u_1,\w\rangle>0$, then $\langle\v-\u_1,\u_2\rangle>\langle\v-\u_1,\u_1\rangle=2$, so $\WW$ is not totally semistable for $M_{\sigma_+}(\v-\u_1)$ by \cref{Prop:LGU walls of low codimension}, as claimed.

\end{proof}
It is worth noting that the possibilities excluded by the condition \eqref{eqn:RestrictionsOnv^2} have either been dealt with already, or are irrelevant.  Indeed, suppose $\langle\v,\u_1\rangle=2$ and $0<\v^2<4$.  If $\v^2=3$, then $(\v-\u_1)^2=-1$, so we must be in \cref{enum:IsotropicLatticeEffectiveExceptional} of \cref{Prop:isotropic lattice}, which was included in \eqref{eqn:RestrictionsOnv^2}.  If $\v^2=2$, then $(\v-\u_1)^2=-2$, so in \cref{enum:IsotropicLatticeEffectiveSpherical} we have already seen (and will prove in \cref{App: exceptional case}) that the divisorial component of the strictly $\sigma_0$-semistable locus is not contracted, while in \cref{enum:IsotropicLatticeNoEffectiveNegatives} we see that $\WW$ is not a wall for $\v$.  This is also the case if $\v^2=1$.  Finally, if $\v^2=4$ in \cref{enum:IsotropicLatticeEffectiveExceptional}, then again we have already seen (and will prove in \cref{App: exceptional case}) that the divisorial component of the strictly $\sigma_0$-semistable locus is not contracted.

Now we consider \cref{enum:IsotropicCodimOne <v.w>=0} of \cref{Prop:LGU walls of low codimension}:
\begin{Lem}\label{Lem: isotropic divisorial l=2 2}
Suppose that $\HH_{\WW}$ contains a primitive isotropic vector $\u$ with $\ell(\u)=2$ and $\v^2>2$.  If $\langle \v,\w\rangle=0$ for an effective spherical class $\w$, then $\WW$ induces a divisorial contraction.
\end{Lem}
\begin{proof}
Let $T$ be the unique $\sigma_0$-stable spherical object of class $\w$.  Consider $\a:=\v-\w$.  Then $\a^2=\v^2-2>0$ and $\langle \a,\w\rangle=2$.  By \cref{Prop:isotropic lattice} we see that $\HH$ is even, and from $\ell(\u)=2$ it follows that $\ell(\u_1)=\ell(\u_2)=2$.  Thus $\langle\a,\u_i\rangle$ is even and at least 2, by \cref{Rem:Even and Odd pairings}.  By \cref{Prop:LGU walls of low codimension}, $M_{\sigma_0}^s(\a)\neq\varnothing$, so letting $A$ vary in $M^s_{\sigma_0}(\a)$, we see that the $\P^1$'s of S-equivalent extensions $$0\to A\to E\to T\to 0$$ sweep out a contracted divisor in $M_{\sigma_+}(v)$.
\end{proof}
With the exact same proof, one can show that the analogous situation for $\w$ exceptional results in a divisor of strictly $\sigma_0$-semistable objects that does not get contracted.
\begin{Lem}\label{Lem:IsotropicNoncontractedDivisor l=2}
Suppose that $\HH_\WW$ contains a primitive isotropic vector $\u$ with $\ell(\u)=2$ and $\v^2\neq 1,4$.  If $\langle\v,\w\rangle=0$ for an effective exceptional class $\w$, then the locus $M_{\sigma_+}(\v)\backslash M_{\sigma_0}^s(\v)$ is a divisor that is not contracted.
\end{Lem}

\begin{Lem}\label{Lem: isotropic divisorial l=2 3}
Let $\WW$ be a potential wall and $\v$ minimal, and suppose that $\HH_\WW$ contains a primitive isotropic class $\u$ such that $\ell(\u)=2$.  Assume that either $\langle \v,\u\rangle=1$, or $\langle\v,\u\rangle=2$ and 
\begin{equation}\label{eqn:RestrictionsOnv^2-2}
    \begin{cases}
    \v^2\geq 4,\;\;\;&\mbox{in \cref{enum:IsotropicLatticeNoEffectiveNegatives} of \cref{Prop:isotropic lattice}};\\
    \v^2\geq 3,\v^2\neq 4,\;\;\;&\mbox{in \cref{enum:IsotropicLatticeEffectiveExceptional} of \cref{Prop:isotropic lattice}};\\   \v^2>2,\;\;\;&\mbox{ in \cref{enum:IsotropicLatticeEffectiveSpherical} of \cref{Prop:isotropic lattice}}.
    \end{cases}
\end{equation}
Then $\WW$ induces a divisorial contraction on $M_{\sigma_+}(\v,L)$.
\end{Lem}
\begin{proof}
The class $ֿֿֿֿֿֿ\u$ is automatically effective.  By \cref{Lem: Hilbert-Chow,Lem: isotropic divisorial l=2 1}, the only remaining case is $\u=\u_2$ and $\langle\v,\u_2\rangle=2$.

First, suppose that $\HH_{\WW}$ admits an effective exceptional or spherical class $\w$.  Then from the minimality of $\v$ and the assumptions in \eqref{eqn:RestrictionsOnv^2-2}, we have $\langle \v,\w\rangle> 0$. Indeed, we may write $\v=x\u_2+y\w$ with $x,y\in\Z$ by \cref{Prop:isotropic lattice}, so the conditions $\langle\v,\u_2\rangle=2$ and $\langle\v,\w\rangle=0$ imply that $\v^2=4$ or $2$ if $\w$ is exceptional or spherical, respectively, contrary to \eqref{eqn:RestrictionsOnv^2-2}.  Thus by \eqref{eqn:u2 and u1} (and the analogue for spherical $\w$), we must have $0<\langle \u_1,\v\rangle<\langle\v,\u_2\rangle=2$, so $\langle\v,\u_1\rangle=1$ and the result follows from \cref{Lem: Hilbert-Chow}.

Now suppose that $\HH_{\WW}$ admits no effective spherical or exceptional classes.  Then the proof of \cref{Lem: isotropic divisorial l=2 1} applies with $\u_2$ instead, giving the result.
\end{proof}

\subsubsection{$\ell(\u_1)=\ell(\u_2)=1$} We begin again by determining necessary conditions for a potential wall $\WW$ to be totally semistable and for $\codim (M_{\sigma_+}(\v)\backslash M_{\sigma_0}^s(\v))=1$.
\begin{Prop}\label{Prop: 1-1 case totally semistable and codim 1}
Let $\WW$ be a potential wall for positive and minimal $\v$ such that $\HH_{\WW}$ is isotropic with $\ell(\u_1)=\ell(\u_2)=1$. 
\begin{enumerate}
    \item\label{enum:IsotropicTotallySemistable l=1} If $\WW$ is totally semistable, then $\langle\v,\w\rangle=0$ and $\langle \v,\u_1\rangle=1$, where $\w$ is an effective spherical class in $\HH_{\WW}$.
    \item If $\codim(M_{\sigma_0}(\v)\backslash M^s_{\sigma_0}(\v))=1$, then
    \begin{enumerate}
        \item\label{enum:IsotropicCodimOne-<v.u>=1 l=1} $\langle \v,\u\rangle=1$ for a primitive isotropic $\u\in\HH_{\WW}$, or
        \item\label{enum:IsotropicCodimOne-<v.w>=0 l=1} $\langle \v,\w\rangle=0$ for a spherical or exceptional class $\w$.
    \end{enumerate} 
\end{enumerate}  
\end{Prop}
\begin{proof}Again we assume that $\u_1$, $\u_2$, and $\w$ are labelled and oriented in accordance with \cref{Rem:Even and Odd pairings} and the discussion preceeding \cref{Prop:LGU walls of low codimension} so that $\langle\u_1,\w\rangle>0$.  

For a given $E\in M_{\sigma_+}(\v)$, let the Harder-Narasimhan filtration of $E$ with respect to $\sigma_-$ correspond to a decomposition $\v=\sum_i \a_i$.  We shall estimate the codimension of the sublocus $\FF(\a_0,\ldots,\a_n)^o$ of destabilized objects which is equal to
 \begin{equation}
\sum_i (\a_i^2-\dim \MM_{\sigma_-}(\a_i))+\sum_{i<j}\langle \a_i,\a_j \rangle.
\end{equation}

We will divide the discussion into two cases as we did for \cref{Prop:LGU walls of low codimension}, depending on whether or not one of the classes has negative square.

(I) We first assume that one of the $\a_i$ satisfies $\a_i^2<0$, say $\a_0=b_0 \w$ for an effective spherical or exceptional class $\w$.  

Assume that $\a_1=b_1\u_1$ and $\a_2=b_2\u_2$ are isotropic.  Then 
 \begin{equation}\label{eq: 1,1 case I}
\begin{split}
& \sum_i (\a_i^2-\dim \MM_{\sigma_-}(\a_i))+\sum_{i<j}\langle \a_i,\a_j \rangle\\
\geq & (\a_0^2-\dim \MM_{\sigma_-}(\a_0))+\sum_{i \geq 1}b_0 \langle \w,\a_i \rangle
-\left\lfloor\frac{b_1}{2}\right\rfloor-\left\lfloor\frac{b_2}{2}\right\rfloor+b_1 b_2 \langle \u_1,\u_2 \rangle\\
\geq & -\dim\MM_{\sigma_-}(\a_0)+b_0 \langle \w,\v \rangle-\left\lfloor\frac{b_1}{2}\right\rfloor-\left\lfloor\frac{b_2}{2}\right\rfloor+b_1 b_2 \langle \u_1,\u_2 \rangle\\
\geq & -\dim\MM_{\sigma_-}(\a_0)-\left\lfloor\frac{b_1}{2}\right\rfloor-\left\lfloor\frac{b_2}{2}\right\rfloor+b_1 b_2 \langle \u_1,\u_2 \rangle,
\end{split}
\end{equation} 
where $b_0\langle\v,\w\rangle\geq 0$ from the minimality of $\v$.

First suppose that $\w^2=-2$, so $\dim\MM_{\sigma_-}(\a_0)=-b_0^2$ and thus 
\begin{equation}\label{eq: 1,1 case I spherical}
\begin{split}
\codim\FF(\a_0,\ldots,\a_n)^o &\geq b_0^2+b_1 b_2-\frac{b_1}{2}-\frac{b_2}{2}\geq 1+b_1 b_2-\frac{b_1}{2}-\frac{b_2}{2}\\
&=1+\frac{b_2(b_1-1)+b_1(b_2-1)}{2}\geq 1.
\end{split}
\end{equation}
Thus if $\codim\FF(\a_0,\dots,\a_n)^o=1$ then we must have $\langle \v, \w\rangle=0$, $\langle \u_1,\u_2\rangle=1$, and $\v=\w+\u_1+\u_2$.  But then $\langle \v,\w\rangle=-2<0$, contrary to assumption.  So we must have $\codim\FF(\a_0,\dots,\a_n)^0\geq 2$ in this case.

If instead $\w^2=-1$, then $\dim\MM_{\sigma_-}(\a_0)=\left\lfloor-\frac{b_0^2}{2}\right\rfloor$ and $2\mid\langle \u_1,\u_2\rangle$ by part \ref{enum:IsotropicLatticeEffectiveExceptional} of \cref{Prop:isotropic lattice}, so
\begin{equation}\label{eq: 1,1 case I exceptional}
\begin{split}
\codim\FF(\a_0,\ldots,\a_n)^o &\geq \frac{b_0^2}{2}+2b_1b_2-\frac{b_1}{2}-\frac{b_2}{2}\geq\frac{1}{2}+2b_1b_2-\frac{b_1}{2}-\frac{b_2}{2}\\
&=\frac{1}{2}+\frac{b_1(2b_2-1)+b_2(2b_1-1)}{2}\geq\frac{3}{2},
\end{split}
\end{equation}
so $\codim\FF(\a_0,\ldots,\a_n)^o\geq 2$ in this case.

Now assume that $\a_1=b_1 \u_j$ and $\a_i^2>0$ for $i>1$.  Then 
\begin{equation}
\begin{split}
 \sum_i (\a_i^2-\dim \MM_{\sigma_-}(\a_i))+\sum_{i<j}\langle \a_i,\a_j \rangle
&\geq  \a_0^2-\dim\MM_{\sigma_-}(\a_0)+\sum_{i \geq 1}b_0 \langle \w,\a_i \rangle
-\left\lfloor\frac{b_1}{2}\right\rfloor+\sum_{i \geq 2} b_1 \langle \u_j,\a_i \rangle\\
&\geq -\dim\MM_{\sigma_-}(\a_0)+b_0 \langle \w,\v \rangle
-\left\lfloor\frac{b_1}{2}\right\rfloor+b_1 \langle \u_j,\a_2 \rangle\\ 
&\geq -\dim\MM_{\sigma_-}(\a_0)+\frac{b_1}{2}.
\end{split}
\end{equation}

If $\w^2=-2$ then $$-\dim\MM_{\sigma_-}(\a_0)+\frac{b_1}{2}=b_0^2+\frac{b_1}{2}\geq\frac{3}{2},$$ while if $\w^2=-1$, then $$\codim\FF(\a_0,\ldots,\a_n)^o\geq-\dim\MM_{\sigma_-}(\a_0)+\frac{b_1}{2}\geq \frac{b_0^2}{2}+\frac{b_1}{2}\geq 1,$$ with equality in the last inequality only if $b_1=b_0=1$, in which case the first inequality is strict.  So we always have $\codim\FF(\a_0,\dots,\a_n)^o\geq 2$ in this case.

We can now assume that there are no positive classes amongst the Harder-Narasimhan factors, i.e. $\v=b_0 \w+b_1 \u_j$.  But $\v^2>0$ forces $j=1$, so we may assume this outright.  Then $0 \leq \langle \v,\w \rangle=b_0 \w^2+b_1 \langle \u_1,\w \rangle$, so our estimate becomes 

 \begin{equation}\label{eq: 1,1 case I no positive}
\begin{split}
\codim\FF(\a_0,\a_1)^o=&\sum_i (\a_i^2-\dim \MM_{\sigma_-}(\a_i))+\sum_{i<j}\langle \a_i,\a_j \rangle\\
= & b_0^2 \w^2-\dim\MM_{\sigma_-}(\a_0)-\left\lfloor\frac{b_1}{2}\right\rfloor+b_0b_1\langle \w,\u_1\rangle\\
\geq & b_0^2 \w^2-\dim\MM_{\sigma_-}(\a_0)+\frac{b_0b_1\langle \w,\u_1\rangle}{2}+\frac{b_1}{2}\left(b_0\langle \w,\u_1\rangle-1\right).
\end{split}
\end{equation}

If $\w^2=-2$, then $b_0^2\w^2-\dim\MM_{\sigma_-}(\a_0)=-b_0^2$, so the last line of \eqref{eq: 1,1 case I no positive} becomes $$-b_0^2+\frac{b_0 b_1\langle \w,\u_1\rangle}{2}+\frac{b_1}{2}(b_0\langle \w,\u_1\rangle-1)=\frac{b_0}{2}\langle \v,\w\rangle+\frac{b_1}{2}(\langle \v,\u_1\rangle-1)>0$$
unless $\langle \v,\w\rangle=0$ and $\langle \v,\u_1\rangle=1$.  But then $\v=\w+2\u_1$ and $\langle \w,\u_1\rangle=1$, in which case indeed $\codim\FF(\a_0,\a_1)^o=0$ as in \cref{enum:IsotropicTotallySemistable l=1}.  Moreover, $\codim \FF(\a_0,\a_1)^o\geq 2$ unless $\langle \w,\u_1\rangle=2$ and $\v=\w+\u_1$, as in both \cref{enum:IsotropicCodimOne-<v.w>=0 l=1,enum:IsotropicCodimOne-<v.u>=1 l=1}, or $\langle \w,\u_1\rangle=1$ and $\v=\w+3\u_1$ or $\w+4\u_1$, as in \cref{enum:IsotropicCodimOne-<v.u>=1 l=1}.  In each of these cases $\codim\FF(\a_0,\a_1)^o=1$.

If $\w^2=-1$, then $b_0^2\w^2-\dim\MM_{\sigma_-}(\a_0)=\left\lceil-\frac{b_0^2}{2}\right\rceil$, so the last line of \eqref{eq: 1,1 case I no positive} gives $$\codim\FF(\a_0,\a_1)^o\geq-\frac{b_0^2}{2}+\frac{b_0 b_1\langle \w,\u_1\rangle}{2}+\frac{b_1}{2}(b_0\langle \w,\u_1\rangle-1)=\frac{b_0}{2}\langle \v,\w\rangle+\frac{b_1}{2}(\langle \v,\u_1\rangle-1)>0$$
unless again $\langle \v,\w\rangle=0$ and $\langle \v,\u_1\rangle=1$, in which case $\v=\w+\u_1$ and $\langle \w,\u_1\rangle=1$.  In this case, however, $\codim \FF(\a_0,\a_1)^o=1$ as in both \cref{enum:IsotropicCodimOne-<v.w>=0 l=1} and \cref{enum:IsotropicCodimOne-<v.u>=1 l=1}.  Moreover, $\codim \FF(\a_0,\a_1)^o\geq 2$ unless $\langle \w,\u_1\rangle=1$ and $\v=\w+2\u_1,2(\w+\u_1)$, in which case $\codim\FF(\a_0,\a_1)^o=1$.  

Finally, assume that other than $\a_0=b_0 \w$, $\a_i^2>0$ for all $i>0$.  Then the estimate becomes 

\begin{equation}\label{eq: 1,1 case I no isotropic}
\begin{split}
\codim\FF(\a_0,\ldots,\a_n)^o=&\sum_i (\a_i^2-\dim \MM_{\sigma_-}(\a_i))+\sum_{i<j}\langle \a_i,\a_j \rangle\\
= & -\dim\MM_{\sigma_-}(\a_0)+b_0\langle \w,\v\rangle+\sum_{0<i<j}\langle \a_i,\a_j\rangle\geq \frac{b_0^2}{2}>0.
\end{split}
\end{equation}
Moreover, $\codim\FF(\a_0,\ldots,\a_n)^o\geq 2$ unless $\langle \w,\v\rangle=0$ and $\v=\w+\a_1$, as in \cref{enum:IsotropicCodimOne-<v.w>=0 l=1}.  In this case we have $\codim\FF(\a_0,\a_1)^o=1$, and we require that $$0<\a_1^2=(\v-\w)^2=\v^2+\w^2,$$ so $\v^2>2$ or $\v^2>1$ if $\w^2=-2$ or $\w^2=-1$, respectively.

(II) We next assume that $\a_i^2 \geq 0$ for all $i$.

We assume $\a_1=b_1 \u_1$ and $\a_2=b_2 \u_2$.
Then 
 \begin{equation}\label{eq: spherical 1,1 case II,a}
\begin{split}
\codim\FF(\a_1,\dots,\a_n)^o=& \sum_i (\a_i^2-\dim \MM_{\sigma_-}(\a_i))+\sum_{i<j}\langle \a_i,\a_j \rangle\\
\geq &
-\left\lfloor\frac{b_1}{2}\right\rfloor-\left\lfloor\frac{b_2}{2}\right\rfloor+b_1 b_2 \langle \u_1,\u_2 \rangle\\
\geq &\frac{b_1(b_2-1)+b_2(b_1-1)}{2}>0,
\end{split}
\end{equation}
unless $\v=\u_1+\u_2$ and $\langle \u_1,\u_2\rangle=1$, in which case $\codim\FF(\a_1,\a_2)^o=1$, as in \cref{enum:IsotropicCodimOne-<v.u>=1 l=1}.  If, say, $b_1=1$ and $b_2\geq 2$, then we have $$\codim\FF(\a_1,\a_2)^o\geq-\left\lfloor\frac{1}{2}\right\rfloor-\left\lfloor\frac{b_2}{2}\right\rfloor+b_2\langle \u_1,\u_2\rangle\geq \frac{b_2}{2}\geq 1,$$ with equality only if $b_2=2$ and $\langle \u_1,\u_2\rangle=1$.  Thus $\codim\FF(\a_1,\dots,\a_n)^o$ is always positive in this case.  Moreover, $\FF(\a_1,\dots,\a_n)^o$ has codimension one only if $\v=\u_1+\u_2,2\u_1+\u_2,\u_1+2\u_2$ with $\langle \u_1,\u_2\rangle=1$, as in \cref{enum:IsotropicCodimOne-<v.u>=1 l=1}.  Otherwise, $\codim\FF(\a_1,\dots,\a_n)^o\geq 2$.

Now we assume that $\a_1=b_1 \u_j$ and $\a_i^2>0$ for $i \geq 2$.
In this case, we also see that
\begin{equation}
\begin{split}
\codim\FF(\a_1,\dots,\a_n)^o=&\sum_i (\a_i^2-\dim \MM_{\sigma_-}(\a_i))+\sum_{i<k}\langle \a_i,\a_k \rangle\\
=&-\left\lfloor\frac{b_1}{2}\right\rfloor+\sum_{i>1}b_1\langle \u_j,\a_i\rangle+\sum_{1<i<k}\langle \a_i,\a_k\rangle\\
\geq&b_1(\langle \v,\u_j\rangle-\frac{1}{2})+\sum_{1<i<k}\langle \a_i,\a_k\rangle\\
\geq&b_1(\langle \v,\u_j\rangle-\frac{1}{2})>1,
\end{split}
\end{equation}
unless $\langle \v,\u_j\rangle=1$, and $\v=b_1\u_j+\a_2$ with $b_1=1,2$, in which case $\codim\FF(\a_1,\a_2)^o=1$, as in \cref{enum:IsotropicCodimOne-<v.u>=1 l=1}.  

Finally, if $\a_i^2>0$ for all $i$, then $\codim\FF(\a_1,\dots,\a_n)^o\geq 2$ by Proposition \ref{Prop:HN filtration all positive classes}.
\end{proof}

We prove the converse to Proposition \ref{Prop: 1-1 case totally semistable and codim 1} in the following lemmas.  We begin with the case of a totally semistable wall as in \cref{enum:IsotropicTotallySemistable l=1} of \cref{Prop: 1-1 case totally semistable and codim 1}.
\begin{Lem}\label{Lem:isotropic totally semistable divisorial contraction l=1}
Suppose that $\HH_\WW$ contains an effective spherical class $\w$ and an isotropic class $\u$ such that $\langle \v,\u\rangle=1=\ell(\u)$ and $\langle\v,\w\rangle=0$.  Then $\WW$ is totally semistable and induces a $\P^1$-fibration on $M_{\sigma_+}(\v,L)$ for $L\equiv D+\frac{\rk \v}{2}K_X\pmod 2$, where $D$ is a nodal cycle.  For the other determinant, $M_{\sigma_+}(\v,L+K_X)\backslash M_{\sigma_0}^s(\v,L+K_X)$ is a divisor which does not get contracted.  
\end{Lem}
\begin{proof}
We first observe that since $\langle\v,\w\rangle=0$, it follows from the analogue of \eqref{eqn:u2 and u1} in the spherical case that $\langle\v,\u_1\rangle=\langle\v,\u_2\rangle$.  As $\ell(\u_1)=\ell(\u_2)$, we may assume that $\u=\u_1$.  By \cref{Prop:isotropic lattice} we may write $\v=x\w+y\u_1$ with $x,y\in\Z$.  Then $1=\langle \v,\u_1\rangle=x\langle \w,\u_1\rangle$ which implies that $x=1=\langle \w,\u_1\rangle$.  But then $\langle\v,\w\rangle=0$ forces $y=2$.  Since $\ell(\u_1)=1$, $M^s_{\sigma_0}(2\u_1,2L')$ is non-empty and two-dimensional by \cref{prop:isotropic}.  Moreover, for the unique $\sigma_0$-stable spherical object $T$ of class $\w$ and any $A\in M^s_{\sigma_0}(2\u_1,2L')$, stability ensures that $\ext^1(T,A)=\langle 2\u_1,\w\rangle=2$.  Then the $\P^1$ worth of extensions $$0\to A\to E\to T\to 0$$ gets contracted by crossing $\WW$, and varying $A$ in $M_{\sigma_0}^s(2\u_1,2L')$ generically sweeps out an entire irreducible component of $M_{\sigma_+}(\v,L)$, where $$L=2L'+\det(T)\equiv \det(T)\equiv D+\frac{\rk \v}{2}K_X\pmod 2.$$

For the other determinant, observe that the only decompositions of $\v$ into effective classes are $\v=\w+2\u_1=\u_1+\u_2$, and from the proof of Proposition \ref{Prop: 1-1 case totally semistable and codim 1} only the former decomposition corresponds to a totally semistable wall.  Moreover, in the case of the decomposition $\v=\w+2\u_1$ for the determinant $L+K_X$, the strictly $\sigma_0$-semistable locus has codimension 1, so $\WW$ is not totally semistable for $M_{\sigma_+}(\v,L+K_X)$.  Indeed, if $E\in M_{\sigma_+}(\v,L+K_X)$ has this decomposition for its Harder-Narasimhan filtration with respect to $\sigma_-$, then the kernel $A$ of the surjection $E\onto T$ would be in $M_{\sigma_-}(2\u_1,2L'+K_X)$.  But $M_{\sigma_-}^s(2\u_1,2L'+K_X)=\varnothing$ by \cref{prop:isotropic}, so $A$ would have to be strictly $\sigma_-$-semistable, and from determinant considerations its Jordan-H\"{o}lder factors would have to be $A_1\in M_{\sigma_-}(\u_1,L')$ and $A_2\in M_{\sigma_-}(\u_1,L'+K_X)$.  But then $A_1\ncong A_2,A_2(K_X)$, as $\det(A_1)\neq\det(A_2)=\det(A_2(K_X))$, so $$\ext^1(A_1,A_2)=\langle \u_1,\u_1\rangle+\hom(A_1,A_2)+\ext^2(A_1,A_2)=\hom(A_2,A_1(K_X))=0,$$ from which it follows that $A=A_1\oplus A_2$.  Using \cite[Lemmas 6.1-6.3]{CH15}, we thus get a unique $\sigma_+$-stable extension $$0\to A_1\oplus A_2\to E\to T\to 0,$$ which is unique in its S-equivalence class with respect to $\sigma_0$, and varying the $A_i$ spans a non-contracted divisor of strictly $\sigma_0$-semistable objects with the prescribed Harder-Narasimhan filtration for $\sigma_-$.  

Now consider the other decomposition, $\v=\u_1+\u_2$, and take $A_1\in M_{\sigma_+}(\u_1,L')$ and $A_2\in M_{\sigma_+}(\u_2)$.  Then any nontrivial extension \begin{equation}\label{eqn:Other decomposition}
0\to A_1\to E\to A_2\to 0
\end{equation} is $\sigma_+$-stable by \cite[Lemma 9.3]{BM14b}, as the parallelogram spanned by $\u_1$ and $\u_2$ has no lattices points other than its vertices.  In order for $\det(E)=L+K_X$ we must have $$\det(A_2)+L'=L+K_X=2L'+\det(T)+K_X,$$ so that $\det(A_2)=L'+\det(T)+K_X$.  As $M_{\sigma_0}^s(\u_2)=\varnothing$, $A_2$ must be strictly $\sigma_0$-semistable with stable factors $T$ and $A_1'\in M_{\sigma_+}(\u_1)$.  It follows that  $$L'+\det(T)+K_X=\det(A_2)=\det(T)+\det(A_1'),$$ so $A_1'\in M_{\sigma_+}(\u_1,L'+K_X)$.  In particular, $\det(A_1)\neq\det(A_1')$, and thus $A_1\ncong A_1',A_1'(K_X)$.  Hence by stability, $$\Hom(T,A_1)=\Hom(A_1,T)=\Hom(A_1',A_1)=\Hom(A_1,A_1'(K_X))=0,$$ from which we see that $\Hom(A_2,A_1)=0=\Ext^2(A_2,A_1)=0$, by applying $\Hom(\blank,A_1)$ to the short exact sequence \eqref{eqn:Other decomposition}.  Thus $\ext^1(A_2,A_1)=1$, and there exists a unique $\sigma_+$-stable extension $E$, which is also unique in its S-equivalence class with respect to $\sigma_0$.  Letting the $A_i$ vary, we again sweep out a divisor that is not contracted by $\WW$, as claimed.
\end{proof}

Now we move on to \cref{enum:IsotropicCodimOne-<v.u>=1 l=1}.
\begin{Lem}\label{Lem:isotropic divisorial l=1 1}
Assume that $\v$ is minimal in $\HH_\WW$, which contains primitive isotropic classes $\u_1$ and $\u_2$ such that $\ell(\u_1)=\ell(\u_2)=1$, and suppose that $\langle \v,\u\rangle=1$ for a primitive isotropic $\u\in\HH_\WW$ with $\ell(\u)=1$.  
\begin{enumerate}
    \item\label{enum:IsotropicDivisorialContraction l=1 v^2>=3} If $\v^2\geq 3$, then $\WW$ induces a divisorial contraction on $M_{\sigma_+}(\v,L)$.
    \item If either
    \begin{enumerate}
        \item\label{enum:IsotropicDivisorialNonContraction l=1 v^2=1} $\v^2=1$ or
        \item\label{enum:IsotropicDivisorialNonContraction l=1 v^2=2} $\v^2=2$, $\HH_\WW$ contains a spherical class, and $L \equiv D+(\frac{\rk \v}{2}+1)K_X \pmod 2$, where $D$ is a nodal cycle,
    \end{enumerate}
    then $M_{\sigma_+}(\v,L)\backslash M^s_{\sigma_0}(\v,L)$ is a divisor which is not contracted by $\WW$.
\end{enumerate}  
\end{Lem}
\begin{proof}
Suppose first that $\HH_{\WW}$ contains an effective spherical or exceptional class $\w$ and $\v^2\geq 3$.  Then by minimality of $\v$ and  \eqref{eqn:u2 and u1} (and its analogue in the spherical case), $\langle\v,\u_2\rangle\geq \langle\v,\u_1\rangle$ with equality only if $\langle\v,\w\rangle=0$.  But if $\langle\v,\w\rangle=0$, then $\u$ could be either $\u_1$ or $\u_2$, and writing $\v=x\w+y\u$ with $x,y\in\Z$, we see that $\langle\v,\u\rangle=1$ and $\langle\v,\w\rangle=0$ imply that $\v=\w+(-\w^2)\u$, where $\langle \u,\w\rangle=1$.  It follows that $\v^2=-\w^2$, which were explicitly excluded.  Thus $\langle\v,\w\rangle>0$ and we see that $\u=\u_1$.  Moreover, we see by the same reasoning that $\langle\v,\u_1\rangle=1$ implies that $\v=\w+y\u_1$ and $\langle\w,\u_1\rangle=1$ so that $\v^2=\w^2+2y\equiv\w^2\pmod 2$.  

We will break the proof up into different parts based on the decomposition of $\v$ we will use.  

Suppose first that $\v^2\geq 3$ if $\w$ is exceptional and $\v^2\geq 8$ if $\w$ is spherical.  Setting $\v'=\v-2\u_1$, it follows that $\v'^2=\v^2-4\geq-1$ or $\v'^2\geq 4$ if $\w$ is exceptional or spherical, respectively.  As $\langle \v',\u_1\rangle=1$, we have $\v'\in C_{\WW}$, and from $\langle\v',\w\rangle=\frac{\v^2+w^2-4}{2}$, which is positive if $\w$ is spherical, we have $M^s_{\sigma_0}(\v')\neq\varnothing$ by Proposition \ref{Prop: 1-1 case totally semistable and codim 1}.  Then for $E_1\in M^s_{\sigma_0}(2\u_1)$ and $E_2\in M^s_{\sigma_0}(\v')$, we have $\Hom(E_2,E_1)=\Hom(E_1,E_2(K_X))=0$ by stability.  It follows that $\ext^1(E_2,E_1)=\langle\v',2\u_1\rangle=2$, so there is a $\P^1$ worth of extensions of the form $$0\to E_1\to E\to E_2\to 0,$$ which gets contracted by $\WW$.  Varying $E_1$ and $E_2$ in their moduli sweeps out a divisor in $M_{\sigma_+}(\v)$.  Moreover, as $\det(E_1)=2L'$ by \cref{prop:isotropic}, we may choose $E_2\in M_{\sigma_0}^s(\v')$ to have either determinant $L''$ or $L''+K_X$, where $[L'' \mod K_X]=c_1(\v')$, to get a divisorial contraction on each $M_{\sigma_+}(\v,2L'+L'')$ and $M_{\sigma_+}(\v,2L'+L''+K_X)$.  We have proven the first claim of the lemma for $\w$ exceptional, and to complete the proof of this claim for $\w$ spherical we must consider when $\v^2=6,4$.

If $\v^2=6$ and $L\equiv D+(\frac{\rk \v}{2}+1)K_X\pmod 2$, then the same argument gives a divisorial contraction on $M_{\sigma_+}(\v,L)$, as then $\v'^2=2$ (i.e. $\langle\v',\w\rangle=0$) and $M_{\sigma_0}^s(\v',L-2L')\neq\varnothing$ by Lemma \ref{Lem:isotropic totally semistable divisorial contraction l=1}.  If instead $L\equiv D+\frac{\rk \v}{2}K_X\pmod 2$, then we will use a different decomposition of $\v$.  The conditions $\langle\v,\u_1\rangle=1$ and $\v^2=6$ force $\v=\w+4\u_1$ with $\langle\u_1,\w\rangle=1$, so instead of the above decomposition ($\v=(\v-2\u_1)+2\u_1$), we use a different one, $\v=\w+\v''$ where $\v'':=2\u_1+2\u_1$.  Indeed, take the unique $\sigma_0$-stable object $T$ of class $\w$, and two non-isomorphic objects $E_1,E_2\in M^s_{\sigma_0}(2\u_1,2L')$.  Then by \cite[Lemmas 6.1-6.3]{CH15} the extensions of the form $$0\to E_1\oplus E_2\to E\to T\to 0$$ are $\sigma_+$-stable and move in a two-dimensional family contracted to the same point by $\WW$.  Varying $(E_1,E_2)\in (M^s_{\sigma_0}(2\u_1,2L')\times M^s_{\sigma_0}(2\u_1,2L'))\backslash\Delta$, where $\Delta$ is diagonal, sweeps out a contracted divisor in $M_{\sigma_+}(\v,L)$.

If $\v^2=4$, then the condition $\langle\v,\u_1\rangle=1$ forces $\v=\w+3\u_1$.  Take the unique $\sigma_0$-stable object $T$ of class $\w$, let $E_1\in M^s_{\sigma_0}(2\u_1,2L')$, and $E_2\in M^s_{\sigma_0}(\u_1)$.  We consider extensions of the form $$0\to 0\to E_1\oplus E_2\to E\to T \to 0.$$  These extensions move in a one-dimensional family by \cite[Lemma 6.3]{CH15} and are $\sigma_+$-stable by \cite[Lemma 6.1]{CH15}.  For fixed $E_i$, this curve of extensions is contracted by $\WW$, and varying the $E_i$ sweeps out a divisor.  As $\det(E_2)$ can be either $L'$ or $L'+K_X$, we get a divisorial contraction in each component as before.  

This concludes the proof of \cref{enum:IsotropicDivisorialContraction l=1 v^2>=3} for $\HH_\WW$ that falls into \cref{enum:IsotropicLatticeEffectiveExceptional,enum:IsotropicLatticeEffectiveSpherical} of \cref{Prop:isotropic lattice}.

For the second claim of the lemma for $\HH_\WW$ falling into \cref{enum:IsotropicLatticeEffectiveExceptional,enum:IsotropicLatticeEffectiveSpherical} of \cref{Prop:isotropic lattice}, we must consider $\v^2=1$ and $\v^2=2$, which occur when $\w$ is exceptional and spherical, respectively.  In the first case, we must have $\v=\w+\u_1$.  Letting $F\in M_{\sigma_0}^s(\u_1)$ and $T$ be the unique (up to $\otimes\OO_X(K_X)$) $\sigma_0$-stable exceptional object of class $\w$, we consider the unique non-trivial extension, $$0\to F\to E\to T\to 0.$$ Varying $F
\in M_{\sigma_0}^s(\u_1)$, we sweep out a non-contracted divisor in each of $M_{\sigma_+}(\v,L)$ and $M_{\sigma_+}(\v,L+K_X)$, giving \cref{enum:IsotropicDivisorialNonContraction l=1 v^2=1}.  The second case has been dealt with in the second statement of Lemma \ref{Lem:isotropic totally semistable divisorial contraction l=1}, giving \cref{enum:IsotropicDivisorialNonContraction l=1 v^2=2}.

Finally, we suppose that $\HH_{\WW}$ contains no effective spherical or exceptional class.  We note that since $\langle \v,\u\rangle=1$, $\HH_\WW$ must be even.  Indeed, it follows from $\langle\v,\u\rangle=1$ that $\HH_\WW=\Z\u+\Z\v$, and if $\v^2$ were odd, then $\frac{\v^2+1}{2}\u-\v$ would be an exceptional class, so $\v^2$ must be even.  But then $\HH_\WW$ must be even as well.  

We first prove \cref{enum:IsotropicDivisorialContraction l=1 v^2>=3} in this case.  So suppose that $\v^2\geq 3$.  Then $\HH_\WW$ being even gives $\v^2\geq 4$, and by \cref{Rem:IsotropicOrientation} we may write $\v=\v'+2\u_1$ as above to get a divisor swept out by contracted $\P^1$'s of extensions $$0\to E_1\to E\to E_2\to 0$$ with $E_1\in M_{\sigma_0}^s(2\u_1,2L')$ and $E_2\in M_{\sigma_0}^s(\v')$.  Note that if $\v^2>4$, so that $\v'^2>0$, $M_{\sigma_0}^s(\v')$ is non-empty by Proposition \ref{Prop: 1-1 case totally semistable and codim 1}, while if $\v^2=4$, so that $\v'^2=0$, then $\v'=\u_2$ and $M_{\sigma_0}^s(\u_2)\neq\varnothing$ because $C_{\WW}=P_{\HH}$ by \cref{Prop:isotropic lattice}. 
Either way, we may choose $E_2$ to have the appropriate determinant to give a divisorial contraction in each $M_{\sigma_+}(\v,L)$.  This completes the proof of \cref{enum:IsotropicDivisorialContraction l=1 v^2>=3} of the Lemma.

The final option to consider is $\v^2=2$, in which case the only possibility for a destabilizing exact sequence is $$0\to E_1\to E\to E_2\to 0$$ for $E_i\in M^s_{\sigma_0}(\u_i)$, which span a divisor which is not contracted as $\ext^1(E_2,E_1)=1$.
\end{proof}
Now we prove the converse to \cref{enum:IsotropicCodimOne-<v.w>=0 l=1} of \cref{Prop: 1-1 case totally semistable and codim 1}.
\begin{Lem}\label{Lem:isotropic divisorial 1-1 2}
Suppose that $\HH_\WW$ is isotropic with $\ell(\u_1)=\ell(\u_2)=1$ and $\langle \v,\w\rangle=0$ for an effective spherical or exceptional class $\w\in\HH_\WW$.  
\begin{enumerate}
    \item\label{enum:IsotropicDivisorialContraction l=1 <v.w>=0} If $\w^2=-2$ and either 
    \begin{enumerate}
        \item $\v^2>2$, or
        \item $\v^2=2$ and $\langle\v,\u_1\rangle>1$,
    \end{enumerate} then $\WW$ induces a divisorial contraction on $M_{\sigma_+}(\v,L)$.
    \item If either 
    \begin{enumerate}
        \item\label{enum:IsotropicDivisorialNonContraction l=1 <v.w>=0 v^2=2}$\w^2=-2$, $\v^2=2$, $\langle\v,\u_1\rangle=1$, and $L \equiv D+(\frac{\rk \v}{2}+1)K_X \pmod2$, or
        \item\label{enum:IsotropicDivisorialNonContraction l=1 <v.w>=0 w^2=-1} $\w^2=-1$,
    \end{enumerate}
    then $\codim(M_{\sigma_+}(\v,L)\backslash M^s_{\sigma_0}(\v,L))=1$ but this divisor is not contracted. 
\end{enumerate}
\end{Lem}
\begin{proof}
Consider the Mukai vector $\a:=\v-\w$.  Then $$\a^2=\v^2+\w^2>\w^2, \langle \a,\v \rangle=\v^2>0,
\mbox{ and }\langle \a,\w\rangle=-\w^2.$$  If $\a^2>0$, then since $\langle \a,\w\rangle>0$, $M^s_{\sigma_0}(\a)\neq\varnothing$ by Proposition \ref{Prop: 1-1 case totally semistable and codim 1}, and we consider the $\sigma_+$-stable extensions $$0\to F\to E\to G\to 0,$$ where $F\in M^s_{\sigma_0}(\a)$ and $G\in M_{\sigma_0}^s(\w)$.  If $\w^2=-2$, then $G$ is the unique $\sigma_0$-stable spherical object $T$ of class $\w$, and by stability $\ext^1(T,F)=\langle\w,\a\rangle=2$, so these extensions span a $\P^1$ which gets contracted by $\WW$.  If $\w^2=-1$, then $G$ is $T$ or $T(K_X)$, and each choice of $G$ gives a unique such non-trivial extension, as $\ext^1(G,F)=\langle\a,\w\rangle=1$ by stability.  A dimension count immediately gives that varying $F$ in $M_{\sigma_0}^s(\a)$ spans a divisor which gets contracted if $\w^2=-2$ but does not if $\w^2=-1$.  Note that we may choose $F$ to have either determinant appropriately to give the claimed behavior for each $M_{\sigma_+}(\v,L)$.

It remains to consider when $\a^2=0$, i.e $\v^2=-\w^2$.  The fact that $\a$ is an effective isotropic class such that $\langle \a,\w\rangle>0$ implies that $\a=k\u_1$.  If $\w^2=-1$, then $k=\langle \u_1,\w\rangle=1$, in which case $\v=\w+\u_1$ and the claim follows from \cref{enum:IsotropicDivisorialNonContraction l=1 v^2=1} of Lemma \ref{Lem:isotropic divisorial l=1 1}.  On the other hand, if $\w^2=-2$, then $$2=-\w^2=\langle \a,\w\rangle=\langle k\u_1,\w\rangle=k\langle \u_1,\w\rangle,$$ so either $\v=\w+2\u_1$ and $\langle \u_1,\w\rangle=1$ or $\v=\w+\u_1$ and $\langle \u_1,\w\rangle=2$.  In the first case, $\v^2=2$ and $\langle\v,\u_1\rangle=1$, as in \cref{enum:IsotropicDivisorialNonContraction l=1 <v.w>=0 v^2=2}, so the claim follows from \cref{Lem:isotropic totally semistable divisorial contraction l=1}.  In the second case, we consider the $\P^1$ worth of extensions $$0\to F\to E\to T\to 0$$ with $F\in M^s_{\sigma_0}(\u_1)$ and $T\in M^s_{\sigma_0}(\w)$ the unique $\sigma_0$-stable spherical object of class $\w$.  Varying $F\in M_{\sigma_0}^s(\u_1,L')$ (or $M_{\sigma_0}^s(\u_1,L'+K_X)$), these extensions span a contracted divisor in each $M_{\sigma_+}(\v,L)$.
\end{proof}

\subsection{Non-minimal case}\label{subsec:non-minimal case}

Finally, we consider the case that $\langle\v,\w\rangle<0$.  As usual, we follow the orientation described in \cref{Rem:IsotropicOrientation} so that $\phi^+(\w)>\phi^+(\v)$, and hence
$\phi^-(\w)<\phi^-(\v)$, and denote by $T$ the unique (up to $-\otimes\OO_X(K_X)$) $\sigma_0$-stable object of class $\w$.
We set 
$\TT_1:=\langle T,T(K_X) \rangle$ if $\w^2=-1$ (resp. $\TT_1:=\langle T\rangle$ if $\w^2=-2$) and
$\FF_1$ is the full subcategory of 
$\PP(1)$ generated by $\sigma_0$-stable objects $E$
with $\phi^+(E)<\phi^+(T)$.
We also let
$\TT_1^*$ be the full subcategory of 
$\PP(1)$ generated by $\sigma_0$-stable objects $E$
with $\phi^-(E)>\phi^-(T)$
and
$\FF_1^*:=\langle T,T(K_X) \rangle$ if $\w^2=-1$ (resp. $\FF_1^*:=\langle T\rangle$ if $\w^2=-2$).
We set $\AA_0=\PP(1)$,
$\AA_1=\langle \TT_1[-1],\FF_1 \rangle$ and
$\AA_1^*:=\langle \TT_1^*,\FF_1^*[1] \rangle$.
 Let $\Phi:\Db(X) \to \Db(X)$ be the equivalence
defined by $T$ as in \cref{eqn:spherical reflection,eqn:weakly spherical reflection}.
Then we have equivalences
\begin{equation}
\begin{split}
\Phi:& \AA_0 \isomor \AA_1\\
\Phi^{-1}:&\AA_0 \isomor \AA_1^*.
\end{split}
\end{equation}
The proof is identical to that for the non-isotropic case.

Write $n:=-2\langle \v,\w \rangle$ if $\w^2=-1$ (resp. $n:=-\langle \v,\w\rangle$ if $\w^2=-2$), and assume that $n>0$.
Then we have isomorphisms
\begin{equation}\label{eq:iso:Phi}
\begin{matrix}
\Phi:& M_{\sigma_+}(\v)& \isomor& M_{\sigma_-}(\v-n\w),\\
\Phi:& M_{\sigma_+}(\v-n\w)& \isomor& M_{\sigma_-}(\v),\\
\end{matrix}
\end{equation}
where $\v-n\w$ is minimal and $\langle\v-n\w,\w\rangle>0$.  By \cref{Prop:LGU walls of low codimension,Prop: 1-1 case totally semistable and codim 1} and the fact that $\langle\v,\w\rangle\neq0$, we have a birational map
\begin{equation}\label{eqn:birational map for v-nw}M_{\sigma_-}(\v-n\w) \dashrightarrow M_{\sigma_+}(\v-n\w)
\end{equation}
unless $\langle\v-n\w,\u\rangle=1$ and $\ell(\u)=2$.  In this case, $\WW$ is totally semistable for $\v-n\w$ as well and induces a divisorial contraction.  We note that the S-equivalence class of $E\in M_{\sigma_+}(\v)$ is determined by that of $\Phi(E)\in M_{\sigma_+}(\v-n\w)$, as $E$ is an extension of $\Phi(E)$ and $(T\oplus T(K_X))^{\oplus n}$ if $\w^2=-1$ (resp. $T^{\oplus n}$ if $\w^2=-2$), so it follows that $\WW$ induces a divisorial contraction for $M_{\sigma_+}(\v)$ as well.  

Otherwise, precomposing and postcomposing the birational map in \eqref{eqn:birational map for v-nw} with the isomorphisms from \eqref{eq:iso:Phi}, we get a birational map
\begin{equation}\label{eqn:PrePostComposition}
\Phi \circ \Phi:M_{\sigma_+}(\v) \dashrightarrow M_{\sigma_-}(\v).
\end{equation}
As the birational map in \eqref{eqn:birational map for v-nw} is isomorphic in codimension one unless 
\begin{enumerate}
    
    \item$\langle\v-n\w,\u\rangle=2=\ell(\u)$; or 
    \item $\langle\v-n\w,\u\rangle=1=\ell(\u),$
\end{enumerate}
the same holds true for the birational map in \eqref{eqn:PrePostComposition}.  In either of these cases, crossing $\WW$ induces a divisorial contraction on $M_{\sigma_+}(\v-n\w)$ by \cref{Lem:isotropic divisorial l=1 1,Lem: isotropic divisorial l=2 3}, so the same holds true for $M_{\sigma_+}(\v)$ by the above discussion regarding S-equivalence in the case $\langle\v-n\w,\u\rangle=1$, $\ell(\u)=2$.

\begin{Rem}\label{Rem:OrthogonalIsomorphism}As in the non-isotropic case, if $\langle \v, \w \rangle=0$, then
$\Phi$ again induces an isomorphism
$$\Phi:M_{\sigma_+}(\v) \to M_{\sigma_-}(\v).$$
\end{Rem}

\begin{proof}[Proof of \cref{Prop:isotropic-classification}]
The proposition follows from the above discussion and \cref{Lem:P1FibrationExceptional,Lem:P1FibrationSpherical,Lem: isotropic divisorial l=2 2,Lem:IsotropicNoncontractedDivisor l=2,Lem: isotropic divisorial l=2 3,Lem:isotropic totally semistable divisorial contraction l=1,Lem:isotropic divisorial l=1 1,Lem:isotropic divisorial 1-1 2}.  The only point that needs mentioning is that the restrictions on $\v^2$ in these lemmas can be equivalently rephrased in terms of either the pairing with $\w$ or $\u$ as in the statement of \cref{Prop:isotropic-classification}.
\end{proof}

\begin{Ex}\label{Ex:ConfusingSmallContraction}
An important example arises in the context of \cref{Rem:OrthogonalIsomorphism}.  Suppose that $\langle\v,\w\rangle=0$ for the effective exceptional class $\w$ in the isotropic lattice $\HH_\WW$.  Then we have seen that if $\langle\v,\u\rangle>\ell(\u)$ for all primitive isotropic $\u\in C_\WW\cap\HH_\WW$, then $M_{\sigma_+}(\v,L)\setminus M_{\sigma_0}^s(\v,L)$ is a divisor that does not get contracted.  We will see in the next section (specifically \cref{prop:flops}) that $\WW$ nevertheless induces a small contraction if $\v$ is primitive.  

This is an interesting situation.  On the one hand, we know from \cref{Rem:OrthogonalIsomorphism} that $\Phi$ induces an isomorphism between $M_{\sigma_+}(\v)$ and $M_{\sigma_-}(\v)$.  Moreover, from its definition, $\Phi$ restricts to the identity on the common open subset $M_{\sigma_0}^s(\v)\subset M_{\sigma_+}(\v)\cap M_{\sigma_-}(\v)$.  On the other hand, as $\pi^+$ is a small contraction, we may flop it to get another minimal model $\tilde{M}$ of $M_{\sigma_+}(\v)$, but $\tilde{M}\ncong M_{\sigma_-}(\v)$.  

This phenomenon, which is present on any Enriques surface, leads one to wonder whether or not the minimal model $\tilde{M}$ can nevertheless be obtained by Bridgeland wall-crossing.  It may be possible to reach $\tilde{M}$ by crossing a different wall bounding the chamber containing $\sigma_+$, or this may be a counter-example to the conjecture that all minimal models of $M_\sigma(\v)$, for $\sigma\in\Stabd(X)$ generic, are isomorphic to $M_\tau(\v)$ for some $\tau\in\Stabd(X)$.
\end{Ex}

\section{Flopping walls}\label{Sec:FloppingWalls}
In \crefrange{Sec:TotallySemistable-non-isotropic}{Sec:Isotropic walls} we have given necessary and sufficient criteria for the wall $\WW$ to be totally semistable, to induce a $\P^1$ fibration, and to induce a divisorial contraction.  In this section, we discuss the remaining possibility for the contraction morphism $\pi^+$.  That is, if $\WW$ does not induce a fibration or a divisorial contraction, then it must either induce a small contraction, that is, the exceptional locus of $\pi^+$ must have codimension at least two, or $\WW$ is fake wall so that $\pi^+$ does not contract any curves.  In the next result, we give precise criteria for when $\WW$ is a genuine wall inducing a small contraction, at least for $\v$ primitive.  It is the only result in our work so far that has assumed that $\v$ is primitive.
\begin{Prop} \label{prop:flops}
Assume that $\v$ is primitive and that $\WW$ induces neither a divisorial contraction nor a $\P^1$-fibration.  If either
\begin{enumerate}
\item \label{enum-prop:sum2positive}
$\v^2\geq 3$ and $\v$ can be written as a sum 
$\v = \a_1 +\a_2$ with $\a_i\in P_\HH$ such that $L\equiv \frac{r}{2}K_X\pmod 2$ if for each $i$, $\a_i^2=0$ and $\ell(\a_i)=2$; or 
\item\label{enum-prop:exceptional} there exists an exceptional class $\w$ and either
\begin{enumerate}
\item\label{enum-prop:exceptionalflop1}
$0< \langle  \w,\v\rangle\leq\frac{\v^2}{2}$, or
\item\label{enum-prop:exceptionalflop2}
$\langle \v,\w\rangle=0$ and $\v^2\geq 3$; or
\end{enumerate}
\item\label{enum-prop:spherical} there exists a spherical class $\w$ and either
\begin{enumerate}
\item\label{enum-prop:sphericalflop1}
$0 < \langle \w, \v\rangle < \frac{\v^2}2$, or
\item\label{enum-prop:sphericalflop2}
$\langle \w,\v\rangle=\frac{\v^2}{2}$ and $\v-\w$ is a spherical class,
\end{enumerate}
\end{enumerate}
then $\WW$ induces a small contraction on $M_{\sigma_+}(\v,L)$.
\end{Prop}

\begin{proof}
Note that it suffices to show that some positive dimensional subvariety of $\sigma_+$-stable objects becomes S-equivalent with respect to $\sigma_0$ and thus gets contracted by $\pi^+$. 

We consider \cref{enum-prop:sum2positive} first, so $\v=\a_1+\a_2$ with $\a_i\in P_{\HH}$.  Using \cite[Lemma 9.2]{BM14b}, we may assume that the parallelogram with vertices $0,\a_1,\v,\a_2$ does not contain any lattice point other than its vertices.  In particular, the $\a_i$ are primitive, and without loss of generality, we may assume that $\phi^+(\a_1)<\phi^+(\a_2)$.  By Theorem \ref{Thm:exist:nodal}, there exist $\sigma_+$-stable objects $A_i$ with $\v(A_i)=\a_i$.  If $\a_i^2>0$ for each $i$, then the signature of $\HH$ forces $\langle \a_1,\a_2\rangle\geq 2$ so that $\ext^1(A_2,A_1)\geq 2$.  If, say, $\a_1^2=0$, then by part \ref{thm:Classification,Divisorial} of Theorem \ref{classification of walls} and the assumptions that $\v^2\geq 3$ and that $\WW$ does not induce a divisorial contraction, we must have either $\ell(\a_1)=2$ and $\langle \v,\a_1\rangle \geq 3$ or $\ell(\a_1)=1$ and $\langle \v,\a_1\rangle\geq 2$.  So either way $\langle \a_2,\a_1\rangle\geq 2$ and again $\ext^1(A_2,A_1)\geq 2$.  By \cite[Lemma 9.3]{BM14b}, any nontrivial extension $$0\to A_1\to E\to A_2\to 0$$ is $\sigma_+$-stable of class $\v$.  All such extensions are non-isomorphic but S-equivalent with respect to $\sigma_0$, giving a projective space of positive dimension contracted by $\pi^+$.  Moreover, in all of the above cases, we may choose $A_1$ or $A_2$ to have the appropriate determinant so that $E$ can have either of the two possible determinants for $\v$, except when $\a_1^2=\a_2^2=0$ and $\ell(\a_1)=\ell(\a_2)=2$.   

When $\a_1^2=\a_2^2=0$ and $\ell(\a_1)=\ell(\a_2)=2$, then the argument above produces a projective space of positive dimension dimension contracted by $\pi^+$ if $$L=\det(A_1)+\det(A_2)\equiv\frac{\rk(A_1)}{2}K_X+\frac{\rk(A_2)}{2}K_X=\frac{r}{2}K_X\pmod 2,$$ as claimed.  Observe further that if $\HH$ contains an exceptional or spherical class $\w$, then $\a_2=\a_1-2\frac{\langle\a_1,\w\rangle}{\w^2}\w$, where we note that if $\w^2=-2$ then $\langle\a_1,\w\rangle$ is even by \cref{Rem:Even and Odd pairings}, so $$\v=\a_1+\a_2=2\a_1-2\frac{\langle\a_1,\w\rangle}{\w^2}\w=2\left(\a_1-\frac{\langle\a_1,\w\rangle}{\w^2}\w\right)$$ is not primitive, contary to our assumptions.  Thus this possibility only occurs in \cref{enum:nonegativeclasses} of \cref{Prop:lattice classification}.

We move on to \cref{enum-prop:exceptionalflop1}.  Assume first that $\v$ is minimal.  Then $\langle \v,\w \rangle>0$ means
$\w$ is effective.  Since $(\v-\w)^2 \geq -1$ and $\langle\v,\w\rangle\leq\frac{\v^2}{2}<\v^2$, we see that $\langle\v,\v-\w\rangle>0$ so that $\v-\w$ must be effective.  From the assumptions, we observe that $\langle \w, \v-\w\rangle=\langle \w,\v\rangle+1\geq 2$.  As in the proof of \cite[Proposition 9.1]{BM14b}, we consider the parallelogram $\mathbf{P}$ with vertices $0,\w,\v,\v-\w$ and the function $f(\a)=\a^2$ on $\mathbf{P}$, and the same argument as given there shows that $f(\a)>-1$ unless $\a\in\{\w,\v-\w\}$.  It follows that if $\mathbf{P}$ contains any lattice point $\a$ other than its vertices, then both $\a^2\geq 0$ and $(\v-\a)^2\geq 0$.  So $\v$ is the sum of two positive classes, and we are in \cref{enum-prop:sum2positive}.  It is easy to see that indeed $\v$ satisfies the extra condition $\v^2\geq 3$ if such an $\a$ is isotropic.  We may therefore assume that no such lattice points exist.  Let $T$ be a $\sigma_+$-stable object of class $\w$ and $F$ be any $\sigma_+$-stable object of class $\v-\w$.  Then, assuming $\phi^+(\w)<\phi^+(\v-\w)$ without loss of generality, we get $\ext^1(F,T)\geq2$, so there is a positive dimensional projective space worth of $\sigma_+$-stable extensions that are S-equivalent with respect to $\sigma_0$ and thus get contracted by $\pi^+$.

Now assume that $\v$ is not minimal.  Consider the composition of spherical and weakly-spherical twists as in Proposition \ref{Prop:NonMinimalIsomorphism} or Proposition \ref{Prop:CompositionSphericalExceptional}, and denote it by $\Phi$.  Then $\Phi_*(\v)$ is minimal, and as $\langle \Phi_*(\w),\Phi_*(\v)\rangle=\langle \w,\v\rangle$, we see that $\Phi_*(\w)$ is an effective exceptional class and satisfies the same inequality in the hypothesis of \ref{enum-prop:exceptionalflop1} for $\Phi_*(\v)$ instead $\v$.  Thus $\Phi_*(\v-\w)$ is effective and satisfies $\Phi_*(\v-\w)^2\geq -1$.  Moreover, depending on the parity of the index of the chamber occupied by $\v$, we either get $M_{\sigma_\pm}(\v)\cong M_{\sigma_\pm}(\Phi_*(\v))$ or $M_{\sigma_{\pm}}(\v)\cong M_{\sigma_\mp}(\Phi_*(\v))$, and the S-equivalence class of $E\in M_{\sigma_+}(\v)$ is determined by that of $\Phi(E)\in M_{\sigma_\pm}(\Phi_*(\v))$, respectively.  Thus the result follows from the work of the previous paragraph.

Now let us consider \cref{enum-prop:exceptionalflop2}.  We first assume that $\v$ is minimal.  As $\v^2\geq3$ by assumption, $(\v-2\w)^2\geq-1$, and since $\langle \v,\w\rangle=0$ and $\langle\v,\v-2\w\rangle=\v^2>0$, we may assume that both $\w$ and $\v-2\w$ are effective.  Note also that $\langle \w,\v-2\w\rangle=2$.  Let $T,T(K_X)$ be the two $\sigma_+$-stable exceptional objects of class $\w$ and let $F$ be a $\sigma_+$-stable object of class $\v-2\w$.  As in the previous cases, we may assume that the parallelogram with vertices $0,\w,\v,\v-\w$ has no additional lattices points so that therefore the parallelogram with vertices $0,2\w,\v,\v-2\w$ has no lattices points other than $0,\w,2\w,\v,\v-\w,\v-2\w$.  Without loss of generality we may assume that $\phi^+(\w)>\phi^+(\v)$.  Then for any extension $$0\to F\to E\to T\oplus T(K_X)\to 0$$ corresponding to non-zero extensions in each of $\Ext^1(T,F)$ and $\Ext^1(T(K_X),F)$, $E$ satisfies $$\Hom(T,E)=\Hom(T(K_X),E)=0.$$  It follows that $E$ is $\sigma_+$-stable of class $\v$.  Indeed, if not, then the class of the maximal destabilizing subobject $A$ would satisfy $\phi^+(\v(A))>\phi^+(\v)$ and thus must either be $\w$ or $2\w$.  But then we would get $\Hom(T,E)\neq 0$ or $\Hom(T(K_X),E)\neq 0$, a contradiction.  Thus we get a $\P^1\times\P^1$ worth of non-isomorphic $\sigma_+$-stable objects of class $\v$ that gets contracted by $\pi^+$.  If $\v$ is not minimal, then as in \cref{enum-prop:exceptionalflop1}, we may apply the composition of spherical and weakly-spherical twists $\Phi$ to reduce to the minimal case.

Finally, we move on to \cref{enum-prop:spherical} and deal with both subcases at the same time.  The proof proceeds exactly as in \cite[Lemma 9.1, case (b)]{BM14b} persuant to the following remark: if there is a lattice point $\a$ in the parallelogram with vertices $0,\w,\v,\v-\w$ that satisfies $\a^2=-1$, then we are in \cref{enum-prop:exceptional}, which we have covered already.  The argument of \cite[Lemma 9.1, case (b)]{BM14b} then carries through without change.
\end{proof}

Now we prove the converse to Proposition \ref{prop:flops}, namely that if $\HH_\WW$ does not fall into any of the above mentioned cases, then $\WW$ is not a genuine wall.
\begin{Prop}\label{prop: fake or non-walls}
Assume that $\v$ is primitive and that $\WW$ induces neither a divisorial contraction nor a $\P^1$-fibration.  Assume further that we are not in  \crefrange{enum-prop:sum2positive}{enum-prop:spherical} of Proposition \ref{prop:flops}.  Then $\WW$ is either a fake wall, or not a wall at all.
\end{Prop}
\begin{proof}
We consider first the case that $\v$ is minimal in its $G_{\HH}$-orbit.  Furthermore, we assume for now that $\v^2\geq 3$ and prove that in this case every $\sigma_+$-stable object $E$ of class $\v$ is $\sigma_0$-stable.  If not, then some such $E$ is strictly $\sigma_0$-semistable, and thus $\sigma_-$-unstable.  Let $\a_1,\ldots,\a_n$ be the Mukai vectors of the HN-filtration factors of $E$ with respect to $\sigma_-$.  By assumption on the failure of condition \ref{enum-prop:sum2positive}, the $\a_i$ cannot all be in $P_{\HH}$, so $E$ must have a destabilizing spherical or exceptional subobject or quotient $T$ with $\v(T)=\w$.  

If there is only one $\sigma_0$-stable spherical or exceptional object (in the latter case, uniqueness is only up to $-\otimes\OO_X(K_X)$ of course), then clearly $\v-\w\in P_{\HH}$, so $\v^2-2\langle \w,\v\rangle+\w^2\geq 0$, contradicting the assumption about the failure of conditions \ref{enum-prop:exceptional} and \ref{enum-prop:spherical} of Proposition \ref{prop:flops}.

Now suppose instead that there are two $\sigma_0$-stable spherical/exceptional objects with Mukai vectors $\w_0,\w_1$.  We must have $\v-\w\in C_{\WW}$, and moreover, by \cite[Lemma 4.6]{Yos16b} any stable factor of $T$ must also be spherical or exceptional, so $\v-\w_0$ or $\v-\w_1$ must be effective as well.  The assumption about the failure of conditions \ref{enum-prop:exceptional} and \ref{enum-prop:spherical} in addition to the minimality assumption on $\v$ force $\langle \v,\w_i\rangle>\frac{\v^2}{2}$, and thus that $(\v-\w_i)^2<\w_i^2$, for $i=0,1$.  But then $\v$ must lie above the concave up (region of the) hyperbola $(\v-\w_1)^2=\w_1^2$ and below the concave down hyperbola $(\v-\w_0)^2=\w_0^2$.  In case $\w_0^2=\w_1^2$, these two hyperbolas intersect at 0 and $\w_0+\w_1$, as pictured in \cref{fig:RegionBetweenHyperbolas}, while if, say, $-1=\w_0^2\neq \w_1^2=-2$, then we must have $(\v-\w_0)^2\leq -2$ and $(\v-\w_1)^2<-2$.  Similarly to the previous case, $\v$ must lie above or on the hyperbola $(\v-\w_0)^2=-2$ and below the concave down hyperbola $(\v-\w_1)^2=-2$.  One can easily check that, writing $\v_i=x_i\w_0+y_i\w_1$, $i=1,2$, these hyperbola intersect at two points $\v_1$ and $\v_2$ both of which satisfy $0<x_i,y_i<1$.  In either case, it follows that $\v$ must be located in the interior of the parallelogram with vertices $0,\w_0,\w_1,\w_0+\w_1$.  But then neither $\v-\w_0$ nor $\v-\w_1$ can be effective, a contradiction.  Thus $\WW$ is not a wall at all as every $E\in M_{\sigma_+}(\v)$ is $\sigma_0$-stable.
\begin{figure}
   \begin{tikzpicture}[scale=1.5]
   \draw [->] (-1.5,0) -- (2,0);
   \draw[->] (0,-.2) -- (0,2);
   \draw [blue,domain=-1.5:0] plot (\x,{sqrt(.5*(pow(\x-1,2)-1))});
   \draw [blue,domain=-.43593:0] plot (\x,{.733002-sqrt(.5*(pow(\x+1.43593,2)-1))});
   \draw [blue,domain=-.43593:.5] plot (\x,{.733002+sqrt(.5*(pow(\x+1.43593,2)-1))});
   \draw [red,domain=-1.5:2] plot (\x,{sqrt(.5*(pow(\x,2)+.25))});
   \filldraw [gray] (-.2,.38079) circle (1pt) node [anchor=north east] {$\v_0$};
   \filldraw [gray] (1,0) circle (1pt) node [anchor=south west] {$\w_0$};
   \filldraw [gray] (-1.43593,0.733002) circle (1pt) node [anchor=east] {$\w_1$};
	\node[above] at (-1,1.5) {$(\v-\w_0)^2=-1$};
	\node[above right] at (.3,1.5) {$(\v-\w_1)^2=-1$};
	\node[left] at (2,1) {$\v^2=\v_0^2$};
   \end{tikzpicture}
   \caption{The region between the two hyperbolas when $\w_0^2=\w_1^2=-1$.}
   \label{fig:RegionBetweenHyperbolas}
\end{figure}
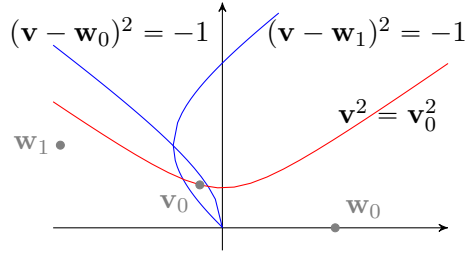

Now let us suppose that $\v$ is still minimal but $\v^2=1,2$.  We will show that though there may be some strictly $\sigma_0$-semistable object $E$, there are no curves of such objects that $\sigma_+$-stable.  Take a strictly $\sigma_0$-semistable object $E$ and consider its Jordan-H\"{o}lder filtration with respect to $\sigma_0$ with $\sigma_0$-stable factors $E_i$ with $\v(E_i)=\a_i$.  Then we may write $\v=\sum_{i=1}^n\a_i$, and the minimality of $\v$ forces $\langle\v,\a_i\rangle\geq 0$ for all $i$, so we may order the $\a_i$ such that $$0\leq\langle\v,\a_1\rangle\leq\langle\v,\a_2\rangle\leq\dots\leq\langle\v,\a_n\rangle.$$

Let us begin with $\v^2=1$, and we observe that for $E$ to be strictly $\sigma_0$-semistable we must have $\langle\v,\a_1\rangle=0$, since otherwise $$1=\v^2=\sum_{i=1}^n\langle\v,\a_i\rangle\geq n.$$  But then the signature of $\HH_\WW$ forces $\a_1^2<0$, so that $\a_1^2=-2$ or $-1$ and $E_1$ is spherical or exceptional, respectively.  We claim that actually $E_1$ must be exceptional.  Indeed, if $E_1$ were spherical then $\HH_\WW$ would be non-isotropic by \cref{Prop:isotropic lattice}, so $\WW$ would induce a divisorial contraction for each choice of $L$ by \cref{Lem:NonisotropicDivisorialContraction}, contrary to assumption.  Thus $E_1$ must be exceptional.  It follows that $\HH_\WW$ is isotropic as $(\v-\a_1)^2=0$, and $\langle\v-\a_1,\a_1\rangle=1$ implies that $\v=\a_1+\u_1$, which is the only decomposition of $\v$ into effective classes.  If $\ell(\u_1)=2$, then by \cref{Lem: Hilbert-Chow} $\WW$ is totally semistable but contracts no curves.  If $\ell(\u_1)=1$, then by \cref{Lem:isotropic divisorial l=1 1} $M_{\sigma_+}(\v,L)\backslash M_{\sigma_0}^s(\v,L)$ is a divisor containing no contracted curves for each choice of $L$, as required.

Now we consider $\v^2=2$.  If $\langle\v,\a_1\rangle>0$, then we must have $n=2$ and $\langle\v,\a_1\rangle=\langle\v,\a_2\rangle=1$.  Moreover, $$\a_1^2=(\v-\a_2)^2=\v^2-2\langle\v,\a_2\rangle+\a_2^2=\a_2^2.$$  Thus if $\a_1^2\geq 0$, then $\a_2^2\geq 0$, so from $\v^2=2$, we see that the only decomposition with $\a_i\in P_\HH$ occurs when $\a_i^2=0$ and $\langle\a_1,\a_2\rangle=1$.  It follows that $\ell(\a_1)=\ell(\a_2)=1$, so we get that $M_{\sigma_+}(\v,L)\backslash M_{\sigma_0}^s(\v,L)$ is a divisor containing no contracted curves by the end of the proof of \cref{Lem:isotropic divisorial l=1 1}.  Otherwise, $\a_i^2<0$, and we have $\langle\v,\a_i\rangle=\frac{\v^2}{2}$ as in \cref{enum-prop:exceptionalflop1,enum-prop:sphericalflop2} of \cref{prop:flops}, contrary to our assumption.  

Thus we must consider $\langle\v,\a_1\rangle=0$, and from the signature of $\HH_\WW$ we must have $\a_1^2<0$ so that $a_1^2=-2$ or $-1$ and $E_1$ is spherical or exceptional, respectively.  

If $E_1$ is spherical, then $(\v-\a_1)^2=0$, so $\HH_\WW$ is isotropic, and we either have $\v-\a_1=\u_1$ with $\langle\u_1,\a_1\rangle=2$ or $\v-\a_1=2\u_1$ with $\langle\u_1,\a_1\rangle=1$ (in which case $\ell(\u_1)=1$).  If $\ell(\u_1)=2$, then we must have $\v=\a_1+\u_1$ with $\langle\a_1,\u_1\rangle=2$, and this is the only decomposition of $\v$.  But then by \cref{Lem:P1FibrationSpherical} $\WW$ is not a wall at all since by assumption $\WW$ does not induce a $\P^1$-fibration.  If $\ell(\u_1)=1$, then from the end of the proof of \cref{Lem:isotropic divisorial 1-1 2}, we cannot have the decomposition $\v=\a_1+\u_1$ as this would lead to a divisorial contraction for each choice of $L$, contrary to assumption.  Thus we must have $\v=\a_1+2\u_1$ with $\langle\a_1,\u_1\rangle=1$.  But then as $\WW$ does not induce a $\P^1$-fibration, we must have $L\equiv D+(\frac{\rk\v}{2}+1)K_X\pmod 2$ and $M_{\sigma_+}(\v,L)\backslash M_{\sigma_0}^s(\v,L)$ is a divisor that contains no contracted curves by \cref{Lem:isotropic totally semistable divisorial contraction l=1}, as required.

If $E_1$ is exceptional, then $\HH_\WW$ must be non-isotropic.  Indeed, if $\HH_\WW$ were isotropic, then we could write $\v=x\u_1+y\a_1$ with $x,y\in\Z$, and the two conditions $\v^2=2$ and $\langle\v,\a_1\rangle=0$ would force $2=y^2$, a contradiction.  If $L\equiv D+\frac{\rk\v}{2}K_X\pmod2$ and $\HH_\WW$ falls into \cref{enum:OneExceptionalOneSpherical} of \cref{Prop:lattice classification}, then from \cref{Lem:ExceptionalDivisorialNonContraction}, $\WW$ would induce a divisorial contraction, contrary to our assumptions.  So either $\HH_\WW$ falls into \cref{enum:OneExceptionalOneSpherical} of \cref{Prop:lattice classification} and $L\notequiv D+\frac{\rk\v}{2}K_X\pmod 2$ or $\HH_\WW$ does not fall into \cref{enum:OneExceptionalOneSpherical} of \cref{Prop:lattice classification}.  But then \cref{Lem:ExceptionalDivisorialNonContraction} shows that no curves get contracted, as required.

If $\v$ is not minimal, then $\v\in\CC_n$ for some $0\neq n\in\Z$ and there exists a minimal class $\v_0$ in the same orbit.  By \cref{Prop:NonMinimalIsomorphism,Prop:CompositionSphericalExceptional}, if $n$ is even then $M_{\sigma_+}(\v)\cong M_{\sigma_+}(\v_0)$, induced by a sequence of spherical/exceptional twists, and if $n$ is odd, then $M_{\sigma_+}(\v)\cong M_{\sigma_+}(\v_1)\cong M_{\sigma_-}(\v_0)$.  Since the assumptions of the proposition are invariant under $G_{\HH}$, the same asumptions apply to $\v_0$, so every $\sigma_\pm$-stable object of class $\v_0$ is $\sigma_0$-stable.  If $\Phi$ is the sequence of spherical/exceptional twists used in this isomorphism, then from the definition of a spherical/exceptional twist, it is easy to see that the $S$-equivalence class of $\Phi(E_0)$ is determined by that of $E_0$.  Since this equivalence is trivial on $M_{\sigma_{\pm}}(\v_0)$, it must be trivial on $M_{\sigma_+}(\v)$ as well, implying that $\pi^+$ is an isomorphism.  Thus $\WW$ is a fake wall, as claimed. 
\end{proof}

Finally we can prove \cref{classification of walls}:
\begin{proof}[Proof of \cref{classification of walls}]
The theorem follows from \cref{prop: fake or non-walls,prop:flops,Prop:isotropic-classification,Lem: condition for totally semistable wall,Prop:NonisotropicDivisorialContraction}.
\end{proof}
We will end this section by observing that \cref{classification of walls,Prop:OrthgonalIsomorphism 2,Prop:CompositionSphericalExceptional} prove part \ref{enum:MT1-two moduli are birational} of \cref{Thm:MainTheorem1}.  Indeed, connect $\sigma$ and $\tau$ by a path, $\sigma(t)$, with $0\leq t\leq 1$, $\sigma(0)=\sigma$, and $\sigma(1)=\tau$.  Observe that as the set of walls is lcoally finite, $\sigma(t)$ only crosses finitely many of them, and by perturbing $\sigma(t)$ slightly, we may assume that $\sigma(t)$ only crosses one wall at a time (that is, if $\sigma(t_0)\in\WW$ then $\sigma(t_0)$ is a generic point of $\WW$).  If $\sigma(t)$ does not cross any totally semistable walls, then $M_\sigma(\v)$ and $M_\tau(\v)$ are clearly birational.  Otherwise, it suffices consider that $\sigma=\sigma_+$ and $\tau=\sigma_-$ are two sufficiently close stability conditions separated by a single totally semistable wall $\WW$.  If there exists a spherical/exceptional class $\w\in\HH_\WW$ such that $\langle\v,\w\rangle=0$, then $M_{\sigma_+}(\v)\cong M_{\sigma_-}(\v)$ by \cref{Prop:OrthgonalIsomorphism 2,Prop:CompositionSphericalExceptional}.  Otherwise, by \cref{classification of walls,Prop:NonMinimalIsomorphism,Prop:CompositionSphericalExceptional}, we may assume that $\v$ is minimal so that we must have $\langle\v,\u\rangle=1$ for an isotropic $\u\in\HH_\WW$ such that $\ell(\u)=2$.  But then we have seen that $\pi^+$ is a divisorial contraction, so we again have $M_{\sigma_+}(\v)$ and $M_{\sigma_-}(\v)$ are birational, as claimed.  

In \cref{sec:Main Theorems} we will prove part \ref{enum:MT1-birational map given by FM transform} of \cref{Thm:MainTheorem1} and the rest of our main results.  In particular, we will show that, under generic conditions, the birational map between $M_{\sigma_+}(\v)$ and $M_{\sigma_-}(\v)$ is induced by a Fourier-Mukai transform.  This is obvious when $\WW$ induces a small contraction, but is more subtle in the case of divisorial contractions.  But while the behavior of divisorial contractions for Bridgeland moduli on Enriques surfaces is analogous to the K3 case for divisorial contractions of Brill-Noether, Hilbert-Chow, and Li-Gieseker-Uhlenbeck type, there is an additional type of divisorial contraction that requires extra care.  We will turn to this subject in the next section.
\section{LGU on the covering K3 surface}
As its name suggests, the induced Li-Gieseker-Uhlenbeck (iLGU) type contraction, which occurs when the hyperbolic lattice $\HH_\WW$ contains an isotropic $\u$ such that $\langle\v,\u\rangle=1=\ell(\u)$, is induced from the covering K3.  The associated HN-filtration factors are objects $F\in M_{\sigma_0}^s(2\u)$.  But while this stable locus has dimension two, it is not proper as $M_{\sigma_0}(2\u)\backslash M_{\sigma_0}^s(2\u)\neq\varnothing$.  So we cannot use the usual machinery as in the regular LGU case where $\ell(\u)=2$.  There we use the universal family associated to the proper two-dimensional moduli space to obtain the requisite Fourier-Mukai transform as in \cite[Corollary 2.8]{BM01}.  Instead, we show that we may induce a Fourier-Mukai transform on $X$ by considering the corresponding LGU type contraction on its K3 cover $\widetilde{X}$.  So we turn now to considering LGU contractions more carefully in the K3 case.

\subsection{Fourier-Mukai transform associated to the Li-Gieseker-Uhlenbeck contraction}\label{Sec:FM transform associated to LGU}

Let $X_1$ be a K3 surface and $\u_1$ 
be a primitive isotropic Mukai vector on $X_1$.
For a general stability condition
$\sigma\in\Stabd(X_1)$, we consider $X_2:=M_\sigma(\u_1)$, which is another K3 surface.
Let $\EE \in \Db(X_1 \times X_2,1_{X_1} \times \alpha)$ be a
universal object for $M_\sigma(\u_1)$, considered as a twisted object where
$\alpha$ is a 2-cocycle of $\OO_{X_2}^{\times}$. 
We set
\begin{equation}
\Phi:=\Phi_{X_1 \to X_2}^{\EE^{\vee}},\;
\Psi:=\Phi_{X_2 \to X_1}^{\EE}.
\end{equation}

For a smooth variety $Y$,
we also set
\begin{equation}
D_Y(E):=E^{\vee}=\RlHom_{\OO_Y}(E,\OO_Y),\; E \in\Db(Y).
\end{equation}

Let $\v_1$ be a Mukai vector on $X_1$ such that $\langle \u_1,\v_1 \rangle=2$.  It is well known that, using the Fourier-Mukai transform $\Phi$, the moduli space $M_{\sigma}(\v_1)$ is isomorphic to a moduli space of rank two Gieseker stable sheaves of Mukai vector $\v_2=\v(\Phi(\v_1))$ on $X_2$.  In the next result, we translate a well-known concrete result about crossing the Uhlenbeck wall for $\v_2$ on $X_2$ to an instrinsic analogous statement for $X_1$.

\begin{Prop}\label{prop:uhl}
There is an object $\FF \in \Db(X_1 \times X_1)$
which induces a Fourier-Mukai transform
$\Phi_{X_1 \to X_1}^{\FF^{\vee}}:\Db(X_1) \to \Db(X_1)$
such that
\begin{enumerate}
\item\label{enum:PreservationOfStabd}
$\Phi_{X_1 \to X_1}^{\FF^{\vee}}$ preserves $\Stab^\dagger(X_1)$.
\item\label{enum:FormulaForLGUFMTransform} The induced action on $H^*(X_1,\Z)$ is given by the formula
\begin{equation}
D_{X_1} \circ \Phi_{X_1 \to X_1}^{\FF^{\vee}}(\v)
=-(\v+\tfrac{\v_1^2}{2} \langle \v,\u_1 \rangle \u_1-\langle \v,\v_1 \rangle \u_1-
\langle \v,\u_1 \rangle \v_1),
\end{equation} for any $\v\in H^*(X_1,\Z)$.
\end{enumerate} 
\end{Prop}

\begin{proof}
For $F \in \Db(X_1)$ with $\v(F)=\v_1$,
we set
\begin{equation}
N:=\det \Phi(F),\mbox{ and }\v_2:=\v(\Phi(F))=(2,\xi,a).
\end{equation}
As mentioned above, $\Phi$ induces an isomorphism between $M_{\sigma}(\v_1)$ and the moduli space $M_\omega(\v_2,\alpha)$ of $\omega$-Gieseker stable $\alpha$-twisted sheaves of Mukai vector $\v_2$ on $X_2$.  It is well known that crossing the LGU wall of the Gieseker chamber for $\v_2$ is induced by the Fourier-Mukai transform $(\blank)^\vee\otimes N=(\otimes N)\circ D_{X_2}$.  Pulling this back to $X_1$, we have a sequence of equivalences:
\begin{equation}
\begin{CD}
\Db(X_1) @>{\Phi}>> \Db(X_2,\alpha^{-1}) @>{D_{X_2}}>> 
\Db(X_2,\alpha)  
@>{\otimes N}>>\Db(X_2,\alpha^{-1})@>{\Psi}>> \Db(X_1),
\end{CD}
\end{equation}
and we denote their composition by
\begin{equation}\label{eq:Xi}
\Xi:=\Psi \circ(\otimes N) \circ D_{X_2} \circ \Phi,
\end{equation}
which is a (contravariant)-equivalence $\Db(X_1) \to \Db(X_1)$.
Thus
$D_{X_1} \circ \Xi$ is an autoequivalence of $\Db(X_1)$. 

Since
$[2] \circ \Xi=
D_{X_1} \circ \Phi_{X_2 \to X_1}^{\EE^{\vee}} \circ(\otimes N^{\vee})  \circ \Phi$
(by Grothendieck-Serre duality),
$D_{X_1} \circ \Xi$ defines a Fourier-Mukai transform
$\Phi_{X_1 \to X_1}^{\FF^{\vee}}:\Db(X_1) \to \Db(X_1)$.

To see that $\Phi_{X_1\to X_1}^{\FF^\vee}$ preserves $\Stabd(X_1)$, note first that $\Phi(\sigma)$ is geometric as $\Phi$ identifies the objects of $M_{\sigma}(\u_1)$ with the points of $X_2$ considered as $\Phi(\sigma)$-stable objects of class $(0,0,1)$.  Thus
$\Phi$ induces an isomorphism $\Stabd(X_1) \cong
\Stabd(X_2)$. Since $\otimes N^{\vee}$ also preserves the space of full numerical stability conditions, $\Phi_{X_1 \to X_1}^{\FF^{\vee}}$
preserves $\Stabd(X_1)$, as claimed in \ref{enum:PreservationOfStabd}.

Now we prove \ref{enum:FormulaForLGUFMTransform}.  Since $\v_2=(2,\xi,a)$, it follows that
$$
(\Q \v_2+\Q \varrho_{X_2})^\perp=\Set{e^{\xi/2}(0,D,0) \ | \ D\in H^2(X_2,\Q)},
$$
where the orthogonal complement is taking inside $H^*(X_2,\Q)$.  Moreover, under $(\otimes N) \circ D_{X_2}$, we have
\begin{equation}
\begin{matrix}
\v_2 & \mapsto & \v_2\\
\varrho_{X_2} & \mapsto & \varrho_{X_2}\\
e^{\xi/2}(0,D,0) & \mapsto & -e^{\xi/2}(0,D,0).
\end{matrix}
\end{equation}
Hence
\begin{equation}
\Xi_{|(\Q \u_1+\Q \v_1)}=1_{(\Q \u_1+\Q \v_1)},\;
\Xi_{|(\Q \u_1+\Q \v_1)^\perp}=-1_{(\Q \u_1+\Q \v_1)^\perp}.
\end{equation}
Then it is easy to see that for any $\v\in H^*(X_1,\Z)$ we have 
\begin{equation}
\Xi(\v)=-(\v+\tfrac{\v_1^2}{2} \langle \v,\u_1 \rangle \u_1-
\langle \v,\v_1 \rangle \u_1-
\langle \v,\u_1 \rangle \v_1).
\end{equation}
\end{proof}

\subsection{The induced Li-Gieseker-Uhlenbeck wall and its Fourier-Mukai transform on an Enriques surface}
Having delved more deeply into the FM transform associated to a wall of LGU type on K3 surfaces, we explore the corresponding picture on the Enriques quotient.  So let $X$ be an Enriques surface with K3 cover $\widetilde{X}$.

\begin{Prop}\label{prop:Enriques-refl}
Assume that $\Pic(\widetilde{X})=\varpi^*\Pic(X)$.
Let $\u_0$ and $\v_0$ be Mukai vectors such that
$\u_0$ is primitive and isotropic,
$\langle \u_0,\v_0 \rangle=1$ and $\ell(\u_0)=1$.
We set
$$
\v':=-(\varrho_X+2\v_0^2 \langle \varrho_X,\u_0 \rangle \u_0
-2\langle \varrho_X,\v_0 \rangle \u_0-2
\langle \varrho_X,\u_0 \rangle \v_0).
$$
Then there is an autoequivalence 
$\Phi_{X \to X}^{\EE^{\vee}}:\Db(X) \to \Db(X)$
such that 
\begin{enumerate}
\item\label{enum:FormulaForLGUFM-Enriques} For any $\v\in\Hal(X,\Z)$,
\begin{equation}
D_X \circ \Phi_{X \to X}^{\EE^{\vee}}(\v)=
-(\v+2\v_0^2 \langle \v,\u_0 \rangle \u_0-2\langle \v,\v_0 \rangle \u_0-
2\langle \v,\u_0 \rangle \v_0).
\end{equation}
\item\label{enum:ClassOfRestriction}
$\v(\EE|_{X \times \{ x\}})=\v'$.
\item
$\varpi^* \circ D_X \circ \Phi_{X \to X}^{\EE^{\vee}}=\Xi \circ \varpi^*$
for $\Xi$ in \eqref{eq:Xi}. 
\end{enumerate}
\end{Prop}

\begin{proof}
We note that $\v'$ is a primitive and isotropic Mukai vector with
$\ell(\v')=2$.
We set $\v_1:=\varpi^*(\v_0)$ and $\u_1:=\varpi^*(\u_0)$.
By \cref{prop:uhl}, we have an object
$\FF \in \Db(\widetilde{X} \times \widetilde{X})$
which defines a Fourier-Mukai transform satisfying \ref{enum:PreservationOfStabd} and \ref{enum:FormulaForLGUFMTransform} of \cref{prop:uhl}.  We will show that $\FF$ descends, in an appropriate sense, to an object $\EE\in\Db(X\times X)$ that defines the desired autoequivalence of $\Db(X)$. 

Since $\Phi_{\widetilde{X} \to \widetilde{X}}^{\FF^{\vee}}$ preserves 
$\Stab^\dagger(\widetilde{X})$, which is isomorphic to $\Stab^\dagger(X)$ from the assumption $\Pic(\widetilde{X})=\varpi^*\Pic(X)$ by  \cite[Theorem 1.2]{MMS09}, there is an $\iota$-invariant stability condition
$\sigma' \in \Stab^\dagger(\widetilde{X})$ such that 
$\FF$ is a family of $\sigma'$-stable objects.
By the $\iota$-invariance of $\sigma'$,
$\iota^*(\FF|_{\widetilde{X} \times \{ \widetilde{X} \}})$ is also $\sigma'$-stable, and moreover, $2\v(\FF|_{\widetilde{X} \times \{ \widetilde{X}\}})=\varpi^*(\v')$.  Hence  
we have  
an isomorphism $\tau:\widetilde{X} \to \widetilde{X}$ 
such that
\begin{equation}
(\iota \times 1_{\widetilde{X}})^*(\FF) \cong
(1_{\widetilde{X}} \times \tau)^*(\FF) \otimes L,
\end{equation}
where $L$ is a line bundle on $\widetilde{X}$.
Then we have a commutative diagram

\begin{equation}
\begin{CD}
\Db(\widetilde{X}) @>{\Phi_{\widetilde{X} \to 
\widetilde{X}}^{\FF^{\vee}}}>> \Db(\widetilde{X})\\
@V{\iota^*}VV @VV{(\otimes L^{\vee}) \circ \tau^*}V \\
\Db(\widetilde{X}) @>>{\Phi_{\widetilde{X} \to 
\widetilde{X}}^{\FF^{\vee}}}>\Db(\widetilde{X}).\\
\end{CD}
\end{equation}
By \cref{prop:uhl}, we see that
\begin{equation}
\iota^* \circ \Phi_{\widetilde{X} \to \widetilde{X}}^{\FF^{\vee}}(\v)
=\Phi_{\widetilde{X} \to \widetilde{X}}^{\FF^{\vee}} \circ \iota^*(\v),\;
 \v \in H^*(\widetilde{X},\Z).
\end{equation}
Hence we get
\begin{equation}
\iota^*(\v)=(\otimes L^{\vee}) \circ \tau^*(\v),\;
\v \in H^*(\widetilde{X},\Z),
\end{equation}
which implies $L=\OO_{\widetilde{X}}$
and $\iota^*=\tau^*$ 
on $H^*(\widetilde{X},\Z)$ (by substituting $\v=\v(\OO_{\widetilde{X}})$).

By the Torelli theorem,
$\iota=\tau$.  Therefore $\Phi_{\widetilde{X} \to \widetilde{X}}^{\FF^{\vee}}$ and 
$\iota$ commute.
Then as in the proof of \cite[Theorem 4.5]{BM98}, there is $\EE \in \Db(X \times X)$ such that
\begin{equation}
(\varpi \times 1_{\widetilde{X}})_*(\FF)
\cong (1_X \times \varpi )^*(\EE).
\end{equation}

It follows that $\Phi_{X \to X}^{\EE^{\vee}}$ defines the desired 
Fourier-Mukai transform.
\end{proof}

\begin{Rem}
The conditions (1) and (2) are cohomological, and
\cite[Theorem 1.2 (ii)]{MMS09} shows that there are autoequivalences of $\Db(X)$
satisfying (1) and (2) 
without the assumption $\Pic(\widetilde{X})=\varpi^*\Pic(X)$.
Since we need the property (3) for the proof of Proposition \ref{prop:LGUK3} below, we 
gave a proof.
\end{Rem}
Now we can relate the induced Fourier-Mukai transform $\Phi_{X\to X}^{\EE^\vee}$ to the birational behavior at a iLGU wall.

\begin{Prop}\label{prop:LGUK3}
Assume that $\Pic(\widetilde{X})=\varpi^*\Pic(X)$.
Let $\WW$ be a wall for $\v$ defined by an isotropic Mukai vector
$\u$ such that $\langle \u,\v \rangle=1$ and $\ell(\u)=1$.
Let $\sigma_0 \in \WW$ be a generic stability condition.
Assume that $\sigma_\pm$ are sufficiently close stability conditions in opposite, adjacent chambers separated by $\WW$.
Then 
$D_X \circ \Phi_{X \to X}^{\EE^{\vee}}$ induces an isomorphism
$M_{\sigma_+}(\v) \cong M_{\sigma_-}(\v)$.
\end{Prop}

\begin{proof}
There are two important consequences of our assumption that $\Pic(\widetilde{X})=\varpi^*\Pic(X)$.  The first is that $H^*_{\alg}(\widetilde{X},{\Bbb Z})=\NS(X)_{tf}(2) \oplus \langle 2 \rangle
\oplus \langle -2 \rangle$, where we recall that $\NS(X)_{tf}(2)$ is the same underlying lattice as the torsion free lattice $\NS(X)_{tf}$ but with pairing multiplied by 2; and the second is that $\Stabd(\widetilde{X})=\Stabd(X)$.  In particular, it follows from the first consequence that setting $X_1:=\widetilde{X}$,
$\v_1:=\varpi^*(\v)$, and $\u_1:=\varpi^*(\u)=\u_1$, there is no $\w\in\Hal(X_1,\Z)$ with $\langle \w,\u_1 \rangle=1$.
Thus $\alpha$ as in \cref{Sec:FM transform associated to LGU} defines a non-trivial Brauer class on $X_2:=M_{\varpi^*(\sigma)}(\u_1)$, and there are no $\alpha$-twisted sheaves on $X_2$ of rank 1.  From the second consequence it follows that $\varpi^*(\sigma_\pm)$ remains generic with respect to $\varpi^*(\v)$.

Considering the autoequivalence $\Phi$ constructed in the proof of \cref{prop:uhl}, we see that for $F \in \MM_{\sigma_+}(\v)$, $\Phi(\varpi^*(F))$ is a $\mu$-stable $\alpha^{-1}$-twisted sheaf of rank two, as there are no twisted sheaves of rank one.  Since $\varpi^*(\sigma_\pm)$ remains generic, $\mu$-stability is taken with respect to a generic polarization, in which case $\mu$-stability is independent of the choice of the B-field.  Thus $\varpi^*(\sigma_+)$ and $\varpi^*(\sigma_-)$ are separated by the single wall associated to the LGU contraction.  The construction of $\Xi$ in \cref{prop:uhl} shows that $\Xi$ switches the chambers containing $\varpi^*(\sigma_\pm)$, so it follows that $\Xi(\varpi^*(\sigma_+))$ and $\varpi^*(\sigma_-)$ belong to the same chamber.  Therefore $\Xi$ induces an isomorphism
$\MM_{\varpi^*(\sigma_+)}(\v_1) \to \MM_{\varpi^*(\sigma_-)}(\v_1)$.  
By \cref{prop:Enriques-refl} (3), it follows that
$D_X \circ \Phi_{X \to X}^{\EE^{\vee}}$ induces
an isomorphism $\MM_{\sigma_+}(\v) \to \MM_{\sigma_-}(\v)$, as claimed.
\end{proof}



\subsection{The case of odd rank}\label{subsec:iLGU odd rank}
For the construction of the equivalence $\Phi_{X \to X}^{\EE^{\vee}}$,
we had to assume that $\Pic(\widetilde{X})=\varpi^*\Pic(X)$. 
Then by using the equivalence, we constructed the isomorphism in Proposition \ref{prop:LGUK3}. In this subsection, we shall construct 
an isomorphism $M_{\sigma_+}(\v,L) \to M_{\sigma_-}(\v,L)$ without 
the construction of $\Phi_{X \to X}^{\EE^{\vee}}$, and thus without the assumption $\Pic(\widetilde{X})=\varpi^*\Pic(X)$, if $\rk \v$ is odd.

For a stability condition $\sigma \in \Stab^\dagger(X)$,
we set $\sigma':=\varpi^*(\sigma)$.
Then we have a morphism 
$\psi:\MM_\sigma(\v) \to \MM_{\sigma'}(\v')$
by sending $E \in \MM_\sigma(\v)$ to
$\varpi^*(E)$, where $\v'=\varpi^*(\v)$.  We need to study when the restriction of $\psi$ to the stable locus gives $\sigma'$-stable objects, the content of the next lemma. 

\begin{Lem}[{\cite[Theorem 8.1]{Nue14b}}]\label{Lem:pullback remains stable} 
For a $\sigma$-stable object $E$, 
$\varpi^*(E)$ is properly $\sigma'$-semistable if and only if
$E=\varpi_*(F)$, where $F$ is a $\sigma'$-stable object
 with $F \not \cong \iota^*(F)$.
In particular $\rk E$ is even.
\end{Lem}

\begin{Lem}\label{Lem:odd rank embedding}
If $\rk\v$ is odd, then 
$\psi:\MM_\sigma^s(\v,L) \to \MM_{\sigma'}(\v')$ is an embedding and
the image is contained in $\MM_{\sigma'}^s(\v')$, where $[L \mod K_X]=c_1(\v)$.
\end{Lem}

\begin{proof}
Since $\rk\v$ is odd, the image is  contained in $\MM_{\sigma'}^s(\v')$ by \cref{Lem:pullback remains stable}.
Combining with \cite[Proposition 7.2]{Nue14b}, we get the claim.
\end{proof}

In particular, if $\v$ has odd rank and $\sigma$ is generic, then \cref{Lem:odd rank embedding} implies that $\psi$ embeds all of $M_\sigma(\v,L)$ into $M_{\sigma'}^s(\v',L)$.  Moreover, for any $\sigma''$ in a sufficiently small neighborhood of $\sigma'$, we also have an embedding
\begin{equation}\label{eq:psi}
M_\sigma^s(\v,L) \hookrightarrow M_{\sigma'}^s(\v') \hookrightarrow M_{\sigma''}(\v')
\end{equation}
by openness of stability.

For the rest of the section, we assume that $\rk \v$ is odd.  Let $\WW\subset\Stabd(X)$ be a wall defined by a primitive isotropic $\u$ with $\langle\v,\u\rangle=1=\ell(\u)$ and $\sigma_0$ a generic stability condition in $\WW$.  We denote by $\sigma_0'=\varpi^*(\sigma_0)$, and we continue to denote by $\v'$ and $\u'$ the pull-backs $\varpi^*(\v)$ and $\varpi^*(\u)$, respectively.  Similarly, let $\WW'$ be the wall for $\v'$ in $\Stabd(\widetilde{X})$ defined by $\u'$.  Take $\sigma'' \in \WW'$ in a neighborhood of $\sigma_0'$ such that $\WW'$ is the unique wall for $\v'$ in a small neighborhood $U_{\sigma''}$ of $\sigma''$ and such that $\sigma''$ is generic with respect to $\u'$.
We also take $\sigma_\pm'' \in U_{\sigma''}$ in opposite and adjacent chambers for $\v'$, separated by the single wall $\WW'$. 
We may assume that $\varpi^*(\sigma_+)$ and $\varpi^*(\sigma_-)$ belong to the same chambers as $\sigma_+''$ and $\sigma_-''$, respectively.
Let $\psi_\pm:M_{\sigma_\pm}(\v,L) \to M_{\sigma_\pm''}(\v')$ be the embeddings for
$\sigma_{\pm}$
in \eqref{eq:psi}.  

We will show in a moment that there is an isomorphism $\varphi:M_{\sigma_+}(\v,L)\to M_{\sigma_-}(\v,L)$ despite $\varphi$ (possibly) not being induced by an autoequivalence of $\Db(X)$.  But first we start with a cohomological lemma.
\begin{Lem}\label{Lem:u and v}
Let $\v'=\varpi^*(\v)$ and $\u'=\varpi^*(\u)$ where $\langle\v,\u\rangle=1=\ell(\u)$ for the primitive isotropic vector $\u$.  Then the sublattice $\Z \u'+\Z \v'$ is saturated in $H^*(\widetilde{X},\Z)$.  

In particular, if $\WW\subset\Stabd(X)$ is the wall for $\v$ defined by $\u$ and $\WW'\subset\Stabd(\widetilde{X})$ is the wall for $\v'$ defined by $\u'$), then we have $\HH_{\WW'}=\Z\v'+\Z\u'=\varpi^*\HH_\WW$ so that $\langle\v_1,\v_2\rangle\in 2\Z$ for any $\v_1,\v_2\in\HH_{\WW'}$.
\end{Lem}

\begin{proof}
Let $L$ be the saturation of $\Z\u'+\Z\v'$ in $H^*(\widetilde{X},\Z)$.  Since $\u'$ is primitive, there is $\w \in L$ such that $L=\Z \u'+\Z \w$.  We set $\v'=a \u'+b \w$ ($a,b \in \Z$).  Then $2=\langle \v',\u' \rangle=b \langle \w,\u' \rangle$.  Hence $b=\pm 1,\pm 2$.  If $b=\pm 2$, then $\rk \v'=a \rk \u'\pm 2 \rk \w \in 2\Z$, which is impossible if $\rk\v'=\rk\v$ is odd.  If $\rk\v$ is even, then let us write $\v=(2n,D_1,s)$ and $\u=(2m,D_2,t)$ so that $\v'=(2n,\varpi^*D_1,2s)$ and $\u'=(2m,\varpi^*D_2,2t)$.  Then $\langle\v',\u'\rangle=2$ implies that $$2(D_1,D_2)=(\varpi^*D_1,\varpi^*D_2)\equiv 2\pmod 4,$$ so $(D_1,D_2)$ is odd.  On the other hand, $\v'=\pm 2\w+a\u'$ gives $\varpi^*D_1=\pm2D_3+a\varpi^*D_2$, where $D_3=c_1(\w)$.  But as $\varpi^*:\Pic(X)/\langle K_X\rangle\into\Pic(\widetilde{X})$ is a primitive embedding, we must have $D_3=\varpi^*D_3'$ for some $D_3\in\Pic(X)$.  It follows that $D_1=\pm2D_3'+aD_2$ and thus $$(D_1,D_2)=\pm2(D_3',D_2)+a(D_2)^2\equiv0\pmod 2,$$ a contradiction.  Thus, in either case, $b=\pm 1$ so that $L=\Z \u'+\Z \v'$, as required.

The first equality of the last statement follows because $\HH_{\WW'}\subset\Hal(\widetilde{X},\Z)$ is a saturated hyperbolic sublattice of rank two containing $\Z\v'+\Z\u'$, and the second equality follows from $\langle\v,\u\rangle=1$ so that $\HH_\WW=\Z\v+\Z\w$. 
\end{proof}
Now we are in position to prove the following result.

\begin{Prop}\label{Prop:phi isomorphism odd rank}
We have an isomorphism $\varphi:M_{\sigma_+}(\v,L) \to M_{\sigma_-}(\v,L)$
with a commutative diagram
\begin{equation}\label{eq:varphi=Xi}
\begin{CD}
M_{\sigma_+}(\v,L) @>{\varphi}>> M_{\sigma_-}(\v,L)\\
@V{\psi_+}VV  @VV{\psi_-}V\\ 
M_{\sigma_+''}(\v') @>>{\Xi}> M_{\sigma_-''}(\v')\\
\end{CD}
\end{equation}
In particular, it is the identity on $M_{\sigma_0}^s(\v,L)$.
\end{Prop}

\begin{proof}
For $X_1:=\widetilde{X}$ and $X_2:=M_{\sigma''}(\u')$, we shall apply Proposition \ref{prop:uhl} to get an
(anti-) autoequivalence $\Xi$ of $\Db(X_1)$.  Recall from the construction of $\Xi$ that there is an equivalence $\Phi:\Db(X_1)\to\Db(X_2,\alpha)$ inducing an isomorphism between $M_{\sigma_+''}(\v')$ and the moduli space $M_\omega(2,D,a)$ of $\omega$-Gieseker semistable $\alpha$-twisted sheaves on $X_2$.  Moreover, $\Xi$ is the pull-back to $X_1$ of the autoequivalence $E\mapsto E^\vee\otimes\det(E)$, which induces an isomorphism $M_{\sigma_+''}(\v')\to M_{\sigma_-''}(\v')$.

We claim that $M_\omega(2,D,a)$ consists only of $\mu$-stable sheaves.  Granting this for the moment, we see that for locally free $E\in M_\omega(2,D,a)$, we have $E^\vee\otimes\det(E)\cong E$, so the isomorphism
$M_{\sigma_+''}(\v') \to M_{\sigma_-''}(\v')$ induced by $\Xi$ is the identity
on $M_{\sigma''}^s(\v')$, which corresponds to the locus $M_\omega^{lf}(2,D,a)$ on $X_2$.  

Now let us prove our claim that $M_\omega(2,D,a)=M^{\mu s}_\omega(2,D,a)$.  We note first that the claim is obvious if $\alpha$ is non-trivial, as we saw in the proof of \cref{prop:LGUK3}.  So suppose that the Brauer class $\alpha$ is trivial.  If there exists a strictly $\mu$-semistable sheaf $E\in M_\omega(2,D,a)$ with subsheaf $F\subset E$ such that $\mu(F)=\mu(E)$, then from the genericity of $\omega$, it would follow that $2\mid D$ and $\v(F)=(1,\tfrac{D}{2},b)$ so that $\langle-\v(F),(0,0,1)\rangle=1$.  But $2\v(F)=(2,D,a)+(2b-a)(0,0,1)$, so writing $-\v(F)=\Phi(\w)$, we get $\w$ such that $2\w\in\Z\v'+\Z\u'$ and $\langle\w,\u'\rangle=1$.  As $\HH_{\WW'}=\Z\v'+\Z\u'$ is saturated by Lemma \ref{Lem:u and v}, this is impossible, proving the claim in this case as well.

By \eqref{eq:psi}, we get an isomorphism $\varphi$ with 
the commutative diagram \eqref{eq:varphi=Xi}.
\end{proof}

We close this subsection by describing the relationship between the Picard groups of $M_{\sigma_\pm}(\v,L)$.
\begin{Prop}\label{Prop:R relates thetas}
Let $R$ be the cohomological action on $K(X)$ given by
$$
R(E)=E+2\v^2 \langle \v(E),\u \rangle U-2\langle \v(E),\v \rangle U-2\langle \v(E),\u \rangle V,\;
E \in K(X),
$$ 
where $U, V \in K(X)$ satisfy $\v(U)=\u,\v(V)=\v$.
Then we have a commutative diagram modulo torsion.
\begin{equation}\label{eqn:ReflectionRelatesThetas}
\begin{CD}
K(X)_\v @>{R}>> K(X)_\v \\
@V{\theta_{\v,\sigma_+}}VV  @VV{\theta_{\v,\sigma_-}}V\\ 
\Pic(M_{\sigma_+}(\v,L)) @<<{\varphi^*}< \Pic(M_{\sigma_-}(\v,L))\\
\end{CD}
\end{equation}
\end{Prop}

\begin{proof}
Consider a quasi-universal family $\EE\in\Db(X\times M_{\sigma_\pm}(\v,L))$ for the moduli space $M_{\sigma_\pm}(\v,L)$ and the map $\varpi\times 1$ fitting into the commutative diagram: 
\[\xymatrix{
\widetilde{X}\times M_{\sigma_\pm}(\v,L) \ar[r]^-{\varpi\times 1} \ar[rd]^{}_{p_{M_{\sigma_\pm}(\v,L)}} &
X\times M_{\sigma_\pm}(\v,L) \ar[d]^{p_{M_{\sigma_\pm}(\v,L)}}\\
&M_{\sigma_\pm}(\v,L)}\]
Then by \eqref{eq:psi}, we may consider $(\varpi\times 1)^*\EE$ to be a family of objects in $M_{\sigma_\pm''}(\v')$ parameterized by $M_{\sigma_\pm}(\v,L)$.  The morphism $\psi_\pm$ is the classifying map associated to the family $(\varpi\times 1)^*\EE$, so by \cite[Theorem 8.1.5]{HL10} we get $$\theta_{(\varpi\times 1)^*\EE}=\psi_\pm^*\circ\theta_{\v',\sigma_\pm''},$$ where $\theta_{(\varpi\times 1)^*\EE}:K(\widetilde{X})_{\v'}\to\Pic(M_{\sigma_\pm}(\v,L))$ is the Donaldson-Mukai homomorphism associated to the family $(\varpi\times 1)^*\EE$.  Precomposing with the pull-back from $K(X)_\v$ gives the following commutative diagram:
\begin{equation}\label{eqn:PullBackDonaldsonMukai}
\xymatrix{
K(X)_\v \ar[r]^{\varpi^*} \ar[d]^{}_{2\theta_{\v,\sigma_{\pm}}}& K(\widetilde{X})_{\v'}\ar[ld]_{\theta_{(\varpi\times 1)^*\EE}}\ar[d]^{\theta_{\v',\sigma_{\pm}''}}\\
\Pic(M_{\sigma_\pm}(\v,L))& \ar[l]^{\psi_\pm^*} \Pic(M_{\sigma_{\pm}''}(\v'))
}
\end{equation}
Indeed, for $E\in K(X)_\v$, we have 
\begin{align*}
    \theta_{(\varpi\times 1)^*\EE}(\varpi^*E)&=\det(p_{M_{\sigma_\pm}(\v,L)!}((\varpi\times 1)^*\EE\otimes p_{\widetilde{X}}^*(\varpi^*E)))\\
    &=\det(p_{M_{\sigma_\pm}(\v,L)!}((\varpi\times 1)^*\EE\otimes (\varpi\times 1)^*(p_X^*E)))\\
    &=\det(p_{M_{\sigma_\pm}(\v,L)!}((\varpi\times 1)^*(\EE\otimes p_X^*E)))\\
    &=\det(p_{M_{\sigma_\pm}(\v,L)!}((\varpi\times 1)_*(\varpi\times 1)^*(\EE\otimes p_X^*E)))\\
    &=\det(p_{M_{\sigma_\pm}(\v,L)!}((\EE\otimes p_X^*(E\otimes(\OO_X\oplus\OO_X(K_X))))))=\theta_{\v,\sigma_\pm}(E)+\theta_{\v,\sigma_\pm}(E(K_X)),
\end{align*}
so modulo torsion, $\theta_{(\varpi\times 1)^*\EE}(\varpi^*E)=2\theta_{\v,\sigma_\pm}(E)$, as claimed.

Using $\varpi^* \circ R=-\Xi \circ \varpi^*$ and $\theta_{\v',\sigma''_+}=
-\Xi^* \circ \theta_{\v',\sigma''_-} \circ \Xi$, where we have abused notation by using $\Xi$ to denote both the autoequivalence and the induced isomorphism $\Xi:M_{\sigma_+''}(\v')\to M_{\sigma_-''}(\v')$, the proposition then follows from \eqref{eq:varphi=Xi}.
\end{proof}
\begin{Rem}\label{Rem:ReflectionBMDivisors}
\begin{enumerate}
\item
Since the claim is independent of the choice of complex structure,
we can reduce to the case where $\Pic(\widetilde{X})=\varpi^*\Pic(X)$.
Then the claim is a consequence of Proposition \ref{prop:Enriques-refl}. 
\item
Since $\varpi^*(\xi_\tau)=2\xi_{\varpi^*(\tau)}$, it follows that  
$$\ell_{\CC_\pm}(\tau)=\theta_{\v,\sigma_\pm}(\xi_\tau)=
\psi_\pm^*(\theta_{\v',\sigma_\pm''}(\xi_{\varpi^*(\tau)}))=\psi_\pm^*\ell_{\CC_\pm''}(\tau).$$  This gives another proof of \cite[Proposition 10.2]{Nue14b}.
\end{enumerate}
\end{Rem}

\subsection{The case of even rank}\label{subsec:iLGU even rank}We show in this subsection that the technique of \cref{subsec:iLGU odd rank} also generalizes to Mukai vectors $\v$ of even rank.  

For the rest of the section, we assume that $\rk \v$ is even.  Let $\WW\subset\Stabd(X)$ be a wall defined by a primitive isotropic $\u$ with $\langle\v,\u\rangle=1=\ell(\u)$ and $\sigma_0$ a generic stability condition in $\WW$.  We denote by $\sigma_0'=\varpi^*(\sigma_0)$, and we again denote by $\v'$ and $\u'$ the pull-backs $\varpi^*(\v)$ and $\varpi^*(\u)$, respectively.  Similarly, let $\WW'\subset \Stabd(\widetilde{X})$ be the wall for $\v'$ defined by $\u'$.  Take $\sigma'' \in \WW'$ in a neighborhood of $\sigma_0'$ such that $\WW'$ is the unique wall for $\v'$ in a small neighborhood $U_{\sigma''}$ of $\sigma''$ and such that $\sigma''$ is generic with respect to $\u'$.
We also take $\sigma_\pm'' \in U_{\sigma''}$ in opposite and adjacent chambers for $\v'$, separated by the single wall $\WW'$. 
We may assume that $\varpi^*(\sigma_+)$ and $\varpi^*(\sigma_-)$ belong to the same chambers as $\sigma_+''$ and $\sigma_-''$, respectively.

We observe first that we still have $\ell(\v)=1$ so that $\varpi^*\v$ is primitive.  Indeed, as in \cref{Rem:Even and Odd pairings}, if $\ell(\v)$ were $2$, then $2\mid\rk(\u)$ would imply that $\langle\v,\u\rangle$ were even, which is absurd.  It therefore follows from \cref{Lem:pullback remains stable} that we have morphisms $\psi_\pm:M_{\sigma_\pm}(\v,L)_{sm} \to M_{\varpi^*(\sigma_\pm)}^s(\v') \into M_{\sigma_\pm''}^s(\v')$ which are \'{e}tale double covers onto their images.  Moreover, by \cref{Lem:isotropic divisorial l=1 1} we have a birational map $\varphi:M_{\sigma_+}(\v,L)\dashrightarrow M_{\sigma_-}(\v,L)$ which is the identity on $M_{\sigma_0}^s(\v,L)$.  As $M_{\sigma_\pm}(\v,L)$ has only terminal l.c.i. singularities and numerically trivial canonical divisor by \cite[Theorem 8.2]{Nue14b}, it follows that $\varphi$ is an isomorphism in codimension one by \cite[Corollary 3.54]{KM98}.  This defines an isomorphism $\varphi^*:\Cl(M_{\sigma_-}(\v,L))\isomor\Cl(M_{\sigma_+}(\v,L))$, where $\Cl(-)$ denotes the group of Weil divisors moduli linear equivalence.  If $\v^2\geq 8$, then $\codim(\Sing(M_{\sigma_\pm}(\v,L))\geq 4$ so that $M_{\sigma_\pm}(\v,L)$ is locally factorial by \cite[Exp. XI, Cor. 3.14]{SGA2} and thus $\Cl(M_{\sigma_\pm}(\v,L))=\Pic(M_{\sigma_\pm}(\v,L))$. We will show in the next result that, even for $2<\v^2<8$, $\varphi$ extends to an isomorphism so that we may nevertheless identify $\Pic(M_{\sigma_\pm}(\v,L))$.
\begin{Prop}\label{Prop:phi isomorphism even rank}
If $\v^2>2$, then the birational map $\varphi:M_{\sigma_+}(\v,L)\dashrightarrow M_{\sigma_-}(\v,L)$ extends to an isomorphism fitting into a commutative diagram modulo torsion.
\begin{equation}\label{eqn:ReflectionRelatesThetasEven}
\begin{CD}
K(X)_\v @>{R}>> K(X)_\v \\
@V{\theta_{\v,\sigma_+}}VV  @VV{\theta_{\v,\sigma_-}}V\\ 
\Pic(M_{\sigma_+}(\v,L)) @<<{\varphi^*}< \Pic(M_{\sigma_-}(\v,L))\\
\end{CD}
\end{equation}
\end{Prop}
\begin{proof}
As in the proof of \cref{Prop:phi isomorphism odd rank}, we have an autoequivalence $\Xi$ that induces an isomorphism $M_{\sigma_+''}(\v')\to M_{\sigma_-''}(\v')$ which is the identity on $M_{\sigma''}^s(\v')$.  Indeed, the same proof goes through since we still have $\HH_{\WW'}=\Z\v'+\Z\u'$ by \cref{Lem:u and v}.

Restricting the birational map $\varphi$ to the smooth locus $M_{\sigma_+}(\v,L)_{sm}$, we get a commutative diagram
\[
\xymatrix{M_{\sigma_+}(\v,L)_{sm}\ar@{-->}[r]^{\varphi}\ar[d]_{\psi_+}&M_{\sigma_-}(\v,L)_{sm}\ar[d]^{\psi_-}\\
M_{\sigma_+''}(\v')\ar[r]^\Xi & M_{\sigma_-''}(\v')}
\] with $\varphi$ the identity on $M_{\sigma_0}^s(\v,L)_{sm}$.  Moreover, we have two homomorphisms 
\begin{equation*}
\begin{split}
    \varphi^*\circ\theta_{\v,\sigma_-}\circ R:K(X)_\v\to \Cl(M_{\sigma_+}(\v,L))\\
    \theta_{\v,\sigma_+}:K(X)_\v\to\Pic(M_{\sigma_+}(\v,L)),
\end{split}
\end{equation*}
where $R$ is defined as in \cref{Prop:R relates thetas} and we have abused notation by using $\varphi^*$ for the restriction of $$\varphi^*:\Cl(M_{\sigma_-}(\v,L))\to\Cl(M_{\sigma_+}(\v,L))$$ to the subgroup $\Pic(M_{\sigma_-}(\v,L))\subset\Cl(M_{\sigma_-}(\v,L))$.  For any $x\in K(X)_\v$, \cref{Prop:R relates thetas}, which had nothing to do with $\rk\v$ being odd, implies that $\theta_{\v,\sigma_+}(x)$ and $\varphi^*\circ\theta_{\v,\sigma_-}\circ R(x)$ agree away from $\Sing(M_{\sigma_+}(\v,L))$, which has codimension at least two since $\v^2>2$.  It follows that $\varphi^*\circ\theta_{\v,\sigma_-}\circ R=\theta_{\v,\sigma_+}$, which gives the commutative diagram \eqref{eqn:ReflectionRelatesThetasEven}.

Finally, we prove that $\varphi$ extends to an isomorphism.  As $\xi_{\sigma_0}\in(\v^\perp\cap\u^\perp)_\R$, $R(\xi_{\sigma_0})=\xi_{\sigma_0}$ and thus $\varphi^*\ell_{\CC^+}(\sigma_0)=\ell_{\CC^-}(\sigma_0)$.  Moreover, letting $\d=\v-\v^2\u\in\HH_{\WW}\cap\v^\perp$, it is easy to see that $R=R_\d$ is just the reflection of $\v^\perp$ in the hyperplane $\d^\perp$ defined by $R_\d(\z)=\z-2\frac{\langle\z,\d\rangle}{\d^2}\d$ (see \cref{Sec:Picard groups} for more on this).  Perturbing $\sigma_\pm$ and $\sigma_0$ slightly, we may assume that we can write $\xi_{\sigma_\pm}=\pm x\d+\xi_{\sigma_0}$ for $x\in\R$.  
In particular, $R(\xi_{\sigma_+})=\xi_{\sigma_-}$.  But then $$\ell_{\CC^+}(\sigma_+)=\varphi^*(\theta_{\v,\sigma_-}(R(\xi_{\sigma_+})))=\varphi^*(\theta_{\v,\sigma_-}(\xi_{\sigma_-}))=\varphi^*\ell_{\CC^-}(\sigma_-).$$  As both $\ell_{\CC^\pm}(\sigma_\pm)$ are ample, $\varphi$ extends to an isomorphism by \cite[Exercise 5.6]{KSC}, as required.
\end{proof}
\section{Main theorems}\label{sec:Main Theorems}
\begin{proof}[Proof of Theorem \ref{Thm:MainTheorem1}]
We proved part \ref{enum:MT1-two moduli are birational} at the end of \cref{Sec:FloppingWalls}, so we focus on part \ref{enum:MT1-birational map given by FM transform}.  Nevertheless, we recall part of the setup form the proof of part \ref{enum:MT1-two moduli are birational}.  We may connect $\sigma$ and $\tau$ by a path which intersects walls in points that lie on no other walls.  As the set of walls is locally finite, the path will intersect only finitely many walls, and thus for the purpose of proving the theorem, it suffices to consider one wall $\WW$, a generic stability condition $\sigma_0\in\WW$, and nearby stability conditions $\sigma_\pm$.  

Suppose that $\langle \v,\w\rangle<0$ for an effective spherical or exceptional class $\w\in\HH_\WW$.  Then suppose that $\v=\v_n\in\CC_n$ with $\v_0$ the minimal Mukai vector in the $G_\HH$-orbit of $\v$.  Letting $\Phi^\pm$ be the sequence of spherical or exceptional twists giving the isomorphism $M_{\sigma_\pm}(\v)\to M_{\sigma_\pm}(\v_0)$ if $n$ is even (resp., $M_{\sigma\pm}(\v)\to M_{\sigma_\mp}(\v_0)$ if $n$ is odd), we see that it suffices to prove the theorem in the case that $\v$ is minimal.  Indeed, if $\Phi$ is an autoequivalence inducing a birational map $M_{\sigma_+}(\v_0)\dashrightarrow M_{\sigma_-}(\v_0)$ as in the theorem, then $(\Phi^-)^{-1}\circ\Phi\circ\Phi^+$ (resp., $(\Phi^-)^{-1}\circ\Phi^{-1}\circ\Phi^+$) proves the theorem for $\v$ if $n$ is even  (resp., if $n$ is odd).

Thus we may assume that $\langle \v,\w\rangle\geq 0$ for all effective spherical and exceptional classes $\w\in\HH_\WW$, and we break up the proof into cases.

Suppose first that $\langle\v,\w\rangle=0$ for some effective spherical or exceptional class $\w\in\HH_\WW$.  But then the minimality of $\v$ forces $\w=\w_0=\v(T_0^+)$ or $\w_1=\v(T_1^+)$ as in \cref{Prop:NonMinimalIsomorphism,Prop:CompositionSphericalExceptional} and the discussions following them.  But we showed there that the spherical/exceptional twists $R_{T_0^+}$ or $R_{T_1^+}$ induce an isomorphism $M_{\sigma_+}(\v)\to M_{\sigma_-}(\v)$.  So we may take $U$ to be the entire moduli space $M_{\sigma_+}(\v)$.  Notice that this discussion covers walls inducing $\P^1$-fibrations, walls of Brill-Noether type, as well as flopping walls induced by an exceptional class $\w\in\HH_\WW$ such that $\langle\v,\w\rangle=0$ when $\v^2\geq 3$.

Otherwise $\langle\v,\w\rangle>0$ for all effective spherical or exceptional classes $\w\in\HH_\WW$.  We first consider the case that $\WW$ is a flopping wall or a fake wall.  Then $\WW$ being a flopping or fake wall implies that $\codim(M_{\sigma_+}(\v)\backslash M_{\sigma_0}^s(\v))\geq 2$ by \cref{Lem:non-isotropic no totally semistable wall,Prop:LGU walls of low codimension,Prop: 1-1 case totally semistable and codim 1}.  In this case we may take $U$ to be the open subset of $\sigma_0$-stable objects, so there is nothing to prove (i.e. we just take $\Phi=\Id$).

Now we consider the case that $\WW$ induces a divisorial contraction and $\langle \v,\w\rangle>0$ for all effective spherical/exceptional classes $\w\in\HH_\WW$, i.e. divisorial contractions coming from Hilbert-Chow, LGU, or iLGU walls.  We will show that we may take $U$ to be the entire moduli space $M_{\sigma_+}(\v)$ in the case of Hilbert-Chow and LGU walls.  That is, we show that there is an autoequivalence that induces an isomorphism $M_{\sigma_+}(\v)\isomor M_{\sigma_-}(\v)$.  We show that the same is true for an iLGU wall if $\Pic(\widetilde{X})=\varpi^*\Pic(X)$, while in general we can still take $U$ to be an open subset of codimension one, proving the theorem. 

\textbf{Hilbert-Chow:} Here we assume that $\HH_\WW$ contains an isotropic vector $\u$ such that $\langle \v,\u\rangle=1$ and $\ell(\u)=2$.  Then by Proposition \ref{Prop:Uhlenbeck morphism} and Lemma \ref{Lem: Hilbert-Chow} we may, up to the shift by 1, identify $M_{\sigma_+}(\v)$ with the $(\beta,\omega)$-Gieseker moduli space $M^\beta_{\omega}(-\v)$ of stable sheaves of rank one.  Up to tensoring with a line bundle, we may assume that $M_{\sigma_+}(\v)$ parametrizes the shifts $I_Z[1]$ of ideal sheaves of 0-dimensional subschemes $Z\in\Hilb^n(X)$, where $n=\frac{\v^2+1}{2}$.  Moreover, $\sigma_-$-stable objects are precisely $I_Z^{\vee}[1]$.  But then $\Phi(\blank):=(\blank)^{\vee}[2]$ provides the required autoequivalence.

\textbf{LGU:} As in \cref{Prop:Uhlenbeck morphism}, shifting by one identifies $M_{\sigma_+}(\v)$ with the moduli space $M_{\omega}(2,c,s)$ of $\omega$-Gieseker stable sheaves $F$ with $\v(F)=(2,c,s)=-\v$.  Choosing $L\in\Pic(X)$ with $[L\mod K_X]=c$, we get that $\Phi(\blank):=(\blank)^{\vee}\otimes \OO(L)[2]$ is the required autoequivalence, as $\Phi(F[1])=F^{\vee}\otimes \OO(L)[1]$ is an object of $M_{\sigma_-}(\v)$, for any $F\in M_{\omega}(2,c,s)$.

\textbf{iLGU:} If we take $U$ to be the open subset $M_{\sigma_0}^s(\v)\subset M_{\sigma_+}(\v)$, then $M_{\sigma_+}(\v)\backslash U$ has codimension one, and we may take $\Phi=\Id$.  We can say more if $\Pic(\widetilde{X})=\varpi^*\Pic(X)$.  Indeed, then we may take $U=M_{\sigma_+}(\v)$ and $\Phi=D_X\circ\Phi_{X\to X}^{\EE^\vee}$ as in \cref{prop:LGUK3}.
\end{proof}

Now let us prove \cref{Thm:application1} as an application of \cref{Thm:MainTheorem1}.
\begin{proof}[Proof of \cref{Thm:application1}]
By the proof of \cite{Yos03}, there is an (anti)-autoequivalence
$\Phi$ such that $\rk \Phi(\v) =1$.
Applying Theorem \ref{Thm:MainTheorem1}, we get the first claim.
For the second claim, we use \cite[sect. 1]{OS11}.
\end{proof}
Similarly, \cref{Thm:application2} follows from \cref{Thm:MainTheorem1}.
\begin{proof}[Proof of \cref{Thm:application2}]
As in \cite[Section 4]{Yos16a} (for generic $X$) or the arguments in \cref{Lem:e-poly,Lem:rank2}, there is an autoequivalence $\Phi$ such that $\Phi(\v)=\w=(0,C,\chi)$ where $C\to\P^1$ is a double cover of $\P^1$ and $\chi\neq 0$.  Then by Theorem \ref{Thm:MainTheorem1},
we get the first claim.
If $\Pic(\widetilde{X})=\varpi^*\Pic(X)$, then by \cite{Yos16a}, \cite[Assumption 2.16]{Sac12} holds
for $M_H(\w,L')$. Hence by 
\cite[Thm. 3.1, Thm. 4.4]{Sac12}, we get the second claim.
\end{proof}

\section{Picard groups of moduli spaces}\label{Sec:Picard groups}

In this section we prove \cref{Cor:Picard}.  Before doing so, however, we prove that the maps $\ell_\CC:\CC\to\Num(M_{\CC}(\v,L))$ fit together to form one coherent map $$\Stabd(X)\to\Num(M_{\sigma_+}(\v,L))$$ whose image is contained in $\Mov(M_{\sigma_+}(\v,L))$.
\subsection{How the Bayer-Macr\`{i} maps fit together}Recall from \cref{eqn:def of xi,eqn:def of ell(sigma)} that $\ell_\CC$ can be written as the composition 
\begin{equation}\label{eqn:ell as composition}
\CC\into\Stabd(X)\mor[\ZZ]\Hal(X,\C)\mor[I]\v^\perp\mor[\theta_{\v,\sigma_\pm}]\Num(M_{\sigma_+}(\v,L)),
\end{equation}
where $\ZZ(\sigma)=\mho_\sigma$, $I(\mho_\sigma)=\xi_\sigma$, and $\theta_{\v,\sigma_\pm}$ are cohomological Donaldson-Mukai maps associated to the universal families for $M_{\sigma_\pm}(\v,L)$ as in \cref{eqn:DonaldsonMukai}.  To show that these fit together it suffices to study how the $\ell_\CC$ are related across a single wall crossing, and  from \eqref{eqn:ell as composition}, it therefore suffices to study the relationship between $\theta_{\v,\sigma_+}$ and $\theta_{\v,\sigma_-}$. 

So, for a given Mukai vector $\v\in\Hal(X,\Z)$ with $\v^2>0$, consider two adjacent chambers $\CC^+$ and $\CC^-$ separated by a wall $\WW$, and, as always, choose a generic $\sigma_0\in\WW$ and nearby stability conditions $\sigma_\pm\in\CC^\pm$.  The signature of $\HH_\WW$ and its saturatedness force $\v^{\perp} \cap \HH_{\WW}=\Z\d$ for some integral $\d$.  We have already seen in \cref{Prop:OrthgonalIsomorphism 2,Prop:CompositionSphericalExceptional} that in some cases $\d=\w_n$ is a spherical or exceptional class, in which case there is an isomorphism $M_{\sigma_-}(\v,L)\isomor M_{\sigma_+}(\v,L)$ induced by $R_{T_n^+}$, where $T_n^+$ is the $\sigma_+$-stable object of class $\w_n$.  When $\v$ is primitive and $\WW$ is defined by an isotropic vector $\u$ with 
\begin{enumerate}
    \item $\langle \v,\u \rangle=1,2$ and $\ell(\u)=2$, or 
    \item $\langle \v,\u \rangle=\ell(\u)=1$,
\end{enumerate}
then we claim that $\d=\v-\frac{\v^2}{\langle \v,\u \rangle}\u$.  We certainly have $\d':=\v-\frac{\v^2}{\langle\v,\u\rangle}\u\in\left(\v^\perp\cap\HH_\WW\right)_\Q$, and when $\langle\v,\u\rangle=1$, it is clear that $\d'\in\v^\perp\cap\HH_\WW$ so that $\d'=m\d$ for some $m\in\Z$.  But then $1=\langle\d',\u\rangle$ forces $m=1$.  Otherwise, $\langle\v,\u\rangle=2$ and $\ell(\u)=2$, forcing $\rk(\v)$, and thus $\v^2$, to be even, so $\d'\in\v^\perp\cap\HH_\WW$ with $\d'=m\d$.  If $m\neq 1$, then $\langle\d',\u\rangle=2$ forces $m=2$, and, as $\d'^2=-\v^2$, we see that $4\mid\v^2$, so $2 \mid \v$, contradicting the assumption that $\v$ was primitive. Therefore $m=1$, and $\d=\v-\frac{\v^2}{\langle\v,\u\rangle}\u$, as claimed.  In any case, we define $R_{\d}$ to be the reflection of $\v^\perp$ in the hyperplane $\d^\perp$.  That is, $R_\d$ is defined by the formula $$R_{\d}(\z)=\z-2 \frac{\langle \z,\d \rangle}{\d^2}\d=\z-2\frac{\langle \z,\u \rangle}{\langle \v,\u \rangle}\d=\z-2\frac{\langle\z,\u\rangle}{\langle\v,\u\rangle}\v+2\v^2\frac{\langle\z,\u\rangle}{\langle\v,\u\rangle^2}\u,\;\;\z\in\v^\perp.$$  Observing that $R_{m\d}(\z)=R_\d(\z)$ for any $m\in\Z$, we see that even if $\d=\frac{1}{2}(\v-\frac{\v^2}{\langle\v,\u\rangle}\u)$, when $\langle\v,\u\rangle=2=\ell(\u)$ and $\v$ is not primitive, the reflection remains the same.

We will prove in the next result that in all of the cases mentioned above $\theta_{\v,\sigma_\pm}$ differ by precomposing with the reflection $R_\d$.

\begin{Prop}\label{Prop:Pic-relation}
Let $\WW$ be a wall for $\v\in\Hal(X,\Z)$ with $\v^2>0$, and for generic $\sigma_0\in\WW$ let $\sigma_\pm$ be sufficiently close stability conditions in the opposite and adjacent chambers separated by $\WW$.  Then we may identify $\Pic(M_{\sigma_+}(\v,L))$ with 
$\Pic(M_{\sigma_-}(\v,L))$.  Moreover, if
\begin{enumerate}
    \item $\langle \v,\w \rangle=0$ for a spherical or exceptional class $\w\in C_\WW\cap\HH_\WW$, or
    \item $\langle \v,\u \rangle=\ell(\u)$ for an isotropic vector $\u\in\HH_\WW$, or
    \item $\langle \v,\u \rangle=1$ for an isotropic vector $\u\in\HH_\WW$ with $\ell(\u)=2$,
\end{enumerate}
then $\theta_{\v,\sigma_+}=\theta_{\v,\sigma_-} \circ R_\d$,
where $\v^\perp\cap\HH_\WW=\Z\d$ and $R_\d$ is the reflection of $\v^\perp$ in the hyperplane $\d^\perp$.  Otherwise, we have
$\theta_{\v,\sigma_+}=\theta_{\v,\sigma_-}$.
\end{Prop}
\begin{proof}
We claim first that we may identify $\Pic(M_{\sigma_+}(\v,L))$ with $\Pic(M_{\sigma_-}(\v,L))$.  Indeed, by \cref{classification of walls,Prop:NonMinimalIsomorphism,Prop:OrthgonalIsomorphism 2,Prop:CompositionSphericalExceptional}, $M_{\sigma_+}(\v,L)\cong M_{\sigma_-}(\v,L)$ if $\langle\v,\w\rangle=0$ for a spherical or exceptional $\w\in C_\WW\cap\HH_\WW$; in particular, the two moduli spaces are isomorphic when $\WW$ is a totally semistable wall of types (TSS3) and (TSS4) and when $\WW$ is a Brill-Noether wall.  The two moduli spaces are also isomorphic when $\WW$ is a wall of Hilbert-Chow or Li-Gieseker-Uhlenbeck type, as we showed in the proof of part \ref{enum:MT1-birational map given by FM transform} of \cref{Thm:MainTheorem1}.  For walls of iLGU type, we proved that $M_{\sigma_+}(\v,L)\cong M_{\sigma_-}(\v,L)$ in \cref{Prop:phi isomorphism odd rank,Prop:phi isomorphism even rank}.  According to \cref{classification of walls}, $M_{\sigma_+}(\v,L)$ and $M_{\sigma_-}(\v,L)$ are isomorphic away from a locus of codimension at least two in all other cases.

Now let us move on to proving the claims about the compatibility of the $\theta$-maps.  As usual, let $\v_0$ be the minimal Mukai vector in the $G_\HH$-orbit of $\v$, and if $\v\in\CC_n$, then we consider the composition of spherical or exceptional twists $\Phi^{\pm}$ giving the isomorphism $M_{\sigma_\pm}(\v)\to M_{\sigma_\pm}(\v_0)$ (resp., $M_{\sigma\pm}(\v)\to M_{\sigma_\mp}(\v_0)$) if $n$ is even (resp., if $n$ is odd). 

Suppose first that $\codim(M_{\sigma_+}(\v_0,L_0)\backslash M_{\sigma_0}^s(\v_0,L_0))\geq 2$.  Then the two moduli spaces $M_{\sigma_\pm}(\v_0,L_0)$ share a common open subset on which the universal families agree.  As the complement of $M_{\sigma_0}^s(\v_0,L_0)$ has codimension at least two and the maps $\theta_{\v_0,\sigma_\pm}$ are determined by their restriction to curves in this open subset, we get $\theta_{\v_0,\sigma_+}=\theta_{\v_0,\sigma_-}$.  For $\v$, the autoequivalence giving the birational map $M_{\sigma_+}(\v,L)\dashrightarrow M_{\sigma_-}(\v,L)$ is induced by $(\Phi^-)^{-1}\circ\Phi^+$.  As the classes of the spherical/exceptional objects occuring in $\Phi^+$ and $\Phi^-$ are identical, this autoequivalence does not change the class of the universal family in the K-group.  Therefore, we again have $\theta_{\v,\sigma_+}=\theta_{\v,\sigma_-}$.

Now suppose that $\codim(M_{\sigma_+}(\v_0,L_0)\backslash M_{\sigma_0}^s(\v_0,L_0))=1$.  Then by \cref{Prop:LGU walls of low codimension,Prop: 1-1 case totally semistable and codim 1} and \cref{Lem:non-isotropic no totally semistable wall}, we must have either $\langle\v_0,\w\rangle=0$ for a spherical or exceptional class $\w\in C_\WW\cap\HH_\WW$ or $\langle\v_0,\u\rangle=\ell(\u)$ for a primitive isotropic $\u\in\HH_\WW$.  

In the first case, we have that $\w=\w_0$ or $\w_1$ and $R_{T_0^+}$ or $R_{T_1^+}$ induces an isomorphism $M_{\sigma_+}(\v_0,L_0)\to M_{\sigma_-}(\v_0,L_0)$ by \cref{Prop:OrthgonalIsomorphism 2,Prop:CompositionSphericalExceptional}.  Considering the classes of the universal families in the K-group we see that $\theta_{\v_0,\sigma_+}=\theta_{\v_0,\sigma_-}\circ R_\w$.  Similarly, we have seen that $(\Phi^{-})^{-1}\circ R_{T_i^+}\circ\Phi^+$ induces an isomorphism $M_{\sigma_+}(\v,L)\to M_{\sigma_-}(\v,L)$ for $i=0$ or 1, so taking the classes of the universal families in the K-group we see that \[\theta_{\v,\sigma_+}=\theta_{\v,\sigma_-}\circ ((\Phi^{-})_*^{-1}\circ R_{\w}\circ(\Phi^+)_*)=\theta_{\v,\sigma_-}\circ R_{(\Phi^+)_*^{-1}(\w)},\] because $(\Phi^+)_*=(\Phi^-)_*$, as we noted above, and $\Phi\circ R_T\circ\Phi^{-1}=R_{\Phi(T)}$ for any autoequivalence $\Phi$ and spherical/exceptional object $T$ \cite[Lemma 8.21]{Hu1}.  Setting $\w':=(\Phi^+)_*^{-1}(\w)$, we get $\theta_{\v,\sigma_+}=\theta_{\v,\sigma_-}\circ R_{\w'}$ for the spherical/exceptional $\w'\in C_\WW\cap\HH_\WW$ such that $\langle\v,\w'\rangle=0$, as claimed.

In the second case, \cite{BM14b} or the arguments in section 10
showed that $\theta_{\v_0,\sigma_+}=\theta_{\v_0,\sigma_-} \circ R_\d$.
As in the previous cases, we see that $$\theta_{\v,\sigma_+}=\theta_{\v,\sigma_-}\circ R_{(\Phi^+)_*^{-1}(\d)},$$ giving the claim for $\v$.

Finally, we must consider when $M_{\sigma_0}^s(\v_0,L)=\varnothing$.  By \cref{classification of walls}, we then have $\v=\v_0$, and either $\WW$ is a Hilbert-Chow wall or $\langle\v,\w\rangle=0$ for a spherical or exceptional $\w\in C_\WW\cap\HH_\WW$ and $\langle\v,\u\rangle=\ell(\u)$ for primitive isotropic $\u\in\HH_\WW$.  These latter cases have already been covered above, so it only remains to consider the case of a Hilbert-Chow wall.  We saw in the proof of part \ref{enum:MT1-birational map given by FM transform} of \cref{Thm:MainTheorem1} that $\Phi(\blank)=(\blank)^\vee[2]$ induces an isomorphism between $M_{\sigma_+}(\v,L)$ and $M_{\sigma_-}(\v,L)$, where we may assume that $\v=-(1,0,\tfrac{1}{2}-n)$, $\u=(0,0,1)$, and $L=\OO_X$.  Then $\v^\perp$ is spanned by $-\d=(1,0,n-\frac{1}{2})$ and classes of the form $(0,c,0)$ for $c\in\Num(X)$.  As the universal family $\EE_-$ of $M_{\sigma_-}(\v,L)$ is given by $\EE_+^\vee[2]$, where $\EE_+$ is the universal family of $M_{\sigma_+}(\v,L)$, we have by Grothendieck duality: $$\theta_{\v,\sigma_+}(\z)=-\theta_{\v,\sigma_-}(\z^\vee).$$  In particular, $\theta_{\v,\sigma_+}((1,0,n-\frac{1}{2}))=-\theta_{\v,\sigma_-}((1,0,n-\frac{1}{2}))$ and $$\theta_{\v,\sigma_+}((0,c,0))=-\theta_{\v,\sigma_-}((0,-c,0))=\theta_{\v,\sigma_-}((0,c,0)).$$  On the other hand, $R_\d((1,0,n-\frac{1}{2}))=-(1,0,n-\frac{1}{2})$ and $R_\d((0,c,0))=(0,c,0)$, so we see that indeed $$\theta_{\v,\sigma_+}=\theta_{\v,\sigma_-}\circ R_\d,$$ as required.
\end{proof}

Since $\xi_{\sigma_0} \in (\HH_\WW^\perp)_\R$, so in particular $\xi_{\sigma_0}\in(\d^\perp\cap\v^\perp)_\R$, we have 
\begin{equation}\label{eqn:nef divisors agree at wall}
\theta_{\v,\sigma_+}(\xi_{\sigma_0})=
\theta_{\v,\sigma_-}(R_\d(\xi_{\sigma_0}))=\theta_{\v,\sigma_-}(\xi_{\sigma_0})
\end{equation}
by Proposition \ref{Prop:Pic-relation}.  It follows that the maps $\ell_{\CC^\pm}$ agree along $\WW$ and thus fit together to give a piece-wise analytic continuous map $$\ell:\Stabd(X)\to\Num(M_{\sigma_+}(\v,L)),$$ as claimed.  It is important to note that in any of the three cases enumerated in \cref{Prop:Pic-relation}, the image under $\ell$ of a path $\sigma:[-\tfrac{1}{2},\frac{1}{2}]\to\Stabd(X)$ with $\sigma(-\tfrac{1}{2})=\sigma_+$, $\sigma(0)=\sigma_0$, and $\sigma(\frac{1}{2})=\sigma_+$ is a path contained entirely in the nef cone of $M_{\sigma_+}(\v,L)$ because of the action of the reflection $R_\d$.  Indeed, as $\sigma(t)$ approaches $\sigma(0)=\sigma_0$, $\ell(\sigma(t))$ approaches $\ell(\sigma_0)$, but then $R_\d$ causes $\ell(\sigma(t))$, $t>0$, to bounce off the wall of $\Nef(M_{\sigma_+}(\v,L))$ containing $\ell(\sigma_0)$ and continue back into the interior.  

Note that while most of these bouncing walls correspond either to divisorial contractions or $\P^1$-fibrations, and thus extremal walls of $\Mov(M_{\sigma}(\v,L))$, the Enriques case differs from the K3 case in that there are bouncing walls that induce small contractions.  Indeed, in the presence of exceptional classes $\w\in C_\WW\cap\HH_\WW\cap\v^\perp$, $\WW$ is a bouncing wall inducing a small contraction if $\v^2\geq 3$ as in \cref{Ex:ConfusingSmallContraction}.

\begin{Rem}
Recall that the contraction morphisms $\pi^\pm:M_{\sigma_\pm}(\v,L)\to \overline{M}_\pm$ are defined by the linear systems
$|n \theta_{\v,\sigma_\pm}(\xi_{\sigma_0})|$ $(n \gg 0)$.  From \eqref{eqn:nef divisors agree at wall}, these  
define morphisms $\pi^\pm:M_{\sigma_{\pm}}(\v,L) \to \P$
such that $\im \pi^+=\im \pi^-$, and thus $\overline{M}_+=\overline{M}_-$, as the normalization of $\im\pi^+=\im\pi^-$.  

Let $M'_{\sigma_0}(\v)$ be the set of 
S-equivalence classes
of $\sigma_0$-semistable objects $E$ with $\v(E)=\v$.
Then $\im \pi^\pm$ is a subset of $M'_{\sigma_0}(\v)$, which is a proper subset in general. 
\end{Rem}

\subsection{Proof of \cref{Cor:Picard}}
Now we shall prove Corollary \ref{Cor:Picard}.
For the computation of the Picard groups in odd rank cases,
we can use deformation of Enriques surfaces, since
$\theta_{\v,\sigma}$ is well-defined for a relative moduli space over
a family of Enriques surfaces.
Thus in either case of \cref{Cor:Picard}, we may assume that $\varpi^*\Pic(X)=\Pic(\widetilde{X})$, where $\widetilde{X}$
is the covering K3 surface.
By \cref{Thm:application1,Thm:application2,Prop:Pic-relation}, 
it is sufficient to prove that $\theta_{\v,\sigma}$ is an isomorphism 
for a special pair of $\v$ and $\sigma$.  Specifically, we may assume that $\v=(1,0,\tfrac{1}{2}-n)$ or 
$\v=(2,D,a)$ with
$(D,\eta)=1$ for a divisor $\eta$ with $(\eta^2)=0$.
In the first case,
$\theta_{\v,\sigma}$ is an isomorphism if
$M_\sigma(\v,L)=\Hilb^n(X)$ $(n \geq 2)$ (see the proof of \cite[Corollary A.4]{Yos16a}).
So we shall treat the second case. 

Replacing $\v$ by $\v \exp(k \eta)$, we may assume that
$D^2=-2,-4$.
Since $X$ is unnodal,
we can take an ample divisor $H$ with $(D,H)=0$.
We set $n:=\frac{D^2}{2}+1-a$.
Then 
$$
\v=(2,D,\tfrac{D^2}{2}+1-n)=\v(\OO_X(D))+\v(I_Z),\;
I_Z \in \Hilb^n(X).
$$
We take $\beta_0 \in \Pic(X)_{\Q}$
with $(\beta_0,D)=\frac{D^2}{2}+n$.
Then $\chi(\OO_X(D-\beta_0))=\chi(I_{Z}(-\beta_0))$.
We take $\beta  \in \Pic(X)_{\Q}$ in a neighborhood of $\beta_0$
such that $\chi(\OO_X(D-\beta))<\chi(I_{Z}(-\beta))$.
We take a stability condition $\sigma$ such that $M_\sigma(\v,L)$ is the 
moduli space of $\beta$-twisted stable sheaves $M_H^\beta(\v,L)$.
For a non-trivial extension
$$
0 \to \OO_X(D) \to E \to I_{Z} \to 0,\; Z \in \Hilb^n(X),
$$
$E$ is a $\beta$-twisted stable sheaf.  Let us denote by $\ZZ$ the universal subscheme $\ZZ\subset X\times\Hilb^n(X)$ and by $p_X$, $p_{\Hilb^n(X)}$ the projections of $X\times\Hilb^n(X)$ onto its first and second factors, respectively.
Since $H^0(X,\OO_X(\pm D))=H^0(X,\OO_X(\pm D+K_X))=0$, we see that 
$\VV:=\lExt^1_{p_{\Hilb_X^n}}(I_{\ZZ},p_X^*\OO_X(D))$ 
is a locally free sheaf 
on $\Hilb^n(X)$ of rank $n-1-\tfrac{D^2}{2}$.
For $n>0$, let $P=\P(\VV^\vee)$ and let $\pi:P \to \Hilb^n(X)$ be the projective bundle
parameterizing non-trivial extensions
of $I_{Z}$ by $\OO_X(D)$ with universal family of extensions
\begin{equation}\label{eq:beta}
0 \to \OO_X(D) \boxtimes \OO_P(\lambda) \to \EE \to
\pi^*I_{\ZZ} \to 0 ,
\end{equation}
where
$\OO_P(\lambda)$ is the tautological line bundle on $P$.
Hence we have a morphism 
$\psi:P \to M_\sigma(\v,L)$.

\begin{Lem}
$\theta_{\v,\sigma}$ is an isomorphism if $\v^2 \geq 4$.
\end{Lem}

\begin{proof}
We note that $\v^2=-D^2+4(n-1)$.
Hence $\v^2 \geq 4$ if and only if 
$n \geq 2$, or $n=1$ and $D^2=-4$. 
We fix $\eta \in \NS(X)$ with $(\eta,D)=1$. 
Then we have a decomposition $\NS(X)=D^\perp \oplus \Z \eta$.
For $\xi \in \NS(X)$, let $u_\xi \in K(X)$ be an element of $K(X)$ such that 
$\v(u_\xi)=(0,\xi,0)$ and $c_1(u_\xi)=\xi$.
Any $u\in K(X)$ can be written $u=r \OO_X+s \C_x+u_\xi$, so from $\langle\v,\v(u_\eta)\rangle=1$ it follows that
\begin{equation}
u-\langle \v,\v(u) \rangle u_\eta=
r \OO_X+s \C_x+u_{\xi-(\xi,D)\eta}+
(r(2+\tfrac{D^2}{2}-n)+2s)u_\eta \in K_\v(X).
\end{equation} 
Hence any $u \in K(X)_\v$ is written as
\begin{equation}\label{eq:u}
 u=r \OO_X+s \C_x+u_{\xi}+
(r(2+\tfrac{D^2}{2}-n)+2s)u_\eta,\; r,s \in\Z, \xi \in D^\perp.
\end{equation}

For any $\alpha\in K(X)$, we set 
$\theta_0(\alpha):=\det p_{\Hilb^n(X) !}(I_{\ZZ} \otimes p_X^*(\alpha^{\vee}))$.
Then $\theta_0(\C_p)=\OO_{\Hilb^n(X)}$.
We have an injective homomorphism
$\NS(X) \to \NS(\Hilb^n(X)) (\cong \Pic(\Hilb^n(X)))$ 
by sending $\xi \in \NS(X)$ to $c_1(\theta_0(u_\xi))$.
We regard $\NS(X)$ as a subgroup of $\NS(\Hilb^n(X))$
by this homomorphism.
If $n \geq 2$, then
$\NS(\Hilb^n(X))=\Z \delta \oplus \NS(X)$ with 
$\delta:=c_1(\theta_0(\OO_X))$.
From the exact sequence \eqref{eq:beta} and \cite[Lemma 8.1.2]{HL10}, we see that writing any $u\in K(X)_\v$ as in \eqref{eq:u},  we have
\begin{equation}\label{eqn:DonaldsonMukai for P}
\begin{split}
c_1 (p_{P!}(\EE \otimes p_X^*(u^{\vee})))=&
-\langle \v(\OO_X(D)),\v(u) \rangle\lambda
+\pi^*c_1(\theta_0(u))\\
=&
r(\delta+(n-1)\lambda+(2+\tfrac{D^2}{2}-n)\eta)
-s(\lambda-2\eta)+\xi,
\end{split}
\end{equation}
where we have surpressed the $\pi^*$ in the second equality since $\pi^*:\NS(\Hilb^n(X))\to\NS(P)$ is injective and
$$
\NS(P)=\NS(X) \oplus \Z \delta \oplus \Z \lambda=
D^\perp \oplus \Z \eta \oplus \Z \delta \oplus \Z \lambda.
$$
By \cite[Theorem 8.1.5]{HL10}, the homomorphism $K(X)_\v\to\NS(P)$ in \eqref{eqn:DonaldsonMukai for P} is precisely $c_1$ composed with the pull-back along the classifying morphism $\psi$ of the universal Donaldson-Mukai map $\theta_{\v,\sigma}$, it follows from the formula in \eqref{eqn:DonaldsonMukai for P} that
$$
c_1 \circ \psi^* \circ \theta_{\v,\sigma}:K(X)_\v \to \Pic(M_H^\beta(\v,L)) \to \Pic(P) \to \NS(P)
$$
is injective (up to torsion) and its image is a direct summand of $\NS(P)$.

If $n=1$ and $D^2=-4$, then $\rk\VV=n+1=2$ and
$\NS(P)=\NS(X) \oplus \Z \lambda$.
Since  
$2+\tfrac{D^2}{2}-n=-1$, 
\begin{equation}
c_1(p_{P!}(\EE \otimes p_X^*(u^{\vee})))=
-r \eta-s(\lambda-2\eta)+\xi.
\end{equation}
Hence $c_1 \circ \psi^* \circ \theta_{\v,\sigma}$
is an isomorphism.

In either case, it follows that $\theta_{\v,\sigma} (K(X)_\v)$ is a direct summand of $\Pic(M_\sigma(\v,L))$.
Since the torsion submodule of $\Pic(M_\sigma(\v,L))$ is
isomorphic to 
$\Z /2 \Z$ and $\rk \Pic(M_\sigma(\v,L))=\rk K(X)_\v$
(\cite[Theorem 5.1]{Sac12}),
$\theta_{\v,\sigma}$ is an isomorphism.
\end{proof}

\section{Appendix: Some non-normal moduli spaces of sheaves}\label{App: exceptional case}
In this final section, we shall describe the two fundamental outlying examples of moduli spaces of sheaves (Bridgeland semistable objects) on Enriques surfaces.  Indeed, by \cref{prop:pss} for $\sigma\in\Stabd(X)$ generic with respect to $\v$ satisfying $\v^2>0$, $M_{\sigma}(\v)$ is normal unless 
\begin{enumerate}
    \item $\v=2\v_0$ with $\v_0^2=1$, or
    \item $\v^2=2$ and $X$ is nodal.
\end{enumerate} 
Moreover, these are precisely the exceptions to $M_\sigma^s(\v)$ having torsion canonical divisor and Gorenstein, terminal, l.c.i. singularities \cite[Theorem 8.2]{Nue14b}.  Similarly, we have seen in \cref{Thm:exist:nodal} when $\v$ is primitive $M_\sigma(\v,L)$ is irreducible if $X$ is unnodal and can only be reducible for nodal $X$ if $\v^2=2$.  

We will prove here that the failure of normality in these two cases is due to the presence of multiple irreducible components.  In the first case, when $\v=2\v_0$ with $\v_0^2=1$, we describe a conneected component of $M_\sigma(\v,L)$ for $L\equiv K_X\pmod 2$ consisting of two irreducible components.  In the second case, when $\v^2=2$, we describe all of $M_\sigma(\v,L)$ for $L\equiv D+\frac{r}{2}K_X\pmod 2$, which consists of precisely two irreducible components.  In both cases, we also describe the other component parametrizing objects with the other determinant.  

\subsection{Some components of $M_\sigma(\v,L)$ when $\v=2\v_0$ with $\v_0^2=1$}
By \cite[Theorem 4.6]{Yos03}, there exists a Fourier-Mukai transform $\Phi:\Db(X)\to\Db(X)$ such that $\Phi_*(\v_0)=(1,0,-\frac{1}{2})$.  By \cref{Thm:MainTheorem1}, it follows that $M_\sigma(\v)$ is birational to the moduli space $M_\omega(2,0,-1)$ of $\omega$-Gieseker semistable sheaves of Mukai vector $2(1,0,-\frac{1}{2})=(2,0,-1)$ for a generic ample divisor $\omega$.  In particular, to understand the number of irreducible components, we are reduced to analyzing the moduli space $M_\omega(2,0,-1)$.

\subsubsection{A connected component of $M_\omega((2,0,-1),K_X)$}

For the Mukai vector $\v=(2,0,-1)$, we shall describe a connected component
of $M_\omega(\v,K_X)$, which is a refinement of \cite[Remark 2.3]{Yos16a}.

Recall that $\varpi:\widetilde{X} \to X$ is the quotient map from the covering K3 surface $\widetilde{X}$ with covering involution $\iota$.
We set $E_0:=\OO_X \oplus \OO_X(K_X)$ and begin by concretely studying the singular locus of the stable locus of $M_\omega^s(\v,K_X)$.
\begin{Lem}\label{Lem:ext^2}
In the open subscheme $M_\omega^s(\v,K_X)$ of stable sheaves,
the singular locus is 
\begin{equation}\label{eq:sing}
\Set{ \varpi_*(I_W(D)) \ | \ I_W \in \Hilb^n(\widetilde{X}),\;\iota^*(D)=-D, n=(D^2)/2+2, I_W(D) \not \cong \iota^*(I_W(D)) },
\end{equation}
where $D=0$ or $(D^2)=-4$.
In particular, the singular locus consists of isolated singular points in 
the open subscheme of locally free sheaves, and 
an irreducible 4-dimensional subscheme. 
\end{Lem}

\begin{proof}
If $M_\omega^s(\v,K_X)$ is singular at $E$, then 
$E(K_X) \cong E$, which implies $E= \varpi_*(I_W(D))$, where 
$I_W$ is the ideal sheaf of a 0-dimensional subscheme $W$ of  $\widetilde{X}$
and $D$ is a divisor on $\widetilde{X}$ (cf. \cite[Lemma 2.13]{Yamada}). 
By the stability of $E$, $\iota^*(I_W(D)) \not \cong I_W(D)$.
Since $\varpi^*(\varpi_*(I_W(D)))=I_W(D) \oplus I_{\iota(W)}(\iota^*(D))$, it follows that
$D+\iota^*(D)=0$ and $\deg W=(D^2)/2+2$.
In particular, as $$(D,\varpi^*(\omega))=(\iota^*D,\iota^*\varpi^*(\omega))=(\iota^*D,\varpi^*(\omega)),$$ we get $(D,\varpi^*(\omega))=0$, so  either $D=0$ or neither $D$ nor $-D$ is effective.  As $\widetilde{X}$ is a K3 surface, we cannot have $(D^2)\geq-2$, since then $D$ or $-D$ would be effective, so we must in fact have $(D^2)\leq -4$.  From $\deg W\geq 0$, however, we get $(D^2) \geq -4$.  We conclude therefore that either $D=0$ or $(D^2)=-4$, giving the description in \eqref{eq:sing}. 

In the open subscheme of locally free sheaves,
$W=\emptyset$, so the singular locus there is isolated, consisting of $\varpi^*(\OO_{\widetilde{X}}(D))$ for the finitely many divisors $D$ such that $\iota^*D=-D$ and $(D^2)=-4$.
If $D=0$ and $\deg W=2$, then 
$\varpi_*(I_W)$ is a non-locally free sheaf
with $\varpi_*(I_W)^{\vee \vee} \cong E_0$,
which gives an irreducible component 
$$
\Set{ \varpi_*(I_W) \in M_\omega^s(\v,K_X) \ | \ I_W \in \Hilb^2(\widetilde{X}) }
$$
of the singular locus of dimension 4. 
\end{proof}
From the lemma it follows that the singularities of the open subscheme $M_\omega^{s,lf}(\v,K_X)$ parametrizing locally free sheaves in $M_\omega^s(\v,K_X)$ are isolated 5-dimensional hypersurface singularities.  Therefore, we get the following corollary:
\begin{Cor}
The open subscheme $M_\omega^{s,lf}(\v,K_X)$ parametrizing $\omega$-stable locally free sheaves of Mukai vector $\v$ and determinant $K_X$ is normal.
\end{Cor}

Now we turn to the 4-dimensional component of the singular locus.
We have an injective morphism $X \to \Hilb^2(\widetilde{X})$ by sending
$I_z$ $(z \in X)$ to $\varpi^*(I_z)=I_{\varpi^{-1}(z)}$.
The image is the $\iota$-invariant Hilbert scheme $\Hilb^2(\widetilde{X})^{\iota}$.
We shall identify $X$ with  $\Hilb^2(\widetilde{X})^{\iota}$.
We also have a morphism
$$
\begin{matrix}
\Hilb^2(\widetilde{X}) & \to & M_\omega(\v,K_X)\\
I_W & \mapsto & \varpi_*(I_W),
\end{matrix}
$$
 which is a double covering to
its image.
If $I_W \in \Hilb_{\widetilde{X}}^2 \setminus X$, then
$\varpi_*(I_W)$ is a stable non-locally free sheaf and
the fiber over $\varpi_*(I_W)$ is $\{ I_W, I_{\iota(W)} \}$.
If $I_W=I_{\varpi^{-1}(z)}$, then 
$$\varpi_*(I_W)=\varpi_*(\varpi^*(I_z))=I_z\otimes(\OO_X\oplus\OO_X(K_X))=I_z \otimes E_0,$$ which is a properly semi-stable
sheaf.

For $E=\varpi_*(I_W)$, we have $-\chi(E,E)/2=2$, and hence
by using
\cite[Fact 2.4]{Yamada} and the paragraph preceeding \cite[Lemma 2.13]{Yamada},
we see that around $E$, 
$M_\omega^s(\v,K_X)$ is analytic locally defined by 
a hypersurface  $F(t_1,t_2,...,t_6)=0$ in $(\C^6,0)$
such that 
$$
F(t_1,t_2,...,t_6)=\sum_{i=1}^n t_i^2+G(t_1,t_2,...,t_6),\;\;
n \geq 2,\; G(t_1,t_2,...,t_6) \in (t_1,t_2,...,t_6)^3.
$$
Since the singular locus is 4-dimensional at $E=\varpi_*(I_W)$, we must have $n=2$. 
 Therefore there are at most two irreducible components intersecting
along the 4-dimensional singular locus.  We shall prove that there are exactly two such irreducible components $M_0$ and $M_1$ and that their union is connected.

Let $\psi:M_\omega(\v,K_X) \to N$ be the contraction map to the 
Uhlenbeck compactification.
\begin{Lem}\label{Lem:M_0}
There is an irreducible component
$M_0$ of $M_\omega(\v,K_X)$ 
such that $M_0$ contains a $\mu$-stable locally free sheaf and
$\psi(M_0)$ contains $M_\omega(\v',K_X) \times S^2 X$,
where $\v'=\v(E_0)$.
\end{Lem}

\begin{proof}
Consider the stack $\MM_\omega(\v,K_X)^{\mu ss}$
of $\mu$-semistable sheaves $E$ with $\v(E)=\v$ and $\det(E)=K_X$, and let $\MM_0$ be the irreducible component containing all locally free sheaves fitting in the short exact sequence
\begin{equation}
0 \to \OO_X \to E \to I_Z (K_X) \to 0,
\end{equation}
where $I_Z \in \Hilb^2(X)$.  Then by the proof of \cite[Lemma 2.8]{Yos16a}, 
$$
\MM_0':=\Set{ E \in \MM_0 \ | \ \text{$E$ is a $\mu$-stable locally free sheaf}}
$$
is a nonempty open and dense substack of $\MM_0$.
Let $M_0$ be the irreducible component of $M_\omega(\v,K_X)$
containing the associated coarse moduli scheme of $\MM_0'$.
We shall prove that $\psi(M_0)$ contains $M_\omega(\v',K_X) \times S^2 X$,
where $\v'=\v(E_0)$.

Let $E\in\MM_0$ be a locally free sheaf fitting in
\begin{equation}
0 \to \OO_X \to E \to I_Z (K_X) \to 0,
\end{equation}
where $I_Z \in \Hilb^2(X)$.  Consider a point $F\in\MM_0'$ and a generic curve $T\subset\MM_0$ connecting $E$ and $F$.  This corresponds to a deformation $\EE$ over the curve $T$ such that $\EE_{t_0}=E$ and $\EE_{t_1}=F$ for points $t_0,t_1\in T$.  In particular, as $T$ is taken generic, we have $\EE_t$ is a $\mu$-stable locally free sheaf for $t\neq t_0$.  Note that $E$ is $\mu$-semistable but not Gieseker semistable, as $$\mu(\OO_X)=0=\mu(I_Z(K_X))\text{ and } \chi(\OO_X)=1>-1=\chi(I_Z(K_X)),$$ so $\EE$ does not induce a morphism $T\to M_0$.  By Langton's theorem \cite[Theorem 2.B.1]{HL10}, however, we obtain another family $\FF$ over $T$ of semistable sheaves $\FF_t$ such that $\FF_t=\EE_t$ for $t \ne t_0$.  This new family does induce a morphism $T \to M_0$.
Since $\psi$ extends to the family of $\mu$-semistable sheaves
$\EE_t$, 
\begin{equation}
\psi(\FF_{t_0})=\lim_{t \to t_0} \psi(\FF_t)=\psi(\EE_{t_0})=
(E_0, [Z]) \in M_\omega(\v',K_X) \times S^2 X,
\end{equation}
where 
$[Z]$ is the 0-cycle defined by $Z$.
Hence our claim holds.
\end{proof}

For the construction of the other irreducible component $M_1$,
we prepare the following lemma.
\begin{Lem}\label{Lem:nonlocally free extension is semistable}
For $Z \in \Hilb^2(X)$ and a non-trivial extension
\begin{equation}\label{eqn:nonlocally free extension}
0 \to I_Z \to E \to \OO_X(K_X) \to 0,
\end{equation}
$E\in M_\omega(\v,K_X)$.
\end{Lem}

\begin{proof}
Suppose not, and let $F$ be a saturated, destabilizing, stable subsheaf of $E$.  In particular, $\rk(F)=1$.  As $E$ is $\mu$-semistable, $\mu_\omega(F)\leq\mu_\omega(E)$, so we must have $\mu_\omega(F)=\mu_\omega(E)$ and $\chi(F)>\frac{\chi(E)}{2}=0$.  It follows from $\chi(I_Z)=-1$ that the composition $\phi:F\into E\onto \OO_X(K_X)$ must be nonzero so that from stability $\chi(F)\leq\chi(\OO_X(K_X))=1$, and thus $\chi(F)=\chi(\OO_X(K_X))$.  Then $\phi$ is an isomorphism, contradicting the non-triviality of the extension \eqref{eqn:nonlocally free extension}.  Hence $E$ is semistable, as claimed. 
\end{proof}
Observe that $E$ as in \eqref{eqn:nonlocally free extension} is not locally free and  dualizing \eqref{eqn:nonlocally free extension} twice we get another short exact sequence 
\begin{equation}\label{eqn:s.e.s. on double duals}
    0\to \OO_X\to E^{\vee\vee}\to\OO_X(K_X)\to 0.
\end{equation}
Moreover, we may put \cref{eqn:nonlocally free extension,eqn:s.e.s. on double duals} together into the following short exact sequence of complexes
$$\begin{CD}
0 @>>> I_Z @>>> E @>>> \OO_X(K_X) @>>> 0\\
& & @VVV  @VVV  @V\Id VV \\
0 @>>> \OO_X @>>> E^{\vee\vee} @>>> \OO_X(K_X)@>>> 0
\end{CD}$$
and apply the Snake lemma to see that $E$ fits into another short exact sequence \begin{equation}\label{eqn:alternative defining short exact sequence}0\to E\to E^{\vee\vee}\to \OO_Z\to 0.\end{equation} As $\Ext^1(\OO_X(K_X),\OO_X)=0$, we see that $E^{\vee\vee}=\OO_X\oplus\OO_X(K_X)$, so we may compose the natural embedding $\OO_X(K_X)\into\OO_X\oplus\OO_X(K_X)$ with the surjection $E^{\vee\vee}\onto \OO_Z$ to obtain a morphism $\OO_X(K_X)\to\OO_Z$.  We will show in the next lemma that we may determine when $E$ as in \eqref{eqn:nonlocally free extension} is in $M_\omega^s(\v,K_X)$ based on the surjectivity of the associated morphism $\OO_X(K_X)\to\OO_Z$.
\begin{Lem}\label{Lem:criterion for extension to be stable}
A sheaf $E$ fitting into the short exact sequence \eqref{eqn:nonlocally free extension} is stable if and only if the associated morphism $\OO_X(K_X)\to\OO_Z$ is surjective.
\end{Lem}
\begin{proof}
The sheaf $E$ is properly semistable if and only if it contains a torsion free subsheaf of Mukai vector $(1,0,-\frac{1}{2})$, that is, either the ideal sheaf of a point $I_z$ or its twist $I_z(K_X)$.  As $\Hom(I_z,I_Z)=0=\Hom(I_z,\OO_X(K_X))$, we see that $E$ is properly semistable if and only if $E$ contains $I_{z_1}(K_X)$ for some $z_1\in X$.  But this is equivalent to $E$ sitting in a short exact sequence \begin{equation}\label{eqn:strictly semistable s.e.s}0\to I_{z_1}(K_X)\to E\to I_{z_2}\to 0\end{equation} for some other point $z_2\in X$.  

Taking double duals of \eqref{eqn:strictly semistable s.e.s}, we again get a short exact sequence of complexes
$$
\begin{CD}
0@>>> I_{z_1}(K_X)@>>> E@>>>I_{z_2}@>>>0\\
&&@VVV @VVV @VVV\\
0@>>>\OO_X(K_X)@>>>E^{\vee\vee}@>>>\OO_X@>>>0,
\end{CD}
$$
so using the Snake lemma we see that, as claimed, the composition $\OO_X(K_X)\into E^{\vee\vee}\onto\OO_Z$ is not surjective, factoring instead through the proper subsheaf $\OO_{z_1}\subsetneq\OO_Z$.  Note that we also see that $Z=\{z_1,z_2\}$.

Conversely, if the composition $\OO_X(K_X)\into E^{\vee\vee}\onto\OO_Z$ is not surjective, then it factors through $\OO_{z_1}$ for some $z_1\in Z$, and we get a short exact sequence of complexes 
$$\begin{CD}
0@>>>\OO_X(K_X)@>>>E^{\vee\vee}@>>>\OO_X@>>>0\\
&&@VVV@VVV@VVV\\
0@>>>\OO_{z_1}@>>>\OO_Z@>>>\OO_{z_2}@>>>0
\end{CD}
$$ for $z_2\in Z$.  It follows that $E$ sits in a short exact sequence as in \eqref{eqn:strictly semistable s.e.s}.
\end{proof}

We have seen that any $E$ fitting into \eqref{eqn:nonlocally free extension} fits into \eqref{eqn:alternative defining short exact sequence}.  We show that the converse holds as well if $E$ is stable.
\begin{Lem}\label{Lem:alternative sequence is equivalent for stable E}
Let $E$ be a stable sheaf fitting into a short exact sequence as in \eqref{eqn:alternative defining short exact sequence}, $$0\to E\to \OO_X\oplus \OO_X(K_X)\to \OO_Z\to 0$$ for $Z\in\Hilb^2(X)$.  Then $E$ fits into a short exact sequence as in \eqref{eqn:nonlocally free extension}.
\end{Lem}
\begin{proof}
Indeed, if $E$ is a stable sheaf fitting into \eqref{eqn:alternative defining short exact sequence}, then the stability of $E$ and $\OO_X(K_X)$ forces the composition $E\into \OO_X\oplus\OO_X(K_X)\mor[p_2]\OO_X(K_X)$ is surjective, where $p_2$ is the projection onto the second factor.  Then similar arguments as above show that the kernel of this composition is $I_Z$, giving the lemma.
\end{proof}

Now we are able to define the second irreducible component $M_1$.  Since $\hom(\OO_X(K_X),I_Z)=0$ and $\ext^2(\OO_X(K_X),I_Z)=\hom(I_Z,\OO_X)=1$, we see that $$\ext^1(\OO_X(K_X),I_Z)=\langle(1,0,\tfrac{1}{2}),(1,0,-\tfrac{3}{2})\rangle+1=2,$$ so we have a family
of semistable non-locally free sheaves 
\begin{equation}
0 \to I_{\ZZ} \otimes \OO_{\P}(1) \to \EE \to \OO_{\P \times X}(K_X) \to 0
\end{equation}
on the projective bundle
$$
q:\P:=\P(\Ext^1_p(\OO_{\Hilb^2(X) \times X}(K_X),I_{\ZZ})) \to \Hilb^2(X),
$$
where $\ZZ$ is the universal family on $\Hilb^2(X) \times X$
and $p:\Hilb^2(X) \times X \to \Hilb^2(X)$ is the projection.
By \cref{Lem:nonlocally free extension is semistable}, $\EE$ gives rise to a morphism 
\begin{equation}
\begin{matrix}
g:& \P & \to & M_\omega(\v,K_X)\\
& z & \mapsto & \EE_z. \\
\end{matrix}
\end{equation}
We set $M_1:=g(\P)$, and prove in the next result that $M_1$ is an irreducible component of $M_\omega(\v,K_X)$.
\begin{Lem}\label{Lem:M_1}
$M_1$ is an irreducible component of $M_\omega(\v,K_X)$ and satisfies $M_1=\psi^{-1}(M_\omega(\v',K_X) \times S^2 X)$. 
\end{Lem}

\begin{proof}
We begin by showing that $g$ is injective
on the stable locus. 
Indeed, for a stable sheaf $E$ in the image of $g$, $Z$ is determined uniquely by $E^{\vee\vee}/E$ so that the exact sequence \eqref{eqn:nonlocally free extension}
is uniquely determined by $E$.

Now we prove that the set  
$$
M_\omega^{pss}(\v,K_X)=\Set{I_{z_1} \oplus I_{z_2}(K_X) \ |\ z_1,z_2 \in X}\subset M_\omega(\v,K_X)
$$
of properly semistable polystable
sheaves is contained in $M_1$.  First assume that $Z=\{z_1,z_2 \}$ $(z_1 \ne z_2)$.
If $E$ is properly semistable, then by \cref{Lem:criterion for extension to be stable}, the associated morphism $\OO_X(K_X) \to \OO_Z$ is not surjective, say at $z_1$, so $E$ fits into an exact sequence
$$
0 \to I_{z_2}(K_X) \to E \to I_{z_1} \to 0,
$$ 
giving the point $I_{z_2}(K_X) \oplus I_{z_1}$ of $M_\omega(\v,K_X)$.
Note, moreover, that it follows from $\ext^1(I_{z_1},I_{z_2}(K_X))=1$ that $g$ is injective if $z_1 \ne z_2$.

If $Z=\{z \}$ with a non-reduced structure
and $E$ is properly semistable, then again by \cref{Lem:criterion for extension to be stable}, the induced morphism $\OO_X(K_X) \to \OO_Z$ is not surjective and $E$ fits into an exact sequence $$0\to I_z(K_X)\to E\to I_z\to 0,$$ giving the point $I_z \oplus I_z(K_X)$.  Thus $M_\omega^{pss}(\v,K_X)\subset M_1$ as claimed.  Note further that the point $I_z(K_X)\oplus I_z$ of $M_\omega(\v,K_X)$ is independent of the choice of a scheme structure on $Z$, so we also see that $g^{-1}(I_z\oplus I_z(K_X))$ has positive dimension.

As $g$ is injective away from $q^{-1}(\Delta_X)$, where $\Delta_X\subset\Hilb^2(X)$ is the locus of non-reduced subschemes, it follows that $\dim(M_1)=5$.  As $M_1$ is proper, irreducible, and of maximal dimension, it is an irreducible component of $M_\omega(\v,K_X)$. 

Since $M_\omega^{pss}(\v,K_X)\subset M_1\subset\psi^{-1}(M_\omega(\v',K_X)\times S^2 X)$, to prove the second claim, we must show that $E \in M_\omega^s(\v,K_X)$ satisfying
$\psi(E) \in M_\omega(\v',K_X) \times S^2 X$ is in $M_1$.  But such an $E$ is a stable sheaf fitting into an exact sequence as in \eqref{eqn:alternative defining short exact sequence}.  It follows from \cref{Lem:alternative sequence is equivalent for stable E} that $E$ fits into the exact sequence
\eqref{eqn:nonlocally free extension}.
Therefore $E\in M_1$, as claimed.  
\end{proof}

We will show that $M_0 \cup M_1$ is a connected component of
$M_\omega(\v,K_X)$, but the first step is to show that $M_1$ does not meet any other components.  In the next lemma, we show that away from the locus of properly semistable sheaves S-equivalent to $I_z\oplus I_z(K_X)$, $M_1$ is normal and thus meets no other components.
\begin{Lem}\label{Lem:normal}
The local deformation space is smooth at $E$ if
$E$ is S-equivalent to $I_{z_1} \oplus I_{z_2} (K_X)$ ($z_1 \ne z_2$).
Hence $I_{z_1} \oplus I_{z_2} (K_X)$ ($z_1 \ne z_2$)
is contained in the normal open subscheme
$$
M_\omega(\v,K_X)^*:=\Set{ E \in M_\omega(\v,K_X) \ |\ \Ext^2(E,E)=0}
$$
of $M_\omega(\v,K_X)$.
\end{Lem}

\begin{proof}
We begin by showing that $M_\omega(\v,K_X)^*$ is normal.  Let $Q$ be an open subscheme of a suitable quot-scheme
$\Quot_{\OO_X(-mH)^{\oplus N}/X}$, which parameterizes quotients
$\OO_X(-mH)^{\oplus N} \to E$,
such that $M_\omega(\v,K_X)$ is a GIT-quotient of
$Q$ by $PGL(N)$.
Let $Q^*$ be the open subscheme of $Q$
such that $\Ext^2(E,E)=0$. Then $Q^*$ is smooth.
Hence the open subscheme $M_\omega(\v,K_X)^*=Q^*/PGL(N)$ of
$M_\omega(\v,K_X)$ is a normal scheme.  

Now suppose that $E$ is S-equivalent to $I_{z_1}\oplus I_{z_2}(K_X)$ with $z_1 \ne z_2$.  Then
$$
\Ext^2(I_{z_1},I_{z_1})=\Ext^2(I_{z_1}, I_{z_2} (K_X))=
\Ext^2(I_{z_2}(K_X),I_{z_1})=\Ext^2(I_{z_2}(K_X),I_{z_2}(K_X))=0.
$$
Hence $\Ext^2(E,E)=0$ so that $E\in Q^*$.  It follows that the polystable representative, $I_{z_1}\oplus I_{z_2}(K_X)$, is in $M_\omega(\v,K_X)^*$, as claimed.
\end{proof}
We can actually prove that $M_\omega(\v,K_X)^*$ is smooth:
\begin{Lem}\label{Lem:M^*} 
$\P \times_{M_\omega(\v,K_X)} M_\omega(\v,K_X)^* 
\to M_\omega(\v,K_X)^*$ is a closed immersion whose
image is a connected component of $M_\omega(\v,K_X)^*$.
In particular $M_\omega(\v,K_X)^* $ is smooth.
\end{Lem}

\begin{proof}
We first show that the image of $g':\P \times_{M_\omega(\v,K_X)} M_\omega(\v,K_X)^*  \to M_\omega(\v,K_X)^*$ a connected component of $M_\omega(\v,K_X)^*$.  First observe that $g'$ is proper since $g$ is a proper morphism.  We further note that the image of $g'$ is $M_1^*:=M_1 \cap M_\omega(\v,K_X)^*$,
which is therefore a connected component of $M_\omega(\v,K_X)^*$ by Lemma \ref{Lem:normal}.

Now let us show that $g'$ is a closed immersion.  If $E\in M^{pss}_\omega(\v,K_X)$ is $S$-equivalent to $I_{z} \oplus I_z(K_X)$, then
we see that $\Ext^2(E,E) \ne 0$. Hence $E$ is not contained in 
$M_\omega(\v,K_X)^*$.
Therefore $g'$ is injective, which implies that
$g'$ is a finite birational map between normal schemes.
Therefore, $g'$ is an isomorphism onto its image, proving the claim.

As $M_\omega^{pss}(\v,K_X)\subset M_1$, the open subscheme $M_\omega(\v,K_X)^*\backslash M_1^*$ is smooth by \cite{Kim98} and the definition of $M_\omega(\v,K_X)^*$.  Since $M_1^*$ is isomorphic to an open subscheme of a $\P^1$ bundle over $\Hilb^2(X)$, it is smooth, which gives the final statement of lemma.
\end{proof}

Finally, we put everything together to show that $M_0\cup M_1$ is a connected component of $M_\omega(\v,K_X)$.
\begin{Prop}\label{prop:connected}
\begin{enumerate}
\item
The singular locus of $M_\omega(\v,K_X)$ consists of a 4-dimensional subscheme
$$
S_1:=\{ \varpi_*(I_W) \mid I_W \in \Hilb^2(\widetilde{X}) \} 
$$
and a finite set of points 
$$
S_2:=\{\varpi_*({\cal O}_{\widetilde{X}}(D)) \mid \iota(D)=-D, (D^2)=-4\}.
$$
\item
$M_0$ and $M_1$
intersect along $S_1$.
\item
$M_0 \cup M_1$ is a connected component of $M_\omega(\v,K_X)$.
\end{enumerate}
\end{Prop}

\begin{proof}
(1) follows from Lemma \ref{Lem:ext^2} and Lemma \ref{Lem:M^*}. 
By Lemma \ref{Lem:M_0} and Lemma \ref{Lem:M_1},
$M_0$ intersects $M_1$ along the 4-dimensional singular locus.
By the description of the singular locus,
$M_0 \cup M_1$ is a connected component of
$M_\omega(\v,K_X)$.
\end{proof}

\begin{Rem}
The restriction of $\psi$ to $S_1=M_0 \cap M_1$ is a double cover onto $M_\omega(\v',K_X) \times S^2 X$.  Indeed, for $\varpi_*(I_{w_1,w_2})$ $(w_1 \ne w_2,\iota(w_2))$, we push forward the short exact sequence on $\widetilde{X}$ 
$$0\to I_{w_1,w_2}\to\OO_{\widetilde{X}}\to\OO_{w_1,w_2}\to 0$$ to obtain the short exact sequence 
$$0\to\varpi_*(I_{w_1,w_2})\to\OO_X\oplus\OO_X(K_X)\to\OO_{\varpi(w_1),\varpi(w_2)}\to 0.$$
From this it is clear that $\psi^{-1}(\psi(\varpi_*(I_{w_1,w_2})))=\{\varpi_*(I_{w_1,w_2}),\varpi_*(I_{w_1,\iota(w_2)})\}$.
\end{Rem}

\subsubsection{A remark on $M_\omega(\v,0)$ $(\v=(2,0,-1))$}
We discuss here the structure of the other component of $M_\omega(\v)$ parametrizing sheaves with determinant $\OO_X$ vis-a-vis the Uhlenbeck contraction.
\begin{Lem}
If $E \in M_\omega(\v,0)$ is not a $\mu$-stable locally free sheaf, then 
$E$ is S-equivalent to $I_{z_1} \oplus I_{z_2}$ or
$I_{z_1}(K_X) \oplus I_{z_2}(K_X)$.  
\end{Lem}

\begin{proof}
If $E$ is not locally free, then we see that
$\chi(E^{\vee \vee})>0$.
Hence $H^0(E^{\vee \vee}) \ne 0$ or
$H^0(E^{\vee}(K_X)) \ne 0$.  It follows from considerations of $\mu$-stability that $E^{\vee\vee}$ contains $\OO_X(D)$, where $D=0$ or $K_X$, respectively, as a saturated subsheaf.  Intersecting with $E$ gives an exact sequence
$$
0 \to I_Z(D) \to E \to I_W(D) \to 0
$$
with $D=0,K_X$.
By the semi-stability of $E$, $\deg(Z)\geq 1$, $\deg(W)\leq 1$, and $\deg(Z)+\deg(W)=2$.  If $\deg(Z)=\deg(W)=1$, then $E$ is S-equivalent to $I_Z(D)\oplus I_W(D)$, as claimed.  If $\deg W=0$, then
$E^{\vee \vee}=\OO_X(D)^{\oplus 2}$ with $E^{\vee\vee}/E\cong\OO_Z$ where now $Z\in\Hilb^2(X)$.  Choosing a torsion free sheaf $F$ with
$E \subset F \subset E^{\vee \vee}$ and $F/E=\OO_{z_1}$, $F$ is $\mu$-semistable and $\chi(F)=1$, so the same argument as for $E^{\vee\vee}$ shows that 
$\Hom(\OO_X(D),F) \ne 0$.  Slope-stability considerations force $\OO_X(D)$ to be a saturated subsheaf, and intersecting with $E$ gives another short exact sequence $$0\to I_{Z'}(D)\to E\to I_{W'}(D)\to 0$$ with $\deg(Z')\geq 1$, $\deg(W')\leq 1$, and $\deg(Z')+\deg(W')=2$.  This time, however, we see that $$0\neq\OO_X(D)/I_{Z'}(D)=(\OO_X(D)+E)/E\subset F/E=\OO_{z_1},$$ so $Z'=z_1$ and $W'=z_2$ for points $z_1,z_2\in X$.
Therefore $E$ is S-equivalent to $I_{z_1}(D)\oplus I_{z_2}(D)$ in this case as well.
\end{proof}

Thus the complement of the open subset of $\mu$-stable locally free sheaves
in $M_\omega(\v,0)$ is parameterized by
two copies of $S^2 X$, and this locus is not contracted by the
morphism to the Uhlenbeck compactification.

\subsection{Irreducible components of $M_\sigma(\v)$
with $\v^2=2$}
In this subsection, we shall prove the following claim. 
\begin{Prop}\label{prop:irred-comp:v^2=2}
Let $\v:=(r,\xi,a)$ be a Mukai vector such that $\v^2=2$.
Assume that $L \equiv D+\frac{r}{2}K_X \pmod 2$, where $D$ is a nodal cycle.
Then $M_\omega(\v,L)$ has two irreducible components.
\end{Prop}
As in the previous subsection, we begin by reducing to studying a certain moduli space of Gieseker semistable sheaves of rank two.  To do so, we make use of a motivic invariant called the virtual Hodge polynomial.  Recall that for a variety $Y$, the virtual Hodge polynomial $e(Y)$ is defined by $e(Y)=\sum_{p,q}e^{p,q}(Y)x^py^q$, where $e^{p,q}(Y):=\sum_k (-1)^kh^{p,q}(H^k_c(Y))$ are the virtual Hodge numbers of the natural mixed Hodge structure on the cohomology of $Y$ with compact support (see \cite{Yos14} and the references therein for a more complete explanation).
\begin{Lem}\label{Lem:e-poly}
Let $\v:=(r,\xi,a)$ be a Mukai vector such that $\v^2$ is even.
\begin{enumerate}
\item
There is $\v'=(2,L',a')$ such that
$e(M_\omega(\v,L))=e(M_\omega(\v',L'))$,
where $L'+K_X \equiv L+\frac{r}{2}K_X \pmod 2$.
\item
For a smooth rational curve $C$,
we have an equality
\begin{equation}\label{eq:(-2)}
e(M_\omega(\v,L))= e(M_{\omega'}(\v',L')),
\end{equation}
where $\sigma'=\Phi(\sigma)$ for the spherical twist $\Phi=R_{\OO_C(-1)}$, 
$\v'=\Phi(\v)=(r,\xi+(C,\xi)C,a)$ and $L'=L+(L,C)C$.
\end{enumerate}
\end{Lem}

\begin{proof}

(1)
Since the $(-1)$-reflection $R_{\OO_X}$ associated to
$\OO_X$ preserves $\Stab^\dagger(X)$, 
we have an isomorphism
\begin{equation}
M_\sigma(\v,L))\cong M_{\sigma'}(\v',L'),
\end{equation}
where $\sigma'=R_{\OO_X}(\sigma)$, 
$\v'=R_{\OO_X}(\v)=(-2a,\xi,-r/2)$ 
and $L'=L+\tfrac{r+2a}{2}K_X$.
By \cite{Nue14a},
we have
\begin{equation}
e(M_\omega(\v,L))= e(M_{\omega'}(\v',L')).
\end{equation}
By similar arguments as in \cite{Nue14a} or \cite{Yos16a},
we get (1).

(2)
Let $\Phi:\Db(X) \to \Db(X)$ be the twist functor
associated to the spherical object $\OO_C(-1)$.
The action of $\Phi$ preserves
$\Stab^\dagger(X)$, so
we have an isomorphism
\begin{equation}
M_\sigma(\v,L) \cong M_{\sigma'}(\v',L)
\end{equation}
where $\sigma'=\Phi(\sigma)$, 
$\v'=(r,\xi+(C,\xi)C,a)$ and $L'=L+(L,C)C$.
By using \cite{Nue14a} again,
we get \eqref{eq:(-2)}.
\end{proof}
In the specific case that we are interested in, we can in fact get more explicit:

\begin{Lem}\label{Lem:rank2}
Let $\v=(r,\xi,a)$ be a Mukai vector with
$\v^2=2$ and $L$ a divisor such that $L \equiv D+\frac{r}{2}K_X \pmod 2$,
where $D$
is a nodal cycle.
Then there is a Mukai vector $\v'=(2,\xi',0)$
and an elliptic fibration $\pi:X \to \P^1$
such that 
\begin{enumerate}
\item
$e(M_\omega(\v,L))=e(M_{\omega'}(\v',L'))$,
\item
$(\xi',C)=2$ for a general fiber $C$ of $\pi$
and $L' \equiv D' +K_X \pmod 2$,
where $D'$ is a nodal cycle.
\end{enumerate}
\end{Lem}

\begin{proof}
For a primitive Mukai vector $\v$ such that $\v^2=2$,
we take $\v'=(2,\xi',a')$ and $L'$ satisfying
Lemma \ref{Lem:e-poly}.
As $\ell(\v)=1$, we can find an isotropic divisor
$\eta$ such that
$(\xi'+2\lambda,\eta)=1$, where $\lambda \in \NS(X)$.   
Replacing $\v'$ by $\v' e^{k \eta+\lambda}$, for an appropriate choice of $k$, we may assume that 
$({\xi'}^2)=2$ and $a'=0$.
Then $\xi'$ is effective or $-\xi'$ is effective.
Since $\v' e^{-\xi'}=(2,-\xi',a')$, we may assume that $\xi'$ is effective.
Since $L \equiv D+\frac{r}{2}K_X \pmod 2$ for a nodal cycle $D$, 
by using Lemma \ref{Lem:nodal} and \eqref{eq:(-2)}, 
we may assume that $\xi'$ is nef and
$L' \equiv D'+K_X \pmod 2$ for a nodal cycle $D'$.

As $(\xi'^2)=2$, $\xi'$ is primitive, so we may find an isotropic divisor $\eta$ with $(\xi',\eta)=1$.  The Riemann-Roch theorem then implies that
$\eta$ is effective.
Let $\eta=f+\sum_i C_i$ be a decomposition of $\eta$ such that 
$f$ is a nef and isotropic divisor, and $C_i$ are smooth rational curves.
Then $(\xi',f), (\xi',C_i) \geq 0$ imply that
$(f,\xi')=1$ by using Hodge index theorem.
The linear system
$|2f|$ induces an elliptic fibration $\pi:X \to \P^1$
satisfying our requirements. 
\end{proof}

\begin{Lem}\label{Lem:nodal}
For nodal cycles $D$ and $D'$, 
 $D+(D,D')D' \equiv D'' \pmod 2$, where $D''$ is another nodal cycle. 
\end{Lem}

\begin{proof}
For nodal cycles $D$ and $D'$ and an ample divisor $H$,
let $E$ be a $H$-stable vector bundle with $\v=\v(E)=(2,D+K_X,0)$
and
$E'$ a $H$-stable vector bundle with $\v'=\v(E')=(2,D'+K_X,0)$.
Then for $E(nH)$ $(n \gg 0)$,
$$
F:=\ker (\Hom(E',E(nH)) \otimes E' \to E(nH))
$$
is a stable vector bundle
with $\v(F)=-e^{nH}(2,D+K_X,0)-\langle e^{nH}\v,\v' \rangle \v'$ (\cite[Thm. 1.7]{Yos09}).  Thus $\v(F)^2=-2$.  Since 
\begin{equation}
\langle e^{nH}\v,\v' \rangle=(D,D')-2n^2(H^2)+2n(H,D'-D),
\end{equation}
we get
\begin{equation}
\begin{split}
\frac{\rk F}{2} \equiv & 1+(D,D')  \pmod 2,\\
c_1(F) \equiv & D+K_X+(D,D')(D'+K_X) \pmod 2\\
\equiv &
D+(D,D')D'+\frac{\rk F}{2}K_X \pmod 2.
\end{split}
\end{equation}
As $F$ is thus a stable spherical bundle, it follows that there is a nodal cycle $D''$ such that $$D''+\frac{\rk F}{2}K_X\equiv c_1(F)\equiv D+(D,D')D'+\frac{\rk F}{2}K_X \pmod 2,$$ and the result follows.
\end{proof}

By Lemma \ref{Lem:rank2}, it is sufficient to 
describe the irreducible components of $M_\omega(\v,L)$
for $\v=(2,\xi,0)$ with $(\xi^2)=2$ and $L$ such that 
$L \equiv D +K_X \pmod 2$ and $(L,C)=2$
for a general fiber $C$ of an elliptic fibration
$\pi:X \to \P^1$.
For this purpose, we first describe a 2-dimensional component
of the singular locus of $M_\omega(\v,L)$.
   
\begin{Lem}\label{ext^2-2}
For a Mukai vector $\v=(2,\xi,0)$ with $(\xi^2)=2$ and a divisor $L$
with $L \equiv D+K_X \pmod 2$, 
the singular locus of $M_\omega(\v,L)$ is
\begin{equation}
\Set{ \varpi_*(I_W(\widetilde{L})) 
\ |\ I_W \in \Hilb^n(\widetilde{X}),\; 
\iota^*(\widetilde{L})=\varpi^*(L)-\widetilde{L},\;
n=(\widetilde{L}^2)/2+1},
\end{equation}
where $(\widetilde{L}^2)=0,-2$.
In particular, the 2-dimensional component of the singular locus is
irreducible.
\end{Lem}

\begin{proof} 
Since $\omega$ is generic, $M_\omega(\v,L)=M_\omega^s(\v,L)$ is of expected dimension, so the singular locus is
$$
\Set{ E \in M_\omega(\v,L) \ |\ \Ext^2(E,E) \ne 0}=\Set{E\in M_\omega(\v,L)\ | \ E\cong E(K_X)}.
$$ 
If $E(K_X) \cong E$, then
we see that $E \cong \varpi_*(I_W(\widetilde{L}))$, 
where $\iota^*(I_W(\widetilde{L}))\ncong I_W(\widetilde{L}))$ and $W=\varnothing$ or $W=\{w \}$.
If $W=\{w \}$, then 
$E^{\vee \vee}=\varpi_*(\OO_{\widetilde{X}}(\widetilde{L}))$
is a spherical vector bundle with $\v(E^{\vee \vee})=(2,\xi,1)$.
Therefore the claim holds.
\end{proof}

\begin{Rem}
For $x \in X$, we set $\varpi^{-1}(x)=\{z,\iota(z)\}$.
Then $\varpi_*(I_z(\widetilde{L}))$ and 
$\varpi_*(I_{\iota(z)}(\widetilde{L}))$
are not locally free at $x$.
Hence the 2-dimensional component of the singular locus
is a double covering of $X$. 
\end{Rem}

The existence of a two dimensional component of the singular locus has the following important consequence for the reducibility of $M_\omega(\v,L)$.
\begin{Lem}\label{Lem:irred-comp-2}
There are at most two irreducible components of
$M_\omega(\v,L)$.
\end{Lem}

\begin{proof}
If there are two irreducible components of $M_\omega(\v,L)$, then
the connectedness of $M_\omega(v,L)$ implies they intersect along 
the 2-dimensional component of the singular locus. 
By \cite{Yamada},
the analytic germ of $M_\omega(\v,L)$ at a singular point $E$ is described as a hypersurface
$F(x_1,x_2,x_3)=0$ in $(\C^3,0)$ with a non-trivial quadratic term. 
If $M_\omega(\v,L)$ is reducible, then each irreducible component is defined
by a factor of $F(x_1,x_2,x_3)$.
Therefore $M_\omega(\v,L)$ has at most two irreducible components.
\end{proof}

Let $f$ be the reduced part of a multiple fiber of $\pi$ and
$E_0 \in M_\omega(\v',L-2f)$ be a spherical vector bundle,
where $\v'=(2,\xi-2f,0)$.
We shall prove Proposition \ref{prop:irred-comp:v^2=2}
by constructing two irreducible components $M_0$ and $M_1$
of $M_\omega(\v,L)$.  We begin by constructing a component $M_0$ containing locally free sheaves using elementary transformations as in \cite[Section 5.2]{HL10}.  To this end we must study how $E_0$ restricts to a general fiber of $\pi$.  
\begin{Lem}\label{Lem:splitting type}
Let $C$ be a general fiber of the elliptic fibration.
Then ${E_0}|_C \cong \OO_C(p) \oplus \OO_C(q)$ $(p \ne q)$.
\end{Lem}

\begin{proof}
For the primitive and isotropic Mukai vector $\u:=(0,C,1)$,   $\ell(\u)=2$ so that $M_\omega(\u,C)$ is a fine moduli space
which is isomorphic to $X$, where $\omega$ is general.
As in \cite{Bri98}, there is an elliptic fibration $M_\omega(\u,C) \to \P^1$ that comes from regarding $M_\omega(\u,C)$ as a smooth compactification of the relative Picard scheme $\Pic^1(X/\P^1) \to \P^1$ of degree 1.
We shall identify $M_\omega(\u,C) \to \P^1$ with
the elliptic fibration $X \to \P^1$.

For a universal family $\EE$ of $M_\omega(\u,C)$, we have a relative Fourier-Mukai transform
$\Phi_{X \to X}^{\EE^{\vee}}: \Db(X) \to \Db(X)$.
Then $F:=\Phi_{X \to X}^{\EE^{\vee}}(E_0)[1]$ is a purely 1-dimensional sheaf
whose support is a double cover of $\P^1$.
Thus for a general fiber $C$ with $F|_C=\C_p \oplus \C_q$, one can check that ${E_0}|_C \cong \OO_C(p) \oplus \OO_C(q)$.
\end{proof}

Let $L$ be a line bundle on a smooth fiber $C \in |2f|$ with $\v(L)=(0,2f,0)$.
We set $L^*:=\EE xt^1_{\OO_X}(L,\OO_X)$.
Then $L^*$ is a line bundle of degree 0 on $C$ so that $\v(L^*)=(0,2f,0)$.  Observe that $$\langle\v(E_0^\vee),\v(L^*)\rangle=\langle(2,2f-\xi,0),(0,2f,0)\rangle=-2<0,$$ so there is necessarily a non-zero homomorphism $\psi:E_0^\vee\to L^*$. 
Moreover we may assume that
$\psi$ is surjective by Lemma \ref{Lem:splitting type}.
We will see in the next two results that these give rise to a component of $M_\omega(\v,L)$ containing $\mu$-stable locally free sheaves.  As $E_0$ is $\mu$-stable with respect to any polarization \cite{Kim94,Qin91}, for any given starting polarization $H_0$, $\{E_0\}= M_{H_0}(\v',L-2f)$.  We will have cause to vary the polarization in the next lemma, but it is important to state that the stability of $E_0$ remains unchanged.
\begin{Lem}\label{Lem:M_1 v^2=2}
For a non-zero homomorphism
$\psi:E_0^{\vee} \to L^*$, 
we set $E:=\RlHom_{\OO_X}(\Cone(\psi),\OO_X)[1]$.
Then $E$ is a $\mu$-stable torsion free sheaf with respect to
$H_0+nf$, $n \gg 0$.
If $\psi$ is surjective, then $E$ is a locally free sheaf.
\end{Lem}

\begin{proof}
We first show that $E$ is torsion free, and locally free if $\psi$ is surjective.  Dualizing the exact triangle $$E_0^\vee\mor[\psi]L^*\to\Cone(\psi)\to E_0^\vee[1],$$ shifting by 1, and taking cohomology sheaves, we immediately see that $E$ is a coherent sheaf fitting into an exact sequence
\begin{equation}
0 \to E_0 \to E \to L \to 0.
\end{equation} 
Moreover, as $\RlHom(E,\OO_X)=\Cone(\psi)[-1]$, we get that $\lExt^1(E,\OO_X)=\HH^{0}(\Cone(\psi))=\coker(\psi)$, which is supporded in codimension two, and $\lExt^2(E,\OO_X)=0$, so
$E$ is torsion free by \cite[Proposition 1.1.10]{HL10}.  Moreover, if $\psi$ is surjective, so that $\coker(\psi)=0$, then $E$ is reflexive (by \cite[\textit{ibid.}]{HL10}) and thus locally free.

Now we show that $E$ is $\mu$-stable.  Let $F$ be a subsheaf of $E$ with $\rk F=1$.
Then $E_0 \cap F$ is a  rank 1 subsheaf of $E_0$.
Since $E_0$ is $\mu$-stable for any ample divisor,
$$
2(c_1(E_0 \cap F),H_0+nf)<(c_1(E_0),H_0+nf)\;\;(n \geq 0).
$$
Hence $2(c_1(E_0 \cap F),f) \leq (c_1(E_0),f)=1$ and
$2(c_1(E_0 \cap F),H_0)<(c_1(E_0),H_0)$.
In particular
$(c_1(E_0 \cap F),f) \leq 0$.
As $c_1(F)=c_1(E_0\cap F)+2f$, $(c_1(F),f) \leq 0$.
If $n>(4f,H_0)$, then $2(c_1(F),H_0+nf)<(c_1(E),H_0+nf)$.
Therefore $E$ is $\mu$-stable with respect to $H_0+nf$ for $n\gg 0$, as claimed.
\end{proof}
Considering the irreducible component containing the sheaves $E$ constructed in \cref{Lem:M_1 v^2=2}, we get the following result.
\begin{Cor}
For $\omega=H_0+n f$ ($n \gg 0$),
there is an irreducible component $M_0$ of $M_\omega(\v,L)$ 
which contains a $\mu$-stable locally free sheaf.
\end{Cor}

The second irreducible component, $M_1$, parametrizes $E$ fitting into an exact sequence
\begin{equation}
0 \to E \to E_0(f) \to \C_x \to 0, \; x \in X.
\end{equation}
In particular, such $E$ are not locally free, so $M_1$ is indeed distinct from $M_0$.  Moreover, as in the previous subsection, $M_1$ has the structure of a $\P^1$-bundle, this time over $X$ (instead of $\Hilb^2(X)$).  Moreover, the fibers of this $\P^1$ bundle are contracted by the Uhlenbeck morphism.

By Lemma \ref{Lem:irred-comp-2}, $M_0$ and $M_1$ are the irreducible components
of $M_\omega(\v,L)$, which shows
Proposition \ref{prop:irred-comp:v^2=2} by Lemma \ref{Lem:rank2}.

\bibliographystyle{plain}
\bibliography{NSF_Research_Proposal}

\end{document}